\documentclass[12pt]{article}

\usepackage{amsmath, amsthm, amssymb}
\usepackage{enumerate}
\usepackage{pdflscape}
\usepackage{caption}

\usepackage{ifpdf}
\ifpdf
\usepackage[pdftex]{graphicx}
\else
\usepackage[dvips]{graphicx}
\fi
\usepackage{tikz}
 	 \usetikzlibrary{arrows,backgrounds}
\usepackage[all]{xy}

\usepackage{multicol}

\usepackage{tocvsec2}

\input xy
\xyoption{all}

\usepackage[pdftex,plainpages=false,hypertexnames=false,pdfpagelabels]{hyperref}
\newcommand{\arxiv}[1]{\href{http://arxiv.org/abs/#1}{\tt arXiv:\nolinkurl{#1}}}
\newcommand{\arXiv}[1]{\href{http://arxiv.org/abs/#1}{\tt arXiv:\nolinkurl{#1}}}
\newcommand{\doi}[1]{\href{http://dx.doi.org/#1}{{\tt DOI:#1}}}

\newcommand{\googlebooks}[1]{(preview at \href{http://books.google.com/books?id=#1}{google books})}

\usepackage{xcolor}
\definecolor{dark-red}{rgb}{0.7,0.25,0.25}
\definecolor{dark-blue}{rgb}{0.15,0.15,0.55}
\definecolor{medium-blue}{rgb}{0,0,.8}
\definecolor{DarkGreen}{RGB}{0,150,0}
\definecolor{plum}{RGB}{153,51,255}
\definecolor{salmon}{RGB}{250,178,102}
\definecolor{wString}{named}{orange}
\definecolor{xString}{named}{red}
\definecolor{yString}{named}{blue}
\definecolor{zString}{named}{DarkGreen}
\definecolor{AString}{named}{red}
\definecolor{BString}{named}{blue}
\definecolor{dString}{named}{orange}
\definecolor{cString}{named}{DarkGreen}

\hypersetup{
   colorlinks, linkcolor={purple},
   citecolor={medium-blue}, urlcolor={medium-blue}
}

\usepackage{longtable}
\usepackage{fullpage}

\theoremstyle{plain}
\newtheorem{thm}{Theorem}[section]
\newtheorem*{thm*}{Theorem}
\newtheorem{thmalpha}{Theorem}

\newtheorem{cor}[thm]{Corollary}

\newtheorem*{cor*}{Corollary}

\newtheorem*{conj*}{Conjecture}
\newtheorem{lem}[thm]{Lemma}

\newtheorem{prop}[thm]{Proposition}

\newtheorem*{quest*}{Question}
\newtheorem*{claim*}{Claim}

\theoremstyle{definition}
\newtheorem{defn}[thm]{Definition}

\newtheorem{nota}[thm]{Notation}

\newtheorem{ex}[thm]{Example}
\newtheorem{sub-ex}[thm]{Sub-Example}
\newtheorem{rem}[thm]{Remark}
\newtheorem*{rem*}{Remark}

\DeclareMathOperator{\coev}{coev}

\DeclareMathOperator{\ev}{ev}
\DeclareMathOperator{\Hom}{Hom}
\renewcommand{\hom}{\Hom}

\DeclareMathOperator{\id}{id}

\DeclareMathOperator{\Tr}{Tr}
\DeclareMathOperator{\tr}{tr}

\newcommand{\comment}[1]{}

\newcommand{\be}{\begin{enumerate}[label=(\arabic*)]}
\newcommand{\ee}{\end{enumerate}}

\newcommand{\Z}{\mathbb{Z}}

\newcommand{\ssZ}{{\scriptscriptstyle \cZ}}

\def\semicolon{;}
\def\applytolist#1{
    \expandafter\def\csname multi#1\endcsname##1{
        \def\multiack{##1}\ifx\multiack\semicolon
            \def\next{\relax}
        \else
            \csname #1\endcsname{##1}
            \def\next{\csname multi#1\endcsname}
        \fi
        \next}
    \csname multi#1\endcsname}

\def\calc#1{\expandafter\def\csname c#1\endcsname{{\mathcal #1}}}
\applytolist{calc}QWERTYUIOPLKJHGFDSAZXCVBNM;
\def\bbc#1{\expandafter\def\csname bb#1\endcsname{{\mathbb #1}}}
\applytolist{bbc}QWERTYUIOPLKJHGFDSAZXCVBNM;
\def\bfc#1{\expandafter\def\csname bf#1\endcsname{{\mathbf #1}}}
\applytolist{bfc}QWERTYUIOPLKJHGFDSAZXCVBNM;
\def\sfc#1{\expandafter\def\csname s#1\endcsname{{\sf #1}}}
\applytolist{sfc}QWERTYUIOPLKJHGFDSAZXCVBNM;

\newcommand{\Mod}{{\sf Mod}}
\newcommand{\Bimod}{{\sf Bimod}}
\renewcommand{\Vec}{{\sf Vec}}

\newcommand{\noshow}[1]{}
\newcommand{\MR}[1]{}

\usepackage[utf8]{inputenc}
\DeclareUnicodeCharacter{2693}{\anchor}
\usepackage{dingbat}

\usetikzlibrary{shapes}
\usetikzlibrary{backgrounds}
\usetikzlibrary{decorations,decorations.pathreplacing,decorations.markings}
\usetikzlibrary{fit,calc,through}
\usetikzlibrary{external}
\tikzset{
	super thick/.style={line width=3pt},
	more thick/.style={line width=1pt},
}
\tikzstyle{shaded}=[fill=red!10!blue!20!gray!30!white]
\tikzstyle{unshaded}=[fill=white]
\tikzstyle{empty box}=[circle, draw, thick, fill=white, opaque, inner sep=2mm]
\tikzstyle{annular}=[scale=.7, inner sep=1mm, baseline]
\tikzstyle{rectangular}=[scale=.75, inner sep=1mm, baseline=-.1cm]
\tikzstyle{mid>}=[decoration={markings, mark=at position 0.5 with {\arrow{>}}}, postaction={decorate}]
\tikzstyle{mid<}=[decoration={markings, mark=at position 0.5 with {\arrow{<}}}, postaction={decorate}]
\tikzstyle{over}=[double, draw=white, super thick, double=]
\tikzstyle{end>}=[decoration={markings, mark=at position 1 with {\arrow{>}}}, postaction={decorate}]

\newcommand{\roundNbox}[6]{
	\draw[rounded corners=5pt, very thick, #1] ($#2+(-#3,-#3)+(-#4,0)$) rectangle ($#2+(#3,#3)+(#5,0)$);
	\coordinate (ZZa) at ($#2+(-#4,0)$);
	\coordinate (ZZb) at ($#2+(#5,0)$);
	\node at ($1/2*(ZZa)+1/2*(ZZb)$) {#6};
}

\newcommand{\roundNboxSize}[7]{
	\draw[rounded corners=5pt, very thick, #1] ($#2+(-#3,-#3)+(-#4,0)$) rectangle ($#2+(#3,#3)+(#5,0)$);
	\coordinate (ZZa) at ($#2+(-#4,0)$);
	\coordinate (ZZb) at ($#2+(#5,0)$);
	\node[scale=#7] at ($1/2*(ZZa)+1/2*(ZZb)$) {#6};
}

\newcommand{\nbox}[6]{
	\draw[thick, #1] ($#2+(-#3,-#3)+(-#4,0)$) rectangle ($#2+(#3,#3)+(#5,0)$);
	\coordinate (ZZa) at ($#2+(-#4,0)$);
	\coordinate (ZZb) at ($#2+(#5,0)$);
	\node at ($1/2*(ZZa)+1/2*(ZZb)$) {#6};
}

\newcommand{\loopIso}[1]{
	\fill[unshaded] ($ #1 - (.3,.3) $) rectangle ($ #1 + (.1,.3) $);
	\draw ($ #1 + (-.3,.2) $) arc (90:270:.2cm);
	\draw ($ #1 + (0,-.3) $)  .. controls ++(90:.2cm) and ++(0:.2cm) .. ($ #1 + (-.3,.2) $);
	\draw[super thick, white] ($ #1 + (0,.3) $)  .. controls ++(270:.2cm) and ++(0:.2cm) .. ($ #1 + (-.3,-.2) $);
	\draw ($ #1 + (0,.3) $)  .. controls ++(270:.2cm) and ++(0:.2cm) .. ($ #1 + (-.3,-.2) $);
}

\newcommand{\loopIsoReverse}[1]{
	\fill[unshaded] ($ #1 - (.3,.3) $) rectangle ($ #1 + (.1,.3) $);
	\draw ($ #1 + (-.3,.2) $) arc (90:270:.2cm);
	\draw ($ #1 + (0,.3) $)  .. controls ++(270:.2cm) and ++(0:.2cm) .. ($ #1 + (-.3,-.2) $);
	\draw[super thick, white]	 ($ #1 + (0,-.3) $)  .. controls ++(90:.2cm) and ++(0:.2cm) .. ($ #1 + (-.3,.2) $);
	\draw ($ #1 + (0,-.3) $)  .. controls ++(90:.2cm) and ++(0:.2cm) .. ($ #1 + (-.3,.2) $);
}

\newcommand{\loopIsoInverse}[1]{
	\fill[unshaded] ($ #1 - (.1,.3) $) rectangle ($ #1 + (.3,.3) $);
	\draw ($ #1 + (.3,.2) $) arc (90:-90:.2cm);
	\draw ($ #1 + (0,-.3) $)  .. controls ++(90:.2cm) and ++(180:.2cm) .. ($ #1 + (.3,.2) $);
	\draw[super thick, white] ($ #1 + (0,.3) $)  .. controls ++(270:.2cm) and ++(180:.2cm) .. ($ #1 + (.3,-.2) $);
	\draw ($ #1 + (0,.3) $)  .. controls ++(270:.2cm) and ++(180:.2cm) .. ($ #1 + (.3,-.2) $);
}

\newcommand{\loopIsoInverseReverse}[1]{
	\fill[unshaded] ($ #1 - (.1,.3) $) rectangle ($ #1 + (.3,.3) $);
	\draw ($ #1 + (.3,.2) $) arc (90:-90:.2cm);
	\draw ($ #1 + (0,.3) $)  .. controls ++(270:.2cm) and ++(180:.2cm) .. ($ #1 + (.3,-.2) $);
	\draw[super thick, white] ($ #1 + (0,-.3) $)  .. controls ++(90:.2cm) and ++(180:.2cm) .. ($ #1 + (.3,.2) $);
	\draw ($ #1 + (0,-.3) $)  .. controls ++(90:.2cm) and ++(180:.2cm) .. ($ #1 + (.3,.2) $);
}

\newcommand{\braiding}[3]{
	\fill[unshaded] ($ #1 - (.1,#3) $) rectangle ($ #1 + (.1,#3) + (#2,0) $);
	\draw ($ #1 + (#2,-#3) $) to[out=90, in=-90] ($ #1 + (0,#3) $);
	\draw[white, line width=4] ($ #1 + (0,-#3) $) to[out=90, in=-90] ($ #1 + (#2,#3) $);
	\draw ($ #1 + (0,-#3) $) to[out=90, in=-90] ($ #1 + (#2,#3) $);
}

\newcommand{\braidingInverse}[3]{
	\fill[unshaded] ($ #1 - (.1,#3) $) rectangle ($ #1 + (.1,#3) + (#2,0) $);
	\draw ($ #1 + (0,-#3) $) to[out=90, in=-90] ($ #1 + (#2,#3) $);
	\draw[white, line width=4] ($ #1 + (#2,-#3) $) to[out=90, in=-90] ($ #1 + (0,#3) $);
	\draw ($ #1 + (#2,-#3) $) to[out=90, in=-90] ($ #1 + (0,#3) $);
}

\newcommand{\Mbox}[4]{
	\pgfmathsetmacro{\planeWidth}{#2};
	\pgfmathsetmacro{\planeDepth}{#3};
	
	\draw[very thick, unshaded] ($ #1 + (-\planeDepth,\planeDepth) $) -- #1 -- ($ #1 + (\planeWidth,0) $) -- ($ #1 + (\planeWidth,0) + (-\planeDepth,\planeDepth) $) -- ($ #1 + (-\planeDepth,\planeDepth) $);

	\node at ($#1+1/2*(#2,0)+1/2*(-#3,#3)$) {\rotatebox{-75}{#4}};
}

\newcommand{\CMbox}[6]{
	\coordinate (#1) at #2;
	\pgfmathsetmacro{\boxWidth}{#3};
	\pgfmathsetmacro{\boxHeight}{#4};
	\pgfmathsetmacro{\boxDepth}{#5};
	\draw[unshaded, very thick] ($(#1) + (\boxWidth,\boxHeight) $) -- ($(#1) + (\boxWidth,\boxHeight) - (\boxDepth,-\boxDepth) $) -- ($(#1) + (0,\boxHeight) - (\boxDepth,-\boxDepth) $) -- ($(#1) - (\boxDepth,-\boxDepth) $);
	\draw[unshaded, very thick] ($(#1) + (0,\boxHeight) $)  -- ($(#1) + (\boxWidth,\boxHeight) $) -- ($(#1) + (\boxWidth,0) $) -- (#1) -- ($(#1) - (\boxDepth,-\boxDepth) $);
	\draw[very thick] ($(#1) + (0,\boxHeight) - (\boxDepth,-\boxDepth) $) -- ($(#1) + (0,\boxHeight) $) -- (#1);
	\node at ($(#1) + 1/2*(\boxWidth,\boxHeight) $) {#6}; 
}
\newcommand{\CMboxSize}[7]{
	\coordinate (#1) at #2;
	\pgfmathsetmacro{\boxWidth}{#3};
	\pgfmathsetmacro{\boxHeight}{#4};
	\pgfmathsetmacro{\boxDepth}{#5};
	\draw[unshaded, very thick] ($(#1) + (\boxWidth,\boxHeight) $) -- ($(#1) + (\boxWidth,\boxHeight) - (\boxDepth,-\boxDepth) $) -- ($(#1) + (0,\boxHeight) - (\boxDepth,-\boxDepth) $) -- ($(#1) - (\boxDepth,-\boxDepth) $);
	\draw[unshaded, very thick] ($(#1) + (0,\boxHeight) $)  -- ($(#1) + (\boxWidth,\boxHeight) $) -- ($(#1) + (\boxWidth,0) $) -- (#1) -- ($(#1) - (\boxDepth,-\boxDepth) $);
	\draw[very thick] ($(#1) + (0,\boxHeight) - (\boxDepth,-\boxDepth) $) -- ($(#1) + (0,\boxHeight) $) -- (#1);
	\node[scale=#7] at ($(#1) + 1/2*(\boxWidth,\boxHeight) $) {#6}; 
}

\newcommand{\CMboxScale}[7]{
	\coordinate (#1) at #2;
	\pgfmathsetmacro{\boxWidth}{#3};
	\pgfmathsetmacro{\boxHeight}{#4};
	\pgfmathsetmacro{\boxDepth}{#5};
	\draw[unshaded, very thick] ($(#1) + (\boxWidth,\boxHeight) $) -- ($(#1) + (\boxWidth,\boxHeight) - (\boxDepth,-\boxDepth) $) -- ($(#1) + (0,\boxHeight) - (\boxDepth,-\boxDepth) $) -- ($(#1) - (\boxDepth,-\boxDepth) $);
	\draw[unshaded, very thick] ($(#1) + (0,\boxHeight) $)  -- ($(#1) + (\boxWidth,\boxHeight) $) -- ($(#1) + (\boxWidth,0) $) -- (#1) -- ($(#1) - (\boxDepth,-\boxDepth) $);
	\draw[very thick] ($(#1) + (0,\boxHeight) - (\boxDepth,-\boxDepth) $) -- ($(#1) + (0,\boxHeight) $) -- (#1);
	\node[scale=#7] at ($(#1) + 1/2*(\boxWidth,\boxHeight) $) {#6}; 
}

\newcommand{\halfDottedEllipse}[3]{
	\draw[thick] #1 arc(-180:0:{#2} and {#3});
	\draw[thick, dotted] ($ #1 + 2*(#2,0)$) arc(0:180:{#2} and {#3});
}

\newcommand{\upTube}[3]{
	\fill[unshaded] #1 -- ($#1 + (0,#3)$) -- ($#1 + (#2,#3)$) -- ($#1 + (#2,0)$);
	\draw[thick] #1 arc(-180:0:{.5*#2} and {.25*#2});
	\draw[thick, unshaded] #1 -- ($#1 + (0,#3)$) ;
	\draw[thick] {($#1 + .5*(#2,0) + (0,#3)$)} ellipse ({.5*#2} and {.25*#2});
	\draw[thick, unshaded] ($#1 + (#2,0)$) -- ($#1 + (#2,#3)$);
}

\newcommand{\upTubeWithString}[4]{
	\fill[unshaded]  ($#1 + (#2,0)$) -- ($#1 + (#2,#3)$) -- ($#1 + (0,#3)$)  -- #1 arc(-180:0:{.5*#2} and {.25*#2});
	\draw[thick] #1 arc(-180:0:{.5*#2} and {.25*#2});
	\draw[thick, unshaded] #1 -- ($#1 + (0,#3)$) ;
	\draw[thick] {($#1 + .5*(#2,0) + (0,#3)$)} ellipse ({.5*#2} and {.25*#2});
	\draw[thick, unshaded] ($#1 + (#2,0)$) -- ($#1 + (#2,.5)$);
	\draw[thick, #4] ($#1 + .5*(#2,0) + .25*(0,-#2) $) -- ($#1 + .5*(#2,0) + (0,#3) + .25*(0,-#2)$);
}

\newcommand{\downTube}[3]{
	\fill[unshaded] #1 -- ($#1 + (0,-#3)$) arc(-180:0:{.5*#2} and {.25*#2}) -- ($#1 + (#2,0)$);
	\draw[thick, unshaded] #1 -- ($#1 + (0,-#3)$) ;
	\halfDottedEllipse{($#1 + (0,-#3)$)}{{.5*#2}}{{.25*#2}}
	\draw[thick, unshaded] ($#1 + (#2,0)$) -- ($#1 + (#2,-#3)$);
}

\newcommand{\downTubeWithString}[4]{
	\draw[thick, unshaded] #1 -- ($#1 + (0,-#3)$) ;
	\halfDottedEllipse{($#1 + (0,-#3)$)}{{.5*#2}}{{.25*#2}}
	\draw[thick, unshaded] ($#1 + (#2,0)$) -- ($#1 + (#2,-#3)$);
	\draw[thick, #4] ($#1 + .5*(#2,0) $) -- ($#1 + .5*(#2,0)  + (0,-#3)+ .25*(0,-#2)$);
}

\newcommand{\straightTubeWithCap}[3]{
	\coordinate (ZZq) at #1;
	\pgfmathsetmacro{\tubeLength}{#3};
	\pgfmathsetmacro{\tubeWidth}{#2};
	\pgfmathsetmacro{\buffer}{.05};	

	\fill[unshaded] ($ (ZZq) + (-\tubeLength,0) + 2*(0,-\buffer) $) -- ($ (ZZq) + 2*(0,-\buffer) $) arc(-90:90:{\tubeWidth+\buffer} and {2*(\tubeWidth+\buffer)}) -- ($ (ZZq) + (-\tubeLength,0) + 4*(0,\tubeWidth) + 2*(0,\buffer) $) arc(90:270:{\tubeWidth+\buffer} and {2*(\tubeWidth+\buffer)}) ;
	\draw[unshaded, thick]  ($ (ZZq) + (-\tubeLength,0) $) -- (ZZq) arc(-90:90:{\tubeWidth} and {2*\tubeWidth}) -- ($ (ZZq) + (-\tubeLength,0) + 4*(0,\tubeWidth) $) ;
	\draw[thick] ($ (ZZq) + (-\tubeLength,0) $) arc(270:90:{2*\tubeWidth});
}

\newcommand{\straightTubeNoString}[3]{
	\coordinate (ZZq) at #1;
	\pgfmathsetmacro{\tubeLength}{#3};
	\pgfmathsetmacro{\tubeWidth}{#2};
	\pgfmathsetmacro{\buffer}{.05};	

	\fill[unshaded] ($ (ZZq) + (-\tubeLength,0) + 2*(0,-\buffer) $) -- ($ (ZZq) + 2*(0,-\buffer) $) arc(-90:90:{\tubeWidth+\buffer} and {2*(\tubeWidth+\buffer)}) -- ($ (ZZq) + (-\tubeLength,0) + 4*(0,\tubeWidth) + 2*(0,\buffer) $) arc(90:270:{\tubeWidth+\buffer} and {2*(\tubeWidth+\buffer)}) ;
	\draw[unshaded, thick]  ($ (ZZq) + (-\tubeLength,0) $) -- (ZZq) arc(-90:90:{\tubeWidth} and {2*\tubeWidth}) -- ($ (ZZq) + (-\tubeLength,0) + 4*(0,\tubeWidth) $) ;
	\draw[thick] ($ (ZZq) + (-\tubeLength,0) $) arc(-90:90:{\tubeWidth} and {2*\tubeWidth}) arc(90:270:{\tubeWidth} and {2*\tubeWidth});
}

\newcommand{\straightTubeWithString}[4]{
	\coordinate (ZZq) at #1;
	\pgfmathsetmacro{\tubeLength}{#3};
	\pgfmathsetmacro{\tubeWidth}{#2};
	\pgfmathsetmacro{\buffer}{.05};	

	\fill[unshaded] ($ (ZZq) + (-\tubeLength,0) + 2*(0,-\buffer) $) -- ($ (ZZq) + 2*(0,-\buffer) $) arc(-90:90:{\tubeWidth+\buffer} and {2*(\tubeWidth+\buffer)}) -- ($ (ZZq) + (-\tubeLength,0) + 4*(0,\tubeWidth) + 2*(0,\buffer) $) arc(90:270:{\tubeWidth+\buffer} and {2*(\tubeWidth+\buffer)}) ;
	\draw[unshaded, thick]  ($ (ZZq) + (-\tubeLength,0) $) -- (ZZq) arc(-90:90:{\tubeWidth} and {2*\tubeWidth}) -- ($ (ZZq) + (-\tubeLength,0) + 4*(0,\tubeWidth) $) ;
	\draw[thick] ($ (ZZq) + (-\tubeLength,0) $) arc(-90:90:{\tubeWidth} and {2*\tubeWidth}) arc(90:270:{\tubeWidth} and {2*\tubeWidth});
	\draw[thick, #4] ($(ZZq) + (\tubeWidth,0) + 2*(0,\tubeWidth) $) -- ($ (ZZq) + (-\tubeLength,0) + 2*(0,\tubeWidth) + (\tubeWidth,0)$);
}

\newcommand{\straightTubeTwoStrings}[4]{
	\fill[unshaded] ($ #1 + (-#2,-.1) $) -- ($ #1 + (0,-.1) $) arc(-90:90:.15cm and .3cm) -- ($ #1 + (-#2,.5) $) ;
	\draw[unshaded, thick]  ($ #1 + (-#2,0) $) -- #1 arc(-90:90:.1cm and .2cm) -- ($ #1 + (-#2,.4) $) ;
	\draw[thick] ($ #1 + (-#2,0) $) arc(-90:90:.1cm and .2cm) arc(90:270:.1cm and .2cm);
	\draw[thick, #3] ($#1 + (.08,.27) $) -- ($ #1 + (-#2,.27) + (.08,0)$);
	\draw[thick, #4] ($#1 + (.08,.13) $) -- ($ #1 + (-#2,.13) + (.08,0)$);
}

\newcommand{\invisibleTube}[1]{
	\fill[unshaded] #1 arc (-90:90:.1cm) arc (270:180:1cm) arc (0:180:.1cm) arc (180:270:1.2cm);
}

\newcommand{\invisibleTubeWithString}[2]{
	\fill[unshaded] #1 arc (-90:90:.1cm) arc (270:180:1cm) arc (0:180:.1cm) arc (180:270:1.2cm);
	\draw[thick, #2] ($#1 + (0,.1) $) arc (270:190:1.1cm) -- ($#1 + (-1.1,1.3) $);
	\draw[thick, #2] ($#1 + (0,.1) $) -- ($#1 + (0,.1) + (.1,0)$);
}

\newcommand{\pairOfPants}[2]{
	\draw[thick] #1 .. controls ++(90:.8cm) and ++(270:.8cm) .. ($ #1 + (.7,1.5) $);
	\draw[thick] ($ #1 + (2,0) $) .. controls ++(90:.8cm) and ++(270:.8cm) .. ($ #1 + (2,0) + (-.7,1.5) $);
	\draw[thick] ($ #1 + (.6,0) $).. controls ++(90:.8cm) and ++(90:.8cm) .. ($ #1 + (1.4,0) $); 
	\halfDottedEllipse{($ #1 + (.7,1.5) $)}{.3}{.1}
	\halfDottedEllipse{#1}{.3}{.1}
	\halfDottedEllipse{($ #1 + (1.4,0) $)}{.3}{.1}
}

\newcommand{\topPairOfPants}[2]{
	\draw[thick] #1 .. controls ++(90:.8cm) and ++(270:.8cm) .. ($ #1 + (.7,1.5) $);
	\draw[thick] ($ #1 + (2,0) $) .. controls ++(90:.8cm) and ++(270:.8cm) .. ($ #1 + (2,0) + (-.7,1.5) $);
	\draw[thick] ($ #1 + (.6,0) $).. controls ++(90:.8cm) and ++(90:.8cm) .. ($ #1 + (1.4,0) $); 
	\draw[thick] ($ #1 + (1,1.5) $) ellipse (.3cm and .1cm);
	\halfDottedEllipse{#1}{.3}{.1}
	\halfDottedEllipse{($ #1 + (1.4,0) $)}{.3}{.1}

}

\newcommand{\invertedPairOfPants}[2]{
	\draw[thick] #1 .. controls ++(270:.8cm) and ++(90:.8cm) .. ($ #1 + (.7,-1.5) $);
	\draw[thick] ($ #1 + (2,0) $) .. controls ++(270:.8cm) and ++(90:.8cm) .. ($ #1 + (2,0) + (-.7,-1.5) $);
	\draw[thick] ($ #1 + (.6,0) $).. controls ++(270:.8cm) and ++(270:.8cm) .. ($ #1 + (1.4,0) $); 
	\halfDottedEllipse{($ #1 + (.7,-1.5) $)}{.3}{.1}

}

\newcommand{\topCylinder}[2]{
	\draw[thick] #1 -- ($ #1 + (0,1) $);
	\draw[thick] ($ #1 + (.6,0) $) -- ($ #1 + (.6,1) $);
	\draw[thick] ($ #1 + (.3,1) $) ellipse (.3cm and .1cm);
}

\newcommand{\bottomCylinder}[3]{
	\draw[thick] #1 -- ($ #1 + (0,#3) $);
	\draw[thick] ($ #1 + 2*(#2,0) $) -- ($ #1 + 2*(#2,0) + (0,#3) $);
	\halfDottedEllipse{#1}{#2}{{1/3*#2}}	
}

\newcommand{\emptyCylinder}[3]{
	\draw[thick] #1 -- ($ #1 + (0,#3) $);
	\draw[thick] ($ #1 + 2*(#2,0) $) -- ($ #1 + 2*(#2,0) + (0,#3) $);	
}

\newcommand{\RightSlantCylinder}[2]{
	\draw[thick] #1 .. controls ++(90:.8cm) and ++(270:.8cm) .. ($ #1 + (.7,1.5) $);
	\draw[thick] ($ #1 + (.6,0) $).. controls ++(90:.8cm) and ++(270:.8cm) .. ($ #1 + (1.3,1.5) $); 
	\halfDottedEllipse{($ #1 + (.7,1.5) $)}{.3}{.1}
	\halfDottedEllipse{#1}{.3}{.1}
}

\newcommand{\LeftSlantCylinder}[2]{
	\draw[thick] #1 .. controls ++(90:.8cm) and ++(270:.8cm) .. ($ #1 + (-.7,1.5) $);
	\draw[thick] ($ #1 + (.6,0) $).. controls ++(90:.8cm) and ++(270:.8cm) .. ($ #1 + (-.1,1.5) $); 
	\halfDottedEllipse{($ #1 + (-.7,1.5) $)}{.3}{.1}
	\halfDottedEllipse{#1}{.3}{.1}
}

\newcommand{\inverseBraid}[3]{

	\coordinate (ZZz) at #1;
	\pgfmathsetmacro{\tubeHeight}{#3};
	\pgfmathsetmacro{\tubeRadius}{#2};
	\pgfmathsetmacro{\buffer}{.2};	

	\draw[thick] (ZZz) .. controls ++(90:.7cm) and ++(270:.7cm) .. ($ (ZZz) + (\tubeHeight,\tubeHeight) - 2*(\tubeRadius,0) $);
	\draw[thick] ($ (ZZz) + 2*(\tubeRadius,0) $) .. controls ++(90:.7cm) and ++(270:.7cm) .. ($ (ZZz) + (\tubeHeight,\tubeHeight) $);

	\fill[unshaded] ($ (ZZz) +(\tubeHeight,0) + (\buffer,0) $) .. controls ++(90:.8cm) and ++(270:.8cm) .. ($ (ZZz) + (0,\tubeHeight) + 2*(\tubeRadius,0) + (\buffer,0) $) -- ($ (ZZz) + (0,\tubeHeight) - (\buffer,0) $)  .. controls ++(270:.8cm) and ++(90:.8cm) .. ($ (ZZz) +  (\tubeHeight,0) - 2*(\tubeRadius,0) - (\buffer,0) $);
	\draw[thick] ($ (ZZz) +(\tubeHeight,0) $) .. controls ++(90:.7cm) and ++(270:.7cm) .. ($ (ZZz) + 2*(\tubeRadius,0) + (0,\tubeHeight) $);
	\draw[thick] ($ (ZZz) + (\tubeHeight,0) - 2*(\tubeRadius,0) $) .. controls ++(90:.7cm) and ++(270:.7cm) .. ($ (ZZz) + (0,\tubeHeight) $);
}

\newcommand{\plane}[3]{
	\pgfmathsetmacro{\planeWidth}{#2};
	\pgfmathsetmacro{\planeDepth}{#3};
	
	\draw[thick] ($ #1 + (-\planeDepth,\planeDepth) $) -- #1 -- ($ #1 + (\planeWidth,0) $) -- ($ #1 + (\planeWidth,0) + (-\planeDepth,\planeDepth) $) -- ($ #1 + (-\planeDepth,\planeDepth) $);
}

  \newcommand{\tikzmath}[2][]
     {\vcenter{\hbox{\begin{tikzpicture}[#1]#2
                     \end{tikzpicture}}}
     }

\makeatletter
\newcommand{\hashdef}[2]{\@namedef{#1}{#2}}
\newcommand{\hashlookup}[1]{\@nameuse{#1}}
\makeatother

\begin{document}
\title{Categorified trace for module tensor categories\\ over braided tensor categories}
\author{Andr\'{e} Henriques, David Penneys, and James Tener}
\date{\today}
\maketitle
\begin{abstract}
Given a braided pivotal category $\cC$ and a pivotal module tensor category $\cM$, we define a functor $\Tr_\cC:\cM \to \cC$, called the associated categorified trace.
By a result of Bezrukavnikov, Finkelberg and Ostrik, 
the functor $\Tr_\cC$ comes equipped with natural isomorphisms $\tau_{x,y}:\Tr_\cC(x \otimes y) \to \Tr_\cC(y \otimes x)$, which we call the traciators.
This situation lends itself to a diagramatic calculus of `strings on cylinders', where the traciator corresponds to wrapping a string around the back of a cylinder.
We show that $\Tr_\cC$ in fact has a much richer graphical calculus in which the tubes are allowed to branch and braid.  
Given algebra objects $A$  and $B$, we prove that $\Tr_\cC(A)$ and $\Tr_\cC(A \otimes B)$ are again algebra objects.
Moreover, provided certain mild assumptions are satisfied, $\Tr_\cC(A)$ and $\Tr_\cC(A \otimes B)$ are semisimple whenever $A$ and $B$ are semisimple.
\end{abstract}

\tableofcontents

\settocdepth{section}							

\section{Introduction}\label{sec:Introduction}

A \emph{tensor category} is a linear category $\cM$, equipped with a functor $\otimes:\cM\times\cM\to \cM$
along with extra data encoding the ideas of associativity and unitality. 
If $x$ is an object of $\cM$, then its dual $x^*$ is characterized, assuming it exists, by adjunctions
$\Hom(y,x\otimes z)\cong\Hom(x^*\otimes y,z)$
and
$\Hom(y\otimes x,z)\cong\Hom(y,z\otimes x^*)$.
The category $\cM$ is \emph{pivotal} if it comes equipped with certain isomorphisms $\varphi_x:x\to x^{**}$ from every object to its double dual.
It is interesting to note that, in a pivotal category, the functor $x\mapsto \Hom(1,x)$ satisfies the following cyclic invariance property:
\[
\begin{split}
\Hom(1,x\otimes y) \,\cong\, \Hom(x^*\otimes 1&,y) \,\cong\, \Hom(1\otimes x^*,y)\\
&\cong\, \Hom(1,y\otimes x^{**}) \,\cong\, \Hom(1,y\otimes x).
\end{split}
\]
We think of $\Hom(1,-)$ as a vector space valued trace $\Tr:\cM\to \mathsf{Vec}$.
This is our prototypic example of a \emph{categorified trace}.

Given a tensor category $\cC$, and two objects $x$, $y$ of some module category $\cM$, the internal hom $\underline{\Hom}(x, y)$ is the object of $\cC$ that represents the exact functor
$c\mapsto \cM(c\cdot x, y)$ \cite[Def.\,3.4]{MR1976459}.
If in addition to being a module category $\cM$ is also a tensor category in its own right, then we may consider the functor 
\[
\Tr_\cC:=\underline{\Hom}(1_\cM, -):\cM\to\cC,
\]
and ask whether it has a similar cyclic invariance property.

The appropriate compatibility between the $\cC$-module structure and the tensor structure of $\cM$ can only be formulated when $\cC$ is braided,
and the resulting notion is what we call a \emph{module tensor category} (Definition \ref{def: central functor}).
We write $\Phi:\cC\to \cM$ for the functor that sends $c\in\cC$ to $c\cdot 1_\cM\in\cM$.
By definition, equipping $\cM$ with the structure of a module tensor category over $\cC$ is the same thing as equipping the functor $\Phi$ with a factorization
\[
\cC\xrightarrow{\,\,\,\Phi^{\scriptscriptstyle \cZ}\,\,\,} \cZ(\cM)\xrightarrow{\,\,\,\,\,\,\,\,\,\,} \cM\,,
\]
where $\Phi^{\scriptscriptstyle \cZ}$ is a braided functor to the Drinfel'd center of $\cM$ \cite[Def.\,1]{MR2074589} \cite[Def.\,2.4]{MR3039775}.
The trace functor can be alternatively described as the right adjoint of $\Phi$:
$$
\begin{tikzpicture}[scale=1.5]
\useasboundingbox (0,-.15) rectangle (3,1.15-.1);
\node[inner sep=5] (x) at (0,0) {$\cC$};
\node[inner sep=5] (y) at (3,0) {$\cM$};
\node[inner sep=5] (up) at (1.5,1-.1) {$\cZ(\cM)$};
\draw[->              , shorten >=2, shorten <=2.5] (x.east)+(0,.03) coordinate(a) --node[above, scale=1.1]{$\scriptstyle \Phi$} (y.west|-a);
\draw[->, dashed, shorten >=2, shorten <=2.5] (y.west)+(0,-.08) coordinate(b) --node[below, scale=1.1]{$\scriptstyle \Tr_\cC$} (x.east|-b);
\draw[->] (x) --node[above, pos=.45, yshift=2, scale=1.1]{$\scriptstyle \Phi^\cZ$} (up);
\draw[->] (up) -- (y);
\end{tikzpicture}
$$

A \emph{categorified trace} (Definition \ref{def:  categorified trace}), also known as a commutator functor  \cite[\S 6]{MR2506324}, \cite[Def. 2.1]{MR3250042}, is a functor $\Tr:\cM\to\cC$ equipped 
with natural isomorphisms
\[
\tau_{x,y}:\Tr(x \otimes y) \to \Tr(y \otimes x),
\]
which we call the \emph{traciators}, subject to the axiom
$\tau_{x, y \otimes z} \,=\, \tau_{z \otimes x, y} \circ \tau_{x \otimes y, z}$.
We denote a categorified trace graphically by lifting objects and morphisms up from the plane (corresponding to $\cM$) onto a cylinder in $3$-space (corresponding to $\cC$)
$$
f=
\begin{tikzpicture}[baseline=-.1cm, scale=.7]
	\plane{(0,-1)}{3}{2}
	\draw[thick, xString] (-1,0) -- (.5,0);
	\draw[thick, yString] (.5,0) -- (2,0);
	\Mbox{(.4,-.3)}{1}{.6}{$f$}
	\node at (2.35,0) {$\scriptstyle y$};
	\node at (-1.35,0) {$\scriptstyle x$};
\end{tikzpicture}
\,\,\,\longmapsto
\quad
\Tr(f)
=
\begin{tikzpicture}[baseline=-.1cm]

	\draw[thick] (-.3,-1) -- (-.3,1);
	\draw[thick] (.3,-1) -- (.3,1);
	\draw[thick] (0,1) ellipse (.3cm and .1cm);
	\halfDottedEllipse{(-.3,-1)}{.3}{.1}
	
	\draw[thick, xString] (0,-1.1) -- (0,0);
	\draw[thick, yString] (0,.9) -- (0,0);
	\nbox{unshaded}{(0,0)}{.3}{-.1}{-.1}{$f$}
	\node at (-.13,.68) {$\scriptstyle y$};
	\node at (-.13,-.77) {$\scriptstyle x$};
\end{tikzpicture}
$$
Following \cite{MR3095324}\footnote{See Remark \ref{rem: rant No1} for a subtle difference between our notion of categorified trace and the one used in~\cite{MR3095324}.}, the traciator is represented graphically by a strand wrapping around the cylinder:
$$
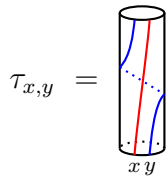
\tau_{x,y} \,\,=\,\,
\begin{tikzpicture}[xscale=1.1, baseline=-.1cm,scale=.9]
	\draw[thick] (-.3,-1) -- (-.3,1);
	\draw[thick] (.3,-1) -- (.3,1);
	\draw[thick] (0,1) ellipse (.3cm and .1cm);
	\halfDottedEllipse{(-.3,-1)}{.3}{.1}
	
	\draw[thick, yString] (.1,-1.1) .. controls ++(85:.2cm) and ++(225:.2cm) .. (.3,-.2);		
	\draw[thick, yString] (-.1,.9) .. controls ++(265:.2cm) and ++(45:.2cm) .. (-.3,.2);
	\draw[thick, yString, dotted] (-.3,.2) -- (.3,-.2);	
	\draw[thick, xString] (-.1,-1.1)  .. controls ++(90:.2) and ++(270:.2) ..  (.1,.9);
	\node at (.11,-1.3) {$\scriptstyle y$};
	\node at (-.1,-1.285) {$\scriptstyle x$};
\end{tikzpicture}
$$
When $\cC$ and $\cM$ are pivotal and $\Phi^{\scriptscriptstyle \cZ}$ is a pivotal functor, then $\Tr_\cC$ is a categorified trace.
In fact, this only depends on $\Phi$ factoring through the Drinfel'd center, and not on it being a tensor functor, as proven in \cite[Prop.\,5]{MR2506324} and \cite[Prop.\,2.5]{MR3250042}.

The fact that $\Phi$ is a tensor functor contributes to the structure of $\Tr_\cC$ in a different way.
Adjoints of tensor functors are lax monoidal \cite{MR0360749}, and so we have unit and multiplication maps $i:1_\cC\to\Tr_\cC(1_\cM)$ and $\mu_{x,y} : \Tr_\cC(x) \otimes \Tr_\cC(y) \to \Tr_\cC(x \otimes y)$
which we represent graphically as follows:
$$
i\,\,=\,\,
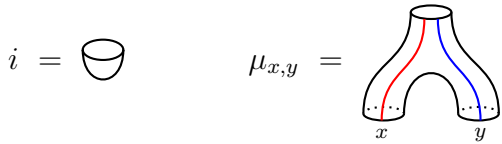
\begin{tikzpicture}[baseline=-.25cm,scale=.9]
	\draw[thick] circle (.3 and .15) (.3,0) to[in=-90, out=-90, looseness=2.5] (-.3,0);
\end{tikzpicture}
\qquad\qquad
\mu_{x,y} \,\,=\,\,
\begin{tikzpicture}[baseline=-.3cm,scale=.9]
	\topPairOfPants{(-1,-1)}{}
	\draw[thick, xString] (-.73,-1.1) .. controls ++(90:.8cm) and ++(270:.8cm) .. (-.1,.4);		
	\draw[thick, yString] (.73,-1.1) .. controls ++(90:.8cm) and ++(270:.8cm) .. (.1,.4);		
	\node[below] at (-.73,-1.09+.03) {$\scriptstyle x$};
	\node[below] at (.73,-1.08+.03) {$\scriptstyle y$};
\end{tikzpicture}
$$

The novelty of this paper is the rich interplay between the above two structures.
Using everything we have, we can assign a morphism to any picture of strands on tubes.
These, in turn, can be used to formulate a number of non-trivial identities, such as
$$
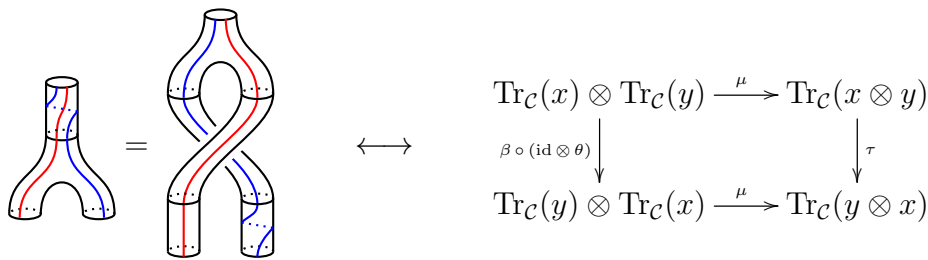
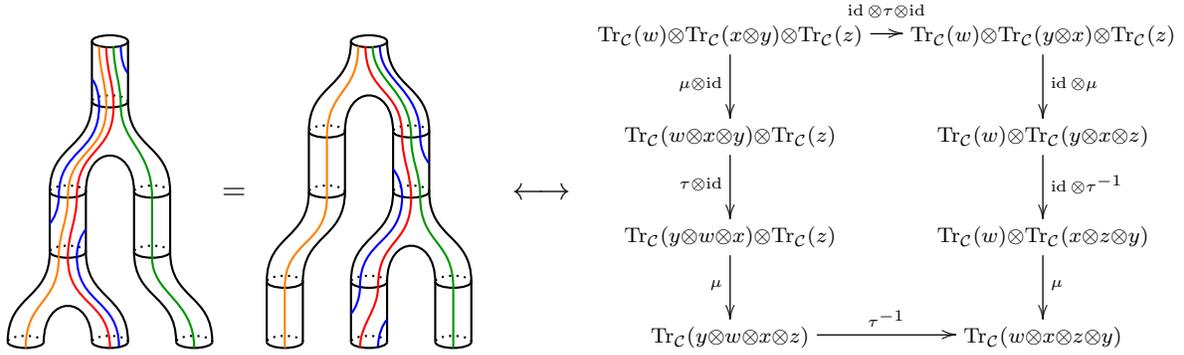
\quad\,\,\,
\begin{tikzpicture}[baseline=.1cm, scale=.7]

	\draw[thick] (-.3,.5) -- (-.3,1.5);
	\draw[thick] (.3,.5) -- (.3,1.5);
	\draw[thick] (0,1.5) ellipse (.3cm and .1cm);
	
	\draw[thick, xString] (-.1,.43) .. controls ++(90:.6cm) and ++(270:.6cm) .. (.1,1.42);		
	\draw[thick, yString] (.1,.43) .. controls ++(90:.2cm) and ++(225:.2cm) .. (.3,.95);		
	\draw[thick, yString] (-.1,1.42) .. controls ++(270:.2cm) and ++(45:.2cm) .. (-.3,1.05);
	\draw[thick, yString, dotted] (.3,.95) -- (-.3,1.05);	

	\draw[thick] (-1,-1) .. controls ++(90:.8cm) and ++(270:.8cm) .. (-.3,.5);
	\draw[thick] (1,-1) .. controls ++(90:.8cm) and ++(270:.8cm) .. (.3,.5);
	\draw[thick] (-.4,-1) .. controls ++(90:.8cm) and ++(90:.8cm) .. (.4,-1); 
	\halfDottedEllipse{(-.3,.5)}{.3}{.1}
	\halfDottedEllipse{(-1,-1)}{.3}{.1}
	\halfDottedEllipse{(.4,-1)}{.3}{.1}
	
	\draw[thick, xString] (-.8,-1.07) .. controls ++(90:.8cm) and ++(270:.8cm) .. (-.1,.42);		
	\draw[thick, yString] (.8,-1.07) .. controls ++(90:.8cm) and ++(270:.8cm) .. (.1,.42);		
\end{tikzpicture}
=
\begin{tikzpicture}[baseline=-.8cm,scale=.7]

	\draw[thick] (-1,-2) -- (-1,-3);
	\draw[thick] (-.4,-2) -- (-.4,-3);
	\halfDottedEllipse{(-1,-3)}{.3}{.1}
	
	\draw[thick, xString] (-.7,-2.1) -- (-.7,-3.1);

	\draw[thick] (1,-2) -- (1,-3);
	\draw[thick] (.4,-2) -- (.4,-3);
	\halfDottedEllipse{(.4,-3)}{.3}{.1}

	\draw[thick, yString] (.7,-3.1) .. controls ++(90:.2cm) and ++(225:.2cm) .. (1,-2.55);		
	\draw[thick, yString] (.7,-2.1) .. controls ++(270:.2cm) and ++(45:.2cm) .. (.4,-2.45);
	\draw[thick, yString, dotted] (.4,-2.45) -- (1,-2.55);		

	\draw[thick] (1,-2) .. controls ++(90:.7cm) and ++(270:.7cm) .. (-.4,0);
	\draw[thick] (.4,-2) .. controls ++(90:.7cm) and ++(270:.7cm) .. (-1,0);
	\draw[thick, yString] (-.7,-.07) .. controls ++(270:.6cm) and ++(90:.8cm) .. (.7,-2.1);		

	\fill[unshaded] (-1.2,-2) .. controls ++(90:.8cm) and ++(270:.8cm) .. (.2,0) -- (1.2,0)  .. controls ++(270:.8cm) and ++(90:.8cm) .. (-.2,-2);
	\draw[thick] (-1,-2) .. controls ++(90:.7cm) and ++(270:.7cm) .. (.4,0);
	\draw[thick] (-.4,-2) .. controls ++(90:.7cm) and ++(270:.7cm) .. (1,0);
	\draw[thick, xString] (.7,-.07) .. controls ++(270:.6cm) and ++(90:.8cm) .. (-.7,-2.1);	

	\halfDottedEllipse{(-1,-2)}{.3}{.1}
	\halfDottedEllipse{(.4,-2)}{.3}{.1}

	\draw[thick] (-1,0) .. controls ++(90:.8cm) and ++(270:.8cm) .. (-.3,1.5);
	\draw[thick] (1,0) .. controls ++(90:.8cm) and ++(270:.8cm) .. (.3,1.5);
	\draw[thick] (-.4,0) .. controls ++(90:.8cm) and ++(90:.8cm) .. (.4,0); 
	\draw[thick] (0,1.5) ellipse (.3cm and .1cm);
	\halfDottedEllipse{(-1,0)}{.3}{.1}
	\halfDottedEllipse{(.4,0)}{.3}{.1}
	
	\draw[thick, yString] (-.7,-.07) .. controls ++(90:.8cm) and ++(270:.8cm) .. (-.1,1.42);		
	\draw[thick, xString] (.7,-.07) .. controls ++(90:.8cm) and ++(270:.8cm) .. (.1,1.42);		
\end{tikzpicture}
\quad\,\,\,\,\,\,\longleftrightarrow\,\,\,\,\quad
\raisebox{.78cm}{
\xymatrix{
\operatorname{Tr}_\mathcal{C}(x) \otimes \Tr_\cC(y)  \ar[r]^(.55){\scriptscriptstyle \mu} \ar[d]_{\scriptscriptstyle \beta \,\circ\, (\id \otimes\, \theta)} &\operatorname{Tr}_\mathcal{C}(x \otimes y) \ar[d]^{\scriptscriptstyle \tau}
\\
\operatorname{Tr}_\mathcal{C}(y) \otimes \operatorname{Tr}_\mathcal{C}(x)  \ar[r]^(.55){\scriptscriptstyle \mu} &\operatorname{Tr}_\mathcal{C}(y \otimes x)
}}
$$
and
$$
\begin{tikzpicture}[baseline=1.1cm, scale=.8]

	\coordinate (a) at (-1,-1);
	\coordinate (a1) at ($ (a) + (1.4,0)$);
	\coordinate (a2) at ($ (a1) + (1.4,0)$);
	\coordinate (b) at ($ (a) + (.7,1.5)$);
	\coordinate (b2) at ($ (b) + (1.4,0)$);
	\coordinate (c) at ($ (b) + (0,1)$);
	\coordinate (c2) at ($ (c) + (1.4,0)$);
	\coordinate (d) at ($ (c) + (.7,1.5)$);	
	
	\pairOfPants{(a)}{}
	\pairOfPants{(c)}{}
	\topCylinder{($ (a) + (1.4,4) $)}{}
	\draw[thick] (b) -- ($ (b) + (0,1) $);
	\draw[thick] ($ (b) + (.6,0) $) -- ($ (b) + (.6,1) $);
	\draw[thick] ($ (b) + (1.4,0) $)  -- ($ (b) + (1.4,1) $);
	\draw[thick] ($ (b) + (2,0) $) -- ($ (b) + (2,1) $);
	\LeftSlantCylinder{($ (a) + (2.8,0) $)}{}

	\draw[thick, wString] ($ (b) + (.15,-.08) $) .. controls ++(90:.4cm) and ++(270:.4cm) .. ($ (b) + 2*(.15,0) + (0,.9) $);
	\draw[thick, xString] ($ (b) + 2*(.15,0) + (0,-.1) $) .. controls ++(90:.4cm) and ++(270:.4cm) .. ($ (b) + 3*(.15,0) + (0,.92) $);		
	\draw[thick, yString] ($ (b) + 3*(.15,0) + (0,-.08) $) .. controls ++(90:.2cm) and ++(225:.1cm) .. ($ (b) + 4*(.15,0) + (0,-.08) + (0,.45)$);
	\draw[thick, yString] ($ (b) + (.15,1) + (0,-.08) $) .. controls ++(270:.2cm) and ++(45:.1cm) .. ($ (b) + (0,1) + (0,-.08) + (0,-.45)$);

	\draw[thick, wString] ($ (d) + 2*(.12,0) + (0,-.1)$) .. controls ++(90:.4cm) and ++(270:.4cm) .. ($ (d) + (.12,0) + (0,-.08) + (0,1) $);
	\draw[thick, xString] ($ (d) + 3*(.12,0) + (0,-.08) $) .. controls ++(90:.4cm) and ++(270:.4cm) .. ($ (d) + 2*(.12,0) + (0,.9) $);		
	\draw[thick, yString] ($ (d) + (.12,0) + (0,-.08) $) .. controls ++(90:.2cm) and ++(-45:.1cm) .. ($ (d) + (0,-.08) + (0,.45)$);
	\draw[thick, yString] ($ (d) + 4*(.12,0) + (0,1) + (0,-.08) $) .. controls ++(270:.2cm) and ++(135:.1cm) .. ($ (d) + 5*(.12,0) + (0,-.08) + (0,-.45) + (0,1)$);
	\draw[thick, zString] ($ (d) + 4*(.12,0) + (0,-.08) $) .. controls ++(90:.4cm) and ++(270:.4cm) .. ($ (d) + 3*(.12,0) + (0,-.08) + (0,1)$);

	\draw[thick, wString] ($ (a) + 2*(.15,0) + (0,-.1)$) .. controls ++(90:.8cm) and ++(270:.8cm) .. ($ (b) + (.15,0) + (0,-.08) $);	
	\draw[thick, xString] ($ (a1) + 2*(.15,0) + (0,-.1)$) .. controls ++(90:.8cm) and ++(270:.8cm) .. ($ (b) + 2*(.15,0) + (0,-.1) $);	
	\draw[thick, yString] ($ (a1) + 3*(.15,0) + (0,-.08)$) .. controls ++(90:.8cm) and ++(270:.8cm) .. ($ (b) + 3*(.15,0) + (0,-.08) $);	
	\draw[thick, zString] ($ (a2) + 2*(.15,0) + (0,-.1)$) .. controls ++(90:.8cm) and ++(270:.6cm) .. ($ (b2) + 2*(.15,0) + (0,-.1) $) -- ($ (c2) + 2*(.15,0) + (0,-.1) $);	
	
	\draw[thick, yString] ($ (c) + (.15,0) + (0,-.08)$) .. controls ++(90:.8cm) and ++(270:.8cm) .. ($ (d) + (.12,0) + (0,-.08) $);
	\draw[thick, wString] ($ (c) + 2*(.15,0) + (0,-.1)$) .. controls ++(90:.8cm) and ++(270:.8cm) .. ($ (d) + 2*(.12,0) + (0,-.08) $);	
	\draw[thick, xString] ($ (c) + 3*(.15,0) + (0,-.08)$) .. controls ++(90:.8cm) and ++(270:.8cm) .. ($ (d) + 3*(.12,0) + (0,-.08) $);	
	\draw[thick, zString] ($ (c2) + 2*(.15,0) + (0,-.1)$) .. controls ++(90:.8cm) and ++(270:.6cm) .. ($ (d) + 4*(.12,0) + (0,-.08) $);	

\end{tikzpicture}
=\,\,
\begin{tikzpicture}[baseline=1.1cm, scale=.8]	

	\coordinate (a1) at (-1,-1);
	\coordinate (a2) at ($ (a1) + (1.4,0)$);
	\coordinate (a3) at ($ (a1) + (2.8,0)$);
	\coordinate (b1) at ($ (a1) + (0,1)$);
	\coordinate (b2) at ($ (b1) + (1.4,0) $);
	\coordinate (b3) at ($ (b2) + (1.4,0) $);
	\coordinate (c1) at ($ (b1) + (.7,1.5)$);
	\coordinate (c2) at ($ (b2) + (.7,1.5)$);
	\coordinate (d1) at ($ (c1) + (0,1)$);
	\coordinate (d2) at ($ (c2) + (0,1)$);
	
	\RightSlantCylinder{(b1)}{}
	\pairOfPants{(b2)}{}
	\draw[thick] (c1) -- ($ (c1) + (0,1) $);
	\draw[thick] ($ (c1) + (.6,0) $) -- ($ (c1) + (.6,1) $);
	\draw[thick] (c2) -- ($ (c2) + (0,1) $);
	\draw[thick] ($ (c2) + (.6,0) $) -- ($ (c2) + (.6,1) $);
	\topPairOfPants{(d1)}{}
	
	\draw[thick] (a1) -- ($ (a1) + (0,1) $);
	\draw[thick] ($ (a1) + (.6,0) $) -- ($ (a1) + (.6,1) $);
	\halfDottedEllipse{(a1)}{.3}{.1}	

	\draw[thick] (a2) -- ($ (a2) + (0,1) $);
	\draw[thick] ($ (a2) + (.6,0) $) -- ($ (a2) + (.6,1) $);
	\halfDottedEllipse{(a2)}{.3}{.1}	

	\draw[thick] (a3) -- ($ (a3) + (0,1) $);
	\draw[thick] ($ (a3) + (.6,0) $) -- ($ (a3) + (.6,1) $);
	\halfDottedEllipse{(a3)}{.3}{.1}	

	\draw[thick, xString] ($ (a2) + (.15,0) + (0,-.1)$) .. controls ++(90:.4cm) and ++(270:.4cm) .. ($ (a2) + 3*(.15,0) + (0,-.08) + (0,1) $);		
	\draw[thick, yString] ($ (a2) + 3*(.15,0) + (0,-.08) $) .. controls ++(90:.2cm) and ++(225:.1cm) .. ($ (a2) + 4*(.15,0) + (0,-.08) + (0,.45)$);
	\draw[thick, yString] ($ (a2) + (.15,1) + (0,-.08) $) .. controls ++(270:.2cm) and ++(45:.1cm) .. ($ (a2) + (0,1) + (0,-.08) + (0,-.45)$);

	\draw[thick, xString] ($ (c2) + 2*(.15,0) + (0,-.1)$) .. controls ++(90:.4cm) and ++(270:.4cm) .. ($ (c2) + (.15,0) + (0,-.08) + (0,1) $);
	\draw[thick, zString] ($ (c2) + 3*(.15,0) + (0,-.08) $) .. controls ++(90:.4cm) and ++(270:.4cm) .. ($ (c2) + 2*(.15,0) + (0,.9) $);		
	\draw[thick, yString] ($ (c2) + (.15,0) + (0,-.08) $) .. controls ++(90:.2cm) and ++(-45:.1cm) .. ($ (c2) + (0,-.08) + (0,.45)$);
	\draw[thick, yString] ($ (c2) + 3*(.15,0) + (0,1) + (0,-.08) $) .. controls ++(270:.2cm) and ++(135:.1cm) .. ($ (c2) + 4*(.15,0) + (0,-.08) + (0,-.45) + (0,1)$);
	
	\draw[thick, wString] ($ (a1) + 2*(.15,0) + (0,-.1)$) -- ($ (b1) + 2*(.15,0) + (0,-.1)$) .. controls ++(90:.8cm) and ++(270:.8cm) .. ($ (c1) + 2*(.15,0) + (0,-.1) $) -- ($ (d1) + 2*(.15,0) + (0,-.1) $) ;	
	\draw[thick, yString] ($ (b2) + (.15,0) + (0,-.08)$) .. controls ++(90:.8cm) and ++(270:.8cm) .. ($ (c2) + (.15,0) + (0,-.08) $);	
	\draw[thick, xString] ($ (b2) + 3*(.15,0) + (0,-.08)$) .. controls ++(90:.8cm) and ++(270:.8cm) .. ($ (c2) + 2*(.15,0) + (0,-.1) $);	
	\draw[thick, zString] ($ (a3) + 2*(.15,0) + (0,-.1)$) -- ($ (b3) + 2*(.15,0) + (0,-.1)$) .. controls ++(90:.8cm) and ++(270:.8cm) .. ($ (c2) + 3*(.15,0) + (0,-.08) $);	

	\draw[thick, wString] ($ (d1) + 2*(.15,0) + (0,-.1)$) .. controls ++(90:.8cm) and ++(270:.8cm) .. ($ (d1) + (.12,0) + (0,-.08) + (.7,1.5) $);	
	\draw[thick, xString] ($ (d2) + (.15,0) + (0,-.08)$) .. controls ++(90:.8cm) and ++(270:.8cm) .. ($ (d1) + 2*(.12,0) + (0,-.08) + (.7,1.5) $);	
	\draw[thick, zString] ($ (d2) + 2*(.15,0) + (0,-.1)$) .. controls ++(90:.8cm) and ++(270:.6cm) .. ($ (d1) + 3*(.12,0) + (0,-.08) + (.7,1.5) $);	
	\draw[thick, yString] ($ (d2) + 3*(.15,0) + (0,-.08)$) .. controls ++(90:.8cm) and ++(270:.6cm) .. ($ (d1) + 4*(.12,0) + (0,-.08) + (.7,1.5) $);
	
\end{tikzpicture}
\quad\longleftrightarrow
\raisebox{2.2cm}{
\xymatrix @C=0cm{
\scriptstyle\Tr_\cC(w) \otimes \Tr_\cC(x\otimes y) \otimes \Tr_\cC(z) \ar[rrr]^*+{\scriptscriptstyle \id\otimes \tau \otimes \id} \ar[d]_{\scriptscriptstyle\mu\otimes \id}
&&&
\scriptstyle\Tr_\cC(w) \otimes \Tr_\cC(y \otimes x) \otimes \Tr_\cC(z) \ar[d]^{\scriptscriptstyle\id\otimes \mu}
\\
\scriptstyle\Tr_\cC(w \otimes x \otimes y)\otimes \Tr_\cC(z) \ar[d]_{\scriptscriptstyle\tau\otimes \id}
&&&
\scriptstyle\Tr_\cC(w) \otimes \Tr_\cC(y \otimes x \otimes z)\ar[d]^{\scriptscriptstyle\id\otimes \tau^{-1}}
\\
\scriptstyle\Tr_\cC(y \otimes w \otimes x)\otimes \Tr_\cC(z) \ar[d]_{\scriptscriptstyle\mu}
&&&
\scriptstyle\Tr_\cC(w) \otimes \Tr_\cC(x \otimes z\otimes y)\ar[d]^{\scriptscriptstyle\mu}
\\
\scriptstyle\Tr_\cC(y \otimes w \otimes x \otimes z) \ar[rrr]^{\scriptscriptstyle\tau^{-1}}
&&&
\scriptstyle\Tr_\cC(w\otimes x \otimes z\otimes y) 
}}
$$
which combine  the traciator and multiplication maps with the braiding and the twist of $\cC$ (see Lemmas~\ref{lem:TwistMultiplicationAndTraciators} and \ref{lem:MultiplicationAssocaitive}).
We summarize the situation in Figure \ref{fig: big figure}.

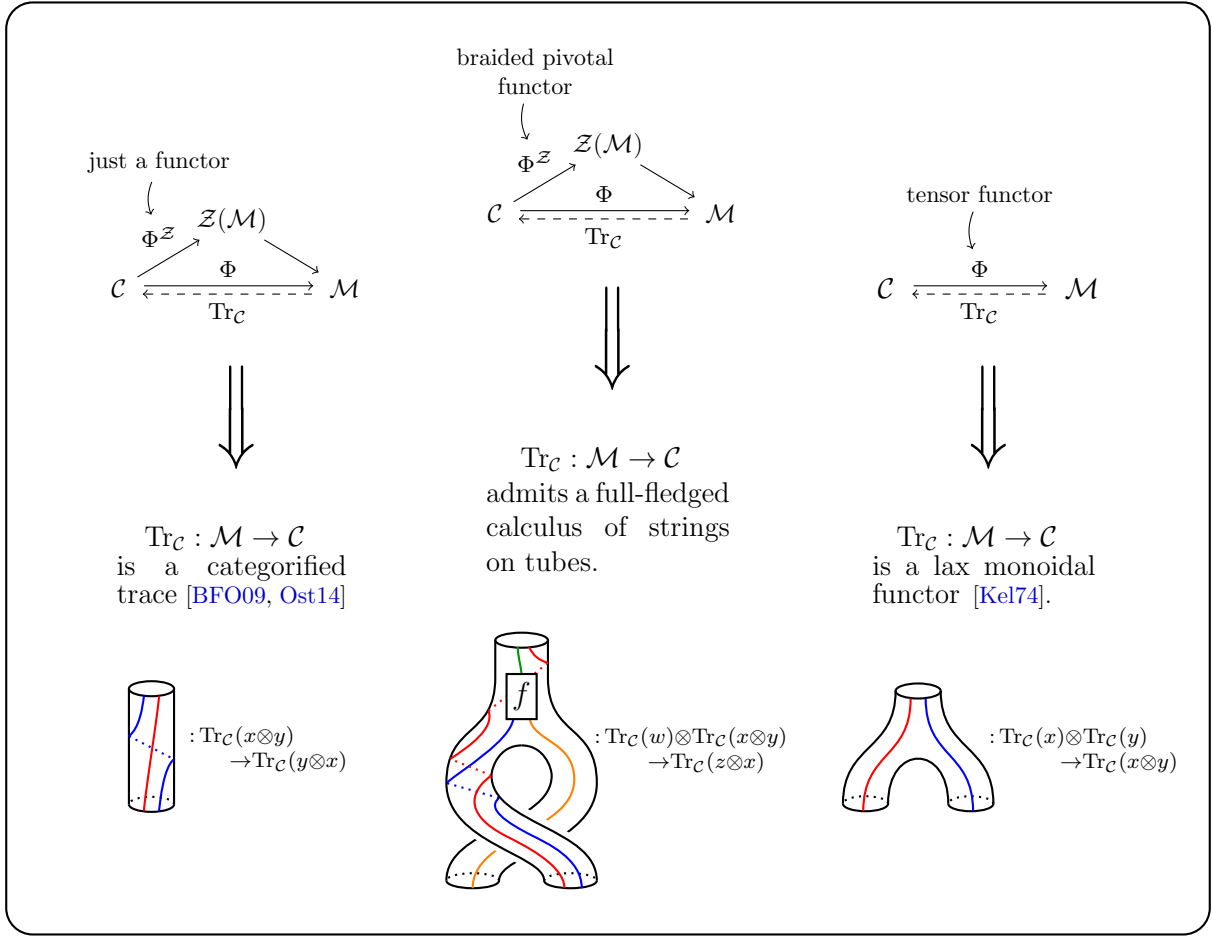
\begin{figure}[!ht]
$$
\begin{tikzpicture}
\useasboundingbox (-6.5,2.9) rectangle (9.5,-9.3);
\draw [thick, rounded corners=10] (-6.5,2.8) rectangle (9.5,-9.6);
\node[inner sep=5, scale=.8] (x) at (0,0) {$ \cC$};
\node[inner sep=5, scale=.8] (y) at (3,0) {$ \cM$};
\node[inner sep=3, scale=.8] (up) at (1.5,.9) {$ \cZ(\cM)$};
\draw[->              , shorten >=2, shorten <=2.5] (x.east)+(0,.03) coordinate(a) --node[above, scale=1.1, yshift=-1]{$\scriptstyle \Phi$} (y.west|-a);
\draw[->, dashed, shorten >=2, shorten <=2.5] (y.west)+(0,-.08) coordinate(b) --node[below, scale=1.1, yshift=1]{$\scriptstyle \Tr_\cC$} (x.east|-b);
\draw[->] (x) --node(w)[above, pos=.35, yshift=2, scale=1.1]{$\scriptstyle \Phi^\cZ$} (up);
\draw[->] (up) -- (y);
\path (w)+(0,1.2) node(z)[scale=.9]{\parbox{2.3cm}{\footnotesize \centerline{braided pivotal} \centerline{functor}}};
\draw[->] (z) to[bend right=20] (w);

\node[inner sep=5, scale=.8] (x1) at (-5,0-1) {$ \cC$};
\node[inner sep=5, scale=.8] (y1) at (3-5,0-1) {$ \cM$};
\node[inner sep=3, scale=.8] (up1) at (1.5-5,.9-1) {$ \cZ(\cM)$};
\draw[->              , shorten >=2, shorten <=2.5] (x1.east)+(0,.03) coordinate(a1) --node[above, scale=1.1, yshift=-1]{$\scriptstyle \Phi$} (y1.west|-a1);
\draw[->, dashed, shorten >=2, shorten <=2.5] (y1.west)+(0,-.08) coordinate(b1) --node[below, scale=1.1, yshift=1]{$\scriptstyle \Tr_\cC$} (x1.east|-b1);
\draw[->] (x1) --node(w1)[above, pos=.35, yshift=2, scale=1.1]{$\scriptstyle \Phi^\cZ$} (up1);
\draw[->] (up1) -- (y1);
\path (w1)+(0,1) node(z1)[scale=.9]{\footnotesize just a functor};
\draw[->] (z1) to[bend right=20] (w1);

\node[inner sep=5, scale=.9] (x2) at (.2+5,0-1) {$\cC$};
\node[inner sep=5, scale=.9] (y2) at (3-.2+5,0-1) {$\cM$};
\draw[->              , shorten >=2, shorten <=2.5] (x2.east)+(0,.03) coordinate(a2) --node(w2)[above, scale=1.1, yshift=-1]{$\scriptstyle \Phi$} (y2.west|-a2);
\draw[->, dashed, shorten >=2, shorten <=2.5] (y2.west)+(0,-.08) coordinate(b2) --node[below, scale=1.1, yshift=1]{$\scriptstyle \Tr_\cC$} (x2.east|-b2);
\path (w2)+(0,1) node(z2)[scale=.9]{\footnotesize tensor functor};
\draw[->] (z2) to[bend right=20] (w2);

\pgftransformyshift{8}

\node[scale=2, rotate=-90, xscale=1.1] at (1.5,-1.95) {$\Longrightarrow$};
\node[scale=2, rotate=-90, xscale=1.1] at (1.5-5,-2-1) {$\Longrightarrow$};
\node[scale=2, rotate=-90, xscale=1.1] at (1.5+5,-2-1) {$\Longrightarrow$};

\pgftransformyshift{5}

\node[below, scale=.9] at (1.5-5,-4.5) {\parbox{3.4cm}{\centerline{$\Tr_\cC:\cM\to \cC\,\,\,$}  is a categorified trace \footnotesize \cite{MR2506324, MR3250042}}};

\node[below, scale=.9] at (1.5,-3.45) {\parbox{3.6cm}{\centerline{$\Tr_\cC:\cM\to \cC\,\,\,$}  admits a full-fledged calculus of strings on tubes.}};

\node[below, scale=.9] at (1.5+5,-4.5) {\parbox{3.3cm}{\centerline{$\Tr_\cC:\cM\to \cC\,\,\,$}  is a lax monoidal functor\;\! \footnotesize \cite{MR0360749}.}};

\pgftransformxshift{-130}
\pgftransformyshift{-215}
\pgftransformyscale{1/1.3}

	\draw[thick] (-.3,-1) -- (-.3,1);
	\draw[thick] (.3,-1) -- (.3,1);
	\draw[thick] (0,1) ellipse (.3cm and .13cm);
	\halfDottedEllipse{(-.3,-1)}{.3}{.13}
	
	\draw[thick, yString] (.1,-1.12) .. controls ++(85:.2cm) and ++(225:.2cm) .. (.3,-.22);		
	\draw[thick, yString] (-.1,.88) .. controls ++(265:.2cm) and ++(45:.2cm) .. (-.3,.18);
	\draw[thick, yString, dotted] (-.3,.18) -- (.3,-.22);	
	\draw[thick, xString] (-.1,-1.12)  .. controls ++(90:.2) and ++(270:.2) ..  (.1,.88);
\node at (1.2,-.05) {$\begin{smallmatrix}:\,\Tr_\cC(x\otimes y)\\ \quad\qquad\to\Tr_\cC(y\otimes x)\end{smallmatrix}$};

\pgftransformyscale{1.3}

\pgftransformxshift{290}
\pgftransformyshift{6.8}

	\topPairOfPants{(-1,-1)}{}
	\draw[thick, xString] (-.73,-1.1) .. controls ++(90:.8cm) and ++(270:.8cm) .. (-.1,.4);		
	\draw[thick, yString] (.73,-1.1) .. controls ++(90:.8cm) and ++(270:.8cm) .. (.1,.4);		
\node at (2,-.29) {$\begin{smallmatrix}:\,\Tr_\cC(x)\otimes \Tr_\cC(y)\\\qquad\quad\,\, \to\Tr_\cC(x\otimes y)\end{smallmatrix}$};

\pgftransformxshift{-150}
\node at (2.25,-.29) {$\begin{smallmatrix}:\,\Tr_\cC(w)\otimes \Tr_\cC(x\otimes y)\\\quad\to\Tr_\cC(z\otimes x)\end{smallmatrix}$};
\pgftransformyshift{5}

	\draw[thick] (-1,-.9) .. controls ++(90:.8cm) and ++(270:.8cm) .. (-.35,.5) -- +(0,.5);
	\draw[thick] (1,-.9) .. controls ++(90:.8cm) and ++(270:.8cm) .. (.35,.5) -- +(0,.5);
	\draw[thick] (0,.5+.5) ellipse (.35cm and .1cm);
	\halfDottedEllipse{(-1,-2.2)}{.35}{.1}
	\halfDottedEllipse{(1,-2.2)}{-.35}{.1}

	\draw[thick, yString, dotted] (-.3,-1.1) -- (-1,-.9);
	\draw[thick, yString] (-1,-.9) .. controls ++(40:.3cm) and ++(260:.3cm) .. (-.11,-.05);	
	\draw[thick, wString] (-.65,-2.3)  .. controls ++(90:.7cm) and ++(270:.7cm) .. (.7,-.85) .. controls ++(90:.5cm) and ++(280:.3cm) .. (.11,-.05);		
	\draw[thick, zString] (0,.5) to[out=90, in=-90] (-.05,.91);
	\draw[thick, xString, dotted] (-.4,-.8) -- (-.95,-.58);
	\draw[thick, xString] (-.95,-.58) .. controls ++(33:.3cm) and ++(290:.25cm) .. (-.47+.03,.06);	
	\draw[thick, xString, dotted] (-.47+.03,.06) -- (.35,.7);
	\draw[thick, xString] (.35,.7) to[bend left=25] (.1,.91);	
	\nbox{unshaded}{(0,.25)}{.3}{-.1}{-.1}{$f$}

\draw[thick] (-1,-2.2) .. controls ++(90:.7cm) and ++(270:.65cm) .. (.4,-.8) .. controls ++(90:.1cm) and ++(0:.3cm) .. (0,-.4);
\draw[thick] (-.3,-2.2)  .. controls ++(90:.35cm) and ++(270:.8cm) ..  (1,-.9);
\fill[white] (0,-1.2) -- (.7,-1.6) -- (0,-2) -- (-.7,-1.6) -- cycle;
\draw[thick] (1,-2.2) .. controls ++(90:.7cm) and ++(270:.65cm) .. (-.4,-.8) .. controls ++(90:.1cm) and ++(180:.3cm) .. (0,-.4);
\draw[thick] (.3,-2.2)  .. controls ++(90:.35cm) and ++(270:.8cm) ..  (-1,-.9);

	\draw[thick, yString] (.8,-2.3) .. controls ++(90:.4cm) and ++(-27:.2cm) .. (.14,-1.6) .. controls ++(180-27:.4cm) and ++(220:.22cm) .. (-.3,-1.1);	
	\draw[thick, xString] (.57,-2.3) .. controls ++(90:.32cm) and ++(-27:.2cm) .. (-.2,-1.61) .. controls ++(180-27:.7cm) and ++(215:.3cm) .. (-.4,-.8);	

\end{tikzpicture}
$$
\caption{Our setup (in the middle), in comparison to previously studied situations.}
\label{fig: big figure}
\end{figure}

In our sequel paper \cite{PAinBTC}, we will prove that the relations established in this paper imply that the morphism assigned to a diagram is invariant under \emph{all} isotopies.
This verification will be performed in the language of planar algebras internal to braided tensor categories.
In our forthcoming paper,
we will later classify planar algebras internal to $\cC$ in terms of module tensor categories for~$\cC$.\bigskip

It is well known that lax tensor functors send algebras to algebras. As a result, if $A$ is an algebra object in $\cM$, then $\Tr_\cC(A)$ is an algebra object in $\cC$. In our situation, more is true.
If $A$ and $B$ are two algebra objects, then $\Tr_\cC(A\otimes B)$ is again an algebra object
(actually, this only requires $\cC$ and $\Phi^{\scriptscriptstyle \cZ}$ to be monoidal, as opposed to braided).
The structure maps are best illustrated by putting the strand corresponding to $A$ on the front of the cylinders, and the one corresponding to $B$ on the back of the cylinders:

\begin{equation*}
m_{\Tr_\cC(A\otimes B)} \,\,=\,\, 
\begin{tikzpicture}[baseline=.6cm, scale=.9]
	\topPairOfPants{(0,0)}{}

\pgftransformxshift{1.3}
	\filldraw[AString] (.9, 1.1) circle (.025cm);
	\draw[thick, AString] (.2,-.09) .. controls ++(92:.8cm) and ++(270:.4cm) .. (.9,1.1);		
	\draw[thick, AString] (1.6,-.09) .. controls ++(90:.8cm) and ++(270:.4cm) .. (.9,1.1);		
	\draw[thick, AString] (.9,1.1) -- (.9,1.41);
\pgftransformxshift{-2.6}
	\filldraw[BString] (1.1, 1.1) circle (.025cm);
	\draw[more thick, BString, densely dotted] (.4,.08) .. controls ++(88:.6cm) and ++(270:.4cm) .. (1.1,1.1);		
	\draw[more thick, BString, densely dotted] (1.8,.08) .. controls ++(90:.6cm) and ++(270:.4cm) .. (1.1,1.1);		
	\draw[more thick, BString, densely dotted] (1.1,1.1) -- (1.1,1.58);
\end{tikzpicture}
\qquad\qquad
i_{\Tr_\cC(A\otimes B)}\,\,=\,\,
\begin{tikzpicture}[baseline=.3cm, scale=.9]
	\draw[thick] (1,1) circle (.3 and .1);
	\draw[thick] (.7,1) -- (.7,.25) arc (-180:0:.3 and .35) -- (1.3,1);		

	\filldraw[AString] (.95, .3+.1) circle (.025cm);
	\filldraw[BString] (1.05, .46+.1) circle (.025cm);
	\draw[thick, AString] (.95,.91) -- (.95,.3+.1);
	\draw[more thick, BString, densely dotted] (1.05,1.09) -- (1.05,.46+.1);
\end{tikzpicture}
\end{equation*}

Let us now assume that $\cC$ and $\cM$ are fusion categories over a field of characteristic zero.
One of our main results is that if $A$ and $B$ are semisimple algebras then so is $\Tr_\cC(A \otimes B)$ (Theorem \ref{thm:TwoAlgebrasPartTwo} and Corollary \ref{CORL our main thm}):

\begin{thmalpha}
\label{thm:TwoAlgebrasIntro}
Let $\cC$ and $\cM$ be tensor categories subject to the above assumptions, and let $\Phi^{\scriptscriptstyle \cZ}:\cC\to \cZ(\cM)$ be a tensor functor. 
If $A,B\in\cM$ are semisimple algebras, then $\Tr_\cC(A \otimes B)$ is also semisimple. 
\end{thmalpha}

\noindent
Alternatively, we can trade the extra assumptions on $\cC$ and $\cM$ for the assumption that $A$ and $B$ are separable algebras (Theorem~\ref{thm:TwoAlgebrasPartTwo}).
We illustrate our construction by computing the algebra $\Tr_\cC(A \otimes B)$ for certain algebra objects related to the Coxeter--Dynkin diagrams $E_7$ and $D_{10}$ (Example \ref{example: E_7}).

In the special case in which we have only one algebra, our result holds in the greater generality of an arbitrary tensor functor $\Phi:\cC \to \cM$ between rigid semisimple tensor categories (Theorem~\ref{thm:SemiSimple} and Corollary~\ref{cor:FrobeniusToFrobenius}): 

\begin{thmalpha}
\label{thm:AlgebraToAlgebraIntro}
Let $\cC$ and $\cM$ be tensor categories subject to the above assumptions, and
let $\Phi:\cC \to \cM$ be a tensor functor with right adjoint $\Tr_\cC$.
If $A\in \cM$ is a semisimple algebra, then $\Tr_\cC(A)$ is as well. In particular, if $A$ is semisimple connected, and thus Frobenius, then so is $\Tr_\cC(A)$. 
\end{thmalpha}

\subsection{Acknowledgements}

The authors would like to thank
Bojko Bakalov,
Bruce Bartlett, 
Alexandru Chirvasitu,
Vaughan Jones,
Scott Morrison,
Richard Ng,
Victor Ostrik,
Ingo Runkel,
Chris Schommer-Pries,
Noah Snyder, and
Kevin Walker
for helpful discussions.

The authors would like to thank Mathematisches Forschungsinstitut Oberwolfach for their hospitality during the workshop on ``Subfactors and conformal field theory.''
Andr\'e Henriques gratefully acknowledges the Leverhulme trust for financing his visiting position in Oxford, and the Hausdorff Institute for its hospitality while visiting James Tener.
James Tener would like to thank the Max Planck Institute for its hospitality and support during the course of this work.
David Penneys and James Tener were partially supported by NSF DMS grant 0856316.
David Penneys was partially supported by an AMS-Simons travel grant and NSF DMS grant 1500387.

\settocdepth{subsection}						

\section{Background}
\label{sec:Background}

The natural setting to study the categorified trace is when $\cC$ is a braided pivotal category, $\cM$ is a pivotal category, and $\Phi^{\scriptscriptstyle \cZ}: \cC\to \cZ(\cM)$ is a braided pivotal functor.
This is perhaps an unfamiliar setting, since most readers are more likely to work with either braided tensor categories or ribbon categories, and braided pivotal categories are a strange intermediate (see Figure~\ref{fig:SynopticChart} in Section~\ref{sec: chart of categories}).

For these reasons, it is important to carefully understand what we can and cannot do in a braided pivotal category.
Moreover, in the literature there are sometimes redundancies within definitions, as well as incorrect statements, so we begin with a short, comprehensive background on various flavours of tensor categories.
Various technical lemmas about braided pivotal categories are deferred to Appendix \ref{sec:Lemmas}.
Finally, we provide a short section on algebras in tensor categories.

\subsection{Tensor categories}\label{sec:  Tensor categories}

We will usually assume that all our categories are linear over some field $k$ (the morphism spaces are finite dimensional vector spaces), 
and that all functors are linear.
However, the large majority of our results do not require this linearity assumption.
The only place where this is really required is for statements involving semisimplicity of certain algebra objects (Theorems \ref{thm:SemiSimple} and \ref{thm:TwoAlgebrasPartTwo}).

\begin{defn}
A \underline{tensor cate}g\!\underline{\,or}y\!\!\underline{\,\,}\footnote{The adjectives `tensor' and `monoidal' are essentially synonymous.
The first one is usually only used when the category is linear; the second one is used regardless of whether the category is linear or not.} consists of the data $(\cC, 1, \otimes, \alpha, \lambda, \rho)$ where $\cC$ is a category, $1\in\cC$ is the distinguished unit object, $\otimes :\cC\boxtimes \cC\to \cC$ is a bifunctor, the associator $\alpha: (a\otimes b)\otimes c \to a\otimes (b\otimes c)$ is a natural isomorphism, and the left and right unitors $\lambda : 1\otimes a\to a$ and $\rho: a\otimes 1\to a$ are natural isomorphisms.
This data must satisfy the well known pentagon and triangle axioms.

A tensor category $\cC$ is called:\footnote{Alternative terminologies include \cite{MR2767048}: rigid = autonomous, pivotal = sovereign, ribbon = tortile.}
\begin{itemize}
\item
\underline{ri}g\!\underline{\,id} (or \underline{has duals})
if for every $a\in\cC$ there is an object $a^* \in \cC$, and there are maps $\coev_a : 1\to a\otimes a^*$ and $\ev_a: a^*\otimes a \to 1$ satisfying the zig-zag axioms
\begin{gather*}
(\id_a\otimes\ev_a)\circ(\coev_a\otimes\id_a)
= 
\id_a
\\
(\ev_a\otimes\id_{a^*})\circ(\id_{a^*}\otimes\coev_a)
= 
\id_{a^*}.
\end{gather*}
Moreover, for every $a\in \cC$, there should exist an object ${}^*\hspace{-.1mm}a\in \cC$ such that $({}^*\hspace{-.1mm}a)^*\cong a$.\\
\phantom{....} Note that being rigid is not data; it is just a property of the category.\footnote{
In any category, an object $a$ is said to \emph{have a dual} if there exist solutions of the zig-zag equations.
The dual object $a^*$ is then unique up to canonical isomorphism, and may thus be referred to as \emph{the} dual of $a$.}
Given $f:a\to b$ in $\cC$, we write $f^*:b^*\to a^*$ for
$(\ev_b\otimes\id_{a^*})\circ(\id_{b^*}\otimes f \otimes\id_{a^*})\circ(\id_{b^*}\otimes\coev_a)$.
\item
\underline{fusion} if $\cC$ is rigid, semisimple, $\cC(1,1)\cong k$, and there are only finitely many isomorphism classes of simple objects.

\item
p\!\!\!\underline{\,\,\,ivotal} if $\cC$ is rigid, and there is a monoidal natural isomorphism $\varphi$ from the identity functor to the double dual functor.%
\footnote{The original definition of Freyd and Yetter \cite{MR1154897} contains the redundant axiom $\varphi_{a^*}=(\varphi_a^{-1})^*$, later reproduced by other authors.
See \cite[Lem. 4.11]{MR2767048} for a proof of that property.
}
The left and right pivotal traces of a morphism $f:a\to a$ are then defined by
\begin{align*}
\tr_L(f)
&= 
\ev_{a} \circ (\id_{a^*}\otimes f)\circ (\id_{a^*}\otimes \varphi_a^{-1}) \circ \coev_{a^*}
\\
\tr_R(f) 
&= 
\ev_{a^*} \circ (\varphi_a \otimes \id_{a^*})\circ (f\otimes \id_{a^*}) \circ \coev_a.
\end{align*}

\item
\underline{s}p\!\underline{\,herical} if $\cC$ is pivotal and for every $c\in \cC$ and $f\in\cC(c,c)$, we have $\tr_L(f)=\tr_R(f)$.

\item
\underline{braided} if there is a family of natural isomorphisms $\beta_{a,b}: a\otimes b \to b\otimes a$ satisfying the two well known hexagon axioms.
We also write $\beta^+_{a,b}$ for $\beta_{a,b}$ and $\beta^-_{a,b}$ for $\beta_{b,a}^{-1}$.

\item
\underline{balanced} if $\cC$ is braided and there are twist isomorphisms $\theta_a:a\to a$ for $a\in \cC$, natural in $a$, satisfying 
$
\theta_{a\otimes b} = \beta_{b,a}\circ \beta_{a,b}\circ(\theta_a\otimes \theta_b)
$
for all $a,b\in\cC$.

\item
\underline{ribbon}%
\footnote{We warn the reader that \cite[Def. 2.2.1]{MR1797619} defines the notion of a braided pivotal category (equivalently, a balanced rigid category -- see Appendix \ref{sec:Lemmas}), rather than that of a ribbon category.
In a braided pivotal category, it is not necessarily the case that $\theta_{a^*}={\theta_a}^{\!\!*}$.
The same problem appears in \cite[Eq.\,(1.6)]{MR1936496}.
}
if $\cC$ is balanced and rigid, and the twist maps satisfy $\theta_{a^*}={\theta_a}^{\!\!*}$ for all $a\in\cC$.

%

\end{itemize}
\end{defn}

There exist graphical calculi for morphisms in various kinds of monoidal or tensor categories.
We may draw different types of diagrams and perform different types of isotopies based on the properties and structures of our category.
We give a helpful guide below (see \cite{MR2767048} for a comprehensive survey).
If $\cC$ is...
\begin{itemize}
\item
\underline{monoidal} (Joyal--Street \cite{MR1113284}),
we may draw our string diagrams so that strands only go up and down.
We read the diagram from bottom to top.
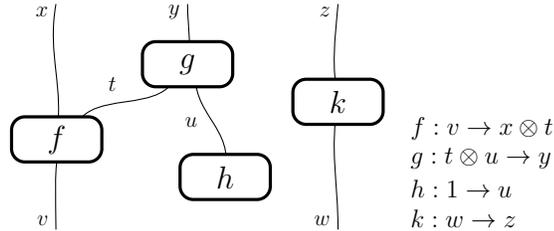
\[\,\,\,
\tikzmath{
\draw (.25,0) to[out=80, in=-100] (.25,1.8) node[scale=.7, left, yshift=-4]{$x$};
\draw (.25,-1.2) node[scale=.7, left, yshift=5]{$v$} to[out=90, in=-100] (.3,0);
\draw (.45,0) to[out=80, in=-100, looseness=1.2] (1.9,1);
\draw (2,1) to[out=80, in=-90] (2,1.8) node[scale=.7, left, yshift=-4]{$y$};
\draw (4,.5) to[out=100, in=-95] (4,1.8) node[scale=.7, left, yshift=-4]{$z$};
\draw (2.5,-.5) to[out=80, in=-95, looseness=1.2] (2.1,1);
\draw (4,-1.2) node[scale=.7, left, yshift=5]{$w$} to[out=90, in=-100] (4,.5);
\node[scale=.7] at (1,.73) {$t$};
\node[scale=.7] at (2.05,.23) {$u$};
\filldraw[fill=white, very thick, rounded corners=5pt]
(.26,0) node{$f$} +(-.6,-.3) rectangle +(.6,.3)
(2,1) node{$g$} +(-.6,-.3) rectangle +(.6,.3)
(2.5,-.5) node{$h$} +(-.6,-.3) rectangle +(.6,.3)
(4,.5) node{$k$} +(-.6,-.3) rectangle +(.6,.3);
\pgftransformxshift{10}
\pgftransformyshift{10}
\node[scale=.8, right] at (4.5,-.2) {$f:v\to x\otimes t$};
\node[scale=.8, right] at (4.5,-.6) {$g: t\otimes u \to y$};
\node[scale=.8, right] at (4.5,-1.0) {$h:1\to u$};
\node[scale=.8, right] at (4.5,-1.4) {$k:w\to z$};
}
\]

\item
\underline{ri}g\!\underline{\,id},
then we can draw any planar diagram. Strings may bend up and down,
however coupons are not allowed to rotate.
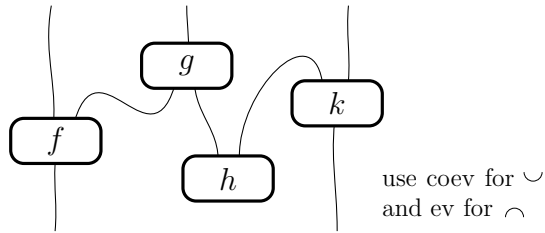
\[
\,\,\tikzmath{
\draw (.2,0) to[out=80, in=-100] (.2,1.8);
\draw (.25,-1.2) to[out=85, in=-93] (.25,-.3);
\draw (.45,0) to[out=80, in=-100, looseness=2.5] (1.9,1);
\draw (2,1) to[out=80, in=-90] (2,1.8);
\draw (4.15,.8) to[out=86, in=-95] (4.15,1.8);
\draw (2.4,-.5) to[out=80, in=-95, looseness=1.2] (2.1,1);
\draw (4,-1.2) to[out=90, in=-100] (4,.5);
\draw (2.7,-.2) to[out=90, in=100, looseness=1.5] (3.8,.8);
\filldraw[fill=white, very thick, rounded corners=5pt]
(.26,0) node{$f$} +(-.6,-.3) rectangle +(.6,.3)
(2,1) node{$g$} +(-.6,-.3) rectangle +(.6,.3)
(2.55,-.5) node{$h$} +(-.6,-.3) rectangle +(.6,.3)
(4,.5) node{$k$} +(-.6,-.3) rectangle +(.6,.3);
\node[scale=.8] at (5.8,-.7) {\parbox{3cm}{
use $\coev$ for $\tikzmath[scale=.15]{\useasboundingbox  (-1,1) -- (1,-1); \draw (-1,.8) arc (-180:0:1);}$\,\,\,\,\,\linebreak and $\ev$ for\, 
$\tikzmath[scale=.15]{\useasboundingbox  (-1,1) -- (1,-1); \draw (-1,-.8) arc (180:0:1);}$}};
}
\]

\item
p\!\!\!\underline{\,\,\,ivotal} (Freyd--Yetter \cite{MR1154897}),
we may rotate coupons by $2\pi$ without changing the morphism.
In defining the right hand side below, we must use the pivotal structure:
\[
\tikzmath{
\draw (.15,.3) to[out=80, in =-95] (.3,1.1);
\draw (-.15,.3) to[out=100, in =-85] (-.3,1.1);
\draw (0,0) -- (0,-1);
\filldraw[fill=white, very thick, rounded corners=5pt]
(0,0) node{$f$} +(-.4,-.3) rectangle +(.4,.3);
}\,\,\,\,
=
\tikzmath{
\clip (1.2,1.1) rectangle (-1.2,-1.1);
\draw (0,-.3) to[out=-100, in=-20] (-.3,-.5) to[out=180-20, in=180+20, looseness=1.25] (0,.68) to[out=20, in=90-15, looseness=1.4] (.8,-.2) to[out=-90-15, in=80] (-.3,-1.6);
\draw (.15,.3) to[out=100, in=180] (.25,.5) to[out=0, in=0, looseness=1.2] (-.1,-.69) to[out=180, in=-90-20, looseness=1.25] (-.7,.2) to[out=90-20, in=-100] (.4,1.3);
\draw (-.15,.3) to[out=80, in=180] (.25,.6) to[out=0, in=0, looseness=1.3] (-.11,-.82) to[out=180, in=-90-20, looseness=1.25] (-.85,.2) to[out=90-20, in=-100] (-.3,1.1);
\filldraw[fill=white, very thick, rounded corners=5pt]
(0,0) node{$f$} +(-.4,-.3) rectangle +(.4,.3);
}
\]

\item
\underline{s}p\!\underline{\,herical},
then the value of a closed diagram is invariant under spherical isotopy.
In particular:
$$
\tr_R(f)=
\begin{tikzpicture}[baseline=-.1cm]
	\roundNbox{unshaded}{(0,0)}{.3}{0}{0}{$f$}	
	\draw (0,.3) arc (180:0:.3cm) -- (.6,-.3) arc (0:-180:.3cm);
\end{tikzpicture}
=
\begin{tikzpicture}[baseline=-.1cm]
	\roundNbox{unshaded}{(0,0)}{.3}{0}{0}{$f$}	
	\draw (0,.3) arc (0:180:.3cm) -- (-.6,-.3) arc (-180:0:.3cm);
\end{tikzpicture}
=
\tr_L(f).
$$

\item
\underline{braided} (Joyal--Street \cite{MR1113284}),
then our diagrams are now in three dimensions, and we draw them projected to the plane.
If $\cC$ is not rigid, then coupons again may only travel up and down, and are not allowed to rotate.
\[
\tikzmath{
\draw (.3,0) to[out=80, in=-100] (.3,1.8);
\draw (.3,-1.2) to[out=90, in=-100] (.35,0);
\draw (.5,0) to[out=80, in=-100, looseness=1.2] (1.9,1);
\draw (2,1) to[out=80, in=-90] (2,1.8);
\draw (2.55,-.65) to[out=75, in=-165, looseness=1.2] (3.95,0);
\draw (4.25,.1) to[out=30, in=-95, looseness=1] (2.1,1);
\draw (2.35,-.5) to[out=80, in=-150, looseness=1.2] (3.4,.26);
\draw (4,.45) to[out=100, in=-95] (4,1.8);
\draw (4,-1.2) to[out=90, in=-80] (4.08,.2);
\draw (3.65,.5) to[out=50, in=-90, looseness=1.2] (3.8,1);
\filldraw[fill=white, very thick, rounded corners=5pt]
(.31,0) node{$f$} +(-.6,-.3) rectangle +(.6,.3)
(2,1) node{$g$} +(-.6,-.3) rectangle +(.6,.3)
(2.6,-.6) node{$h$} +(-.6,-.3) rectangle +(.6,.3)
(3.95,1) node{$k$} +(-.6,-.3) rectangle +(.6,.3);
\node[scale=.8] at (5.8,-.7) {\parbox{2.8cm}{
use $\beta^+$\, for\, 
$\tikzmath[scale=.15]{\draw (1,1) -- (-1,-1) (1,-1) -- (.4,-.4) (-1,1) -- (-.4,.4);}$
\,\,\,\,\,\linebreak and\, 
$\beta^-$ for $\tikzmath[scale=.15]{\draw (-1,1) -- (1,-1) (-1,-1) -- (-.4,-.4) (1,1) -- (.4,.4);}$
}};
}
\]

\item
\underline{balanced} (Shum \cite{MR1268782}),
then strands are replaced by ribbons, which can twirl on themselves, but cannot bend up and down.
Coupons may rotate around the $z$-axis.
\[
\,\,\tikzmath{
\draw[double, double distance = 2.5, line cap = rect] (.3,0) to[out=80, in=-95] (.3,1.8);
\draw[double, double distance = 2.5, line cap = rect] (.3,-1.2) to[out=90, in=-100] (.35,0);
\draw[double, double distance = 2.5, line cap = rect] (.5,0) to[out=80, in=-100, looseness=1.2] (1.9,1);
\draw[double, double distance = 2.5, line cap = rect] (2,1) to[out=80, in=-90] (2,1.8);
\draw[double, double distance = 2.5, line cap = butt] (2.6,-.5) to[out=80, in=180+30, looseness=.9] (4.25,.14);
\draw[double, double distance = 2.5, line cap = rect] (2.35,-.5) to[out=80, in=-90, looseness=1.2] (3.8,1);
\draw[double, double distance = 2.5, line cap = rect] (4.13,-1.2) to[out=85, in=-95] (4,1.8);
\draw[double, double distance = 2.5, line cap = butt] (4.25,.14) to[out=30, in=-95, looseness=.9] (2.1,1);
\fill[white] (4.25,.14) circle (.043);
\fill[white] (.22,1.35) rectangle (.385,1);
\draw (.342,1.35) to[out=-90, in=90] (.249,1.175);
\draw (.351,1.175) to[out=-90, in=90] (.263,1);
\draw[double, draw=white, double=black, double distance=.4, very thick] (.24,1.35) to[out=-90, in=90] (.351,1.175);
\draw[double, draw=white, double=black, double distance=.4, very thick] (.249,1.175) to[out=-90, in=90] (.365,1);
\pgftransformxshift{108.95}
\pgftransformyshift{-43}
\fill[white] (.22,1.35) rectangle (.385,1);
\draw (.342,1.35) to[out=-90, in=90] (.249,1.175);
\draw (.351,1.175) to[out=-90, in=90] (.263,1);
\draw[double, draw=white, double=black, double distance=.4, very thick] (.24,1.35) to[out=-90, in=90] (.351,1.175);
\draw[double, draw=white, double=black, double distance=.4, very thick] (.249,1.175) to[out=-90, in=90] (.365,1);
\pgftransformxshift{-108.95}
\pgftransformyshift{43}
\filldraw[fill=white, very thick, rounded corners=5pt]
(.31,0) node{$f$} +(-.6,-.3) rectangle +(.6,.3)
(2,1) node{$g$} +(-.6,-.3) rectangle +(.6,.3)
(2.55,-.5) node{$h$} +(-.6,-.3) rectangle +(.6,.3)
(3.95,1) node{$k$} +(-.6,-.3) rectangle +(.6,.3);
\node[scale=.8] at (5.9,-.7) {\parbox{2.5cm}{
use\;\! $\theta$\;\! for\;\! 
$\tikzmath[scale=1.2, line cap = rect]{
\draw (0,.175) to[out=-90, in=90] (.1,0);
\draw (0,.35) to[out=-90, in=90] (.1,.175);
\draw[double, draw=white, double=black, double distance=.4, very thick, line cap = butt] (.1,.175) to[out=-90, in=90] (0,0);
\draw[double, draw=white, double=black, double distance=.4, very thick, line cap = butt] (.1,.35) to[out=-90, in=90] (0,.175);
\draw (.1,.35) -- (.1,.5) -- (0,.5) -- (0,.35);
\draw (.1,0) -- (.1,-.15) -- (0,-.15) -- (0,0);}
$
\,\,\,\,\,\linebreak and\, 
$\theta^{-1}$\;\! for\;\! 
$\tikzmath[scale=1.2, line cap = rect]{
\draw (.1,.175) to[out=-90, in=90] (0,0);
\draw (.1,.35) to[out=-90, in=90] (0,.175);
\draw[double, draw=white, double=black, double distance=.4, very thick, line cap = butt] (0,.175) to[out=-90, in=90] (.1,0);
\draw[double, draw=white, double=black, double distance=.4, very thick, line cap = butt] (0,.35) to[out=-90, in=90] (.1,.175);
\draw (.1,.35) -- (.1,.5) -- (0,.5) -- (0,.35);
\draw (.1,0) -- (.1,-.15) -- (0,-.15) -- (0,0);}
$
}};}
\]

\item
\underline{braided }p\!\underline{\,ivotal} (Freyd--Yetter \cite{MR1154897}), then the second and third Reidemeister moves are allowed, but not the first.
Strands behave like ribbons pressed flat against the plane.

\newcommand{\VarLoopIsoReverse}[1]{
	\fill[unshaded] ($ #1 - (.3,.3) $) rectangle ($ #1 + (.1,.3) $);
	\draw ($ #1 + (-.3,.2) $) arc (90:270:.2cm);
	\draw ($ #1 + (0,.3) $)  .. controls ++(270:.2cm) and ++(0:.2cm) .. ($ #1 + (-.3,-.2) $);
	\draw[line width=5, white]	 ($ #1 + (0,-.3) $)  .. controls ++(90:.2cm) and ++(0:.2cm) .. ($ #1 + (-.3,.2) $);
	\draw ($ #1 + (0,-.3) $)  .. controls ++(90:.2cm) and ++(0:.2cm) .. ($ #1 + (-.3,.2) $);
}
\newcommand{\VarLoopIsoInverse}[1]{
	\fill[unshaded] ($ #1 - (.1,.3) $) rectangle ($ #1 + (.3,.3) $);
	\draw ($ #1 + (.3,.2) $) arc (90:-90:.2cm);
	\draw ($ #1 + (0,-.3) $)  .. controls ++(90:.2cm) and ++(180:.2cm) .. ($ #1 + (.3,.2) $);
	\draw[line width=5, white] ($ #1 + (0,.3) $)  .. controls ++(270:.2cm) and ++(180:.2cm) .. ($ #1 + (.3,-.2) $);
	\draw ($ #1 + (0,.3) $)  .. controls ++(270:.2cm) and ++(180:.2cm) .. ($ #1 + (.3,-.2) $);
}
\[
\quad\,\,\,\tikzmath{
\draw (.3,0) to[out=80, in=-100] (.3,1.8);
\draw (.3,-1.2) to[out=90, in=-100] (.35,0);
\draw (.45,0) to[out=80, in=-100, looseness=2.5] (1.9,1);
\draw (2,1) to[out=80, in=-90] (2,1.8);
\draw (2.55,-.65) to[out=75, in=-165, looseness=1.2] (3.95,0);
\draw (4.25,.1) to[out=30, in=-95, looseness=1] (2.1,1);
\draw (2.35,-.5) to[out=80, in=-150, looseness=1.2] (3.4,.26);
\draw (4,.45) to[out=100, in=-95] (4,1.8);
\draw (4,-1.2) to[out=90, in=-80] (4.08,.2);
\draw (3.65,.5) to[out=50, in=-90, looseness=1.2] (3.8,1);
\filldraw[fill=white, very thick, rounded corners=5pt]
(.31,0) node{$f$} +(-.6,-.3) rectangle +(.6,.3)
(2,1) node{$g$} +(-.6,-.3) rectangle +(.6,.3)
(2.6,-.6) node{$h$} +(-.6,-.3) rectangle +(.6,.3)
(3.95,1) node{$k$} +(-.6,-.3) rectangle +(.6,.3);
\pgftransformscale{.7}
\pgftransformrotate{3.35}
	\VarLoopIsoReverse{(.496,1.7)}
\pgftransformrotate{-3.35}
\pgftransformscale{1/.7}
\pgftransformscale{.7}
\pgftransformrotate{-5.5}
	\VarLoopIsoInverse{(5.861,.076)}
\pgftransformrotate{5.5}
\pgftransformscale{1/.7}
\node[scale=.8] at (6.3,-.7) {\parbox{3.3cm}{
in general\,
$
\begin{tikzpicture}[baseline=-.1cm, scale=.55]
	\draw (0,-.6) -- (0,.6);
	\loopIsoInverseReverse{(0,0)}
\end{tikzpicture}
\;\!\not =\;\!
\begin{tikzpicture}[baseline=-.1cm, scale=.55]
	\draw (0,-.6) -- (0,.6);
	\loopIso{(0,0)}
\end{tikzpicture}
$}};
}
\]
There are two ways of making such a category into a balanced one, by letting $\theta$ be either
$
\begin{tikzpicture}[baseline=-.1cm, scale=.5]
	\draw (0,-.6) -- (0,.6);
	\loopIsoInverseReverse{(0,0)}
\end{tikzpicture}
$
or
$
\begin{tikzpicture}[baseline=-.1cm, scale=.5]
	\draw (0,-.6) -- (0,.6);
	\loopIso{(0,0)}
\end{tikzpicture}
$\, (see Lemma \ref{lem:CenterPivotalBalanced} in the appendix).

\item
\underline{ribbon} (Shum \cite{MR1268782}),
then the stands are replaced by ribbons, and all three dimensional isotopies are allowed.
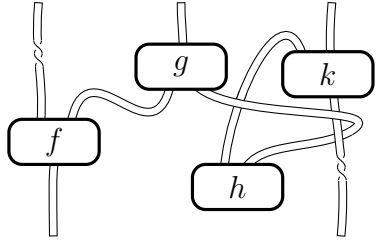
\[
\,\,\tikzmath{
\draw[double, double distance = 2.5, line cap = rect] (.1,0) to[out=80, in=-95] (.1,1.8);
\draw[double, double distance = 2.5, line cap = rect] (.3,-1.2) to[out=90, in=-100] (.35,0);
\draw[double, double distance = 2.5, line cap = rect] (.5,0) to[out=85, in=-95, looseness=2.5] (1.9,1.05);
\draw[double, double distance = 2.5, line cap = rect] (2,1) to[out=80, in=-90] (2,1.8);
\draw[double, double distance = 2.5, line cap = butt] (2.8,-.5) to[out=80, in=180+30, looseness=.9] (4.25,.14);
\draw[double, double distance = 2.5, line cap = rect] (2.55,-.5) to[out=83, in=-110, looseness=1] (2.8,.6);
\draw[double, double distance = 2.5, line cap = rect] (2.8,.6) to[out=70, in=110, looseness=1.5] (3.6,1.2);
\draw[double, double distance = 2.5, line cap = rect] (4.13,-1.2) to[out=85, in=-95] (4,1.8);
\draw[double, double distance = 2.5, line cap = butt] (4.25,.14) to[out=30, in=-95, looseness=.9] (2.1,1);
\fill[white] (4.25,.14) circle (.043);
\fill[white] (.02,1.35) rectangle (.185,1);
\draw (.142,1.35) to[out=-90, in=90] (.049,1.175);
\draw (.151,1.175) to[out=-90, in=90] (.063,1);
\draw[double, draw=white, double=black, double distance=.4, very thick] (.04,1.35) to[out=-90, in=90] (.151,1.175);
\draw[double, draw=white, double=black, double distance=.4, very thick] (.049,1.175) to[out=-90, in=90] (.165,1);
\pgftransformxshift{108.95}
\pgftransformyshift{-43}
\fill[white] (.22,1.35) rectangle (.385,1);
\draw (.342,1.35) to[out=-90, in=90] (.249,1.175);
\draw (.351,1.175) to[out=-90, in=90] (.263,1);
\draw[double, draw=white, double=black, double distance=.4, very thick] (.24,1.35) to[out=-90, in=90] (.351,1.175);
\draw[double, draw=white, double=black, double distance=.4, very thick] (.249,1.175) to[out=-90, in=90] (.365,1);
\pgftransformxshift{-108.95}
\pgftransformyshift{43}
\filldraw[fill=white, very thick, rounded corners=5pt]
(.31,0) node{$f$} +(-.6,-.3) rectangle +(.6,.3)
(2,1) node{$g$} +(-.6,-.3) rectangle +(.6,.3)
(2.75,-.6) node{$h$} +(-.6,-.3) rectangle +(.6,.3)
(3.95,.9) node{$k$} +(-.6,-.3) rectangle +(.6,.3);
}
\qquad\quad\,\,\,
\]

\end{itemize}

Given a functor between categories with extra structure, one can ask for the functor to be compatible with that structure.
We list the relevant conditions for the notions mentioned above.
Let $\cC$ and $\cD$ be tensor categories, and let $F:\cC \to \cD$ be a functor. 
\begin{itemize}
\item
We say that $F$ is a \underline{tensor functor} if it comes equipped with 
an invertible\footnote{
If $\nu$ and $i$ are not assumed invertible, then this is called a \emph{lax tensor functor}.
We warn the reader that some people use `strong tensor functor' for tensor functors, and `tensor functor' for lax tensor functors.
}
 natural transformation $\nu_{x,y}:F(x)\otimes F(y)\to F(x\otimes y)$ and an isomorphism $i:1_\cD\to F(1_\cC)$,
subject to
\begin{equation}\label{eq: def: tensor functor}
\begin{split}
\nu_{x,y\otimes z}\circ(1_{F(x)}\otimes\nu_{y,z})\circ \alpha_{F(x),F(y),F(z)}  = F(\alpha_{x,y,z})\circ \nu_{x\otimes y, z}\circ (\nu_{x,y}\otimes 1_{F(z)}),
\\
F(\lambda_x) \circ \nu_{1,x}\circ (i\otimes 1_{F(x)})=\lambda_{F(x)},
\quad\,\,\,\text{and}\quad\,\,\,
 F(\rho_x)\circ \nu_{x,1}\circ(1_{F(x)}\otimes i)=\rho_{F(x)}.
\end{split}
\end{equation}
\item
If $\cC$ and $\cD$ are braided, then $F$ is a \underline{braided functor} if the condition $F(\beta)=\nu \circ \beta\circ \nu^{-1}$ is satisfied.
\item
If $\cC$ and $\cD$ are balanced, then $F$ is a \underline{balanced functor} if we also have $F(\theta)=\theta$.
\end{itemize}
\noindent
If $\cC$ is rigid and $F:\cC\to\cD$ is a tensor functor, then every object in the essential image of $F$ has a dual, and there is a canonical isomorphism $\delta_x:F(x^*)\to F(x)^*$,
characterised by $F(\ev_x) = i \circ \ev_{F(x)}\circ(\delta_x\otimes \id_{F(x)})\circ \nu_{x^*,x}^{-1}$ (this requires no extra data).
\begin{itemize}
\item If $\cC$ and $\cD$ are pivotal, then a p\!\!\!\underline{\,\,\,ivotal functor} is a tensor functor satisfying the extra axiom $F(\varphi_x)=\delta_{x^*}^{-1} \circ \delta_x^* \circ \varphi_{F(x)}$.
\end{itemize}

\subsection{The Drinfel'd center}
The Drinfel'd center construction takes as input a tensor category and produces as output a braided tensor category \cite{MR1107651}.
It is a categorification of the operation of taking the center of a ring.

\begin{defn}
Let $\cC$ be a tensor category.
Its center $\cZ(\cC)$ is the tensor category whose objects are pairs $(a, e_a)$ where $e_a=e_{a,\scriptscriptstyle \bullet}$ is a half-braiding for $a\in \cC$.
A half-braiding for $a\in\cC$ is a family of isomorphisms 
$$
e_{a,b}\,
=
\begin{tikzpicture}[baseline=-.1cm]
	\draw (0,-.4)  .. controls ++(90:.3cm) and ++(270:.3cm) .. (-.6,.4);
	\draw[super thick, white] (-.6,-.4)  .. controls ++(90:.3cm) and ++(270:.3cm) .. (0,.4);
	\draw (-.6,-.4)  .. controls ++(90:.3cm) and ++(270:.3cm) .. (0,.4);
	\node at (-.8,-.3) {\scriptsize{$a$}};
	\node at (.2,-.3) {\scriptsize{$b$}};
\end{tikzpicture}
:a\otimes b \to b\otimes a
$$ 
which is natural with respect to $b$, and satisfies the hexagon axiom
\begin{equation}\label{eq: hexagon identity}
(\id_{b} \otimes e_{a,c})\circ \alpha \circ (e_{a,b} \otimes \id_{c}) = \alpha\circ e_{a,b\otimes c}\circ \alpha.
\end{equation}
The morphisms in $\cZ(\cC)$ from $(a,e_a)$ to $(b,e_b)$ are all the $f\in\cC(a,b)$ which ``pass over all crossings'', i.e., such that for all $c\in\cC$,
$$
e_{b,c}\circ(f\otimes \id_c) = (\id_c\otimes f) \circ e_{a,c}.
$$

Note that $\cZ(\cC)$ is always braided by defining $\beta_{(a, e_a) (b, e_b)} =e_{a,b}$.
In fact, $\cZ(\cC)$ retains many of the properties and structures from $\cC$.
For example, if $\cC$ is rigid, fusion, pivotal, or spherical, then so is $\cZ(\cC)$ (e.g., see \cite[Thm. 2.15]{MR2183279} and  \cite[Prop. 3.9]{MR1966525}).
\end{defn}

The proof of pivotality of $\cZ(\cC)$ in \cite[Prop. 3.9]{MR1966525} assumes $\cC$ to be a \emph{strict} pivotal category.
We are not aware of a reference that proves the result in full generality, so we provide a short argument of the most non-trival part of the proof:

\begin{prop}\label{prop: communicated by Noah Snyder}
Let $\cC$ be a rigid tensor category.
Then a pivotal structure on $\cC$ induces a pivotal structure on $\cZ(\cC)$.
\end{prop}

\begin{proof}[Proof \rm(communicated by Noah Snyder).]
The dual of an object $(a,e_a)\in \cZ(\cC)$ is given by $(a^*,e_{a^*})$, with $e_{a^*,b}:=e_{a,{}^*\hspace{-.1mm}b}^*$.
The somewhat non-trivial claim is that the map $\varphi_a:a\to a^{**}$ provided by the pivotal structure of $\cC$ defines a morphism $(a,e_a)\to (a,e_a)^{**}$ in $\cZ(\cC)$.
We need to show that $(\id_b\otimes \varphi_a)\circ e_{a,b} = e_{a^{**},b}\circ (\varphi_a\otimes \id_b)$ for every $b\in\cC$.
This is done as follows:
\[
\begin{split}
e_{a^{**},b}\circ (\varphi_a\otimes \id_b) &= 
e_{a,{}^{**}\hspace{-.1mm}b}^{**}\circ \varphi_{a\;\!\otimes {}^{**}\hspace{-.2mm}b}\circ (\id_a\otimes \varphi_{{}^{**}\hspace{-.2mm}b}^{-1}) \\&=
\varphi_{{}^{**}\hspace{-.2mm}b\;\!\otimes\;\! a}\circ e_{a,{}^{**}\hspace{-.1mm}b} \circ (\id_a\otimes \varphi_{{}^{**}\hspace{-.2mm}b}^{-1}) \\&=
\varphi_{{}^{**}\hspace{-.2mm}b\otimes a}\circ(\varphi_{{}^{**}\hspace{-.2mm}b}^{-1}\otimes\id_a)\circ e_{a,b} =(\id_b\otimes \varphi_a)\circ e_{a,b}
.\qedhere
\end{split}
\]
\end{proof}



\subsection{Synoptic chart of tensor categories}
\label{sec: chart of categories}
We present here a chart (Figure~\ref{fig:SynopticChart}) that summarizes the various notions of tensor category.
The two ways of going between balanced rigid categories and braided pivotal categories will be discussed in Appendix \ref{sec:Lemmas}.

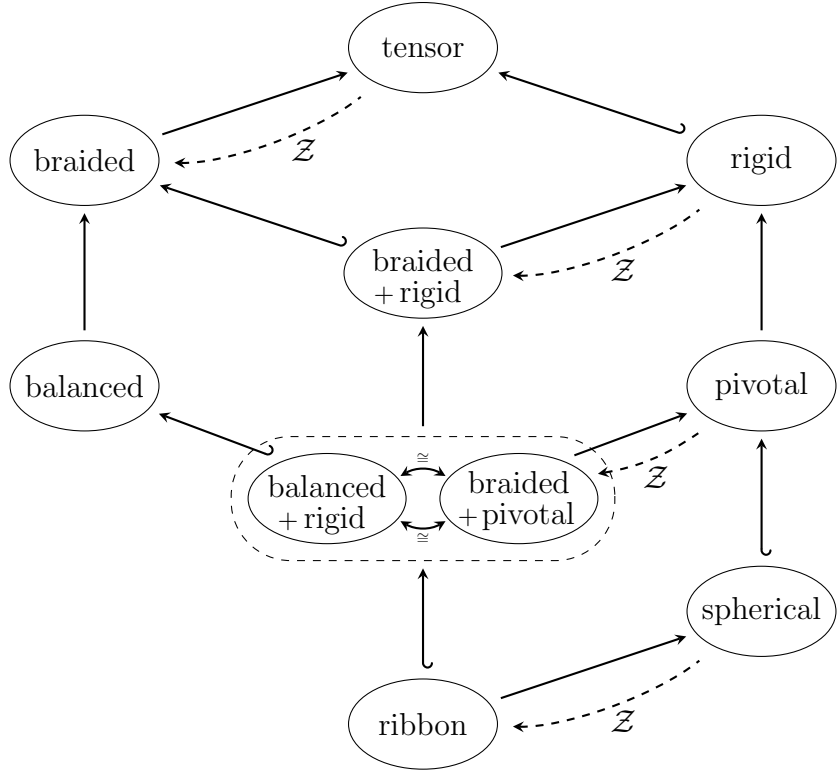
\begin{figure}[!ht]
\[
\begin{tikzpicture}[scale=1.5]
	\draw (3,6) coordinate (A) circle (.66 and .4) node {tensor}; 
	\draw (0,5) coordinate (B) circle (.66 and .4) node {braided}; 
	\draw (6,5) coordinate (C) circle (.66 and .4) node {rigid}; 
	\draw (3,4) coordinate (D) circle (.7 and .4) +(0,-.05) node {\parbox{2cm}{\centerline{braided}\vspace{-.1cm} \centerline{{\footnotesize+}\,rigid\,\,}}}; 
	\draw (0,3) coordinate (E) circle (.66 and .4) node {balanced}; 
	\draw (6,3) coordinate (F) circle (.66 and .4) node {pivotal}; 
	\draw (2.15,2) coordinate (G) circle (.7 and .4) +(0,-.05) node {\parbox{2cm}{\centerline{balanced}\vspace{-.1cm} \centerline{{\footnotesize+}\,rigid\,\,}}}; 
	\draw (3.85,2) coordinate (H) circle (.7 and .4) +(0,-.02) node {\parbox{2cm}{\centerline{braided}\vspace{-.1cm} \centerline{{\footnotesize+}\;\!pivotal\,}}}; 
	\draw (6,1) coordinate (I) circle (.66 and .4) node {spherical}; 
	\draw (3,0) coordinate (J) circle (.66 and .4) node {ribbon}; 
\begin{scope}[thick,-stealth, shorten >=30, shorten <=31]
	\draw (B) -- (A);
	\draw (D) -- (C);
	\draw (H)+(-.19,.13) -- (F);
	\draw (J) -- (I);
\end{scope}
\begin{scope}[dashed, thick, stealth-, shorten >=30, shorten <=34]
	\draw (B) to[bend right=20]node[below, pos=.58, xshift=5]{$\cZ$} (A);
	\draw (D) to[bend right=20]node[below, pos=.52, xshift=5]{$\cZ$} (C);
	\draw (J) to[bend right=20]node[below, pos=.52, xshift=5]{$\cZ$} (I);
\end{scope}
\draw[dashed, thick, stealth-] (4.55,2.22) to[bend right=12]node[below, xshift=2, yshift=1.9]{$\cZ$} (5.44,2.58);
\begin{scope}[thick,-stealth, shorten >=20, shorten <=21]
	\draw (E) -- (B);
	\draw (F) -- (C);
	\draw[left hook-stealth] (I) -- (F);
	\draw[left hook-stealth] (J) -- (3,1.85);
	\draw (3,2.15) --(D);
\end{scope}
\begin{scope}[thick,left hook-stealth, shorten >=30, shorten <=31]
	\draw (C) -- (A);
	\draw (D) -- (B);
	\draw (G)+(.19,.12) -- (E);
\end{scope}
\draw[thick,stealth-stealth] (3,2) +(-.2,.2) to[bend left=30]node[above, scale=.5]{$\cong$} +(.2,.2);
\draw[thick,stealth-stealth] (3,2) +(-.2,-.2) to[bend right=30]node[below, scale=.5]{$\cong$} +(.2,-.2);
\draw[dashed, rounded corners=22.5] (3,2) +(1.7,.55) rectangle +(-1.7,-.55);
\end{tikzpicture}
\]
\caption{Synoptic chart of tensor categories}
\label{fig:SynopticChart}
\end{figure}

In Figure~\ref{fig:SynopticChart}, we use the following notations:
\begin{itemize}
\item
An arrow 
$\tikz[baseline=-.1cm]{\node[draw, ellipse, inner ysep=2] (A) at (0,0) {A};\node[draw, ellipse, inner ysep=2] (B) at (1.5,0) {B};\draw[thick, -stealth, shorten >=2, shorten <=2] (A) -- (B);}$
indicates that notion B can be obtained from notion A by forgetting part of the data.\medskip
\item
An arrow 
$\tikz[baseline=-.1cm]{\node[draw, ellipse, inner ysep=2] (A) at (0,0) {A};\node[draw, ellipse, inner ysep=2] (B) at (1.5,0) {B};\draw[thick, right hook-stealth, shorten >=2, shorten <=2] (A) -- (B);}$
indicates that notion A can be obtained from notion B by imposing extra axioms.
\item
A dashed arrow
$\tikz[baseline=-.1cm]{\node[draw, ellipse, inner ysep=2] (A) at (0,0) {A};\node[draw, ellipse, inner ysep=2] (B) at (1.5,0) {B};\draw[thick, -stealth, dashed, shorten >=2, shorten <=2] (A) --node[above]{$\scriptstyle \cZ$} (B);}$
indicates that the Drinfel'd center construction goes from notion A to notion B.\medskip
\item
A double arrow 
$\tikz[baseline=-.1cm]{\node[draw, ellipse, inner ysep=2] (A) at (0,0) {A};\node[draw, ellipse, inner ysep=2] (B) at (1.5,0) {B};\draw[thick, stealth-stealth, shorten >=2, shorten <=2] (A) --node[above, scale=.5]{$\cong$} (B);}$
indicates an equivalence between notions A and B.
\end{itemize}

\noindent
As mentioned above, there are \emph{two} ways of going between balanced rigid categories and braided pivotal categories.
To avoid any confusion, we will only use \emph{one} of those two equivalences in the remainder of this paper.
Specifically, given a braided pivotal category $(\cC, \beta,\varphi)$, we will always equip it with the twists given by
\begin{equation}\label{def: theta1 -- PRE}
\theta_{a}
:= 
(\id_a\otimes \ev_{a^*})
\circ
(\beta_{a^{**},a}\otimes\id_{a^*})
\circ
(\id_{a^{**}}\otimes \coev_a)
\circ
\varphi_a
=
\begin{tikzpicture}[rotate=180, baseline=.35cm]
	\draw (0,-1.6) -- (0,.8);
	\loopIso{(0,-1)}
	\roundNbox{unshaded}{(0,0)}{.35}{0}{0}{$\varphi_a$}	
	\node at (-.2+.4,.6) {\scriptsize{$a$}};
	\node at (-.2+.4,-1.4) {\scriptsize{$a$}};
	\node at (-.3+.55,-.55) {\scriptsize{$a^{**}$}};
\end{tikzpicture}
\end{equation}
See Appendix \ref{sec:Lemmas} for more details.


\subsection{Algebra objects}
\label{sec:Frobenius}

We now discuss algebra objects in tensor categories. 
General references include \cite[\S 1]{MR1936496}, \cite[\S 2.4]{MR3039775}, \cite[\S 3]{MR1976459}, and \cite{MR1966524}.

\begin{defn}
An algebra object in a tensor category $\cC$ is a triple $(A,m:A\otimes A\to A,i:1\to A)$ which satisfies the following axioms:
\begin{align*}
&m( m\otimes \id_A) =  m(\id_A \otimes  m)
&&\text{(Associativity)}
\\
& m(i\otimes \id_A)=\id_A =  m(\id_A\otimes\, i)
&&\text{(Unitality)}
\end{align*}
Dually, a coalgebra is a triple $(A,\Delta:A\to A\otimes A,\epsilon:A\to 1)$ satisfying analogous coassociativity and counitality axioms.
\end{defn}

We represent the multiplication $ m$ and unit $i$ by a trivalent, respectively univalent, vertex
$$
 m_A \,=\! 
\begin{tikzpicture}[baseline=-.1cm]
	\coordinate (a) at (0,0);
	\filldraw (a) circle (.05cm);
	\draw (-.3,-.4) -- (a) -- (.3,-.4);
	\draw (a) -- (0,.4);
	\node at (-.45,-.3) {\scriptsize{$A$}};
	\node at (.45,-.3) {\scriptsize{$A$}};
	\node at (.15,.3) {\scriptsize{$A$}};
\end{tikzpicture}
\qquad\qquad
i_A\,=\,
\begin{tikzpicture}[baseline=-.1cm]
	\coordinate (a) at (0,-.1);
	\filldraw (a) circle (.05cm);
	\draw (a) -- (0,.3);
	\node at (.15,.2) {\scriptsize{$A$}};
\end{tikzpicture}\,.
$$
An algebra object is called \emph{connected} if $\cC(1, A)$ is one dimensional.
If the ambient category $\cC$ is semisimple, then an algebra object is:
\begin{itemize}
\item
\emph{semisimple} if the category $\Mod_\cC(A)$ of $A$-modules in $\cC$ is semisimple,
\item
\emph{separable} if $m_A: A\otimes A\to A$ splits as a morphism of $A$-$A$-bimodules.
\end{itemize}
At first glance, the notions of semisimple and separable algebras do not seem very related.
However, a special case of the following proposition (namely setting $B=1$) shows that separable always implies semisimple:

\begin{prop}
\label{prop: ostrik bimod semisimple}
Let $A$ and $B$ be separable algebras in a semisimple tensor category $\cC$.
Then the category $\Bimod_{\cC}(A,B)$ of $A$-$B$-bimodules is semisimple.
\end{prop}
\begin{proof}[Proof \rm (communicated by Victor Ostrik)]
Let $s_A : A \to A \otimes A$ and $s_B: B \to B \otimes B$ be bimodule maps that split $m_A$ and $m_B$.
Given $M\in\Bimod_{\cC}(A,B)$, the action map $A \otimes M \otimes B \to M$ admits a splitting 
$$
M \cong A \otimes_A M \otimes_B B \xrightarrow{s_A \otimes \id \otimes s_B}  (A \otimes  A) \otimes_A M \otimes_B  (B \otimes B) \cong A \otimes M \otimes B,
$$
which is furthermore a bimodule map.
$M$ is therefore isomorphic to a direct summand of $A \otimes M \otimes B$.
As the latter is projective, so is $M$.
Every object of $\Bimod_\cC(A,B)$ is projective, every exact sequence splits, and the category is semisimple.
\end{proof}

We also get the following well-known equivalent definition of separability:

\begin{prop}[{\cite[Prop. 2.7]{MR3039775}}]
\label{prop:EquivalentSeparabilityDefinition}
An algebra object $A$ in a semisimple tensor category $\cC$ is separable if and only if $\Bimod_{\cC}(A,A)$ is semisimple.
\end{prop}
\begin{proof}
The multiplication map $m_A: A\otimes A\to A$ admits a right inverse (given by $\id_A\otimes\, i_A$) and is thus always an epimorphism.
If $\Bimod_{\cC}(A,A)$ is semisimple, $m_A$ therefore splits.
The other direction follows from the previous proposition by setting $B=A$.
\end{proof}

\begin{rem}
\label{rem:WhenSimpleImpliesSeparable}
In positive characteristic there exist semisimple algebras which are not separable, e.g., 
the group algebra $\mathbb F_p[G]$ viewed as a $G$-graded vector space, for some $p$-group~$G$.

However, if $\cC$ is a fusion category over a field of characteristic zero, then algebra objects in $\cC$ are semisimple if and only if they are separable.
To see that, consider a semisimple algebra object $A$, which we assume without loss of generality to be simple (not a direct sum of smaller algebras).
Then $\cM:=\Mod_\cC(A)$ is an indecomposable module category, and
the category $\mathsf{Func}_\cC(\cM,\cM)=\Bimod_\cC(A,A)$ 
of $\cC$-linear endofunctors of $\cM$ is semisimple by \cite[Lem. 2.15]{MR2183279}.
\end{rem}

We now briefly introduce the notion of a Frobenius algebra object: 

\begin{defn}\label{def:FrobAlg}
A Frobenius algebra in a tensor category $\cC$ is a quintuplet $(A, m,\Delta,i,\epsilon)$ such that 
$(A, m,i)$ is an algebra,
$(A,\Delta,\epsilon)$ is a coalgebra,
and the following compatibility condition is satisfied:
\begin{align*}
( m\otimes \id_A)\circ(\id_A\otimes \Delta)
&=
\Delta\circ  m
=
(\id_A \otimes  m)\circ(\Delta \otimes \id_A)
&&\text{(Frobenius)}
\end{align*}
\end{defn}

It is interesting to note that a Frobenius algebra is entirely determined by its underlying algebra $(A, m,i)$ and counit $\epsilon$.
Moreover, for a given algebra $A=(A, m,i)$, a morphism $\epsilon:A\to 1$ is the counit of a Frobenius structure if and only if the pairing $\epsilon\circ  m : A\otimes A\to 1$ is \emph{non-degenerate}
(a pairing $p:A\otimes B\to 1$ is non-degenerate if $A$ and $B$ are dualizable and $(p\otimes \id_{B^*})\circ(\id_A \otimes \coev_B) : A\to B^*$ is an isomorphism):
\begin{prop}[{\cite[Lem.\,3.7]{MR1940282}, \cite[Prop.\,8]{MR2500035}}]
\label{prop   :Frobenius}
Let $\cC$ be a tensor category, and let $(A, m, i)$ be an algebra object in $\cC$.
Let $\epsilon: A\to 1$ be such that the pairing $\epsilon\circ  m$ is non-degenerate.
Then, there exists a unique comultiplication $\Delta:A\to A\otimes A$ making the quintuplet $(A, m,\Delta,i,\epsilon)$ into a Frobenius algebra.
\end{prop}

\section{Categorified traces for module tensor categories}
\label{sec: categorified trace for module tensor categories}

In this section, we explain the general notion of categorified trace, along with some basic properties thereof.
We then explain how to construct a categorified trace $\Tr_\cC:\cM\to \cC$ from a module tensor category $\cM$ over a braided tensor category $\cC$.

\subsection{Categorified traces}
\label{sec: categorified traces}
Recall that a trace on an algebra is a linear map to the ground field satisfying $\operatorname{tr}(xy) = \operatorname{tr}(yx)$. The categorification of the notion of algebra is that of tensor category, where notably the associativity constraint gets replaced by the additional data of associators.
The notion of a \emph{categorified trace} (also known as a \emph{commutator functor} \cite[\S 6]{MR2506324}, \cite[Def. 2.1]{MR3250042}) is defined in the same spirit:
\begin{defn}\label{def:  categorified trace}
Let $\cM$ be a tensor category, and let $\cC$ be a category.
A \emph{categorified trace} is a functor $\Tr:\cM\to \cC$ along with natural isomorphisms
\[
\tau_{x,y}:\Tr(x \otimes y) \to \Tr(y \otimes x)\qquad\quad x,y \in \cM
\]
which we call the \emph{traciators}, subject to the following axiom:
\begin{equation}\label{eq: axiom for traciator}
\tau_{x, y \otimes z} \,=\, \tau_{z \otimes x, y} \circ \tau_{x \otimes y, z}.
\end{equation}
Here, we have suppressed the associators for readability.
The axiom including the necessary associators reads as follows: $\Tr(\alpha_{y,z,x})\circ \tau_{x, y \otimes z}\circ \Tr(\alpha_{x,y,z}) = \tau_{z \otimes x, y} \circ \Tr(\alpha_{z,x,y})^{-1} \circ \tau_{x \otimes y, z}$.
\end{defn}

We will sometimes write $\tau^+_{x,y}$ for $\tau_{x,y}$, and $\tau^-_{x,y}$ for $\tau_{y,x}^{-1}$.
The latter is called the inverse traciator.
We point out that $\tau^-$ does not in general satisfy equation \eqref{eq: axiom for traciator}, but instead a mirrored version of that axiom.

\begin{rem}\label{rem: rant No1}
A slightly less general notion of categorified trace was studied in \cite{MR3095324} under the name ``shadow''.
It differs from Definition \ref{def:  categorified trace} in that the authors also included the axioms\footnote{We have again suppressed the coherences for readability.
The actual axioms are $\Tr(\lambda_x) \circ \tau_{x,1}=\Tr(\rho_x)$ and $\Tr(\rho_x) \circ \tau_{1,x}=\Tr(\lambda_x)$.}
\[
\tau_{x,1}=\id \qquad\quad\text{and}\qquad\quad \tau_{1,x}=\id.
\]
The first axiom above is redundant: see Lemma~\ref{lem: Ponto-Shulman axiom redundant} below.
The second axiom is not satisfied by the trace functors we will consider (see Remark \ref{rem: rant No2}),
and we prefer not to include it in the general definition of a categorified trace.
\end{rem}

\begin{lem}[{\cite[Rem. 2.2.]{MR3250042}}]\label{lem: Ponto-Shulman axiom redundant}
Let $(\Tr,\tau)$ be a categorified trace. Then $\tau_{x,1}=\id_{\Tr(x)}$.
\end{lem}
\begin{proof}
Omitting unitors and associators, we have $\tau_{x,1}=\tau_{x,1\otimes 1}=\tau_{1\otimes x, 1}\circ \tau_{x\otimes1, 1}=\tau_{x, 1}\circ \tau_{x,1}$, where the second equality holds by Equation \eqref{eq: axiom for traciator}.
As $\tau_{x,1}$ is invertible, it follows that $\tau_{x,1}=\id_{\Tr(x)}$.
\end{proof}

In Section \ref{sec:  Tensor categories} we have reviewed the classical string diagram notation for objects and morphisms in tensor categories.
As explained in \cite{MR3095324}, a good graphical notation for categorified traces is to take the strands that represent the objects of $\cM$, and place them on the surface of a cylinder.
For example, if we denote $x\in \mathcal{M}$ as a strand on a horizontal plane, then we represent $\Tr(x)\in\cC$ as the same strand on the surface of a cylinder:
$$
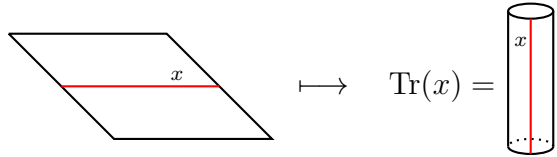
\begin{tikzpicture}[baseline=-.1cm, scale=.7]
	\plane{(0,-1)}{3}{2}
	\draw[thick, xString] (-1,0) -- (2,0);
\node at (1.2,.2) {$\scriptstyle x$};
\end{tikzpicture}
\,\,\,
\longmapsto
\quad
\Tr(x)=
\begin{tikzpicture}[baseline=-.1cm]

	\draw[thick] (-.3,-1+.2) -- (-.3,1);
	\draw[thick] (.3,-1+.2) -- (.3,1);
	\draw[thick] (0,1) ellipse (.3cm and .1cm);
	\halfDottedEllipse{(-.3,-1+.2)}{.3}{.1}
	
	\draw[thick, xString] (0,-1.1+.2) -- (0,.9);
	\node at (-.13,.6) {$\scriptstyle x$};
\end{tikzpicture}
$$
Morphisms $f\in \cM(x, y)$ are denoted by coupons in the horizontal plane, and we represent $\Tr(f):\Tr(x)\to \Tr(y)$ as the same coupon, but now on a cylinder:
$$
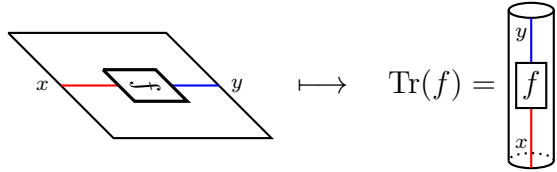
\begin{tikzpicture}[baseline=-.1cm, scale=.7]
	\plane{(0,-1)}{3}{2}
	\draw[thick, xString] (-1,0) -- (.5,0);
	\draw[thick, yString] (.5,0) -- (2,0);
	\Mbox{(.4,-.3)}{1}{.6}{$f$}
	\node at (2.35,0) {$\scriptstyle y$};
	\node at (-1.35,0) {$\scriptstyle x$};
\end{tikzpicture}
\,\,\,\longmapsto
\quad
\Tr(f)
=
\begin{tikzpicture}[baseline=-.1cm]

	\draw[thick] (-.3,-1) -- (-.3,1);
	\draw[thick] (.3,-1) -- (.3,1);
	\draw[thick] (0,1) ellipse (.3cm and .1cm);
	\halfDottedEllipse{(-.3,-1)}{.3}{.1}
	
	\draw[thick, xString] (0,-1.1) -- (0,0);
	\draw[thick, yString] (0,.9) -- (0,0);
	\nbox{unshaded}{(0,0)}{.3}{-.1}{-.1}{$f$}
	\node at (-.13,.68) {$\scriptstyle y$};
	\node at (-.13,-.77) {$\scriptstyle x$};
\end{tikzpicture}
$$

The traciators $\tau^\pm : \Tr_\cC(x\otimes y) \to \Tr_\cC(y\otimes x)$ are depicted
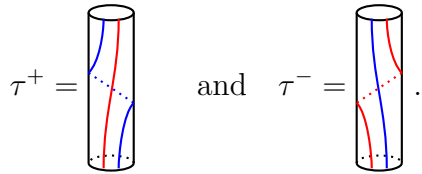
\[
\tau^+ =
\begin{tikzpicture}[baseline=-.1cm]

	\draw[thick] (-.3,-1) -- (-.3,1);
	\draw[thick] (.3,-1) -- (.3,1);
	\draw[thick] (0,1) ellipse (.3cm and .1cm);
	\halfDottedEllipse{(-.3,-1)}{.3}{.1}
	
	\draw[thick, xString] (-.1,-1.07) .. controls ++(90:.8cm) and ++(270:.8cm) .. (.1,.92);		
	\draw[thick, yString] (.1,-1.07) .. controls ++(90:.2cm) and ++(225:.2cm) .. (.3,-.2);		
	\draw[thick, yString] (-.1,.92) .. controls ++(270:.2cm) and ++(45:.2cm) .. (-.3,.2);
	\draw[thick, yString, dotted] (-.3,.2) -- (.3,-.2);	
\end{tikzpicture}
\qquad\text{and}\quad
\tau^- =
\begin{tikzpicture}[baseline=-.1cm, xscale=-1]

	\draw[thick] (-.3,-1) -- (-.3,1);
	\draw[thick] (.3,-1) -- (.3,1);
	\draw[thick] (0,1) ellipse (.3cm and .1cm);
	\halfDottedEllipse{(-.3,-1)}{.3}{.1}
	
	\draw[thick, yString] (-.1,-1.07) .. controls ++(90:.8cm) and ++(270:.8cm) .. (.1,.92);		
	\draw[thick, xString] (.1,-1.07) .. controls ++(90:.2cm) and ++(225:.2cm) .. (.3,-.2);		
	\draw[thick, xString] (-.1,.92) .. controls ++(270:.2cm) and ++(45:.2cm) .. (-.3,.2);
	\draw[thick, xString, dotted] (-.3,.2) -- (.3,-.2);	
\end{tikzpicture}\,\,.
\]
The naturality of the traciator and the traciator axiom \eqref{eq: axiom for traciator} are given in diagrams by
$$
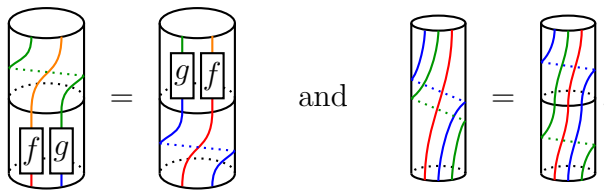
\begin{tikzpicture}[baseline=-.1cm]

	\draw[thick] (-.5,-1) -- (-.5,1);
	\draw[thick] (.5,-1) -- (.5,1);
	\draw[thick] (0,1) ellipse (.5cm and .2cm);
	\halfDottedEllipse{(-.5,-1)}{.5}{.2}
	\halfDottedEllipse{(-.5,0)}{.5}{.2}
	
	\draw[thick, xString] (-.2,-1.17) -- (-.2,-.7);		
	\draw[thick, wString] (-.2,-.7) -- (-.2,-.17) .. controls ++(90:.6cm) and ++(270:.6cm) .. (.2,.82);		
	\draw[thick, yString] (.2,-1.17) -- (.2,-.7);
	\draw[thick, zString] (.2,-.7) -- (.2,-.17) .. controls ++(90:.3cm) and ++(225:.2cm) .. (.5,.3);		
	\draw[thick, zString] (-.2,.82) .. controls ++(270:.3cm) and ++(45:.2cm) .. (-.5,.4);
	\draw[thick, zString, dotted] (.5,.3) -- (-.5,.4);	

	\nbox{unshaded}{(-.2,-.7)}{.3}{-.15}{-.15}{$f$}
	\nbox{unshaded}{(.2,-.7)}{.3}{-.15}{-.15}{$g$}
\end{tikzpicture}
\,\,\,=\,\,\,
\begin{tikzpicture}[baseline=-.1cm]

	\draw[thick] (-.5,-1) -- (-.5,1);
	\draw[thick] (.5,-1) -- (.5,1);
	\draw[thick] (0,1) ellipse (.5cm and .2cm);
	\halfDottedEllipse{(-.5,-1)}{.5}{.2}
	\halfDottedEllipse{(-.5,0)}{.5}{.2}
	
	\draw[thick, xString] (-.2,-1.17) .. controls ++(90:.6cm) and ++(270:.6cm) .. (.2,-.17) -- (.2,.3);		
	\draw[thick, wString] (.2,.3) -- (.2,.83);		
	\draw[thick, yString] (.2,-1.17) .. controls ++(90:.3cm) and ++(225:.2cm) .. (.5,-.7);		
	\draw[thick, yString] (-.2,.3) -- (-.2,-.17) .. controls ++(270:.3cm) and ++(45:.2cm) .. (-.5,-.6);
	\draw[thick, zString] (-.2,.83) -- (-.2,.3);
	\draw[thick, yString, dotted] (-.5,-.6) -- (.5,-.7);	

	\nbox{unshaded}{(-.2,.3)}{.3}{-.15}{-.15}{$g$}
	\nbox{unshaded}{(.2,.3)}{.3}{-.15}{-.15}{$f$}
\end{tikzpicture}
\qquad \text{and} \qquad
\begin{tikzpicture}[baseline=-.1cm, xscale=1.2]

	\draw[thick] (-.3,-1) -- (-.3,1);
	\draw[thick] (.3,-1) -- (.3,1);
	\draw[thick] (0,1) ellipse (.3cm and .1cm);
	\halfDottedEllipse{(-.3,-1)}{.3}{.1}
	
	\draw[thick, xString] (-.15,-1.09) .. controls ++(90:.8cm) and ++(270:.8cm) .. (.15,.91);		

	\draw[thick, yString] (0,-1.1) .. controls ++(90:.2cm) and ++(220:.28cm) .. (.3,-.05);		
	\draw[thick, yString] (-.15,.9) .. controls ++(270:.2cm) and ++(45:.1cm) .. (-.3,.3-.05);
	\draw[thick, yString, dotted] (-.3,.3-.05) -- (.3,0-.05);	
	\draw[thick, zString] (0,.9) .. controls ++(270:.2cm) and ++(40:.28cm) .. (-.3,0);
	\draw[thick, zString] (.15,-1.09) .. controls ++(90:.2cm) and ++(225:.1cm) .. (.3,-.3);		
	\draw[thick, zString, dotted] (-.3,0) -- (.3,-.3);	

\end{tikzpicture}
\,\,\,=\,\,\,
\begin{tikzpicture}[baseline=-.1cm, xscale=1.2]

	\draw[thick] (-.3,-1) -- (-.3,1);
	\draw[thick] (.3,-1) -- (.3,1);
	\draw[thick] (0,1) ellipse (.3cm and .1cm);
	\halfDottedEllipse{(-.3,-1)}{.3}{.1}
	\halfDottedEllipse{(-.3,0)}{.3}{.1}
		
	\draw[thick, xString] (-.15,-1.09) .. controls ++(90:.4cm) and ++(270:.4cm) .. (0,-.1);		
	\draw[thick, yString] (0,-1.1) .. controls ++(90:.4cm) and ++(270:.4cm) .. (.15,-.08);	
	\draw[thick, xString] (0,-.1) .. controls ++(90:.4cm) and ++(270:.4cm) .. (.15,.91);	
	\draw[thick, zString] (-.15,-.08) .. controls ++(90:.4cm) and ++(270:.4cm) .. (0,.9);

	\draw[thick, zString] (-.15,-.08) .. controls ++(270:.2cm) and ++(45:.1cm) .. (-.3,-.52);
	\draw[thick, zString] (.15,-1.1) .. controls ++(90:.2cm) and ++(225:.1cm) .. (.3,-.62);		
	\draw[thick, zString, dotted] (-.3,-.52) -- (.3,-.62);	
	\draw[thick, yString] (.15,-.08) .. controls ++(90:.2cm) and ++(225:.1cm) .. (.3,.38);
	\draw[thick, yString] (-.15,.92) .. controls ++(270:.2cm) and ++(45:.1cm) .. (-.3,.48);
	\draw[thick, yString, dotted] (-.3,.48) -- (.3,.38);	

\end{tikzpicture}\,.
$$


To illustrate the graphical notation, we prove a general basic property of categorified traces.

\begin{lem}\label{lem:ZigZagWithTraciator}
Let $\Tr: \cM \to \cC$ be a categorified trace. Then for any dualizable object $x\in \cM$, the following equality holds:
\[
\tau^+_{y,x^*}=\Tr(\id_{x^*\otimes y}\otimes \ev_x)\circ \tau^-_{x,x^*\otimes y \otimes x^*}\circ \Tr(\coev_x\otimes \id_{y\otimes x^*}).
\]
In diagrams:
$$
\begin{tikzpicture}[xscale=1.2, baseline=-.1cm]

	\draw[thick] (-.5,-1) -- (-.5,1);
	\draw[thick] (.5,-1) -- (.5,1);
	\draw[thick] (0,1) ellipse (.5cm and .2cm);
	\halfDottedEllipse{(-.5,-1)}{.5}{.2}
	\halfDottedEllipse{(-.5,-.5)}{.5}{.2}
	\halfDottedEllipse{(-.5,.5)}{.5}{.2}
	
	\draw[thick, xString] (.1,-1.2) .. controls ++(90:.4cm) and ++(270:.4cm) .. (.2,.45) .. controls ++(90:.2cm) and ++(90:.2cm) .. (.35,.45) .. controls ++(270:.1cm) and ++(135:.1cm) .. (.5,.1);	
	\draw[thick, xString] (-.1,.8) .. controls ++(270:.4cm) and ++(90:.4cm) .. (-.2,-.75) .. controls ++(270:.2cm) and ++(270:.2cm) .. (-.35,-.75) .. controls ++(90:.1cm) and ++(-45:.1cm) .. (-.5,-.3);
	\draw[thick, xString, dotted] (-.5,-.3) -- (.5,.1);	

\draw[thick, yString] (-.1,-1.2)  .. controls ++(90:.2) and ++(270:.2) ..  (.1,.8);
	
	\node at (.15,-1.34) {$\scriptstyle x^*$};
	\node at (-.1,-1.4) {$\scriptstyle y$};
\end{tikzpicture}
\,\,=\,\,
\begin{tikzpicture}[xscale=1.2, baseline=-.1cm]

	\draw[thick] (-.3,-1) -- (-.3,1);
	\draw[thick] (.3,-1) -- (.3,1);
	\draw[thick] (0,1) ellipse (.3cm and .1cm);
	\halfDottedEllipse{(-.3,-1)}{.3}{.1}
	
	\draw[thick, xString] (.1,-1.1) .. controls ++(85:.2cm) and ++(225:.2cm) .. (.3,-.2);		
	\draw[thick, xString] (-.1,.9) .. controls ++(265:.2cm) and ++(45:.2cm) .. (-.3,.2);
	\draw[thick, xString, dotted] (-.3,.2) -- (.3,-.2);	

\draw[thick, yString] (-.1,-1.1)  .. controls ++(90:.2) and ++(270:.2) ..  (.1,.9);
	
	\node at (.15,-1.24) {$\scriptstyle x^*$};
	\node at (-.1,-1.3) {$\scriptstyle y$};
\end{tikzpicture}
$$
\end{lem}

\begin{proof}
We prove this identity after composing with the inverse traciator $\tau^-_{x^*,y}$:
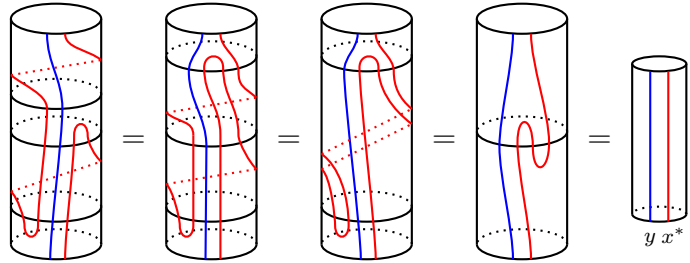
$$
\begin{tikzpicture}[xscale=1.2, baseline=.3cm]

	\draw[thick] (-.5,-1) -- (-.5,2);
	\draw[thick] (.5,-1) -- (.5,2);
	\halfDottedEllipse{(-.5,-1)}{.5}{.2}
	\halfDottedEllipse{(-.5,-.5)}{.5}{.2}
	\halfDottedEllipse{(-.5,.5)}{.5}{.2}
	\halfDottedEllipse{(-.5,1)}{.5}{.2}
	\draw[thick] (0,2) ellipse (.5cm and .2cm);
	
	\draw[thick, xString] (.1,-1.2) .. controls ++(90:.4cm) and ++(270:.4cm) .. (.2,.45) .. controls ++(90:.2cm) and ++(90:.2cm) .. (.35,.45) .. controls ++(270:.1cm) and ++(135:.1cm) .. (.5,.1);	
	\draw[thick, xString] (-.1,.8) .. controls ++(270:.4cm) and ++(90:.4cm) .. (-.2,-.75) .. controls ++(270:.2cm) and ++(270:.2cm) .. (-.35,-.75) .. controls ++(90:.1cm) and ++(-45:.1cm) .. (-.5,-.3);
	\draw[thick, xString, dotted] (-.5,-.3) -- (.5,.1);	
	\draw[thick, xString] (-.1,.8) .. controls ++(90:.2cm) and ++(-45:.2cm) .. (-.5,1.25);		
	\draw[thick, xString] (.1,1.8) .. controls ++(270:.2cm) and ++(135:.2cm) .. (.5,1.45);
	\draw[thick, xString, dotted] (-.5,1.25) -- (.5,1.45);	

\draw[thick, yString] (-.1,1.8) .. controls ++(270:.4cm) and ++(90:.4cm) .. (0.07,.8) .. controls ++(270:.4cm) and ++(90:.4cm) .. (-.065,-1.2);

\end{tikzpicture}
\,\,=\,\,
\begin{tikzpicture}[xscale=1.2, baseline=.3cm]

	\draw[thick] (-.5,-1) -- (-.5,2);
	\draw[thick] (.5,-1) -- (.5,2);
	\halfDottedEllipse{(-.5,-1)}{.5}{.2}
	\halfDottedEllipse{(-.5,-.5)}{.5}{.2}
	\halfDottedEllipse{(-.5,.5)}{.5}{.2}
	\halfDottedEllipse{(-.5,1.5)}{.5}{.2}
	\draw[thick] (0,2) ellipse (.5cm and .2cm);
	
	\draw[thick, xString] (.1,-1.2) .. controls ++(90:.4cm) and ++(270:.4cm) .. (.1,.45) .. controls ++(90:.4cm) and ++(270:.4cm) ..  (-.08,1.35) .. controls ++(90:.2cm) and ++(90:.2cm) .. (.13,1.35) .. controls ++(270:.4cm) and ++(90:.4cm) .. (.3,.45) .. controls ++(270:.1cm) and ++(135:.1cm) .. (.5,0);	
	\draw[thick, xString] (-.2,.45) .. controls ++(270:.4cm) and ++(90:.4cm) .. (-.2,-.75) .. controls ++(270:.2cm) and ++(270:.2cm) .. (-.35,-.75) .. controls ++(90:.1cm) and ++(-45:.1cm) .. (-.5,-.2);
	\draw[thick, xString, dotted] (-.5,-.2) -- (.5,0);	
	\draw[thick, xString] (-.2,.45) .. controls ++(90:.2cm) and ++(-45:.2cm) .. (-.5,.75);		
	\draw[thick, xString] (.1,1.8) .. controls ++(270:.2cm) and ++(90:.2cm) .. (.3,1.3) .. controls ++(270:.2cm) and ++(135:.2cm) .. (.5,.95);
	\draw[thick, xString, dotted] (-.5,.75) -- (.5,.95);	

\draw[thick, yString] (-.1,1.8) .. controls ++(270:.2cm) and ++(90:.2cm) .. (-.25,1.3) .. controls ++(270:.4cm) and ++(90:.4cm) .. (-.055,.3) .. controls ++(270:.4cm) and ++(90:.4cm) .. (-.065,-1.2);

\end{tikzpicture}
\,\,=\,\,
\begin{tikzpicture}[xscale=1.2, baseline=.3cm]

	\draw[thick] (-.5,-1) -- (-.5,2);
	\draw[thick] (.5,-1) -- (.5,2);
	\halfDottedEllipse{(-.5,-1)}{.5}{.2}
	\halfDottedEllipse{(-.5,-.5)}{.5}{.2}
	\halfDottedEllipse{(-.5,1.5)}{.5}{.2}
	\draw[thick] (0,2) ellipse (.5cm and .2cm);
	
	\draw[thick, xString] (.1,-1.2) .. controls ++(90:.4cm) and ++(270:.4cm) .. (-.08,1.35) .. controls ++(90:.2cm) and ++(90:.2cm) .. (.13,1.35) .. controls ++(270:.4cm) and ++(140:.1cm) .. (.5,.6);	
	\draw[thick, xString] (-.5,.2) .. controls ++(-40:.25cm) and ++(90:.4cm) .. (-.2,-.75) .. controls ++(270:.2cm) and ++(270:.2cm) .. (-.35,-.75) .. controls ++(90:.1cm) and ++(-45:.15cm) .. (-.5,0);
	\draw[thick, xString] (.1,1.8) .. controls ++(270:.2cm) and ++(90:.2cm) .. (.3,1.3) .. controls ++(270:.2cm) and ++(135:.2cm) .. (.5,.8);
\draw[thick, yString] (-.1,1.8) .. controls ++(270:.2cm) and ++(90:.2cm) .. (-.25,1.3) .. controls ++(270:.4cm) and ++(90:.4cm) .. (-.065,-1.2);
 	
	\draw[thick, xString, dotted] (-.5,.0) -- (.5,.6);	
	\draw[thick, xString, dotted] (-.5,.2) -- (.5,.8);	
\end{tikzpicture}
\,\,=\,\,
\begin{tikzpicture}[xscale=1.2, baseline=.3cm]

	\draw[thick] (-.5,-1) -- (-.5,2);
	\draw[thick] (.5,-1) -- (.5,2);
	\halfDottedEllipse{(-.5,-1)}{.5}{.2}
	\halfDottedEllipse{(-.5,.5)}{.5}{.2}
	\draw[thick] (0,2) ellipse (.5cm and .2cm);
	
	\draw[thick, xString] (.1,-1.2) .. controls ++(90:.4cm) and ++(270:.4cm) .. (-.05,.35) .. controls ++(90:.4cm) and ++(90:.4cm) .. (.12,.3) .. controls ++(270:.4cm) and ++(270:.4cm) .. (.2+.1,.35) .. controls ++(90:.4cm) and ++(270:.4cm) .. (.1,1.8);	
	\draw[thick, yString] (-.1,-1.2) .. controls ++(90:.4cm) and ++(270:.4cm) .. (-.25,.3) .. controls ++(90:.4cm) and ++(270:.4cm) .. (-.1,1.8);
\end{tikzpicture}
\,\,=\,\,
\begin{tikzpicture}[xscale=1.2, baseline=-.1cm]

	\draw[thick] (-.3,-1) -- (-.3,1);
	\draw[thick] (.3,-1) -- (.3,1);
	\draw[thick] (0,1) ellipse (.3cm and .1cm);
	\halfDottedEllipse{(-.3,-1)}{.3}{.1}
	
	\draw[thick, xString] (.1,-1.1) -- (.1,.9);
	\node at (.16,-1.28) {$\scriptstyle x^*$};

	\draw[thick, yString] (-.1,-1.1) -- (-.1,.9);
	\node at (-.1,-1.34) {$\scriptstyle y$};
\end{tikzpicture}
$$
The first equality holds by naturality of $\tau^-$.
The second one is the analog of \eqref{eq: axiom for traciator} for inverse traciators.
Finally, the third equality follows from naturality, along with Lemma~\ref{lem: Ponto-Shulman axiom redundant}.
\end{proof}

The main focus of this paper is the construction of a particular class of examples of categorical traces,
which have the further property that the cylinders can branch and braid (see Figure \ref{fig: big figure}). To obtain such a categorified trace $\Tr:\cM \to \cC$, we must assume that $\cM$ is equipped with the structure of a pivotal module tensor category over $\cC$, a notion we describe below.

\subsection{Module tensor categories}

If $\cM$ and $\cC$ are tensor categories, a functor $\Phi:\cC \to \cM$ equips $\cM$ with the structure of a left module category via $c \cdot m := \Phi(c) \otimes m$. If $\cC$ is braided, then $\Phi$ actually equips $\cM$ with \emph{two} distinct left $\cC$-module structures (one coming from the left action above, and one coming from the right action, twisted by the braiding). These two module structures are equivalent precisely when $\Phi$ is given the structure of a \emph{braided central functor}:
\begin{defn}\label{def: central functor}
Let $\cC$ be a category and $\cM$ a tensor category.
A \emph{central functor}\footnote{
The reader is cautioned that our usage of the term `central functor' agrees with \cite{MR2506324,MR3250042}, but differs from \cite{MR2074589,MR3039775} as these require that a central functor be braided.
} 
from $\cC$ to $\cM$ is a functor $\Phi: \cC \to \cM$ equipped with a factorisation
\[
\Phi\,:\,\,\cC\xrightarrow{\,\,\,\Phi^{\scriptscriptstyle \cZ}\,\,\,} \cZ(\cM)\xrightarrow{\,\,\,\,F\,\,\,\,} \cM
\]
where $F$ is the forgetful functor \cite{MR2506324}\cite[\S2.1]{MR3250042}. When $\cC$ is monoidal, a central functor $\Phi: \cC \to \cM$ is called monoidal if $\Phi^\ssZ$ is monoidal.
When $\cC$ is braided, a central functor is called braided if the functor $\Phi^\ssZ$ is braided.
\end{defn}

Equivalently, a functor $\Phi$ is a central functor if it is equipped with the additional data of half-braidings $e_{\Phi(c), x}:\Phi(c)\otimes x \to x \otimes \Phi(c)$, natural in $c$ and $x$, satisfying
$$
e_{\Phi(c), x \otimes y} = (\id_x \otimes e_{\Phi(c),y}) \circ (e_{\Phi(c),x} \otimes \id_y) \hspace{2.2cm} \text{(Central functor).}\hspace{.4cm}
$$
To go from a central functor to a monoidal central functor, one imposes the additional axiom
$$
e_{\Phi(c \otimes d),x} = (e_{\Phi(c),x} \otimes \id_{\Phi(d)}) \circ (\id_{\Phi(c)} \otimes e_{\Phi(d),x})\hspace{1cm} \text{(Monoidal central functor).}
$$
Finally, a braided central functor satisfies the further axiom 
$$
\hspace{2.1cm}e_{\Phi(c),\Phi(d)} = \Phi(\beta_{c,d})\hspace{3.7cm} \text{(Braided central functor).}
$$ 
Note that in formulating the above axioms we have omitted the associators and the isomorphisms $\Phi(x) \otimes \Phi(y) \to \Phi(x \otimes y)$, for readability. 

Let $\cC$ be a braided tensor category.
A braided central functor $\Phi:\cC \to \cM$ makes the tensor category $\cM$ into a $\cC$-module category in way that is compatible with the tensor structure of $\cM$ and the braiding of $\cC$.
We call such an $\cM$ a \emph{module tensor category} over $\cC$:

\begin{defn}\label{def: module tensor category}
If $\cM$ is a tensor category and $\cC$ is a braided category, then $\cM$ is a \emph{module tensor category} over $\cC$ if it is equipped with a braided central functor $\Phi: \cC \to \cM$.

If $\cM$, $\cC$ and $\Phi^{\scriptscriptstyle \cZ}$ are pivotal, then we say that $\cM$ is a pivotal module tensor category.
\end{defn}

As module tensor categories might be an unfamiliar concept, we provide a couple of alternative viewpoints:\medskip

\label{rem:WhyTensorModuleCategories?}
A tensor category can be thought of as a $2$-category with a single object.
Similarly, a pair consisting of a tensor category $\cC$ and module category $\cM$ can be encoded by a $2$-category with exactly two objects $1$ and $*$, and only non-trivial homs given by $\cC:=\hom(1,1)$ and $\cM:=\hom(1,*)$.

On the other hand, a braided tensor category can be thought of as a 3-category with one object and one 1-morphism \cite[`Periodic Table', Fig. 21]{MR1355899}.
The pair consisting of a braided category $\cC$ and module tensor category $\cM$ can be encoded by a $3$-category with exactly two objects $1$ and $*$,
and with only non-trivial homs given by the $2$-categories $\hom(1,1)$ and $\hom(1,*)$, which are furthermore required to have only one object.
We recover the braided category and its module tensor category as $\cC=\hom_{\hom(1,1)}(1_1,1_1)$ and $\cM=\hom_{\hom(1,*)}(\cdot,\cdot)$, respectively, where $\cdot:1\to *$ is the unique morphism.
\medskip

Alternatively, for those who like to think of a braided tensor category as a category which is an algebra over the little discs operad,
then a pair consisting of a braided tensor category $\cC$ and a module tensor category $\cM$ is the same thing as an algebra over Voronov's Swiss-cheese operad \cite{MR1718089}.

\subsection{The categorified trace associated to a module tensor category}

We now introduce the main construction of this paper, the categorified trace associated to a module tensor category.

\begin{defn}
\label{defn:AssociatedCategorifiedTrace}
Let $\cC$ be a braided pivotal category, and let $\cM$ be a pivotal module tensor category.
The \emph{associated categorified trace} $\Tr_\cC:\cM \to \cC$ is the right adjoint of the action functor $\Phi$.
\end{defn} 

The existence of a right adjoint is a mild assumption, which can be easily checked in many circumstances
(see, e.g., \cite[Cor. 1.9]{1406.4204}):

\begin{lem}
\label{lem:AdjointExists}
If $\cC$ and $\cM$ are semisimple (linear over some field $k$) and $\cC$ has finitely many types of simple objects,
then any linear functor $F:\cC\to \cM$ admits a right adjoint.
\end{lem}
\begin{proof}
Let $\{c_i\}$ be a basis of $\cC$ (representatives of the isomorphism classes of simple objects), and let $\{m_j\}$ be a basis of $\cM$.
Let $A=(A_{ij})$ be the matrix of finite dimensional vector spaces given by $A_{ij}:=\cM(m_j,F(c_i))$.
Note that $A$ has only finitely many non-zero entries.

We have canonical isomorphisms $F(c_i) \cong \bigoplus_j A_{ij} \otimes m_j$, and 
the functor $F$ can be recovered
by the formula
\[\textstyle
F\big(\bigoplus_i V_i \otimes c_i\big) \cong \bigoplus_{i,j} V_i \otimes A_{ij} \otimes m_j.
\]
Let $A^*$ be the `conjugate transpose' matrix of vector spaces, given by $A^*_{ji}=\Hom_k(A_{ij},k)$.
It is then an easy exercise to check that the functor $G\big(\bigoplus_j W_j \otimes m_j\big) := \bigoplus_{i,j} W_j \otimes A_{ji}^* \otimes c_i$
is a right adjoint (also a left adjoint) of $F$.
\end{proof}

Alternatively one can define $\Tr_\cC := \underline{\Hom}(1_\cM,-)$, where $\underline{\Hom}$ is the internal hom for module categories \cite{MR1976459}.
By definition, $\underline{\Hom}(1_\cM,x)$ represents the functor $c\mapsto\Hom(c\cdot 1_\cM, x)$ and so, recalling that $\Phi(c)=c\cdot 1_\cM$, we indeed have the adjunction
$
\Hom(\Phi(c), x) \cong \Hom(c,\underline{\Hom}(1_\cM,x)).
$

In Section \ref{sec:Traciator} we will construct traciators for $\Tr_\cC$ and prove that they satisfy the axioms of a categorified trace (Lemmas \ref{lem:TraciatorComposition} and \ref{lem:MoveThroughTraciator}), a result which first appeared in \cite[Prop.\,5]{MR2506324} and \cite[Prop.\,2.5]{MR3250042}.

We illustrate our construction by a couple of examples.

\begin{ex}
Any pivotal category $\cM$ is naturally a pivotal module tensor category over $\Vec$.  For every $x\in \cM$ we have $\cC(1,\Tr_\cC(x))\cong \cM(\Phi(1_\cC),x)=\cM(1_\cM,x)$, the invariant vectors.
This means that we can naturally identify $\Tr_\cC(x)$ with the vector space $\cM(1_\cM,x)$.
\end{ex}

\begin{ex}\label{ex: the short version of the thing in the appendix}
Let $\cC$ be a braided pivotal tensor category,  let $a\in \cC$ be a commutative algebra object, 
and let $\cM = \Mod_\cC(a)$ be the category of $a$-modules in $\cC$.
Then $\cM$ is both a tensor category \cite{MR1315904}, and a $\cC$-module category \cite[\S1]{MR1936496}.
Moreover, it is a $\cC$-module tensor category, as noted in \cite[\S3.4]{MR3039775}. The functor
$\Phi:\cC\to \cM$ is the \emph{free module} functor $\Phi(x):=a\otimes x$, and the half-braidings are given by
$$
e_{\Phi(x),y}:\Phi(x)\otimes_a y\cong x\otimes y\xrightarrow{\,\,\,\beta_{x,y}\,\,\,}y\otimes x\cong y\otimes_a\Phi(x).
$$

The associated categorified trace $\Tr_\cC:\Mod_\cC(a) \to \cC$ is given by the forgetful functor, so that $\Tr_\cC(x)=x$ for all $x \in \cM$.
Indeed, using the adjunction, we have 
$$
\cC(c, \Tr_\cC(x)) \cong \cM(\Phi(c), x) = \cM(a\otimes c, x) := \cC_{a}(a\otimes c, x) \cong \cC(c,x),
$$
where $\cC_{a}(-,-)$ denotes the $a$-module maps.
%

When $a$ is furthermore separable and with trivial twist, then $\Mod_\cC(a)$ is a pivotal module tensor category.
For more details on this example, and in particular for the proof that $\Phi^{\scriptscriptstyle \cZ}$ is a pivotal functor, we refer the reader to Appendix \ref{sec:Examples}.
\end{ex}

\begin{sub-ex}
\label{ex:trace of su2}
Take $\cC$ to be $SU(2)_k$ (the semi-simplification of the category of tilting modules for Lusztig's integral form of $U_q(\mathfrak{sl}(2))$, specialized at $q=\exp(\frac{2\pi i}{2(k+2)})$\footnote{We use the convention according to which $[2]_q=q+q^{-1}$.} -- see \cite{MR2286123} for details), which is an $A_{k+1}$ modular tensor category. 
The commutative algebra objects $a\in\cC$ are completely classified \cite{MR1907188,MR1936496}, and yield $A_n$, $D_{2n}$, $E_6$, and $E_8$ fusion categories, whose simple objects are classified by the nodes of the corresponding Dynkin diagrams.
Combining Example~\ref{ex: the short version of the thing in the appendix} with the computation at the end of \cite[\S6]{MR1936496}, we get the following explicit description of the trace functor:

{\it Example:} $D_4$\,.\quad
This is a module tensor category over $SU(2)_4$ 
(with simple objects $\mathbf 1,\ldots, \mathbf 5$).
Its unit object is
$
1_\cM=\tikz[scale=.3]{\draw (0,0) -- (1,0) -- (2,0) (1,0) -- (1,1);
\filldraw[fill=black] (0,0) circle(.15);
\filldraw[fill=white] (1,0) circle(.15);
\filldraw[fill=white] (2,0) circle(.15);
\filldraw[fill=white] (1,1) circle(.15);
}
$\,, and
the adjacency matrix of the Dynkin diagram
encodes the operation of tensoring by
$
\Phi(\mathbf 2)=
\tikz[scale=.3]{\draw (0,0) -- (1,0) -- (2,0) (1,0) -- (1,1);
\filldraw[fill=white] (0,0) circle(.15);
\filldraw[fill=black] (1,0) circle(.15);
\filldraw[fill=white] (2,0) circle(.15);
\filldraw[fill=white] (1,1) circle(.15);
}
$\,.
The trace functor is given (at the level of isomorphism classes of simple objects) by:\bigskip

\centerline{
\begin{tabular}{|c|c|c|c|c|}
\hline
$x\in \cM$\phantom{\Big|}\!\!
&
$
\tikz[scale=.3]{\draw (0,0) -- (1,0) -- (2,0) (1,0) -- (1,1);
\filldraw[fill=black] (0,0) circle(.15);
\filldraw[fill=white] (1,0) circle(.15);
\filldraw[fill=white] (2,0) circle(.15);
\filldraw[fill=white] (1,1) circle(.15);
}
$
&
$
\tikz[scale=.3]{\draw (0,0) -- (1,0) -- (2,0) (1,0) -- (1,1);
\filldraw[fill=white] (0,0) circle(.15);
\filldraw[fill=black] (1,0) circle(.15);
\filldraw[fill=white] (2,0) circle(.15);
\filldraw[fill=white] (1,1) circle(.15);
}
$
&
$
\tikz[scale=.3]{\draw (0,0) -- (1,0) -- (2,0) (1,0) -- (1,1);
\filldraw[fill=white] (0,0) circle(.15);
\filldraw[fill=white] (1,0) circle(.15);
\filldraw[fill=white] (2,0) circle(.15);
\filldraw[fill=black] (1,1) circle(.15);
}
$
&
$
\tikz[scale=.3]{\draw (0,0) -- (1,0) -- (2,0) (1,0) -- (1,1);
\filldraw[fill=white] (0,0) circle(.15);
\filldraw[fill=white] (1,0) circle(.15);
\filldraw[fill=black] (2,0) circle(.15);
\filldraw[fill=white] (1,1) circle(.15);
}
$
\\
$\Tr_\cC(x)\in SU(2)_4$
&
$\mathbf 1\oplus \mathbf 5$
&
$\mathbf 2\oplus \mathbf 4$
&
$\mathbf 3$
&
$\mathbf 3$\\\hline
\end{tabular}}\bigskip

{\it Example:} $E_6$\,.\quad
This is a module tensor category over $SU(2)_{10}$ 
(with simple objects $\mathbf 1,\ldots, \mathbf {11}$).
The unit object is
$
\tikz[scale=.3]{\draw (-1,0) -- (0,0) -- (1,0) -- (2,0) -- (3,0) (1,0) -- (1,1);
\filldraw[fill=black] (-1,0) circle(.15);
\filldraw[fill=white] (0,0) circle(.15);
\filldraw[fill=white] (1,0) circle(.15);
\filldraw[fill=white] (2,0) circle(.15);
\filldraw[fill=white] (3,0) circle(.15);
\filldraw[fill=white] (1,1) circle(.15);
}
$\,, and
the adjacency matrix of the Dynkin diagram
encodes the operation of tensoring by
$
\Phi(\mathbf 2)=
\tikz[scale=.3]{\draw (-1,0) -- (0,0) -- (1,0) -- (2,0) -- (3,0) (1,0) -- (1,1);
\filldraw[fill=white] (-1,0) circle(.15);
\filldraw[fill=black] (0,0) circle(.15);
\filldraw[fill=white] (1,0) circle(.15);
\filldraw[fill=white] (2,0) circle(.15);
\filldraw[fill=white] (3,0) circle(.15);
\filldraw[fill=white] (1,1) circle(.15);
}
$\,.
The trace functor is:\bigskip

\centerline{
\begin{tabular}{|c|c|c|c|c|c|c|}
\hline
$x\in \cM$\phantom{\Big|}\!\!
&
$
\tikz[scale=.3]{\draw (-1,0) -- (0,0) -- (1,0) -- (2,0) -- (3,0) (1,0) -- (1,1);
\filldraw[fill=black] (-1,0) circle(.15);
\filldraw[fill=white] (0,0) circle(.15);
\filldraw[fill=white] (1,0) circle(.15);
\filldraw[fill=white] (2,0) circle(.15);
\filldraw[fill=white] (3,0) circle(.15);
\filldraw[fill=white] (1,1) circle(.15);
}
$
&
$
\tikz[scale=.3]{\draw (-1,0) -- (0,0) -- (1,0) -- (2,0) -- (3,0) (1,0) -- (1,1);
\filldraw[fill=white] (-1,0) circle(.15);
\filldraw[fill=black] (0,0) circle(.15);
\filldraw[fill=white] (1,0) circle(.15);
\filldraw[fill=white] (2,0) circle(.15);
\filldraw[fill=white] (3,0) circle(.15);
\filldraw[fill=white] (1,1) circle(.15);
}
$
&
$
\tikz[scale=.3]{\draw (-1,0) -- (0,0) -- (1,0) -- (2,0) -- (3,0) (1,0) -- (1,1);
\filldraw[fill=white] (-1,0) circle(.15);
\filldraw[fill=white] (0,0) circle(.15);
\filldraw[fill=black] (1,0) circle(.15);
\filldraw[fill=white] (2,0) circle(.15);
\filldraw[fill=white] (3,0) circle(.15);
\filldraw[fill=white] (1,1) circle(.15);
}
$
&
$
\tikz[scale=.3]{\draw (-1,0) -- (0,0) -- (1,0) -- (2,0) -- (3,0) (1,0) -- (1,1);
\filldraw[fill=white] (-1,0) circle(.15);
\filldraw[fill=white] (0,0) circle(.15);
\filldraw[fill=white] (1,0) circle(.15);
\filldraw[fill=white] (2,0) circle(.15);
\filldraw[fill=white] (3,0) circle(.15);
\filldraw[fill=black] (1,1) circle(.15);
}
$
&
$
\tikz[scale=.3]{\draw (-1,0) -- (0,0) -- (1,0) -- (2,0) -- (3,0) (1,0) -- (1,1);
\filldraw[fill=white] (-1,0) circle(.15);
\filldraw[fill=white] (0,0) circle(.15);
\filldraw[fill=white] (1,0) circle(.15);
\filldraw[fill=black] (2,0) circle(.15);
\filldraw[fill=white] (3,0) circle(.15);
\filldraw[fill=white] (1,1) circle(.15);
}
$
&
$
\tikz[scale=.3]{\draw (-1,0) -- (0,0) -- (1,0) -- (2,0) -- (3,0) (1,0) -- (1,1);
\filldraw[fill=white] (-1,0) circle(.15);
\filldraw[fill=white] (0,0) circle(.15);
\filldraw[fill=white] (1,0) circle(.15);
\filldraw[fill=white] (2,0) circle(.15);
\filldraw[fill=black] (3,0) circle(.15);
\filldraw[fill=white] (1,1) circle(.15);
}
$
\\
$\Tr_\cC(x)\in SU(2)_{10}$
&
$\!\scriptstyle \mathbf 1\,\oplus\, \mathbf 7\!$
&
$\!\scriptstyle \mathbf 2\,\oplus\, \mathbf 6\,\oplus\, \mathbf 8\!$
&
$\!\scriptstyle \mathbf 3\,\oplus\, \mathbf 5\,\oplus\, \mathbf 7\,\oplus\, \mathbf 9\!$
&
$\!\scriptstyle \mathbf 4\,\oplus\, \mathbf 8\!$
&
$\!\scriptstyle \mathbf 4\,\oplus\, \mathbf 6\,\oplus\, \mathbf {10}\!$
&
$\!\scriptstyle \mathbf 5\,\oplus\, \mathbf {11}\!$
\\\hline
\end{tabular}}\bigskip

{\it Example:} $E_8$\,.\quad
This is a module tensor category over $SU(2)_{28}$ 
(with simple objects $\mathbf 1,\ldots, \mathbf {29}$).
The unit object is
$
\tikz[scale=.3]{\draw (-3,0) -- (-2,0) -- (-1,0) -- (0,0) -- (1,0) -- (2,0) -- (3,0) (1,0) -- (1,1);
\filldraw[fill=black] (-3,0) circle(.17);
\filldraw[fill=white] (-2,0) circle(.17);
\filldraw[fill=white] (-1,0) circle(.17);
\filldraw[fill=white] (0,0) circle(.17);
\filldraw[fill=white] (1,0) circle(.17);
\filldraw[fill=white] (2,0) circle(.17);
\filldraw[fill=white] (3,0) circle(.17);
\filldraw[fill=white] (1,1) circle(.17);
}
$\,. Once again,
the edges of the Dynkin diagram
encode tensoring by
$
\Phi(\mathbf 2)=
\tikz[scale=.3]{\draw (-3,0) -- (-2,0) -- (-1,0) -- (0,0) -- (1,0) -- (2,0) -- (3,0) (1,0) -- (1,1);
\filldraw[fill=white] (-3,0) circle(.17);
\filldraw[fill=black] (-2,0) circle(.17);
\filldraw[fill=white] (-1,0) circle(.17);
\filldraw[fill=white] (0,0) circle(.17);
\filldraw[fill=white] (1,0) circle(.17);
\filldraw[fill=white] (2,0) circle(.17);
\filldraw[fill=white] (3,0) circle(.17);
\filldraw[fill=white] (1,1) circle(.17);
}
$\,.
The trace functor is given by:\bigskip

\centerline{
\begin{tabular}{|c|c|c|c|c|c|c|c|c|}
\hline
$\tikz[baseline=-.1cm]{\node[scale=1.2]{$\scriptstyle x$};}$\phantom{\Big|}\!\!
&
\!\!\!\;$
\tikz[scale=.25]{\draw (-3,0) -- (-2,0) -- (-1,0) -- (0,0) -- (1,0) -- (2,0) -- (3,0) (1,0) -- (1,1);
\filldraw[fill=black] (-3,0) circle(.17);
\filldraw[fill=white] (-2,0) circle(.17);
\filldraw[fill=white] (-1,0) circle(.17);
\filldraw[fill=white] (0,0) circle(.17);
\filldraw[fill=white] (1,0) circle(.17);
\filldraw[fill=white] (2,0) circle(.17);
\filldraw[fill=white] (3,0) circle(.17);
\filldraw[fill=white] (1,1) circle(.17);
}
$\!\!\!\;
&
\!\!\!\;$
\tikz[scale=.25]{\draw (-3,0) -- (-2,0) -- (-1,0) -- (0,0) -- (1,0) -- (2,0) -- (3,0) (1,0) -- (1,1);
\filldraw[fill=white] (-3,0) circle(.17);
\filldraw[fill=black] (-2,0) circle(.17);
\filldraw[fill=white] (-1,0) circle(.17);
\filldraw[fill=white] (0,0) circle(.17);
\filldraw[fill=white] (1,0) circle(.17);
\filldraw[fill=white] (2,0) circle(.17);
\filldraw[fill=white] (3,0) circle(.17);
\filldraw[fill=white] (1,1) circle(.17);
}
$\!\!\!\;
&
\!\!\!\;$
\tikz[scale=.25]{\draw (-3,0) -- (-2,0) -- (-1,0) -- (0,0) -- (1,0) -- (2,0) -- (3,0) (1,0) -- (1,1);
\filldraw[fill=white] (-3,0) circle(.17);
\filldraw[fill=white] (-2,0) circle(.17);
\filldraw[fill=black] (-1,0) circle(.17);
\filldraw[fill=white] (0,0) circle(.17);
\filldraw[fill=white] (1,0) circle(.17);
\filldraw[fill=white] (2,0) circle(.17);
\filldraw[fill=white] (3,0) circle(.17);
\filldraw[fill=white] (1,1) circle(.17);
}
$\!\!\!\;
&
\!\!\!\;$
\tikz[scale=.25]{\draw (-3,0) -- (-2,0) -- (-1,0) -- (0,0) -- (1,0) -- (2,0) -- (3,0) (1,0) -- (1,1);
\filldraw[fill=white] (-3,0) circle(.17);
\filldraw[fill=white] (-2,0) circle(.17);
\filldraw[fill=white] (-1,0) circle(.17);
\filldraw[fill=black] (0,0) circle(.17);
\filldraw[fill=white] (1,0) circle(.17);
\filldraw[fill=white] (2,0) circle(.17);
\filldraw[fill=white] (3,0) circle(.17);
\filldraw[fill=white] (1,1) circle(.17);
}
$
&
$
\tikz[scale=.25]{\draw (-3,0) -- (-2,0) -- (-1,0) -- (0,0) -- (1,0) -- (2,0) -- (3,0) (1,0) -- (1,1);
\filldraw[fill=white] (-3,0) circle(.17);
\filldraw[fill=white] (-2,0) circle(.17);
\filldraw[fill=white] (-1,0) circle(.17);
\filldraw[fill=white] (0,0) circle(.17);
\filldraw[fill=black] (1,0) circle(.17);
\filldraw[fill=white] (2,0) circle(.17);
\filldraw[fill=white] (3,0) circle(.17);
\filldraw[fill=white] (1,1) circle(.17);
}
$\!\!\!\;
&
\!\!\!\;$
\tikz[scale=.25]{\draw (-3,0) -- (-2,0) -- (-1,0) -- (0,0) -- (1,0) -- (2,0) -- (3,0) (1,0) -- (1,1);
\filldraw[fill=white] (-3,0) circle(.17);
\filldraw[fill=white] (-2,0) circle(.17);
\filldraw[fill=white] (-1,0) circle(.17);
\filldraw[fill=white] (0,0) circle(.17);
\filldraw[fill=white] (1,0) circle(.17);
\filldraw[fill=white] (2,0) circle(.17);
\filldraw[fill=white] (3,0) circle(.17);
\filldraw[fill=black] (1,1) circle(.17);
}
$\!\!\!\;
&
\!\!\!\;$
\tikz[scale=.25]{\draw (-3,0) -- (-2,0) -- (-1,0) -- (0,0) -- (1,0) -- (2,0) -- (3,0) (1,0) -- (1,1);
\filldraw[fill=white] (-3,0) circle(.17);
\filldraw[fill=white] (-2,0) circle(.17);
\filldraw[fill=white] (-1,0) circle(.17);
\filldraw[fill=white] (0,0) circle(.17);
\filldraw[fill=white] (1,0) circle(.17);
\filldraw[fill=black] (2,0) circle(.17);
\filldraw[fill=white] (3,0) circle(.17);
\filldraw[fill=white] (1,1) circle(.17);
}
$\!\!\!\;
&
\!\!\!\;$
\tikz[scale=.25]{\draw (-3,0) -- (-2,0) -- (-1,0) -- (0,0) -- (1,0) -- (2,0) -- (3,0) (1,0) -- (1,1);
\filldraw[fill=white] (-3,0) circle(.17);
\filldraw[fill=white] (-2,0) circle(.17);
\filldraw[fill=white] (-1,0) circle(.17);
\filldraw[fill=white] (0,0) circle(.17);
\filldraw[fill=white] (1,0) circle(.17);
\filldraw[fill=white] (2,0) circle(.17);
\filldraw[fill=black] (3,0) circle(.17);
\filldraw[fill=white] (1,1) circle(.17);
}
$\!\!\!\;
\\
\!\!\!\;$\scriptstyle \Tr_\cC(x)$\!\!\!\;
&
\!\!\!$\scriptscriptstyle \mathbf 1\oplus \mathbf {11}\oplus \mathbf {19}\oplus \mathbf {29}$\!\!\!
&
\!\!$\begin{smallmatrix}\scriptscriptstyle \mathbf {2}\oplus \mathbf {10}\oplus \mathbf {12}\oplus\\\scriptscriptstyle \mathbf {18}\oplus \mathbf {20}\oplus \mathbf {28}\end{smallmatrix}$\!\!
&
\!\!\!$\begin{smallmatrix}\scriptscriptstyle \mathbf {3}\oplus \mathbf {9}\oplus \mathbf {11}\oplus \mathbf {13}\oplus\\\scriptscriptstyle \mathbf {17}\oplus \mathbf {19}\oplus \mathbf {21}\oplus \mathbf {27}\end{smallmatrix}$\!\!\!
&
\!\!\!$\begin{smallmatrix}\scriptscriptstyle \mathbf {4}\oplus \mathbf {8}\oplus \mathbf {10}\oplus \mathbf {12}\oplus \mathbf {14}\oplus\\\scriptscriptstyle \mathbf {16}\oplus \mathbf {18}\oplus \mathbf {20}\oplus \mathbf {22}\oplus \mathbf {26}\end{smallmatrix}$\!\!\!
&
\!\!\!$\begin{smallmatrix}\scriptscriptstyle \mathbf {5}\oplus \mathbf {7}\oplus \mathbf {9}\oplus \mathbf {11}\oplus \mathbf {13}\oplus \mathbf {15}\oplus\\\scriptscriptstyle \mathbf {15}\oplus \mathbf {17}\oplus \mathbf {19}\oplus \mathbf {21}\oplus \mathbf {23}\oplus \mathbf {25}\end{smallmatrix}$\!\!\!
&
\!\!\!$\begin{smallmatrix}\scriptscriptstyle \mathbf {6}\oplus \mathbf {10}\oplus \mathbf {14}\oplus\\\scriptscriptstyle \mathbf {16}\oplus \mathbf {20}\oplus \mathbf {24}\end{smallmatrix}$\!\!\!
&
\!\!\!$\begin{smallmatrix}\scriptscriptstyle \mathbf {6}\oplus \mathbf {8}\oplus \mathbf {12}\oplus \mathbf {14}\oplus\\\scriptscriptstyle \mathbf {16}\oplus \mathbf {18}\oplus \mathbf {22}\oplus \mathbf {24}\end{smallmatrix}$\!\!\!
&
\!\!\!$\scriptscriptstyle \mathbf {7}\oplus \mathbf {13}\oplus \mathbf {17}\oplus \mathbf {23}$\!\!\!
\\\hline
\end{tabular}}\bigskip

Note that $\Tr_\cC(1_\cM)$ (the first entry of the above tables) recovers the commutative algebra object $a$ for which $\cM = \Mod_\cC(a)$ (see \cite[Table 1]{MR1936496}).
\end{sub-ex}

\begin{rem}
The category $SU(2)_k$ is closely related to the Temperley-Lieb $A_{k+1}$ ribbon category with $q=\exp(\frac{2\pi i}{2(k+2)})$ and braiding
\begin{equation}
\label{eq:BraidingInTL}
\begin{tikzpicture}[baseline=-.1cm]
	\braiding{(0,0)}{.6}{.4}
\end{tikzpicture}
=
iq^{1/2}\,
\begin{tikzpicture}[baseline=-.1cm]
	\nbox{unshaded}{(0,0)}{.4}{0}{0}{}
	\draw (.2,-.6)--(.2,.6);
	\draw (-.2,-.6)--(-.2,.6);
\end{tikzpicture}
-
iq^{-1/2}\,
\begin{tikzpicture}[baseline=-.1cm]
	\draw (.2,-.6)--(.2,.6);
	\draw (-.2,-.6)--(-.2,.6);
	\nbox{unshaded}{(0,0)}{.4}{0}{0}{}
	\draw (-.2,-.4) arc (180:0:.2cm);
	\draw (-.2,.4) arc (-180:0:.2cm);
\end{tikzpicture}\,,
\end{equation}
familiar from the definition of the Jones polynomial \cite[\S 1.1.3]{MR2783128}.
The classification of commutative algebra objects in Temperley-Lieb is very similar to that of $SU(2)_k$.
We again have $A_n$, $D_{2n}$, $E_6$, and $E_8$, but now we also have commutative algebras coming from the tadpole diagrams $T_n$ (compare \cite[Thm. 3.12]{MR2046203}).
Interestingly, for the braiding \eqref{eq:BraidingInTL}, only the $T_{2n}$ arise from commutative algebras, while if we negate \eqref{eq:BraidingInTL}, only the $T_{2n+1}$ arise from commutative algebras.

The skein theory for the $D_{2n}$ planar algebras was computed in \cite{MR2559686}, and the skein theory for the $E_6$ and $E_8$ planar algebras was computed in \cite{MR2577673}.
To give evaluation algorithms for these planar algebras, the authors used the existence of overbraiding relations of the form
$$
\begin{tikzpicture}[baseline = -.3cm]
	\draw (-.7,-.8)--(-.7,0) arc (180:0:.7cm)--(.7,-.8);
	\draw (-.15,0)--(-.15,-.8) (.15,0)--(.15,-.8);
	\node[scale=.4] at (0,-.6) {$\ldots$};
	\node at (0,-1.05) {{\scriptsize{$4n$}}};
	\node[rotate=90] at (0,-.9) {$\{$};
	\draw[thick, unshaded] (0,0) circle (.4);
	\node at (0,0) {$S$};
\end{tikzpicture}
\,=
\begin{tikzpicture}[baseline = -.2cm]
	\draw (-.15,0)--(-.15,-.8) (.15,0)--(.15,-.8);
	\node[scale=.4] at (0,-.6) {$\ldots$};
	\node at (0,-1.05) {{\scriptsize{$4n$}}};
	\node[rotate=90] at (0,-.9) {$\{$};
	\draw[super thick, white] (-.7,-.8) .. controls ++(90:.45cm) and ++(90:.45cm) .. (.7,-.8);
	\draw (-.7,-.8) .. controls ++(90:.45cm) and ++(90:.45cm) .. (.7,-.8);
	\draw[thick, unshaded] (0,.25) circle (.4);
	\node at (0,.25) {$S$};
\end{tikzpicture}
$$
where the crossing is the braiding \eqref{eq:BraidingInTL} in the appropriate $A_{k}$ category.
The reason this overbraiding exists is that the $D_{2n}$, $E_6$ and $E_8$ fusion categories arise as module tensor categories of the form $\Mod_\cC(a)$ for $a$ a separable commutative algebra with trivial twist,
as explained above.
\end{rem}

\begin{rem}\label{rem: rant No2}
The calculations in Example \ref{ex:trace of su2} show that our trace functor is typically not a shadow in the sense of \cite{MR3095324}, as it may fail to satisfy $\tau_{1,x} = \id$.
Indeed, $\tau_{1,x}$ is equal to the twist by Proposition \ref{prop:ThetaAndTraciator}, and the latter is almost always nontrivial (see \cite[\S8.2]{MR1358783},\cite[(12.8.11)]{MR1104219} for explicit formulas). 
\end{rem}



\section{Properties of the trace functor}
\label{sec:InternalTrace}

In this section we establish all the basic properties of the categorified trace $\Tr_\cC$ associated to a pivotal module tensor category (Definition \ref{defn:AssociatedCategorifiedTrace}).  
In Sections \ref{sec:AdjointsOfTensorFunctors}--\ref{sec:UnitMap}, we discuss the properties of $\Tr_\cC$ which follow from the fact that it is the right adjoint of a tensor functor, notably the existence of the multiplication map $\mu_{x,y}:\Tr_\cC(x) \otimes \Tr_\cC(y) \to \Tr_\cC(x \otimes y)$. 
The construction of the traciator $\tau_{x,y}:\Tr_\cC(x \otimes y)\to \Tr_\cC(y \otimes x)$ is done in Section~\ref{sec:Traciator}.
This only depends on the fact that $\cM$ is pivotal and that $\Phi$ factors through $\cZ(\cM)$.
Finally, in Section~\ref{sec:InteractionBetweenTraciatorAndBraiding} we assume our full set of hypothesis: that $\cM$ is a pivotal module tensor category over $\cC$.
In this context we establish various compatibility relations which combine the traciator and multiplication maps with the braiding and the twist of $\cC$.

In terms of Figure \ref{fig: big figure} from the introduction, Sections~\ref{sec:AdjointsOfTensorFunctors}--\ref{sec:UnitMap} deal with the right column,
Section~\ref{sec:Traciator} deals with the left column, and Section~\ref{sec:InteractionBetweenTraciatorAndBraiding} deals with the middle column.

\subsection{The attaching map}
\label{sec:AdjointsOfTensorFunctors}

Adjoints of tensor functors have been studied systematically \cite{MR0360749} and arise in many contexts (e.g.\;\cite{MR2863377}, \cite{MR3039775}).
We will reprove several known properties as a way of introducing our graphical notation for $\Tr_\cC$.
For Sections \ref{sec:AdjointsOfTensorFunctors}-\ref{sec:UnitMap}, we fix the following notation:

\begin{samepage}
\begin{nota}
\label{nota:CM}
\mbox{}

\begin{itemize}
\item
$\cC,\cM$ are tensor categories, 
\item
$\Phi: \cC\to \cM$ is a tensor functor, and
\item 
$\Tr_\cC : \cM \to \cC$ is the right adjoint of $\Phi$.
\end{itemize}
\end{nota}
\end{samepage}

\noindent
The reader is cautioned that under these limited assumptions (compare Notation \ref{nota:SecTraciator}) the functor $\Tr_\cC$ may fail to be a categorified trace.
Nevertheless, we shall use the notation $\Tr_\cC$ for notational consistency.

Recall from Section \ref{sec: categorified traces} that we depict an object $x \in \cM$ by a strand in the plane, and $\Tr_\cC(x)$ by a cylinder with an $x$-strand on its surface.
Given an object $c\in\cC$, the adjunction $\cC(c,\Tr_\cC(x)) \cong \cM(\Phi(c),x)$ is represented diagrammatically by:
\begin{equation}\label{eq: open umbrella}
\qquad\quad\begin{tikzpicture}[baseline=-1.3cm, scale=.9]
	\draw[thick, cString] (0,-2) -- (0,-3.3);
	\CMbox{box}{(-.2,-2)}{.8}{.8}{.4}{$g$}
	\fill[unshaded] (-.37,-.9) -- (-.37,-.6) -- (.37,-.6) -- (.37,-1.07);
	\draw[thick] (-.3,-1) -- (-.3,.5);
	\draw[thick] (.3,-1) -- (.3,.5);
	\draw[thick] (0,.5) ellipse (.3cm and .1cm);
	\halfDottedEllipse{(-.3,-1)}{.3}{.1}
	\draw[thick, xString] (0,.4) -- (0,-1.1);
	\node at (-.25,-3.1) {$\scriptstyle c$};
	\node at (-.15,0) {$\scriptstyle x$};
\end{tikzpicture}
\,\,\,\in\,\cC(c,\Tr_\cC(x))
\quad\longleftrightarrow
\begin{tikzpicture}[baseline=.1cm, scale=.8]
	\plane{(-.5,-.8)}{3}{2}
	\draw[thick, xString] (.8,0) -- (1.7,0);
	\CMbox{box}{(0,-.4)}{.8}{.8}{.4}{$g$}
	\invisibleTubeWithString{(-.3,.1)}{cString}
	\draw[thick, cString] (-1.4,1.4) .. controls ++(90:.2cm) and ++(0:.2cm) .. (-2,1.7);
	\node at (-1.75,1.9) {$\scriptstyle c$};
	\node at (1.25,-.22) {$\scriptstyle x$};
\end{tikzpicture}
\!\!\!\in\, \cM(\Phi(c),x)
\end{equation}
Visually, we can think of the above adjunction as the operation of opening the tube (like an umbrella) and then flattening it to a plane.


\begin{defn}[the attaching map $\varepsilon$]
\label{defn:NaturalMap}
Given $x\in \cM$, we define $\varepsilon_x:\Phi(\Tr_\cC(x))\to x$ to be the image of $\id_{\Tr_\cC(x)}$ under the correspondence \eqref{eq: open umbrella}.
Equivalently, $\varepsilon:\Phi\circ \Tr_\cC\to \id_\cM$ is the counit of the adjunction $\Phi\dashv \Tr_\cC$.

We represent the morphism $\varepsilon_x$ as follows:
$$
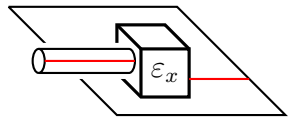
\begin{tikzpicture}[baseline=.1cm, scale=.8]
	\plane{(-.4,-.8)}{2.8}{1.8}

	\CMbox{box}{(0,-.5)}{.8}{.8}{.4}{$\varepsilon_x$}

	\straightTubeWithString{(-.2,-.1)}{.1}{1.5}{xString}
	\draw[thick, xString] (.8,-.2) -- (1.8,-.2);

\end{tikzpicture}
$$
A possibly better graphical notation for $\varepsilon_x$ is that of a tube coming out of the plane:
$$
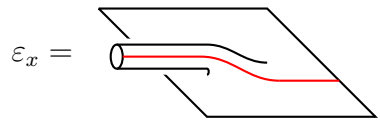
\varepsilon_x
=\,\,\,
\newcommand{\straightTubeNoEnd}[4]{
	\coordinate (ZZq) at #1;
	\pgfmathsetmacro{\tubeLength}{#3};
	\pgfmathsetmacro{\tubeWidth}{#2};
	\pgfmathsetmacro{\buffer}{.05};	
	\fill[unshaded] ($ (ZZq) + (-\tubeLength,0) + 2*(0,-\buffer) $) -- ($ (ZZq) + 2*(0,-\buffer) $) arc(-90:90:{\tubeWidth+\buffer} and {2*(\tubeWidth+\buffer)}) -- ($ (ZZq) + (-\tubeLength,0) + 4*(0,\tubeWidth) + 2*(0,\buffer) $) arc(90:270:{\tubeWidth+\buffer} and {2*(\tubeWidth+\buffer)}) ;
	\draw[unshaded, thick]  ($ (ZZq) + (-\tubeLength,0) $) -- (ZZq) to[out=0,in=0] ($ (ZZq) + (0,-.1) $)($ (ZZq) + (-\tubeLength,0) + 4*(0,\tubeWidth) $) -- ($ (ZZq) + 2*(0,2*\tubeWidth) $) to[out=0,in=180] ($ (ZZq) + 2*(0,2*\tubeWidth) + (1,-.3)$);
	\draw[thick] ($ (ZZq) + (-\tubeLength,0) $) arc(-90:90:{\tubeWidth} and {2*\tubeWidth}) arc(90:270:{\tubeWidth} and {2*\tubeWidth});
	\draw[thick, #4] ($ (ZZq) + (-\tubeLength,0) + 2*(0,\tubeWidth) + (\tubeWidth,0)$) -- +(1.3,0) to[out=0,in=180] (.8,-.2);
}
\begin{tikzpicture}[baseline=.1cm, scale=.8]
	\plane{(-.4,-.8)}{2.8}{1.8}
	\straightTubeNoEnd{(-.4,0)}{.1}{1.5}{xString}
	\draw[thick, xString] (.8,-.2) -- (1.8,-.2);
\end{tikzpicture}
$$
but in order to show that this graphical depiction of $\varepsilon_x$ is valid, we would first need to prove properties such as 
\eqref{eq: pants o epsilon}, or \eqref{eqn:TrSplittingIso}.
The next two lemmas are general results that hold for any adjunction.
\end{defn}

\begin{lem}
\label{lem:PushDownMorphism}
Let $a\in\cC$ and $x\in\cM$. 
Then, under the adjunction $\cC(a,\Tr_\cC(x))\cong \cM(\Phi(a), x)$,
a morphism $f\in \cC(a,\Tr_\cC(x))$ corresponds to $\varepsilon_x\circ \Phi(f)\in \cM(\Phi(a), x)$.
\end{lem}
\begin{proof}
Since the adjunction is natural, the following diagram commutes:
\[
\xymatrix@C=0pt{
\varepsilon_x \ar@{|->}[d]
&\in&
\cM(\Phi(\Tr_\cC(x)), x) \ar[d]_{\Phi(f)^*}
&\cong&
\cC(\Tr_\cC(x),\Tr_\cC(x)) \ar[d]_{f^*}
&\ni&
\id_{\Tr_\cC(x)}, \ar@{|->}[d]
\\
\varepsilon_x \circ \Phi(f)
&\in& 
\cM(\Phi(a), x) 
&\cong& 
\cC(a, \Tr_\cC(x))
&\ni&
f
}
\]
where $f^*(g)= g\circ f$.
\end{proof}

\begin{lem}
\label{lem:MoveThroughNaturalMap}
Suppose $f\in\cM(x,y)$.  Then $\varepsilon_y\circ \Phi(\Tr_\cC(f)) = f\circ \varepsilon_x$, i.e.,
\begin{equation} \label{pic: f passes through epsilon}
\begin{tikzpicture}[baseline=.1cm, scale=.8]
	\plane{(-.4,-.8)}{3.2}{1.8}

	\CMbox{box}{(0,-.5)}{.8}{1.2}{.5}{$\varepsilon_y$}

	\straightTubeNoString{(-.3,.1)}{.15}{2}{xString}

	\draw[thick, xString] (-2.15,.4) -- (-1.2,.4);
	\draw[thick, yString] (-.15,.4) -- (-1.2,.4);

	\nbox{unshaded}{(-1.2,.4)}{.2}{.2}{.2}{\rotatebox{-75}{$f$}}
	\draw[thick, yString] (.8,-.2) -- (2.2,-.2);
\end{tikzpicture}
=
\begin{tikzpicture}[baseline=.1cm, scale=.8]
	\plane{(-.4,-.8)}{3.8}{1.8}

	\CMbox{box}{(0,-.5)}{.8}{.8}{.4}{$\varepsilon_x$}

	\straightTubeWithString{(-.2,-.1)}{.1}{1.5}{xString}
	\draw[thick, xString] (.8,-.2) -- (2,-.2);
	\draw[thick, yString] (2,-.2) -- (2.8,-.2);

	\Mbox{(1.6,-.5)}{1}{.6}{$f$}
\end{tikzpicture}
\end{equation}
\end{lem}
\begin{proof}
The counit of an adjunction is always a natural transformation \cite[\S IV.1]{MR1712872}. 
\end{proof}

\subsection{The multiplication map}
\label{sec:Multiplication}

\begin{defn}\label{def: multiplication map}
For $x,y\in \cM$, there is a canonical map $\mu_{x,y} : \Tr_\cC(x) \otimes \Tr_\cC(y) \to \Tr_\cC(x \otimes y)$ given as the mate of 
$$
\varepsilon_x\otimes \varepsilon_y 
=
\begin{tikzpicture}[baseline=.4cm, scale=.8]
	\plane{(-.2,-.8)}{2.4}{3}
	\draw[thick, yString] (.8,-.1) -- (1.5,-.1);
	\draw[thick, xString] (-.6,1.3) -- (.1,1.3);
	\CMbox{box}{(-1.4,.8)}{.8}{.8}{.4}{$\varepsilon_x$}
	\CMbox{box}{(0,-.6)}{.8}{.8}{.4}{$\varepsilon_y$}
	\straightTubeWithString{(-1.6,1.1)}{.1}{1.2}{xString}
	\straightTubeWithString{(-.2,-.3)}{.1}{1.2}{yString}
\end{tikzpicture}
$$
under the isomorphisms
\begin{align*}
\cM(\Phi(\Tr_\cC(x)) \otimes \Phi(\Tr_\cC(y)), x\otimes y)
&\cong 
\cM(\Phi(\Tr_\cC(x) \otimes \Tr_\cC(y)), x\otimes y)
\\&\cong
\cC(\Tr_\cC(x) \otimes \Tr_\cC(y) , \Tr_\cC(x \otimes y)),
\end{align*}
where we used that $\Phi$ is a tensor functor.
Diagrammatically, we denote $\mu$ by a pair of pants:
$$
\mu_{x,y} \,\,=\,\,
\begin{tikzpicture}[baseline=-.3cm,scale=.9]
	\topPairOfPants{(-1,-1)}{}
	\draw[thick, xString] (-.73,-1.1) .. controls ++(90:.8cm) and ++(270:.8cm) .. (-.1,.4);		
	\draw[thick, yString] (.73,-1.1) .. controls ++(90:.8cm) and ++(270:.8cm) .. (.1,.4);		
	\node[below] at (-.73,-1.09+.03) {$\scriptstyle x$};
	\node[below] at (.73,-1.08+.03) {$\scriptstyle y$};
\end{tikzpicture}
$$
\end{defn}

Our next task will be to show that $\mu$ is associative.
We begin by showing that it is compatible with the attaching map $\varepsilon$:

\begin{lem}\label{lem:MultiplicationCompatibleWithNatural}
For any $x,y\in \cM$, we have
$\varepsilon_{x\otimes y} \circ \Phi(\mu_{x,y})= \varepsilon_x\otimes \varepsilon_y$.
\begin{equation}
\begin{tikzpicture}[baseline=.4cm, scale=.7]
	\plane{(-.4,-1.1)}{4.8}{3.1}
	\draw[thick, yString] (1.4,0) -- (3.3,0);
	\draw[thick, xString] (1.4,.3) -- (3,.3);
	\CMbox{box}{(.2,-.4)}{1.2}{1.4}{.5}{$\varepsilon_{x\otimes y}$}
	\fill[unshaded] (-1,1.7) -- (-3.3,1.7) -- (-2.25,.65) -- ++(.2,0) -- ++(.3,-.3) -- ++(-.2,0) -- ++(.8,-.8) -- ++(.5,.1) -- (-.5,.1) -- (0,.1) -- (0,1.1) -- ++(-1,0);
	\draw[thick, xString] (-2.9,1.2) to[out=0,in=180] (-.1,.7) -- (.08,.7);
	\draw[thick, yString] (-1.6,-.2) to[out=10,in=180] (-.15,.5) -- (.08,.5);
	\draw[thick] (-3.1,1.56) to[out=0,in=180] (-.1,1);
	\draw[thick] (-1.8,-.6) to[out=5,in=180] (-.1,.2);
	\draw[thick] (-1.8,.2) to[out=10,in=2, looseness=2] (-3.1,.84);
	\draw[thick] (-3.1,1.2) circle (.2 and .36);
	\draw[thick] (-1.8,-.2) circle (.2 and .4);
	\draw[thick] (-.1,.2) arc (-90:90:.2 and .4);
\end{tikzpicture}
\!\!=\;\;
\begin{tikzpicture}[baseline=.65cm, scale=.8]
	\plane{(-.2,-.8)}{2.7}{3}
	\draw[thick, yString] (.8,-.1) -- (1.8,-.1);
	\draw[thick, xString] (-.6,1.3) -- (.4,1.3);
	\CMbox{box}{(-1.4,.8)}{.8}{.8}{.4}{$\varepsilon_x$}
	\CMbox{box}{(0,-.6)}{.8}{.8}{.4}{$\varepsilon_y$}
	\straightTubeWithString{(-1.6,1.1)}{.1}{1.2}{xString}
	\straightTubeWithString{(-.2,-.3)}{.1}{1.2}{yString}
\end{tikzpicture}
\label{eq: pants o epsilon}
\end{equation}
\end{lem}
\begin{proof}
By definition, $\varepsilon_x\otimes \varepsilon_y$ is the mate of $\mu_{x,y}$ under the adjunction.
Now apply Lemma \ref{lem:PushDownMorphism} to see that $\mu_{x,y}$ is also the mate of $\varepsilon_{x\otimes y} \circ \Phi(\mu_{x,y})$.
\end{proof}

As a corollary, we see $\mu_{x,y}$ that is natural in $x$ and $y$:

\begin{cor}
\label{cor:MoveTensorThroughMultiplication}
For all $x,y,z\in \cM$ and all $f:y\to z$ in $\cM$, we have $\mu_{x,z}\circ (\id_{\Tr_\cC(x)}\otimes \Tr_\cC(f))=\Tr_\cC(\id_x\otimes f) \circ \mu_{x,y}$.
In diagrams,
\begin{equation}\label{eq: f through pants}
\begin{tikzpicture}[baseline=.1cm]

	\coordinate (a1) at (-1,-1);
	\coordinate (a2) at ($ (a1) + (1.4,0) $);
	\coordinate (b1) at ($ (a1) + (0,1)$);
	\coordinate (b2) at ($ (a2) + (0,1)$);
	\coordinate (c) at ($ (b1) + (.7,1.5)$);

	\topPairOfPants{(b1)}{}
	
	\draw[thick, xString] ($ (a1) + (.3,-.1) $) -- ($ (b1) + (.3,-.1) $) .. controls ++(90:.8cm) and ++(270:.8cm) .. ($ (c) + (.2,-.08) $);		
	\draw[thick, yString] ($ (a2) + (.3,-.1) $) -- ($ (a2) + (.3,.4) $);
	\draw[thick, zString] ($ (a2) + (.3,.5) $) -- ($ (b2) + (.3,-.1) $).. controls ++(90:.8cm) and ++(270:.8cm) .. ($ (c) + (.4,-.08) $);		

	\nbox{unshaded}{($ (a2) + (.3,.4) $)}{.2}{-.05}{-.05}{\scriptsize{$f$}}

	\draw[thick] (a1) -- (b1);
	\draw[thick] ($ (a1) + (.6,0) $) -- ($ (a1) + (.6,1) $);
	\halfDottedEllipse{(a1)}{.3}{.1}	
	
	\draw[thick] (a2) -- (b2);
	\draw[thick] ($ (a2) + (.6,0) $) -- ($ (a2) + (.6,1) $);
	\halfDottedEllipse{(a2)}{.3}{.1}	
\end{tikzpicture}
\,\,=
\begin{tikzpicture}[baseline=.1cm]

	\coordinate (a1) at (-1,-1);
	\coordinate (a2) at ($ (a1) + (1.4,0) $);
	\coordinate (b) at ($ (a1) + (.7,1.5)$);
	\coordinate (c) at ($ (b) + (0,1)$);

	\pairOfPants{(a1)}{}
	
	\draw[thick, xString] ($ (a1) + (.3,-.1) $) .. controls ++(90:.8cm) and ++(270:.8cm) .. ($ (b) + (.2,-.08) $) -- ($ (c) + (.2,-.08) $);		
	\draw[thick, zString] ($ (b) + (.4,.4) $) -- ($ (c) + (.4,-.08) $);
	\draw[thick, yString] ($ (a2) + (.3,-.1) $) .. controls ++(90:.8cm) and ++(270:.8cm) .. ($ (b) + (.4,-.08) $) -- ($ (b) + (.4,.4) $);		

	\nbox{unshaded}{($ (b) + (.4,.4) $)}{.2}{-.07}{-.07}{\scriptsize{$f$}}

	\draw[thick] (b) -- (c);
	\draw[thick] ($ (b) + (.6,0) $) -- ($ (c) + (.6,0) $);
	\draw[thick] ($ (c) + (.3,0) $) ellipse (.3cm and .1cm);
\end{tikzpicture}.
\end{equation}
Similarly, for $g:w\to x$, we have $\mu_{x,y}\circ (\Tr_\cC(g)\otimes \id_{\Tr_\cC(y)})=\Tr_\cC(g\otimes \id_{y}) \circ \mu_{w,y}$.
\end{cor}
\begin{proof}
We only show the first statement; the other is similar.
By taking mates, and using Lemma \ref{lem:PushDownMorphism}, Equation \eqref{eq: f through pants} becomes
$\varepsilon\circ\Phi\big(\mu_{x,z}\circ (\id_{\Tr_\cC(x)}\otimes\Tr_\cC(f))\big)=
\varepsilon\circ\Phi\big(\Tr_\cC(\id_x\otimes f) \circ \mu_{x,y}\big)$
\[
\begin{tikzpicture}[baseline=.2cm, scale=.7]
	\plane{(-1.9,-1.1)}{6.3}{3.1}
	\draw[thick, zString] (1.4,0) -- (3.3,0);
	\draw[thick, xString] (1.4,.3) -- (3,.3);
	\CMbox{box}{(.2,-.4)}{1.2}{1.4}{.5}{$\varepsilon_{x\otimes y}$}
	\fill[unshaded] (-1,1.7) -- (-4.8,1.7) -- (-3.8,.7) -- ++(.2,0) -- ++(.35,-.35) -- ++(-.2,0) -- ++(1.1,-1.1) -- ++(.5,0) -- (-.5,.1) -- (0,.1) -- (0,1.1) -- ++(-1,0);
	\draw[thick, xString] (-2.9,1.2)+(-1.7,0)--+(0,0) to[out=0,in=180] (-.1,.7) -- (.08,.7);
	\draw[thick, zString] (-1.6,-.2)+(-.7,0)--+(0,0) to[out=10,in=180] (-.15,.5) -- (.08,.5);
	\draw[thick, yString] (-1.6,-.2)+(-.7,0) -- +(-1.7,0);
	\draw[thick] (-3.1,1.56) to[out=0,in=180] (-.1,1);
	\draw[thick] (-1.8,-.6) to[out=5,in=180] (-.1,.2);
	\draw[thick] (-1.8,.2) to[out=10,in=2, looseness=2] (-3.1,.84);
	\draw[thick] (-3.1,1.2)+(-1.7,0) circle (.2 and .36);
	\draw[thick] (-1.8,-.2)+(-1.7,0) circle (.2 and .4);
	\draw[thick] (-3.1,1.2)++(-1.7,-.36) -- ++(1.7,0) arc (-90:90:.2 and .36) -- ++(-1.7,0);
	\draw[thick] (-1.8,-.2)++(-1.7,-.4) -- ++(1.7,0) arc (-90:90:.2 and .4) -- ++(-1.7,0);
	\draw[thick] (-.1,.2) arc (-90:90:.2 and .4);
	\nbox{unshaded}{(-2.5,-.2)}{.2}{.2}{.2}{\rotatebox{-75}{$f$}}
\end{tikzpicture}
\!=\,\,
\begin{tikzpicture}[baseline=.2cm, scale=.7]
	\plane{(-1.9,-1.1)}{6.3}{3.1}
	\draw[thick, zString] (1.4,0) -- (3.3,0);
	\draw[thick, xString] (1.4,.3) -- (3,.3);
	\CMbox{box}{(.2,-.4)}{1.2}{1.4}{.5}{$\varepsilon_{x\otimes y}$}
	\fill[unshaded] (-1,1.7) -- (-4.8,1.7) -- (-3.75,.65) -- ++(.2,0) -- ++(.3,-.3) -- ++(-.2,0) -- ++(.7,-.7) -- ++(.5,.1) -- (-.5,.1) -- (0,.1) -- (.15,.2) -- (.15,.9) -- (0,1.1) -- ++(-1,0);
	\draw[thick, xString] (-2.9-1.7,1.2) to[out=0,in=180] (-.1-1.7,.7) -- (.08-1.7,.7) -- (.08,.7);
	\draw[thick, yString] (-1.6-1.7,-.2) to[out=10,in=180] (-.15-1.7,.5) -- (.08-1.7,.5) -- ++(.5,0);
	\draw[thick, zString] (.08-1.7,.5) ++(.7,0) -- ++(1,0);
	\draw[thick] (-3.1-1.7,1.56) to[out=0,in=180] (-.1-1.7,1);
	\draw[thick] (-1.8-1.7,-.6) to[out=5,in=180] (-.1-1.7,.2);
	\draw[thick] (-1.8-1.7,.2) to[out=10,in=2, looseness=2] (-3.1-1.7,.84);
	\draw[thick] (-3.1,1.2)+(-1.7,0) circle (.2 and .36);
	\draw[thick] (-1.8,-.2)+(-1.7,0) circle (.2 and .4);
	\draw[thick] (-.1-1.7,.2) arc (-90:90:.2 and .4);
	\draw[thick] (-.1,.2)+(-1.7,0) -- +(0,0) arc (-90:90:.2 and .4) -- +(-1.7,0);
	\nbox{unshaded}{(-.7,.45)}{.15}{.35}{.2}{\rotatebox{-75}{$f$}}
\end{tikzpicture}
\]
which follows readily from Lemmas \ref{lem:MultiplicationCompatibleWithNatural} and \ref{lem:MoveThroughNaturalMap}, along with the fact that $\Phi$ is a tensor functor.
\end{proof}

\begin{lem}\label{lem:Associative}
The multiplication map $\mu$ is associative, i.e., the following diagram commutes:
\[
\xymatrix{
\Tr_\cC(x) \otimes \Tr_\cC(y) \otimes \Tr_\cC(z) \ar[rr]^{\id_{\Tr_\cC(x)}\otimes\mu}\ar[d]_{\mu\otimes \id_{\Tr_\cC(z)}} 
&&
\Tr_\cC(x) \otimes \Tr_\cC(y \otimes z) \ar[d]^{\mu}
\\
\Tr_\cC(x \otimes y) \otimes \Tr_\cC(z) \ar[rr]^{\mu}  
&& 
\Tr_\cC(x \otimes y \otimes z) 
}
\]
In diagrams,
\begin{equation}\label{eqn:RHS -- 1st version}
\begin{tikzpicture}[baseline=.3cm, scale=.7]
	\pgfmathsetmacro{\voffset}{.08};
	\pgfmathsetmacro{\hoffset}{.15};

	\coordinate (a) at (-1,-1);
	\coordinate (a1) at ($ (a) + (1.4,0)$);
	\coordinate (a2) at ($ (a1) + (1.4,0)$);
	\coordinate (b) at ($ (a) + (.7,1.5)$);
	\coordinate (b2) at ($ (b) + (1.4,0)$);
	\coordinate (c) at ($ (b) + (.7,1.5)$);	

	\pairOfPants{(a)}{}
	\topPairOfPants{(b)}{}
	\LeftSlantCylinder{(a2)}{}

	\draw[thick, xString] ($ (a) + 2*(\hoffset,0) + (0,-.1)$) .. controls ++(90:.8cm) and ++(270:.8cm) .. ($ (b) + (\hoffset,0) + (0,-\voffset) $);	
	\draw[thick, yString] ($ (a1) + 2*(\hoffset,0) + (0,-.1)$) .. controls ++(90:.8cm) and ++(270:.8cm) .. ($ (b) + 2*(\hoffset,0) + (0,-.1) $);	
	\draw[thick, zString] ($ (a2) + 2*(\hoffset,0) + (0,-.1)$) .. controls ++(90:.8cm) and ++(270:.6cm) .. ($ (b2) + 2*(\hoffset,0) + (0,-.1) $) -- ($ (b2) + 2*(\hoffset,0) + (0,-.1) $);	
	
	\draw[thick, xString] ($ (b) + (\hoffset,0) + (0,-\voffset)$) .. controls ++(90:.8cm) and ++(270:.8cm) .. ($ (c) + (\hoffset,0) + (0,-\voffset) $);
	\draw[thick, yString] ($ (b) + 2*(\hoffset,0) + (0,-.1)$) .. controls ++(90:.8cm) and ++(270:.8cm) .. ($ (c) + 2*(\hoffset,0) + (0,-\voffset) $);	
	\draw[thick, zString] ($ (b2) + 2*(\hoffset,0) + (0,-.1)$) .. controls ++(90:.8cm) and ++(270:.6cm) .. ($ (c) + 3*(\hoffset,0) + (0,-\voffset) $);	
\end{tikzpicture}
=
\begin{tikzpicture}[baseline=.3cm, scale=.7, xscale=-1]
	\pgfmathsetmacro{\voffset}{.08};
	\pgfmathsetmacro{\hoffset}{.15};

	\coordinate (a) at (-1,-1);
	\coordinate (a1) at ($ (a) + (1.4,0)$);
	\coordinate (a2) at ($ (a1) + (1.4,0)$);
	\coordinate (b) at ($ (a) + (.7,1.5)$);
	\coordinate (b2) at ($ (b) + (1.4,0)$);
	\coordinate (c) at ($ (b) + (.7,1.5)$);	

	\pairOfPants{(a)}{}
	\topPairOfPants{(b)}{}
	\LeftSlantCylinder{(a2)}{}

	\draw[thick, zString] ($ (a) + 2*(\hoffset,0) + (0,-.1)$) .. controls ++(90:.8cm) and ++(270:.8cm) .. ($ (b) + (\hoffset,0) + (0,-\voffset) $);	
	\draw[thick, yString] ($ (a1) + 2*(\hoffset,0) + (0,-.1)$) .. controls ++(90:.8cm) and ++(270:.8cm) .. ($ (b) + 2*(\hoffset,0) + (0,-.1) $);	
	\draw[thick, xString] ($ (a2) + 2*(\hoffset,0) + (0,-.1)$) .. controls ++(90:.8cm) and ++(270:.6cm) .. ($ (b2) + 2*(\hoffset,0) + (0,-.1) $) -- ($ (b2) + 2*(\hoffset,0) + (0,-.1) $);	
	
	\draw[thick, zString] ($ (b) + (\hoffset,0) + (0,-\voffset)$) .. controls ++(90:.8cm) and ++(270:.8cm) .. ($ (c) + (\hoffset,0) + (0,-\voffset) $);
	\draw[thick, yString] ($ (b) + 2*(\hoffset,0) + (0,-.1)$) .. controls ++(90:.8cm) and ++(270:.8cm) .. ($ (c) + 2*(\hoffset,0) + (0,-\voffset) $);	
	\draw[thick, xString] ($ (b2) + 2*(\hoffset,0) + (0,-.1)$) .. controls ++(90:.8cm) and ++(270:.6cm) .. ($ (c) + 3*(\hoffset,0) + (0,-\voffset) $);	
\end{tikzpicture}
.
\end{equation}
\end{lem}

\begin{proof}
We claim that, under the adjunction, each map individually is the mate of
$$
\qquad\begin{tikzpicture}[baseline=.9cm, scale=.8]
	\plane{(0,-.6)}{3}{3.6}

	\draw[thick, xString] (-1.8,2) -- (.4,2);
	\draw[thick, yString] (-.8,1) -- (1.4,1);
	\draw[thick, zString] (.2,0) -- (2.4,0);

	\CMbox{box}{(-1.8,1.6)}{.8}{.8}{.4}{$\varepsilon_x$}
	\CMbox{box}{(-.8,.6)}{.8}{.8}{.4}{$\varepsilon_y$}
	\CMbox{box}{(.2,-.4)}{.8}{.8}{.4}{$\varepsilon_{z}$}

	\straightTubeWithString{(-2,2)}{.1}{1.4}{xString}
	\straightTubeWithString{(-1,1)}{.1}{1.4}{yString}
	\straightTubeWithString{(0,0)}{.1}{1.4}{zString}

\end{tikzpicture}
\hspace{-1cm}
\in\,\, \cM\Big(\Phi\big(\Tr_\cC(x)\otimes \Tr_\cC(y)\otimes \Tr_\cC(z)\big),\, x\otimes y\otimes z\Big).
$$
We only prove this for the right hand side of Equation \eqref{eqn:RHS -- 1st version} (the left hand side is similar).
The mate of the right hand side is given by
\begin{align*}
&\varepsilon_{x\otimes y\otimes z}
\circ
\Phi\big(\mu_{x,y\otimes z}\circ (\id_{\Tr_\cC(x)}\otimes \mu_{y,z})\big)
=
\varepsilon_{x\otimes y\otimes z}
\circ
\Phi(\mu_{x,y\otimes z})\circ \Phi(\id_{\Tr_\cC(x)}\otimes \mu_{y,z})
\\
&\,\,\,=
(\varepsilon_x\otimes\varepsilon_{y\otimes z})\circ\big(\Phi(\id_{\Tr_\cC(x)})\otimes\Phi(\mu_{y,z})\big)
=
\varepsilon_x \otimes 
\big(
\varepsilon_{y\otimes z}
\circ
\Phi(\mu_{y,z})
\big)
=
\varepsilon_x\otimes (\varepsilon_y\otimes \varepsilon_z),
\end{align*}
where the first equality is the functoriality of $\Phi$, the second equality is Lemma \ref{lem:MultiplicationCompatibleWithNatural} and the fact that $\Phi$ is a tensor functor, and the fourth equality is again Lemma \ref{lem:MultiplicationCompatibleWithNatural}.
\end{proof}

\subsection{The unit of the adjunction}
\label{sec:UnitMap}
In Section \ref{sec:AdjointsOfTensorFunctors} we studied the attaching map $\varepsilon$, which is the counit of the adjunction $\Phi\dashv \Tr_\cC$.
We now study the unit of the adjunction, which we denote $\eta$.

Thus, for any object $c\in \cC$, we have a morphism $\eta_c:c \to \Tr_\cC(\Phi(c))$.
We will also write $i:1_\cC\to \Tr_\cC(1_\cM)$ for the unit $\eta$ evaluated on the unit object $1_\cC\in\cC$.
We represent $\eta$ and $i$ by the following diagrams:
$$
\eta_c \,\,=\,\, 
\begin{tikzpicture}[baseline=-.1cm]
	\coordinate (a1) at (0,0);
	\coordinate (b1) at (0,.4);
	\draw[thick] (a1) -- (b1);
	\draw[thick] ($ (a1) + (.6,0) $) -- ($ (b1) + (.6,0) $);
	\draw[thick] ($ (b1) + (.3,0) $) ellipse (.3 and .1);
	\draw[thick] (a1) arc (-180:0:.3cm);
	\draw[super thick, white] (0,-.6) -- (0,-.5) .. controls ++(90:.2cm) and ++(270:.2cm) .. ($ (a1) + (.3,0) $) -- ($ (b1) + (.3,-.05) $);
	\draw[thick, cString] (0,-.6) -- (0,-.5) .. controls ++(90:.2cm) and ++(270:.2cm) .. ($ (a1) + (.3,0) $) -- ($ (b1) + (.3,-.05) $);
\end{tikzpicture} 
\qquad\qquad\quad
i \,\,=\,\, 
\begin{tikzpicture}[baseline=-.1cm]
	\coordinate (a1) at (0,0);
	\coordinate (b1) at (0,.4);
	\draw[thick] (a1) -- (b1);
	\draw[thick] ($ (a1) + (.6,0) $) -- ($ (b1) + (.6,0) $);
	\draw[thick] ($ (b1) + (.3,0) $) ellipse (.3 and .1);
	\draw[thick] (a1) arc (-180:0:.3cm);
\end{tikzpicture} 
\,.
$$

Since we have an adjunction $\Phi\dashv \Tr_\cC$, the counit $\varepsilon$ and the unit $\eta$ interact in the following way:
\begin{enumerate}[(1)]
\item
For $x\in \cM$, the mate of $\varepsilon_{x}$ is $\id_{\Tr_\cC(x)}$.
Hence $\Tr_\cC(\varepsilon_x)\circ \eta_{\Tr_\cC(x)}=\id_{\Tr_\cC(x)} $.
We represent this diagrammatically by the following relation.
\begin{equation}
\label{eq:CAdjointRelation}
\begin{tikzpicture}[baseline=.3cm]
	\coordinate (a1) at (0,0);
	\coordinate (b1) at (0,1.5);

	\draw[thick] (a1) -- (b1);
	\draw[thick] ($ (a1) + (.6,0) $) -- ($ (b1) + (.6,0) $);
	\draw[thick] ($ (b1) + (.3,0) $) ellipse (.3 and .1);
	\draw[thick] (a1) arc (-180:0:.3cm);
	\downTubeWithString{(-.2,.2)}{.2}{1}{xString}
	\CMbox{box}{(-.2,.2)}{.4}{.4}{.2}{$\varepsilon$}
	\draw[thick, dotted] (-.4,-.4) rectangle (.8,0);
	\node at (1.4,-.2) {\scriptsize{$\eta_{\Tr_\cC(x)}$}};
	\draw[thick, xString] (.1,.7) -- (.1,1.42) ;
\end{tikzpicture}
=\,\,
\begin{tikzpicture}[baseline=.9cm]
	\coordinate (a1) at (0,0);
	\coordinate (b1) at (0,2);

	\draw[thick] (a1) -- (b1);
	\draw[thick] ($ (a1) + (.6,0) $) -- ($ (b1) + (.6,0) $);
	\draw[thick] ($ (b1) + (.3,0) $) ellipse (.3 and .1);
	\halfDottedEllipse{(a1)}{.3}{.1}
	\draw[thick, xString] (.3,-.1) -- (.3,1.9) ;
\end{tikzpicture}
\end{equation}
\item
For $c\in\cC$, the mate of $\eta_c$ is $\id_{\Phi(c)}$.
So by Lemma \ref{lem:PushDownMorphism}, $\Phi(\eta_c)\circ \varepsilon_{\Phi(c)}=\id_{\Phi(c)}$. 
We represent this diagrammatically by the following relation.
\begin{equation}
\label{eq:MAdjointRelation}
\begin{tikzpicture}[baseline=.1cm, scale=.8]
	\plane{(-.4,-.8)}{2.8}{1.8}

	\CMbox{box}{(0,-.5)}{.8}{.8}{.4}{$\varepsilon$}

	\straightTubeWithCap{(-.2,-.1)}{.1}{.5}
	\draw[super thick, white] (.9,0) -- (2.2,0);
	\draw[thick, cString] (.8,0) -- (2.2,0);

	\draw[super thick, white] (-2,.3) -- (-1.1,.3) .. controls ++(0:.2cm) and ++(180:.2cm) .. (-.6,.1) -- (-.07,.1) ;
	\draw[thick, cString] (-2,.3) -- (-1.1,.3) .. controls ++(0:.2cm) and ++(180:.2cm) .. (-.6,.1) -- (-.07,.1) ;
\end{tikzpicture}
=
\begin{tikzpicture}[baseline=.1cm, scale=.8 ]
	\plane{(0,-.8)}{2}{1.8}

	\draw[super thick, white] (-1.4,0) -- (1.8,0);
	\draw[thick, cString] (-1.4,0) -- (1.8,0);
\end{tikzpicture}
\end{equation}
\end{enumerate}

The following lemmas demonstrate how $\eta$ and $i$ interact with the graphical calculus introduced so far.

\begin{lem}\label{lem:CreateUnit}
The following diagram commutes:
\[
\xymatrix{
1_\cC \otimes \Tr_\cC(x)\ar[d]^{i\otimes  \id_{\Tr_\cC(x)}}\ar[rr] && \Tr_\cC(x)\ar[d]^{\id_{\Tr_\cC(x)}} && \Tr_\cC(x)\otimes 1_\cC \ar[ll]\ar[d]^{\id_{\Tr_\cC(x)}\otimes i}
\\
\Tr_\cC(1) \otimes \Tr_\cC(x) \ar[r]^(.55)\mu & \Tr_\cC(1 \otimes x)\ar[r] &\Tr_\cC(x)& \Tr_\cC(x\otimes 1)\ar[l] & \Tr_\cC(x)\otimes \Tr_\cC(1) \ar[l]_(.55)\mu.
}
\]
In pictures, this is
$$
\begin{tikzpicture}[baseline=1.15cm]

	\coordinate (a1) at (0,0);
	\coordinate (a2) at (1.4,0);
	\coordinate (b1) at (0,1);
	\coordinate (b2) at (1.4,1);
	\coordinate (c1) at (.7,2.5);
	
	\draw[thick] (a2) -- (b2);
	\draw[thick] ($ (a2) + (.6,0) $) -- ($ (b2) + (.6,0) $);
	\halfDottedEllipse{(a2)}{.3}{.1}
	\draw[thick] (b1) arc (-180:0:.3cm);
	\topPairOfPants{(b1)}{}
		
	\draw[thick, xString] ($ (a2) + (.3,-.1) $) -- ($ (b2) + (.3,-.1) $) .. controls ++(90:.8cm) and ++(270:.8cm) .. ($ (c1) + (.3,-.1) $);
\end{tikzpicture}
=
\begin{tikzpicture}[baseline=1.15cm]

	\coordinate (a1) at (0,0);
	\coordinate (b1) at (0,2.5);
	
	\draw[thick] (a1) -- (b1);
	\draw[thick] ($ (a1) + (.6,0) $) -- ($ (b1) + (.6,0) $);
	\halfDottedEllipse{(a1)}{.3}{.1}
	\draw[thick] ($ (b1) + (.3,0) $) ellipse (.3cm and .1cm);
		
	\draw[thick, xString] ($ (a1) + (.3,-.1) $) -- ($ (b1) + (.3,-.1) $);
\end{tikzpicture}
=
\begin{tikzpicture}[baseline=1.15cm, xscale=-1]

	\coordinate (a1) at (0,0);
	\coordinate (a2) at (1.4,0);
	\coordinate (b1) at (0,1);
	\coordinate (b2) at (1.4,1);
	\coordinate (c1) at (.7,2.5);
	
	\draw[thick] (a2) -- (b2);
	\draw[thick] ($ (a2) + (.6,0) $) -- ($ (b2) + (.6,0) $);
	\halfDottedEllipse{(a2)}{.3}{.1}
	\draw[thick] (b1) arc (-180:0:.3cm);
	\topPairOfPants{(b1)}{}
		
	\draw[thick, xString] ($ (a2) + (.3,-.1) $) -- ($ (b2) + (.3,-.1) $) .. controls ++(90:.8cm) and ++(270:.8cm) .. ($ (c1) + (.3,-.1) $);
\end{tikzpicture}
$$
\end{lem}
\begin{proof}
We check that the left square commutes upon taking mates (the right one is similar).
Using Lemma \ref{lem:PushDownMorphism}, we compute:
\begin{align*}
\varepsilon_x\circ\Phi\big(&\mu_{1,x}\circ (i\otimes \id_{\Tr_\cC(x)})\big)
=
\varepsilon_x\circ\Phi\big(\mu_{1,x}\big)\circ \Phi\big(i\otimes \id_{\Tr_\cC(x)})\big)
\\&=
(\varepsilon_{1}\otimes \varepsilon_x)\circ \big(\Phi(i)\otimes \Phi(\id_{\Tr_\cC(x)})\big)
=
\big(\varepsilon_{1}\circ \Phi(i)\big)\otimes \varepsilon_x =\varepsilon_x.
\end{align*}
We have used Lemma \ref{lem:MultiplicationCompatibleWithNatural} for the second equality, and Equation \eqref{eq:MAdjointRelation} for the last one.
\end{proof}

Note that a lot of what we have done so far is to reprove in our special case the well known fact that the adjoint of a tensor functor is a lax tensor functor \cite{MR0360749}.
The data is provided by the natural transformations $\mu$ and $\eta$, and the axioms are verified in Lemmas~\ref{lem:Associative} and~\ref{lem:CreateUnit}.

\begin{lem}
\label{lem:EtaSplitting}
For $c,d\in\cC$, we have $\eta_{c\otimes d} = \mu_{\Phi(c),\Phi(d)} \circ (\eta_c \otimes \eta_d)$.
In diagrams:
$$
\begin{tikzpicture}[baseline=-.1cm]
	\coordinate (a1) at (0,0);
	\coordinate (b1) at (0,.4);
	\draw[thick] (a1) -- (b1);
	\draw[thick] ($ (a1) + (.6,0) $) -- ($ (b1) + (.6,0) $);
	\draw[thick] ($ (b1) + (.3,0) $) ellipse (.3 and .1);
	\draw[thick] (a1) arc (-180:0:.3cm);
	\draw[super thick, white] (.05,-.6) -- (.05,-.5) .. controls ++(90:.2cm) and ++(270:.2cm) .. ($ (a1) + (.35,0) $) -- ($ (b1) + (.35,-.05) $);
	\draw[super thick, white] (-.05,-.6) -- (-.05,-.5) .. controls ++(90:.2cm) and ++(270:.2cm) .. ($ (a1) + (.25,0) $) -- ($ (b1) + (.25,-.05) $);
	\draw[thick, dString] (.05,-.6) -- (.05,-.5) .. controls ++(90:.2cm) and ++(270:.2cm) .. ($ (a1) + (.35,0) $) -- ($ (b1) + (.35,-.05) $);
	\draw[thick, cString] (-.05,-.6) -- (-.05,-.5) .. controls ++(90:.2cm) and ++(270:.2cm) .. ($ (a1) + (.25,0) $) -- ($ (b1) + (.25,-.05) $);
\end{tikzpicture}
\,\,\,=\,\,
\begin{tikzpicture}[baseline=.8cm]
	\coordinate (a1) at (0,0);
	\coordinate (b1) at (0,.4);
	\coordinate (a2) at (1.4,0);
	\coordinate (b2) at (1.4,.4);
	\draw[thick] (a1) -- (b1);
	\draw[thick] ($ (a1) + (.6,0) $) -- ($ (b1) + (.6,0) $);
	\draw[thick] (a2) -- (b2);
	\draw[thick] ($ (a2) + (.6,0) $) -- ($ (b2) + (.6,0) $);
	\topPairOfPants{(b1)}{}
	\draw[thick] (a1) arc (-180:0:.3cm);
	\draw[thick] (a2) arc (-180:0:.3cm);
	\draw[super thick, white] ($ (a1) + (0,-.6) $) --($ (a1) + (0,-.5) $) .. controls ++(90:.2cm) and ++(270:.2cm) .. ($ (a1) + (.3,0) $) -- ($ (b1) + (.3,-.05) $) .. controls ++(90:.8cm) and ++(270:.8cm) .. ($ (b1) + (.95,1.45) $);
	\draw[thick, cString] ($ (a1) + (0,-.6) $) --($ (a1) + (0,-.5) $) .. controls ++(90:.2cm) and ++(270:.2cm) .. ($ (a1) + (.3,0) $) -- ($ (b1) + (.3,-.05) $) .. controls ++(90:.8cm) and ++(270:.8cm) .. ($ (b1) + (.95,1.45) $);
	\draw[super thick, white] ($ (a2) + (0,-.6) $) --($ (a2) + (0,-.5) $) .. controls ++(90:.2cm) and ++(270:.2cm) .. ($ (a2) + (.3,0) $) -- ($ (b2) + (.3,-.05) $) .. controls ++(90:.8cm) and ++(270:.8cm) .. ($ (b1) + (1.05,1.45) $);
	\draw[thick, dString] ($ (a2) + (0,-.6) $) --($ (a2) + (0,-.5) $) .. controls ++(90:.2cm) and ++(270:.2cm) .. ($ (a2) + (.3,0) $) -- ($ (b2) + (.3,-.05) $) .. controls ++(90:.8cm) and ++(270:.8cm) .. ($ (b1) + (1.05,1.45) $);
\end{tikzpicture}
$$
\end{lem}
\begin{proof}
By Equation \eqref{eq:MAdjointRelation}, we know that $\Phi(\eta_{c\otimes d})\circ\varepsilon_{\Phi(c\otimes d)}=\id_{\Phi(c\otimes d)}$, so
it suffices to show that the mate of the right hand side is equal to $\id_{\Phi(c\otimes d)}$.
The result now follows from Lemma~\ref{lem:MultiplicationCompatibleWithNatural} together with two applications of the relation~\eqref{eq:MAdjointRelation}.
\begin{align*}
\begin{tikzpicture}[baseline=.4cm, scale=.7]
	\plane{(-.4,-1.1)}{4.8}{3.1}
	\draw[super thick, white] (1.5,0) -- (3.7,0);
	\draw[super thick, white] (1.5,.3) -- (3.4,.3);
	\draw[thick, dString] (1.4,0) -- (3.7,0);
	\draw[thick, cString] (1.4,.3) -- (3.4,.3);
	\CMbox{box}{(.2,-.4)}{1.2}{1.4}{.5}{$\varepsilon$}
	\fill[unshaded] (-1,1.7) -- (-3.3,1.7) -- (-2.25,.65) -- ++(.2,0) -- ++(.3,-.3) -- ++(-.2,0) -- ++(.8,-.8) -- ++(.5,.1) -- (-.5,.1) -- (0,.1) -- (0,1.1) -- ++(-1,0);
	\draw[thick] (-3.1,1.56) to[out=0,in=180] (-.1,1);
	\draw[thick] (-1.8,-.6) to[out=5,in=180] (-.1,.2);
	\draw[thick] (-1.8,.2) to[out=10,in=2, looseness=2] (-3.1,.84);
	\draw[thick] (-3.1,.84) arc (270:90:.36cm);
	\draw[thick] (-1.8,-.6) arc (270:90:.4cm);
	\draw[thick] (-.1,.2) arc (-90:90:.2 and .4);
	\fill[unshaded] (.1,.6) circle (.04cm);
	\draw[super thick, white] (.05,.5) -- (.125,.5);
	\draw[super thick, white] (.05,.7) -- (.125,.7);
	\draw[thick, cString] (-2.9,1.2) to[out=0,in=180] (-.1,.7) -- (.12,.7);
	\draw[thick, dString] (-1.6,-.2) to[out=10,in=180] (-.15,.5) -- (.12,.5);
	\draw[super thick, white] (-4.3,1.56) -- (-3.9,1.56) .. controls ++(0:.4cm) and ++(180:.4cm) .. (-2.9,1.2);
	\draw[thick, cString] (-4.3,1.56) -- (-3.9,1.56) .. controls ++(0:.4cm) and ++(180:.4cm) .. (-2.9,1.2);
	\draw[super thick, white] (-3,.2) -- (-2.6,.2) .. controls ++(0:.4cm) and ++(180:.4cm) .. (-1.6,-.2);
	\draw[thick, dString] (-3,.2) -- (-2.6,.2) .. controls ++(0:.4cm) and ++(180:.4cm) .. (-1.6,-.2);
\end{tikzpicture}
\!\!&=\;\;
\begin{tikzpicture}[baseline=.65cm, scale=.8]
	\plane{(-.5,-.8)}{3}{3}
	\CMbox{box}{(-1.4,.8)}{.8}{.8}{.4}{$\varepsilon$}
	\CMbox{box}{(0,-.6)}{.8}{.8}{.4}{$\varepsilon$}
	\straightTubeWithCap{(-.2,-.3)}{.1}{.5}
	\draw[super thick, white] (-2,.1) -- (-1.1,.1) .. controls ++(0:.2cm) and ++(180:.2cm) .. (-.6,-.1) -- (-.07,-.1) ;
	\draw[thick, dString] (-2,.1) -- (-1.1,.1) .. controls ++(0:.2cm) and ++(180:.2cm) .. (-.6,-.1) -- (-.07,-.1) ;
	\straightTubeWithCap{(-1.6,1.1)}{.1}{.5}
	\draw[super thick, white] (-3.4,1.5) -- (-2.5,1.5) .. controls ++(0:.2cm) and ++(180:.2cm) .. (-2,1.3) -- (-1.47,1.3) ;
	\draw[thick, cString] (-3.4,1.5) -- (-2.5,1.5) .. controls ++(0:.2cm) and ++(180:.2cm) .. (-2,1.3) -- (-1.47,1.3) ;
	\draw[super thick, white] (.9,-.1) -- (2.2,-.1);
	\draw[thick, dString] (.8,-.1) -- (2.2,-.1);
	\draw[super thick, white] (-.5,1.3) -- (.8,1.3);
	\draw[thick, cString] (-.6,1.3) -- (.8,1.3);
\end{tikzpicture}
=
\begin{tikzpicture}[baseline=.65cm, scale=.8]
	\plane{(.5,-.8)}{2}{3}
	\draw[super thick, white] (-.8,-.1) -- (2.4,-.1);
	\draw[thick, dString] (-.8,-.1) -- (2.4,-.1);
	\draw[super thick, white] (-2.2,1.3) -- (1,1.3);
	\draw[thick, cString] (-2.2,1.3) -- (1,1.3);
\end{tikzpicture}
\qedhere
\end{align*}
\end{proof}

\begin{lem}
\label{lem:TrCSplitting}
The following two maps $\Tr_\cC(x)\otimes c \to \Tr_\cC(x\otimes \Phi(c))$ are equal:
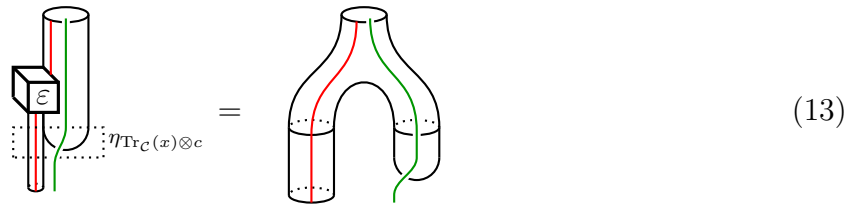
\begin{equation}
\begin{tikzpicture}[baseline=.1cm]
	\coordinate (a1) at (0,0);
	\coordinate (b1) at (0,1.5);

	\draw[thick] (a1) -- (b1);
	\draw[thick] ($ (a1) + (.6,0) $) -- ($ (b1) + (.6,0) $);
	\draw[thick] ($ (b1) + (.3,0) $) ellipse (.3 and .1);
	\draw[thick] (a1) arc (-180:0:.3cm);
	\downTubeWithString{(-.2,.2)}{.2}{1}{xString}
	\CMbox{box}{(-.2,.2)}{.4}{.4}{.2}{$\varepsilon$}
	\draw[thick, dotted] (-.4,-.4) rectangle (.8,0);
	\node at (1.5,-.2) {\scriptsize{$\eta_{\Tr_\cC(x)\otimes c}$}};
	\draw[thick, xString] (.1,.7) -- (.1,1.42) ;
	\draw[super thick, white] ($ (a1) + (.15,-.85) $) -- ($ (a1) + (.15,-.5) $) .. controls ++(90:.2cm) and ++(270:.2cm) .. ($ (a1) + (.3,0) $) -- ($ (b1) + (.3,-.05) $) -- ($ (b1) + (.3,-.05) $);
	\draw[thick, cString] ($ (a1) + (.15,-.85) $) -- ($ (a1) + (.15,-.5) $) .. controls ++(90:.2cm) and ++(270:.2cm) .. ($ (a1) + (.3,0) $) -- ($ (b1) + (.3,-.05) $) -- ($ (b1) + (.3,-.05) $);
\end{tikzpicture}
\!=\,\,\,\,\,\,\,
\begin{tikzpicture}[baseline=.5cm]
	\coordinate (a1) at (0,-.5);
	\coordinate (b1) at (0,.4);
	\coordinate (a2) at (1.4,0);
	\coordinate (b2) at (1.4,.4);
	\draw[thick] (a1) -- (b1);
	\draw[thick] ($ (a1) + (.6,0) $) -- ($ (b1) + (.6,0) $);
	\halfDottedEllipse{(a1)}{.3}{.1}
	\draw[thick] (a2) -- (b2);
	\draw[thick] ($ (a2) + (.6,0) $) -- ($ (b2) + (.6,0) $);
	\topPairOfPants{(b1)}{}
	\draw[thick] (a2) arc (-180:0:.3cm);
	\draw[thick, xString] ($ (a1) + (.3,-.1) $) -- ($ (b1) + (.3,-.1) $) .. controls ++(90:.8cm) and ++(270:.8cm) .. ($ (b1) + (.9,1.41) $);
	\draw[super thick, white] ($ (a2) + (0,-.6) $) --($ (a2) + (0,-.5) $) .. controls ++(90:.2cm) and ++(270:.2cm) .. ($ (a2) + (.3,0) $) -- ($ (b2) + (.3,-.05) $) .. controls ++(90:.8cm) and ++(270:.8cm) .. ($ (b1) + (1.08,1.45) $);
	\draw[thick, cString] ($ (a2) + (0,-.6) $) --($ (a2) + (0,-.5) $) .. controls ++(90:.2cm) and ++(270:.2cm) .. ($ (a2) + (.3,0) $) -- ($ (b2) + (.3,-.05) $) .. controls ++(90:.8cm) and ++(270:.8cm) .. ($ (b1) + (1.08,1.45) $);
\end{tikzpicture}
\label{eqn:TrSplittingIso}
\end{equation}
Moreover, when $\cC$ is rigid, the above map \eqref{eqn:TrSplittingIso} is invertible.
\end{lem}
\begin{proof}
To show that the two sides of \eqref{eqn:TrSplittingIso} are equal,
apply Lemma~\ref{lem:EtaSplitting} to the left hand side and use Equation~\eqref{eq:CAdjointRelation}.

If we assume that $c$ is dualizable,
then we can write down an inverse to the above map:
\[
\begin{tikzpicture}[baseline=.4cm]
	\coordinate (a1) at (0,-.7);
	\coordinate (b1) at (0,.4);
	\coordinate (a2) at (1.4,0);
	\coordinate (b2) at (1.4,.4);
	\draw[thick] (a1) -- (b1);
	\draw[thick] ($ (a1) + (.6,0) $) -- ($ (b1) + (.6,0) $);
	\halfDottedEllipse{(a1)}{.3}{.1}
	\draw[thick] (a2) -- (b2);
	\draw[thick] ($ (a2) + (.6,0) $) -- ($ (b2) + (.6,0) $);
	\topPairOfPants{(b1)}{}
	\draw[thick] (a2) arc (-180:0:.3cm);
	\draw[thick, xString] ($ (a1) + (.2,-.09) $) -- ($ (b1) + (.2,-.09) $) .. controls ++(90:.8cm) and ++(270:.5cm) .. ($ (b1) + (.95,1.41) $);
	\draw[super thick, white](2.4,2)  -- (2.4,-.2) .. controls ++(269:.7cm) and ++(270:.35cm) .. (1.6,-.4) .. controls ++(90:.2cm) and ++(270:.2cm) .. ($ (a2) + (.3,0) $) -- ($ (b2) + (.3,-.05) $) arc(0:180:.65 and .85)
	--($ (a1) + (.4,-.18) $);
	\draw[thick, cString] (2.4,2)  -- (2.4,-.2) .. controls ++(269:.7cm) and ++(270:.35cm) .. (1.6,-.4) .. controls ++(90:.2cm) and ++(270:.3cm) .. ($ (a2) + (.3,0) $) -- ($ (b2) + (.3,-.05) $) arc(0:180:.65 and .85)
	--($ (a1) + (.4,-.18) $);
\end{tikzpicture}\,\,\,:\Tr_\cC(x\otimes \Phi(c))\to \Tr_\cC(x)\otimes c
\]
In formulas, this is:
\[
[
(\Tr_\cC(\ev_{\Phi({}^*\hspace{-.01cm}c)})
\circ\mu_{x\otimes \Phi(c),\Phi({}^*\hspace{-.01cm}c)}
\circ(\id_{\Tr_\cC(x\otimes \Phi(c))}\otimes \eta_{{}^*\hspace{-.01cm}c}))\otimes \id_c]
\circ(\id_{\Tr_\cC(x\otimes \Phi(c))}\otimes\coev_{{}^*\hspace{-.01cm}c}),
\]
where we have used the canonical identification $\Phi({}^*\hspace{-.015cm}c)\cong {}^*\Phi(c)$ provided by the fact that $\Phi$ is a tensor functor.
The two maps compose to the identity by a straightforward application of Lemma \ref{lem:Associative} (the associativity of $\mu$), Lemma \ref{lem:EtaSplitting}, and the naturality of $\eta$.
\end{proof}


\subsection{Construction of the traciator}
\label{sec:Traciator}
So far we have constructed the multiplication $\mu_{x,y}: \Tr_\cC(x) \otimes \Tr_\cC(y) \to \Tr_\cC(x \otimes y)$ and the corresponding unit map under the assumption that $\cC$ and $\cM$ are tensor categories and that $\Phi:\cC\to \cM$ is a tensor functor. 
However, the construction of the traciator, along with the proof that it equips $\Tr_\cC$ with the structure of a categorified trace, only depends on $\Phi:\cC\to \cM$ factoring through $\cZ(\cM)$, and not on it being a tensor functor.
The material of this section can be found in \cite[Prop. 5]{MR2506324} and \cite[Prop. 2.5]{MR3250042}. 

We fix the following notation for Section~\ref{sec:Traciator}.
\begin{nota}
\label{nota:SecTraciator}
\mbox{}\vspace{-.1cm}

\begin{itemize}
\item
$\cC$ is a category
\item
$\cM$ is a pivotal category,
\item
$\Phi^{\scriptscriptstyle \cZ}: \cC\to \cZ(\cM)$ is a functor,
\item
the trace functor $\Tr_\cC:\cM\to \cC$ is the right adjoint of $\Phi:=F\circ \Phi^{\scriptscriptstyle \cZ}$. 
\end{itemize}
\end{nota}

\begin{defn}
\label{defn:Traciator}
For $x,y\in \cM$, we define the traciator $\tau_{x,y}: \Tr_\cC(x\otimes y) \to \Tr_\cC(y\otimes x)$ in the following way.
Let $c= \Tr_\cC(x\otimes y)$,
and let us write $\tilde\ev_y:y\otimes y^*\to 1$ for the composite $\ev_{y^*}\circ\,(\varphi_y\otimes\id_{y^*})$, with $\varphi_y:y\to y^{**}$ the pivotal structure.
Then $\tau_{x,y}$ is the mate of 
$$
	(\id_{y\otimes x}\otimes\, \tilde\ev_y)
\circ	(\id_y\otimes\, \varepsilon_{x\otimes y}\otimes \id_{y^*})
\circ	(e_{\Phi(c),y}\otimes \id_{y^*})
\circ	(\id_{\Phi(c)}\otimes\coev_y)
=
\begin{tikzpicture}[baseline=.1cm, scale=.8]
	\plane{(-.5,-.8)}{3.2}{2}
	\draw[thick, yString] (1.2,0) arc (90:-90:.3cm) -- (0,-.6) arc (270:90:.8cm) -- (.9,1);
	\draw[thick, xString] (1.2,.2) -- (1.7,.2);
	\CMbox{box}{(0,-.4)}{1.2}{.8}{.4}{$\varepsilon_{x\otimes y}$}
	\straightTubeTwoStrings{(-.2,0)}{2}{xString}{yString}
\end{tikzpicture}
$$
under the adjunction $\cM(\Phi(\Tr_\cC(x\otimes y)),y\otimes x)\cong \cC(\Tr_\cC(x\otimes y),\Tr_\cC(y\otimes x))$.
\end{defn}

As we will see shortly, the traciator is always invertible.
We will sometimes write $\tau^+_{x,y}$ for $\tau_{x,y}$, and $\tau^-_{x,y}$ for $\tau_{y,x}^{-1}$. 
In terms of $3$-dimensional diagrams, the traciators $\tau^\pm : \Tr_\cC(x\otimes y) \to \Tr_\cC(y\otimes x)$ are depicted
\[
\tau^+ =
\begin{tikzpicture}[baseline=-.1cm]

	\draw[thick] (-.3,-1) -- (-.3,1);
	\draw[thick] (.3,-1) -- (.3,1);
	\draw[thick] (0,1) ellipse (.3cm and .1cm);
	\halfDottedEllipse{(-.3,-1)}{.3}{.1}
	
	\draw[thick, xString] (-.1,-1.07) .. controls ++(90:.8cm) and ++(270:.8cm) .. (.1,.92);		
	\draw[thick, yString] (.1,-1.07) .. controls ++(90:.2cm) and ++(225:.2cm) .. (.3,-.2);		
	\draw[thick, yString] (-.1,.92) .. controls ++(270:.2cm) and ++(45:.2cm) .. (-.3,.2);
	\draw[thick, yString, dotted] (-.3,.2) -- (.3,-.2);	
\end{tikzpicture}
\qquad\,\,\,\text{and}\quad\,\,\,\,
\tau^- =
\begin{tikzpicture}[baseline=-.1cm, xscale=-1]

	\draw[thick] (-.3,-1) -- (-.3,1);
	\draw[thick] (.3,-1) -- (.3,1);
	\draw[thick] (0,1) ellipse (.3cm and .1cm);
	\halfDottedEllipse{(-.3,-1)}{.3}{.1}
	
	\draw[thick, yString] (-.1,-1.07) .. controls ++(90:.8cm) and ++(270:.8cm) .. (.1,.92);		
	\draw[thick, xString] (.1,-1.07) .. controls ++(90:.2cm) and ++(225:.2cm) .. (.3,-.2);		
	\draw[thick, xString] (-.1,.92) .. controls ++(270:.2cm) and ++(45:.2cm) .. (-.3,.2);
	\draw[thick, xString, dotted] (-.3,.2) -- (.3,-.2);	
\end{tikzpicture}\,\,.
\]

\begin{lem}\label{lem:TraciatorsComposeToIdentity}
The traciator $\tau_{y,x}$ is invertible, and its inverse $\tau^-_{x,y}$ is the mate of
\begin{equation}\label{mate of tau-}
\begin{tikzpicture}[baseline=.1cm, scale=.8]
	\plane{(-.5,-.8)}{3.7}{2}
	\draw[thick, xString] (2.7+.3,-.6) -- (-.05,-.6) arc (270:90:.8cm) -- +(.8,0) to[in=0,out=0, looseness=2] (1,.1);
	\draw[thick, yString] (1.2,-.1) -- (2.5,-.1);
	\CMbox{box}{(0,-.4)}{1.2}{.8}{.4}{$\varepsilon_{x\otimes y}$}
	\straightTubeTwoStrings{(-.2,0)}{2}{xString}{yString}
\end{tikzpicture}
\!\!\in\,
\cM(\Phi(\Tr_\cC(x\otimes y)),y\otimes x)
\end{equation}
\end{lem}
\begin{proof}
Let us define $\tilde\tau^-_{x,y}$ to be the mate of \eqref{mate of tau-} and let us agree, for the purpose of this proof, to reserve the graphical notation\, 
\(
\begin{tikzpicture}[baseline=-.12cm, xscale=-.5, yscale=.5, line width=.5]
	\draw (-.3,-.7) -- (-.3,.7);
	\draw (.3,-.7) -- (.3,.7);
	\draw (0,.7) ellipse (.3cm and .1cm);
	\draw (-.3,-.7) arc (180:360:.3cm and .1cm);	
	\draw[yString] (-.1,-1.07+.3) .. controls ++(90:.6cm) and ++(270:.6cm) .. (.1,.92-.3);		
	\draw[xString] (.1,-1.07+.3) .. controls ++(90:.2cm) and ++(225:.2cm) .. (.3,-.15);		
	\draw[xString] (-.1,.92-.3) .. controls ++(270:.2cm) and ++(45:.2cm) .. (-.3,.15);
	\draw[xString, dotted] (-.3,.15) -- (.3,-.15);	
\end{tikzpicture}
\)\,
for $\tilde\tau^-_{x,y}$.
We need to show that $\tilde\tau^-_{x,y}=\tau^{-}_{x,y}$.
Equivalently, we need to show that the following two equations hold:
$$
\tau_{y,x}\circ \tilde\tau^-_{x,y}=\id_{\Tr_\cC(x\otimes y)}=\tilde\tau^-_{y,x}\circ \tau_{x,y}\,\,:\quad
\begin{tikzpicture}[baseline=-.1cm]
\pgftransformxscale{-1}
	\draw[thick] (-.3,-1) -- (-.3,1);
	\draw[thick] (.3,-1) -- (.3,1);
	\draw[thick] (0,1) ellipse (.3cm and .1cm);
	\halfDottedEllipse{(-.3,-1)}{.3}{.1}
	\halfDottedEllipse{(-.3,0)}{.3}{.1}
		
	\draw[thick, yString] (-.1,-1.08) .. controls ++(90:.4cm) and ++(270:.4cm) .. (.1,-.08);		
	\draw[thick, yString] (-.1,.92) .. controls ++(270:.4cm) and ++(90:.4cm) .. (.1,-.08);		

	\draw[thick, xString] (.1,-1.08) .. controls ++(90:.2cm) and ++(225:.1cm) .. (.3,-.62);		
	\draw[thick, xString] (-.1,-.08) .. controls ++(270:.2cm) and ++(45:.1cm) .. (-.3,-.52);
	\draw[thick, xString] (-.1,-.08) .. controls ++(90:.2cm) and ++(-45:.1cm) .. (-.3,.32);		
	\draw[thick, xString] (.1,.92) .. controls ++(270:.2cm) and ++(135:.1cm) .. (.3,.42);
	\draw[thick, xString, dotted] (.3,.42) -- (-.3,.32);
	\draw[thick, xString, dotted] (.3,-.62) -- (-.3,-.52);	
\end{tikzpicture}
\,\,=\,\,
\begin{tikzpicture}[baseline=-.1cm]

	\draw[thick] (-.3,-1) -- (-.3,1);
	\draw[thick] (.3,-1) -- (.3,1);
	\draw[thick] (0,1) ellipse (.3cm and .1cm);
	\halfDottedEllipse{(-.3,-1)}{.3}{.1}
	\halfDottedEllipse{(-.3,0)}{.3}{.1}
		
	\draw[thick, xString] (-.1,-1.08) -- (-.1,.92);		
	\draw[thick, yString] (.1,-1.08) -- (.1,.92);		
\end{tikzpicture}
\,\,=\,\,
\begin{tikzpicture}[baseline=-.1cm]

	\draw[thick] (-.3,-1) -- (-.3,1);
	\draw[thick] (.3,-1) -- (.3,1);
	\draw[thick] (0,1) ellipse (.3cm and .1cm);
	\halfDottedEllipse{(-.3,-1)}{.3}{.1}
	\halfDottedEllipse{(-.3,0)}{.3}{.1}
		
	\draw[thick, xString] (-.1,-1.08) .. controls ++(90:.4cm) and ++(270:.4cm) .. (.1,-.08);		
	\draw[thick, xString] (-.1,.92) .. controls ++(270:.4cm) and ++(90:.4cm) .. (.1,-.08);		

	\draw[thick, yString] (.1,-1.08) .. controls ++(90:.2cm) and ++(225:.1cm) .. (.3,-.62);		
	\draw[thick, yString] (-.1,-.08) .. controls ++(270:.2cm) and ++(45:.1cm) .. (-.3,-.52);
	\draw[thick, yString] (-.1,-.08) .. controls ++(90:.2cm) and ++(-45:.1cm) .. (-.3,.32);		
	\draw[thick, yString] (.1,.92) .. controls ++(270:.2cm) and ++(135:.1cm) .. (.3,.42);
	\draw[thick, yString, dotted] (.3,.42) -- (-.3,.32);
	\draw[thick, yString, dotted] (.3,-.62) -- (-.3,-.52);	
\end{tikzpicture}\,\,.
$$
We only treat the first one (the other is entirely similar): upon taking mates, we get
\begin{gather*}
\begin{tikzpicture}[baseline=.4cm, scale=.8]
	\plane{(-.5,-1)}{4.8}{3.1}
	\draw[thick, yString] (1.4,0) -- (3.3,0);
	\draw[thick, xString] (1.4,.3) -- (3,.3);
	\CMbox{box}{(.2,-.4)}{1.2}{1.4}{.5}{$\varepsilon_{x\otimes y}$}
	\straightTubeNoString{(-.1,.2)}{.2}{3}
\pgftransformyshift{17}
\pgftransformyscale{1.3}
\pgftransformxshift{-43.5}
\pgftransformxscale{1.5}
\pgftransformrotate{90}
	\draw[thick, yString] (-.1,-1.08) .. controls ++(90:.4cm) and ++(270:.4cm) .. (.1,-.08);		
	\draw[thick, yString] (-.1,.92) .. controls ++(270:.4cm) and ++(90:.4cm) .. (.1,-.08);		
	\draw[thick, xString] (.1,-1.08) .. controls ++(90:.2cm) and ++(225:.1cm) .. (.3,-.62);		
	\draw[thick, xString] (-.1,-.08) .. controls ++(270:.2cm) and ++(45:.1cm) .. (-.3,-.52);
	\draw[thick, xString] (-.1,-.08) .. controls ++(90:.2cm) and ++(-45:.1cm) .. (-.3,.32);		
	\draw[thick, xString] (.1,.92) .. controls ++(270:.2cm) and ++(135:.1cm) .. (.3,.42);
	\draw[thick, xString, dotted] (.3,.42) -- (-.3,.32);
	\draw[thick, xString, dotted] (.3,-.62) -- (-.3,-.52);	
\end{tikzpicture}
\!\!\!\!\!=\!
\begin{tikzpicture}[baseline=.4cm, scale=.8]
	\plane{(-.5,-1)}{4.8}{3.1}
	\draw[thick, xString] (1.4,0) arc (90:-90:.3cm) -- (.2,-.6) arc (270:90:1.2cm) -- (1.5,1.8);
	\draw[thick, yString] (1.4,.3) -- (3,.3);
	\CMbox{box}{(.2,-.4)}{1.2}{1.4}{.5}{$\varepsilon_{y\otimes x}$}
	\straightTubeNoString{(-.1,.2)}{.2}{2.4}
	\draw[thick, yString] (-2.3,.5) .. controls ++(0:.4cm) and ++(180:.4cm) .. (.08,.7);
	\draw[thick, xString] (-2.3,.7) .. controls ++(0:.2cm) and ++(-135:.1cm) .. (-1.3,1);
	\draw[thick, xString] (.08,.5) .. controls ++(180:.2cm) and ++(45:.1cm) .. (-1.1,.2);
	\draw[thick, xString, dotted] (-1.1,.2) -- (-1.3,1);
\end{tikzpicture}\!\!\!\!=\,\,\,\,
\\
=
\begin{tikzpicture}[baseline=.4cm, scale=.8]
	\plane{(-.5,-1)}{4.8}{3.1}
	\draw[thick, xString] (1.3,.5) to[in=0,out=0, looseness=2] (.7,1.7) -- (.2,1.7) arc (-270:-90:1.15cm) -- (2,-1.8+1.2) arc (90:-90:.1cm) -- (.2,-2+1.2) arc (-90:-270:1.35cm) -- (1.4,1.9);
	\draw[thick, yString] (1.4,.3) -- (3,.3);
	\CMbox{box}{(.2,-.4)}{1.2}{1.4}{.5}{$\varepsilon_{x\otimes y}$}
	\straightTubeNoString{(-.1,.2)}{.2}{2.4}
	\draw[thick, yString] (-2.3,.5) -- (.08,.5);
	\draw[thick, xString] (-2.3,.7) -- (.08,.7);
\end{tikzpicture}
\!\!\!\!\!=\!
\begin{tikzpicture}[baseline=.4cm, scale=.8]
	\plane{(-.5,-1)}{4.8}{3.1}
	\draw[thick, yString] (1.4,0) -- (3.3,0);
	\draw[thick, xString] (1.4,.3) -- (3,.3);
	\CMbox{box}{(.2,-.4)}{1.2}{1.4}{.5}{$\varepsilon_{x\otimes y}$}
	\straightTubeNoString{(-.1,.2)}{.2}{2.4}
	\draw[thick, yString] (-2.3,.5) -- (.08,.5);
	\draw[thick, xString] (-2.3,.7) -- (.08,.7);
\end{tikzpicture}
\end{gather*}
where we have used Lemma \ref{lem:PushDownMorphism} for the first two equalities.
The last equal sign follows from the zig-zag equations satisfied  by the (co)evaluation morphisms $\ev$, $\tilde \ev$, $\coev$, and $\tilde \coev$.
\end{proof}

By Lemmas \ref{lem:TraciatorsComposeToIdentity} and \ref{lem:PushDownMorphism}, the traciators satisfy
\addtocounter{equation}{1}
\begin{equation}\tag{\arabic{equation}.a}\label{pic: twopics - a}
\begin{tikzpicture}[baseline=.1cm, scale=.9]
	\plane{(-.5,-.8)}{3.4}{2}
	\draw[thick, yString] (1,.1) -- (2,.1);
	\draw[thick, xString] (1.2,-.1) -- (2.2,-.1);
	\CMbox{box}{(0,-.4)}{1.2}{.8}{.4}{$\varepsilon_{y\otimes x}$}
	\straightTubeNoString{(-.2,0)}{.1}{2}
\pgftransformxscale{.84}
\pgftransformxshift{-6}
\pgftransformyshift{13.8}
\pgftransformyscale{-.48}
	\draw[thick, xString] (-2.3,.5) .. controls ++(0:.4cm) and ++(180:.4cm) .. (.08,.7);
	\draw[thick, yString] (-2.3,.7) .. controls ++(0:.2cm) and ++(-135:.1cm) .. (-1.4,1);
	\draw[thick, yString] (.08,.5) .. controls ++(180:.2cm) and ++(45:.1cm) .. (-1,.2);
	\draw[thick, yString, densely dotted] (-1,.2) -- (-1.4,1);
\end{tikzpicture}
\!\!=\,\,
\begin{tikzpicture}[baseline=.1cm, scale=.9]
	\plane{(-.5,-.8)}{3.2+.2}{2}
	\draw[thick, yString] (1.2,0) arc (90:-90:.3cm) -- (0,-.6) arc (270:90:.8cm) -- (.9+.2,1);
	\draw[thick, xString] (1.2,.2) -- (1.7+.2,.2);
	\CMbox{box}{(0,-.4)}{1.2}{.8}{.4}{$\varepsilon_{x\otimes y}$}
	\straightTubeTwoStrings{(-.2,0)}{2}{xString}{yString}
\end{tikzpicture}
\end{equation}
%
%
\begin{equation}\tag{\arabic{equation}.b}\label{pic: twopics - b}
\hspace{-1cm}\text{and}\,\,\,\hspace{1cm}
\begin{tikzpicture}[baseline=.1cm, scale=.9]
	\plane{(-.5,-.8)}{3.4}{2}
	\draw[thick, yString] (1,.1) -- (2,.1);
	\draw[thick, xString] (1.2,-.1) -- (2.2,-.1);
	\CMbox{box}{(0,-.4)}{1.2}{.8}{.4}{$\varepsilon_{y\otimes x}$}
	\straightTubeNoString{(-.2,0)}{.1}{2}
\pgftransformxscale{.84}
\pgftransformxshift{-6}
\pgftransformyscale{.48}
\pgftransformyshift{-5}
	\draw[thick, yString] (-2.3,.5) .. controls ++(0:.4cm) and ++(180:.4cm) .. (.08,.7);
	\draw[thick, xString] (-2.3,.7) .. controls ++(0:.2cm) and ++(-135:.1cm) .. (-1.4,1);
	\draw[thick, xString] (.08,.5) .. controls ++(180:.2cm) and ++(45:.1cm) .. (-1,.2);
	\draw[thick, xString, densely dotted] (-1,.2) -- (-1.4,1);
\end{tikzpicture}
\!\!=\,\,
\begin{tikzpicture}[baseline=.1cm, scale=.9]
	\plane{(-.5,-.8)}{3.7}{2}
	\draw[thick, xString] (3,-.6) -- (-.05,-.6) arc (270:90:.8cm) -- +(.8,0) to[in=0,out=0, looseness=2] (1,.1);
	\draw[thick, yString] (1.2,-.1) -- (2.5,-.1);
	\CMbox{box}{(0,-.4)}{1.2}{.8}{.4}{$\varepsilon_{x\otimes y}$}
	\straightTubeTwoStrings{(-.2,0)}{2}{xString}{yString}
\end{tikzpicture}
\end{equation}

\begin{rem}
\label{rem:DoubleDual}
In the event that $\cM$ is rigid but not pivotal, we only get isomorphisms $\tau^+:\Tr_\cC(x\otimes y)\to\Tr_\cC({}^{**}\hspace{-.1mm}y\otimes x)$ and $\tau^-:\Tr_\cC(x\otimes y)\to\Tr_\cC(y\otimes x^{**})$.
All of our theorems can be generalised to this more general situation, provided double duals are inserted in the appropriate places.
\end{rem}

We now prove the naturality of the traciator, and that it satisfies the axiom of a categorified trace.

\begin{lem}
\label{lem:TraciatorComposition}
The two maps $\Tr_\cC(x \otimes y \otimes z) \to  \Tr_\cC(y \otimes z \otimes x)$ given by $\tau_{x, y \otimes z}$ and by  $\tau_{z \otimes x, y} \circ \tau_{x \otimes y, z}$ are equal.
In diagrams:
$$
\begin{tikzpicture}[baseline=-.1cm, xscale=1.1]

	\draw[thick] (-.3,-1) -- (-.3,1);
	\draw[thick] (.3,-1) -- (.3,1);
	\draw[thick] (0,1) ellipse (.3cm and .1cm);
	\halfDottedEllipse{(-.3,-1)}{.3}{.1}
	
	\draw[thick, xString] (-.15,-1.09) .. controls ++(90:.8cm) and ++(270:.8cm) .. (.15,.91);		

	\draw[thick, yString] (0,-1.1) .. controls ++(90:.2cm) and ++(220:.28cm) .. (.3,-.05);		
	\draw[thick, yString] (-.15,.9) .. controls ++(270:.2cm) and ++(45:.1cm) .. (-.3,.3-.05);
	\draw[thick, yString, dotted] (-.3,.3-.05) -- (.3,0-.05);	
	\draw[thick, zString] (0,.9) .. controls ++(270:.2cm) and ++(40:.28cm) .. (-.3,0);
	\draw[thick, zString] (.15,-1.09) .. controls ++(90:.2cm) and ++(225:.1cm) .. (.3,-.3);		
	\draw[thick, zString, dotted] (-.3,0) -- (.3,-.3);	

\end{tikzpicture}
\,\,\,=\,\,\,
\begin{tikzpicture}[baseline=-.1cm, xscale=1.1]

	\draw[thick] (-.3,-1) -- (-.3,1);
	\draw[thick] (.3,-1) -- (.3,1);
	\draw[thick] (0,1) ellipse (.3cm and .1cm);
	\halfDottedEllipse{(-.3,-1)}{.3}{.1}
	\halfDottedEllipse{(-.3,0)}{.3}{.1}
		
	\draw[thick, xString] (-.15,-1.09) .. controls ++(90:.4cm) and ++(270:.4cm) .. (0,-.1);		
	\draw[thick, yString] (0,-1.1) .. controls ++(90:.4cm) and ++(270:.4cm) .. (.15,-.08);	
	\draw[thick, xString] (0,-.1) .. controls ++(90:.4cm) and ++(270:.4cm) .. (.15,.91);	
	\draw[thick, zString] (-.15,-.08) .. controls ++(90:.4cm) and ++(270:.4cm) .. (0,.9);

	\draw[thick, zString] (-.15,-.08) .. controls ++(270:.2cm) and ++(45:.1cm) .. (-.3,-.52);
	\draw[thick, zString] (.15,-1.1) .. controls ++(90:.2cm) and ++(225:.1cm) .. (.3,-.62);		
	\draw[thick, zString, dotted] (-.3,-.52) -- (.3,-.62);	
	\draw[thick, yString] (.15,-.08) .. controls ++(90:.2cm) and ++(225:.1cm) .. (.3,.38);
	\draw[thick, yString] (-.15,.92) .. controls ++(270:.2cm) and ++(45:.1cm) .. (-.3,.48);
	\draw[thick, yString, dotted] (-.3,.48) -- (.3,.38);	

\end{tikzpicture}\,.
$$
Similarly, we have $\tau^-_{x\otimes y,z}=\tau^-_{y, z\otimes x}\circ\tau^-_{x, y\otimes z}$.
\end{lem}
\begin{proof}
The first statement is equivalent, upon taking mates, to the equation
\[
\varepsilon_{ y\otimes z\otimes x}\circ \Phi(\tau_{z\otimes x, y}\circ \tau_{x\otimes y, z})=
\varepsilon_{y\otimes z\otimes x}\circ \Phi(\tau_{x, y \otimes z}).
\]
The argument is similar to the one in the previous lemma:
\begin{gather*}
\begin{tikzpicture}[baseline=.4cm, scale=.8]
	\plane{(-.5,-1)}{4.8}{3.1}
	\draw[thick, yString] (1.4,.2) -- (3.1,.2);
	\draw[thick, zString] (1.4,0) -- (3.3,0);
	\draw[thick, xString] (1.4,-.2) -- (3.5,-.2);
	\CMboxScale{box}{(.2,-.4)}{1.2}{1.4}{.5}{$\varepsilon_{y\otimes z\otimes x}$}{.75}
	\straightTubeNoString{(-.1,.2)}{.2}{2.4}
	\draw[thick, xString] (-2.32,.8) .. controls ++(0:.4cm) and ++(180:.4cm) .. (.08,.4);
	\draw[thick, zString] (-2.32,.4) .. controls ++(0:.2cm) and ++(135:.1cm) .. (-1.7,.2);
	\draw[thick, yString] (.08,.8) .. controls ++(180:.2cm) and ++(-45:.1cm) .. (-.6,1);
	\draw[thick, yString] (-2.3,.6) .. controls ++(0:.3cm) and ++(135:.4cm) .. (-.8,.2);
	\draw[thick, zString] (.1,.6) .. controls ++(180:.3cm) and ++(-45:.4cm) .. (-1.5,1);
	\draw[thick, zString, dotted] (-1.7,.2) -- (-1.5,1);
	\draw[thick, yString, dotted] (-.8,.2) -- (-.6,1);
\end{tikzpicture}
\!\!\!\!\!=\!
\begin{tikzpicture}[baseline=.4cm, scale=.8]
	\plane{(-.5,-1)}{4.8}{3.1}
	\draw[thick, yString] (1.4,-.2) arc (90:-90:.2cm) -- (.2,-.6) arc (270:90:1.2cm) -- (1.5,1.8);
	\draw[thick, zString] (1.4,.2) -- (3.1,.2);
	\draw[thick, xString] (1.4,0) -- (3.3,0);
	\CMboxScale{box}{(.2,-.4)}{1.2}{1.4}{.5}{$\varepsilon_{z\otimes x\otimes y}$}{.75}
	\straightTubeNoString{(-.1,.2)}{.2}{2.4}
	\draw[thick, xString] (-2.32,.8) .. controls ++(0:.4cm) and ++(180:.4cm) .. (.1,.6);
	\draw[thick, yString] (-2.3,.6) .. controls ++(0:.4cm) and ++(180:.4cm) .. (.08,.4);
	\draw[thick, zString] (-2.32,.4) .. controls ++(0:.2cm) and ++(135:.1cm) .. (-1.3,.2);
	\draw[thick, zString] (.08,.8) .. controls ++(180:.2cm) and ++(-45:.1cm) .. (-1.1,1);
	\draw[thick, zString, dotted] (-1.1,1) -- (-1.3,.2);
\end{tikzpicture}
=\\=
\begin{tikzpicture}[baseline=.4cm, scale=.8]
	\plane{(-.5,-1)}{4.8}{3.1}
	\draw[thick, zString] (1.4,-.2) arc (90:-90:.2cm) -- (.2,-.6) arc (270:90:1.2cm) -- (1.5,1.8);
	\draw[thick, yString] (1.4,0) arc (90:-90:.4cm) -- (.2,-.8) arc (270:90:1.4cm) -- (1.3,2);
	\draw[thick, xString] (1.4,.2) -- (3.1,.2);
	\CMboxScale{box}{(.2,-.4)}{1.2}{1.4}{.5}{$\varepsilon_{x\otimes y\otimes z}$}{.75}
	\straightTubeNoString{(-.1,.2)}{.2}{2.4}
	\draw[thick, xString] (-2.32,.8) -- (.08,.8);
	\draw[thick, yString] (-2.3,.6) -- (.1,.6);
	\draw[thick, zString] (-2.32,.4) -- (.08,.4);
\end{tikzpicture}
=
\begin{tikzpicture}[baseline=.4cm, scale=.8]
	\plane{(-.5,-1)}{4.8}{3.1}
	\draw[thick, yString] (1.4,.2) -- (3.1,.2);
	\draw[thick, zString] (1.4,0) -- (3.3,0);
	\draw[thick, xString] (1.4,-.2) -- (3.5,-.2);
	\CMboxScale{box}{(.2,-.4)}{1.2}{1.4}{.5}{$\varepsilon_{y\otimes z\otimes x}$}{.75}
	\straightTubeNoString{(-.1,.2)}{.2}{2.4}
	\draw[thick, xString] (-2.32,.8) .. controls ++(0:.4cm) and ++(180:.4cm) .. (.08,.4);
	\draw[thick, zString] (-2.32,.4) .. controls ++(0:.2cm) and ++(135:.1cm) .. (-1.5,.2);
	\draw[thick, yString] (.08,.8) .. controls ++(180:.2cm) and ++(-45:.1cm) .. (-.7,1);
	\draw[thick, yString] (-2.3,.6) .. controls ++(0:.3cm) and ++(135:.4cm) .. (-1.2,.2);
	\draw[thick, zString] (.1,.6) .. controls ++(180:.3cm) and ++(-45:.4cm) .. (-1,1);
	\draw[thick, zString, dotted] (-1.5,.2) -- (-1,1);
	\draw[thick, yString, dotted] (-1.2,.2) -- (-.7,1);
\end{tikzpicture}
\end{gather*}
Here, we have used Equation \eqref{pic: twopics - a} for the first two equalities,
and we have used Equation \eqref{pic: twopics - a} along with the identities
\begin{align*}
\widetilde{\ev}_{y\otimes z}
&=
\widetilde{\ev}_{y}\circ (\id_{y}\otimes \widetilde{\ev}_{z}\otimes \id_{y^*})
\\
\coev_{y\otimes z}
&=
(\id_{y}\otimes \coev_{z}\otimes \id_{y^*})\circ \coev_y
\end{align*}
for the last equality.
The second statement follows from the first one by taking inverses.
\end{proof}

Naturality of the traciator is given by the following lemma.

\begin{lem}\label{lem:MoveThroughTraciator}
For $f\in \cM(x,w)$ and $g\in \cM(y,z)$, we have $\tau_{w,z}\circ \Tr_\cC(f\otimes g)=\Tr_\cC(g\otimes f)\circ \tau_{x,y}$, and similarly for $\tau^-$.
In diagrams:
\begin{equation}
\begin{tikzpicture}[baseline=-.1cm]

	\draw[thick] (-.5,-1) -- (-.5,1);
	\draw[thick] (.5,-1) -- (.5,1);
	\draw[thick] (0,1) ellipse (.5cm and .2cm);
	\halfDottedEllipse{(-.5,-1)}{.5}{.2}
	\halfDottedEllipse{(-.5,0)}{.5}{.2}
	
	\draw[thick, xString] (-.2,-1.17) -- (-.2,-.7);		
	\draw[thick, wString] (-.2,-.7) -- (-.2,-.17) .. controls ++(90:.6cm) and ++(270:.6cm) .. (.2,.82);		
	\draw[thick, yString] (.2,-1.17) -- (.2,-.7);
	\draw[thick, zString] (.2,-.7) -- (.2,-.17) .. controls ++(90:.3cm) and ++(225:.2cm) .. (.5,.3);		
	\draw[thick, zString] (-.2,.82) .. controls ++(270:.3cm) and ++(45:.2cm) .. (-.5,.4);
	\draw[thick, zString, dotted] (.5,.3) -- (-.5,.4);	

	\nbox{unshaded}{(-.2,-.7)}{.3}{-.15}{-.15}{$f$}
	\nbox{unshaded}{(.2,-.7)}{.3}{-.15}{-.15}{$g$}
\end{tikzpicture}
\,\,\,=\,\,\,
\begin{tikzpicture}[baseline=-.1cm]

	\draw[thick] (-.5,-1) -- (-.5,1);
	\draw[thick] (.5,-1) -- (.5,1);
	\draw[thick] (0,1) ellipse (.5cm and .2cm);
	\halfDottedEllipse{(-.5,-1)}{.5}{.2}
	\halfDottedEllipse{(-.5,0)}{.5}{.2}
	
	\draw[thick, xString] (-.2,-1.17) .. controls ++(90:.6cm) and ++(270:.6cm) .. (.2,-.17) -- (.2,.3);		
	\draw[thick, wString] (.2,.3) -- (.2,.83);		
	\draw[thick, yString] (.2,-1.17) .. controls ++(90:.3cm) and ++(225:.2cm) .. (.5,-.7);		
	\draw[thick, yString] (-.2,.3) -- (-.2,-.17) .. controls ++(270:.3cm) and ++(45:.2cm) .. (-.5,-.6);
	\draw[thick, zString] (-.2,.83) -- (-.2,.3);
	\draw[thick, yString, dotted] (-.5,-.6) -- (.5,-.7);	

	\nbox{unshaded}{(-.2,.3)}{.3}{-.15}{-.15}{$g$}
	\nbox{unshaded}{(.2,.3)}{.3}{-.15}{-.15}{$f$}
\end{tikzpicture}\,.
\label{pic: traciator with f and g}
\end{equation}
\end{lem}
\begin{proof}
By Equation \eqref{pic: twopics - a} and Lemma \ref{lem:MoveThroughNaturalMap}, the two sides of \eqref{pic: traciator with f and g} are mates of 
$$
\begin{tikzpicture}[baseline=.4cm, scale=.8]
	\plane{(-.5,-.8)}{4}{2.8}

	\draw[thick, zString] (1.4,0) arc (90:-90:.3cm) -- (.2,-.6) arc (270:90:1.2cm) -- (.9,1.8);
	\draw[thick, wString] (1.4,.3) -- (2.4,.3);
	
	\CMbox{box}{(.2,-.4)}{1.2}{1.45}{.5}{$\varepsilon_{w\otimes z}$}

	\straightTubeNoString{(-.1,.175)}{.225}{2.4}

	\draw[thick, yString] (-2.3,.45) -- (-1,.45);
	\draw[thick, zString] (.1,.45) -- (-1,.45);
	\draw[thick, xString] (-2.3,.85) -- (-1,.85);
	\draw[thick, wString] (-1,.85) -- (.1,.85);

	\nbox{unshaded}{(-1,.85)}{.15}{.2}{.2}{\rotatebox{-75}{$f$}}
	\nbox{unshaded}{(-1,.45)}{.15}{.2}{.2}{\rotatebox{-75}{$g$}}
\end{tikzpicture}
=
\begin{tikzpicture}[baseline=.4cm, scale=.8]
	\plane{(-.5,-.8)}{5.3}{2.8}

	\draw[thick, yString] (1,0) -- (2.5,0);
	\draw[thick, zString] (2.5,0) -- (3.5,0) arc (90:-90:.3cm) -- (.2,-.6) arc (270:90:1.2cm) -- (2.2,1.8);
	\draw[thick, xString] (1.4,.6) -- (2.2,.6);
	\draw[thick, wString] (2.2,.6) -- (3.4,.6);
	\Mbox{(2.5,-.2)}{1}{.4}{$g$}
	\Mbox{(2,.4)}{1}{.4}{$f$}
	
	\CMbox{box}{(.2,-.4)}{1.2}{1.4}{.5}{$\varepsilon_{x\otimes y}$}

	\straightTubeNoString{(-.1,.2)}{.2}{2.4}

	\draw[thick, yString] (-2.3,.5) -- (.08,.5);
	\draw[thick, xString] (-2.3,.7) -- (.08,.7);

\end{tikzpicture}
$$
and
$$
\quad
\begin{tikzpicture}[baseline=.4cm, scale=.8]
	\plane{(-.5,-.8)}{5.3}{2.8}

	\draw[thick, xString] (1,0) -- (2.5,0);
	\draw[thick, wString] (2.5,0) -- (4,0);
	\draw[thick, yString] (1.4,.6) -- (2.2,.6);
	\draw[thick, zString] (2.2,.6) -- (3.4,.6);
	\Mbox{(2.5,-.2)}{1}{.4}{$f$}
	\Mbox{(2,.4)}{1}{.4}{$g$}
	
	\CMbox{box}{(.2,-.4)}{1.2}{1.4}{.5}{$\varepsilon_{y\otimes x}$}

	\straightTubeNoString{(-.1,.2)}{.2}{2.4}

	\draw[thick, xString] (-2.3,.7) .. controls ++(0:.4cm) and ++(180:.4cm) .. (.08,.5);
	\draw[thick, yString] (-2.3,.5) .. controls ++(0:.2cm) and ++(135:.1cm) .. (-1.3,.2);
	\draw[thick, yString] (.08,.7) .. controls ++(180:.2cm) and ++(-45:.1cm) .. (-1.1,1);
	\draw[thick, yString, dotted] (-1.1,1) -- (-1.3,.2);

\end{tikzpicture}
\!\!=
\begin{tikzpicture}[baseline=.4cm, scale=.8]
	\plane{(-.5,-.8)}{5}{2.8}

	\draw[thick, yString] (1.4,0) arc (90:-90:.3cm) -- (.2,-.6) arc (270:90:1cm) -- (.8,1.4);
	\draw[thick, xString] (1.4,.2) -- (2.3,.2);
	\draw[thick, wString] (2.3,.2) -- (3.5,.2);
	\draw[thick, zString] (1.6,1.4) -- (2.3,1.4);
	\Mbox{(1.1,1.1)}{1}{.6}{$g$}
	\Mbox{(2.3,-.1)}{1}{.6}{$f$}
	
	\CMbox{box}{(.2,-.4)}{1.2}{.8}{.4}{$\varepsilon_{x\otimes y}$}

	\straightTubeTwoStrings{(0,-.1)}{2}{xString}{yString}

\end{tikzpicture}$$
respectively.
The latter are equal by pivotality in $\cM$. 
\end{proof}


\subsection{Interaction between traciator and braiding}
\label{sec:InteractionBetweenTraciatorAndBraiding}

In Sections~\ref{sec:AdjointsOfTensorFunctors}--\ref{sec:UnitMap}, we saw that when $\Tr_\cC$ is the adjoint of a tensor functor, it can be described by a graphical calculus of strings on tubes.
The tubes are allowed to split, but the strings must remain on the fronts of the tubes.
In Section~\ref{sec:Traciator}, we also saw that when $\Phi$ is a central functor (not assumed monoidal) $\Tr_\cC$ admits a graphical calculus of strings winding around a single tube.

In this section, we start by assuming that $\Phi$ is a monoidal central functor.
The tubes, with the strings on their surface are now allowed to branch, but they must remain in a single plane.
Finally, we go on to assume our strongest hypothesis, namely that $\Phi$ is a braided central functor.
The tubes can now braid freely in three dimensions (see Figure \ref{fig: big figure} in the introduction).\bigskip

We begin with the following assumptions:
\begin{nota}\mbox{}\vspace{-.15cm}
\begin{itemize}
\item
$\cC$ and $\cM$ are pivotal categories
\item
$\Phi^{\scriptscriptstyle \cZ}: \cC\to \cZ(\cM)$ is a pivotal tensor functor,
\item
the trace functor $\Tr_\cC:\cM\to \cC$ is the right adjoint of $\Phi:=F\circ \Phi^{\scriptscriptstyle \cZ}$. 
\end{itemize}
\end{nota}

\noindent 
Under these assumptions, we can establish a generalization of the associativity of~$\mu$:

\begin{lem}\label{lem:MultiplicationAssocaitive}
Given $w,x,y,z\in\cM$, the following two morphisms in $\cC\big(\Tr_\cC(w) \otimes \Tr_\cC(x\otimes y) \otimes \Tr_\cC(z),\Tr_\cC(w\otimes x \otimes z\otimes y)\big)$ are equal:
\begin{equation}
\begin{tikzpicture}[baseline=1.1cm]

	\pgfmathsetmacro{\voffset}{.08};
	\pgfmathsetmacro{\hoffset}{.15};
	\pgfmathsetmacro{\hoffsetTop}{.12};

	\coordinate (a) at (-1,-1);
	\coordinate (a1) at ($ (a) + (1.4,0)$);
	\coordinate (a2) at ($ (a1) + (1.4,0)$);
	\coordinate (b) at ($ (a) + (.7,1.5)$);
	\coordinate (b2) at ($ (b) + (1.4,0)$);
	\coordinate (c) at ($ (b) + (0,1)$);
	\coordinate (c2) at ($ (c) + (1.4,0)$);
	\coordinate (d) at ($ (c) + (.7,1.5)$);	
	
	\pairOfPants{(a)}{}
	\pairOfPants{(c)}{}
	\topCylinder{($ (a) + (1.4,4) $)}{}
	\draw[thick] (b) -- ($ (b) + (0,1) $);
	\draw[thick] ($ (b) + (.6,0) $) -- ($ (b) + (.6,1) $);
	\draw[thick] ($ (b) + (1.4,0) $)  -- ($ (b) + (1.4,1) $);
	\draw[thick] ($ (b) + (2,0) $) -- ($ (b) + (2,1) $);
	\LeftSlantCylinder{($ (a) + (2.8,0) $)}{}

	\draw[thick, wString] ($ (b) + (\hoffset,-\voffset) $) .. controls ++(90:.4cm) and ++(270:.4cm) .. ($ (b) + 2*(\hoffset,0) + (0,.9) $);
	\draw[thick, xString] ($ (b) + 2*(\hoffset,0) + (0,-.1) $) .. controls ++(90:.4cm) and ++(270:.4cm) .. ($ (b) + 3*(\hoffset,0) + (0,.92) $);		
	\draw[thick, yString] ($ (b) + 3*(\hoffset,0) + (0,-\voffset) $) .. controls ++(90:.2cm) and ++(225:.1cm) .. ($ (b) + 4*(\hoffset,0) + (0,-\voffset) + (0,.45)$);
	\draw[thick, yString] ($ (b) + (\hoffset,1) + (0,-\voffset) $) .. controls ++(270:.2cm) and ++(45:.1cm) .. ($ (b) + (0,1) + (0,-\voffset) + (0,-.45)$);

	\draw[thick, wString] ($ (d) + 2*(\hoffsetTop,0) + (0,-.1)$) .. controls ++(90:.4cm) and ++(270:.4cm) .. ($ (d) + (\hoffsetTop,0) + (0,-\voffset) + (0,1) $);
	\draw[thick, xString] ($ (d) + 3*(\hoffsetTop,0) + (0,-\voffset) $) .. controls ++(90:.4cm) and ++(270:.4cm) .. ($ (d) + 2*(\hoffsetTop,0) + (0,.9) $);		
	\draw[thick, yString] ($ (d) + (\hoffsetTop,0) + (0,-\voffset) $) .. controls ++(90:.2cm) and ++(-45:.1cm) .. ($ (d) + (0,-\voffset) + (0,.45)$);
	\draw[thick, yString] ($ (d) + 4*(\hoffsetTop,0) + (0,1) + (0,-\voffset) $) .. controls ++(270:.2cm) and ++(135:.1cm) .. ($ (d) + 5*(\hoffsetTop,0) + (0,-\voffset) + (0,-.45) + (0,1)$);
	\draw[thick, zString] ($ (d) + 4*(\hoffsetTop,0) + (0,-\voffset) $) .. controls ++(90:.4cm) and ++(270:.4cm) .. ($ (d) + 3*(\hoffsetTop,0) + (0,-\voffset) + (0,1)$);

	\draw[thick, wString] ($ (a) + 2*(\hoffset,0) + (0,-.1)$) .. controls ++(90:.8cm) and ++(270:.8cm) .. ($ (b) + (\hoffset,0) + (0,-\voffset) $);	
	\draw[thick, xString] ($ (a1) + 2*(\hoffset,0) + (0,-.1)$) .. controls ++(90:.8cm) and ++(270:.8cm) .. ($ (b) + 2*(\hoffset,0) + (0,-.1) $);	
	\draw[thick, yString] ($ (a1) + 3*(\hoffset,0) + (0,-\voffset)$) .. controls ++(90:.8cm) and ++(270:.8cm) .. ($ (b) + 3*(\hoffset,0) + (0,-\voffset) $);	
	\draw[thick, zString] ($ (a2) + 2*(\hoffset,0) + (0,-.1)$) .. controls ++(90:.8cm) and ++(270:.6cm) .. ($ (b2) + 2*(\hoffset,0) + (0,-.1) $) -- ($ (c2) + 2*(\hoffset,0) + (0,-.1) $);	
	
	\draw[thick, yString] ($ (c) + (\hoffset,0) + (0,-\voffset)$) .. controls ++(90:.8cm) and ++(270:.8cm) .. ($ (d) + (\hoffsetTop,0) + (0,-\voffset) $);
	\draw[thick, wString] ($ (c) + 2*(\hoffset,0) + (0,-.1)$) .. controls ++(90:.8cm) and ++(270:.8cm) .. ($ (d) + 2*(\hoffsetTop,0) + (0,-\voffset) $);	
	\draw[thick, xString] ($ (c) + 3*(\hoffset,0) + (0,-\voffset)$) .. controls ++(90:.8cm) and ++(270:.8cm) .. ($ (d) + 3*(\hoffsetTop,0) + (0,-\voffset) $);	
	\draw[thick, zString] ($ (c2) + 2*(\hoffset,0) + (0,-.1)$) .. controls ++(90:.8cm) and ++(270:.6cm) .. ($ (d) + 4*(\hoffsetTop,0) + (0,-\voffset) $);	

\node at (-.7,-1.25) {$\scriptstyle w$};
\node at (.6,-1.25) {$\scriptstyle x$};
\node at (.85,-1.28) {$\scriptstyle y$};
\node at (2.1,-1.25) {$\scriptstyle z$};
\end{tikzpicture}
\quad=\quad\,\,\,
\begin{tikzpicture}[baseline=1.1cm]	
\node at (-.7,-1.25) {$\scriptstyle w$};
\node at (.57,-1.25) {$\scriptstyle x$};
\node at (.83,-1.28) {$\scriptstyle y$};
\node at (2.1,-1.25) {$\scriptstyle z$};

	\pgfmathsetmacro{\voffset}{.08};
	\pgfmathsetmacro{\hoffset}{.15};
	\pgfmathsetmacro{\hoffsetTop}{.12};

	\coordinate (a1) at (-1,-1);
	\coordinate (a2) at ($ (a1) + (1.4,0)$);
	\coordinate (a3) at ($ (a1) + (2.8,0)$);
	\coordinate (b1) at ($ (a1) + (0,1)$);
	\coordinate (b2) at ($ (b1) + (1.4,0) $);
	\coordinate (b3) at ($ (b2) + (1.4,0) $);
	\coordinate (c1) at ($ (b1) + (.7,1.5)$);
	\coordinate (c2) at ($ (b2) + (.7,1.5)$);
	\coordinate (d1) at ($ (c1) + (0,1)$);
	\coordinate (d2) at ($ (c2) + (0,1)$);
	
	\RightSlantCylinder{(b1)}{}
	\pairOfPants{(b2)}{}
	\draw[thick] (c1) -- ($ (c1) + (0,1) $);
	\draw[thick] ($ (c1) + (.6,0) $) -- ($ (c1) + (.6,1) $);
	\draw[thick] (c2) -- ($ (c2) + (0,1) $);
	\draw[thick] ($ (c2) + (.6,0) $) -- ($ (c2) + (.6,1) $);
	\topPairOfPants{(d1)}{}
	
	\draw[thick] (a1) -- ($ (a1) + (0,1) $);
	\draw[thick] ($ (a1) + (.6,0) $) -- ($ (a1) + (.6,1) $);
	\halfDottedEllipse{(a1)}{.3}{.1}	

	\draw[thick] (a2) -- ($ (a2) + (0,1) $);
	\draw[thick] ($ (a2) + (.6,0) $) -- ($ (a2) + (.6,1) $);
	\halfDottedEllipse{(a2)}{.3}{.1}	

	\draw[thick] (a3) -- ($ (a3) + (0,1) $);
	\draw[thick] ($ (a3) + (.6,0) $) -- ($ (a3) + (.6,1) $);
	\halfDottedEllipse{(a3)}{.3}{.1}	

	\draw[thick, xString] ($ (a2) + (\hoffset,0) + (0,-.1)$) .. controls ++(90:.4cm) and ++(270:.4cm) .. ($ (a2) + 3*(\hoffset,0) + (0,-\voffset) + (0,1) $);		
	\draw[thick, yString] ($ (a2) + 3*(\hoffset,0) + (0,-\voffset) $) .. controls ++(90:.2cm) and ++(225:.1cm) .. ($ (a2) + 4*(\hoffset,0) + (0,-\voffset) + (0,.45)$);
	\draw[thick, yString] ($ (a2) + (\hoffset,1) + (0,-\voffset) $) .. controls ++(270:.2cm) and ++(45:.1cm) .. ($ (a2) + (0,1) + (0,-\voffset) + (0,-.45)$);

	\draw[thick, xString] ($ (c2) + 2*(\hoffset,0) + (0,-.1)$) .. controls ++(90:.4cm) and ++(270:.4cm) .. ($ (c2) + (\hoffset,0) + (0,-\voffset) + (0,1) $);
	\draw[thick, zString] ($ (c2) + 3*(\hoffset,0) + (0,-\voffset) $) .. controls ++(90:.4cm) and ++(270:.4cm) .. ($ (c2) + 2*(\hoffset,0) + (0,.9) $);		
	\draw[thick, yString] ($ (c2) + (\hoffset,0) + (0,-\voffset) $) .. controls ++(90:.2cm) and ++(-45:.1cm) .. ($ (c2) + (0,-\voffset) + (0,.45)$);
	\draw[thick, yString] ($ (c2) + 3*(\hoffset,0) + (0,1) + (0,-\voffset) $) .. controls ++(270:.2cm) and ++(135:.1cm) .. ($ (c2) + 4*(\hoffset,0) + (0,-\voffset) + (0,-.45) + (0,1)$);
	
	\draw[thick, wString] ($ (a1) + 2*(\hoffset,0) + (0,-.1)$) -- ($ (b1) + 2*(\hoffset,0) + (0,-.1)$) .. controls ++(90:.8cm) and ++(270:.8cm) .. ($ (c1) + 2*(\hoffset,0) + (0,-.1) $) -- ($ (d1) + 2*(\hoffset,0) + (0,-.1) $) ;	
	\draw[thick, yString] ($ (b2) + (\hoffset,0) + (0,-\voffset)$) .. controls ++(90:.8cm) and ++(270:.8cm) .. ($ (c2) + (\hoffset,0) + (0,-\voffset) $);	
	\draw[thick, xString] ($ (b2) + 3*(\hoffset,0) + (0,-\voffset)$) .. controls ++(90:.8cm) and ++(270:.8cm) .. ($ (c2) + 2*(\hoffset,0) + (0,-.1) $);	
	\draw[thick, zString] ($ (a3) + 2*(\hoffset,0) + (0,-.1)$) -- ($ (b3) + 2*(\hoffset,0) + (0,-.1)$) .. controls ++(90:.8cm) and ++(270:.8cm) .. ($ (c2) + 3*(\hoffset,0) + (0,-\voffset) $);	

	\draw[thick, wString] ($ (d1) + 2*(\hoffset,0) + (0,-.1)$) .. controls ++(90:.8cm) and ++(270:.8cm) .. ($ (d1) + (\hoffsetTop,0) + (0,-\voffset) + (.7,1.5) $);	
	\draw[thick, xString] ($ (d2) + (\hoffset,0) + (0,-\voffset)$) .. controls ++(90:.8cm) and ++(270:.8cm) .. ($ (d1) + 2*(\hoffsetTop,0) + (0,-\voffset) + (.7,1.5) $);	
	\draw[thick, zString] ($ (d2) + 2*(\hoffset,0) + (0,-.1)$) .. controls ++(90:.8cm) and ++(270:.6cm) .. ($ (d1) + 3*(\hoffsetTop,0) + (0,-\voffset) + (.7,1.5) $);	
	\draw[thick, yString] ($ (d2) + 3*(\hoffset,0) + (0,-\voffset)$) .. controls ++(90:.8cm) and ++(270:.6cm) .. ($ (d1) + 4*(\hoffsetTop,0) + (0,-\voffset) + (.7,1.5) $);
	
\end{tikzpicture}
\label{eqn:RHS}
\end{equation}
\end{lem}

\begin{rem}
\label{rem:BackOfTubes}
In the above equation, it is best to imagine the blue strands as running along the backs of the tubes:
pre- and postcomposing by the ``half-traciators''
\(
\begin{tikzpicture}[baseline=-.12cm, scale=.5, line width=.5]
	\draw (-.3,-.7) -- (-.3,.7);
	\draw (.3,-.7) -- (.3,.7);
	\draw (0,.7) ellipse (.3cm and .1cm);
	\draw (-.3,-.7) arc (180:360:.3cm and .1cm);	
	\draw[dotted] (-.3,-.7) arc (180:0:.3cm and .1cm);	
	\draw[xString] (-.1,-1.07+.27) -- (-.1,.92-.3);		
	\draw[yString] (.06,.92-.3) .. controls ++(270:.2cm) and ++(135:.2cm) .. (.3,0);		
	\draw[yString, densely dotted, line width=.6] (.06,-1.07+.15+.3) .. controls ++(90:.2cm) and ++(-135:.2cm) .. (.3,0);
\useasboundingbox ($(current bounding box.north east)+(0,-.5)$) rectangle ($(current bounding box.south west)+(0,.5)$);
\end{tikzpicture}
\)\;\!
and\;\!
\(
\begin{tikzpicture}[baseline=-.12cm, scale=.5, line width=.5]
	\pgfmathsetmacro{\hoffsetTop}{.12};
	\pgfmathsetmacro{\voffset}{.08};

	\draw (-.3,-.7) -- (-.3,.7);
	\draw (.3,-.7) -- (.3,.7);
	\draw (0,.7) ellipse (.3cm and .1cm);
	\draw (-.3,-.7) arc (180:360:.3cm and .1cm);	
	\draw[dotted] (-.3,-.7) arc (180:0:.3cm and .1cm);	
	\draw[wString] ($ (-.3,-.7) + (\hoffsetTop,0) + (0,-\voffset)$) -- ($ (-.3,.7) + (\hoffsetTop,0) + (0,-\voffset)$);		
	\draw[xString] ($ (-.3,-.7) + 2*(\hoffsetTop,0) + (0,-\voffset)$) -- ($ (-.3,.7) + 2*(\hoffsetTop,0) + (0,-\voffset)$);		
	\draw[zString] ($ (-.3,-.7) + 3*(\hoffsetTop,0) + (0,-\voffset)$) .. controls ++(90:.6cm) and ++(270:.6cm) .. ($ (-.3,.7) + 4*(\hoffsetTop,0) + (0,-\voffset)$);		
	\draw[yString] ($ (-.3,-.7) + 4*(\hoffsetTop,0) + (0,-\voffset)$)  .. controls ++(90:.2cm) and ++(-135:.15cm) .. (.3,0);		
	\draw[yString, densely dotted, line width=.6] ($ (-.3,.7) + 3*(\hoffsetTop,0) + (0,\voffset)$) .. controls ++(270:.2cm) and ++(135:.2cm) .. (.3,0);		
\useasboundingbox ($(current bounding box.north east)+(0,-.5)$) rectangle ($(current bounding box.south west)+(0,.5)$);
\end{tikzpicture}
\)\;\!,
Equation \eqref{eqn:RHS} becomes
$$
\begin{tikzpicture}[baseline=.5cm,scale=.8]

	\pgfmathsetmacro{\voffset}{.08};
	\pgfmathsetmacro{\hoffset}{.15};
	\pgfmathsetmacro{\hoffsetTop}{.12};

	\coordinate (a) at (-1,-1);
	\coordinate (a1) at ($ (a) + (1.4,0)$);
	\coordinate (a2) at ($ (a1) + (1.4,0)$);
	\coordinate (b) at ($ (a) + (.7,1.5)$);
	\coordinate (b2) at ($ (b) + (1.4,0)$);
	\coordinate (c) at ($ (b) + (.7,1.5)$);	
	
	\pairOfPants{(a)}{}
	\topPairOfPants{(b)}{}
	\LeftSlantCylinder{($ (a) + (2.8,0) $)}{}

	\draw[thick, wString] ($ (a) + 2*(\hoffset,0) + (0,-.1)$) .. controls ++(90:.8cm) and ++(270:.8cm) .. ($ (b) + (\hoffset,0) + (0,-\voffset) $);	
	\draw[thick, xString] ($ (a1) + 2*(\hoffset,0) + (0,-.1)$) .. controls ++(90:.8cm) and ++(270:.8cm) .. ($ (b) + 2*(\hoffset,0) + (0,-.1) $);		
	\draw[thick, zString] ($ (a2) + 2*(\hoffset,0) + (0,-.1)$) .. controls ++(90:.8cm) and ++(270:.6cm) .. ($ (b2) + 2*(\hoffset,0) + (0,-.1) $);	
	
	\draw[thick, wString] ($ (b) + (\hoffset,0) + (0,-\voffset)$) .. controls ++(90:.8cm) and ++(270:.8cm) .. ($ (c) + (\hoffsetTop,0) + (0,-\voffset) $);
	\draw[thick, xString] ($ (b) + 2*(\hoffset,0) + (0,-.1)$) .. controls ++(90:.8cm) and ++(270:.8cm) .. ($ (c) + 2*(\hoffsetTop,0) + (0,-\voffset) $);	
	\draw[thick, zString] ($ (b2) + 2*(\hoffset,0) + (0,-.1)$) .. controls ++(90:.7cm) and ++(270:.8cm) .. ($ (c) + 4*(\hoffsetTop,0) + (0,-\voffset) $);	

	\draw[more thick, densely dotted, yString] ($ (a1) + 3*(\hoffset,0) + (0,\voffset)$) .. controls ++(90:.6cm) and ++(270:.8cm) .. ($ (b) + 3*(\hoffset,0) + (0,-\voffset) $) .. controls ++(90:.8cm) and ++(270:1cm) .. ($ (c) + 3*(\hoffsetTop,0) + (0,\voffset) $);

\end{tikzpicture}
\quad=\quad\,\,\,
\begin{tikzpicture}[baseline=.5cm, scale=.8]	

	\pgfmathsetmacro{\voffset}{.08};
	\pgfmathsetmacro{\hoffset}{.15};
	\pgfmathsetmacro{\hoffsetTop}{.12};

	\coordinate (a1) at (-1,-1);
	\coordinate (a2) at ($ (a1) + (1.4,0) $);
	\coordinate (a3) at ($ (a2) + (1.4,0) $);
	\coordinate (b1) at ($ (a1) + (.7,1.5)$);
	\coordinate (b2) at ($ (a2) + (.7,1.5)$);
	
	\RightSlantCylinder{(a1)}{}
	\pairOfPants{(a2)}{}
	\topPairOfPants{(b1)}{}
	
	\draw[thick, wString] ($ (a1) + 2*(\hoffset,0) + (0,-.1)$) .. controls ++(90:.8cm) and ++(270:.6cm) .. ($ (b1) + 2*(\hoffset,0) + (0,-.1) $);	
	\draw[thick, xString] ($ (a2) + (\hoffset,0) + (0,-\voffset)$) .. controls ++(90:.8cm) and ++(270:.8cm) .. ($ (b2) + (\hoffset,0) + (0,-\voffset) $);	
	\draw[thick, zString] ($ (a3) + 2*(\hoffset,0) + (0,-.1)$) .. controls ++(90:.8cm) and ++(270:.8cm) .. ($ (b2) + 3*(\hoffset,0) + (0,-\voffset) $);	

	\draw[thick, wString] ($ (b1) + 2*(\hoffset,0) + (0,-.1)$) .. controls ++(90:.7cm) and ++(270:.8cm) .. ($ (b1) + (\hoffsetTop,0) + (0,-\voffset) + (.7,1.5) $);	
	\draw[thick, xString] ($ (b2) + (\hoffset,0) + (0,-\voffset)$) .. controls ++(90:.8cm) and ++(270:.8cm) .. ($ (b1) + 2*(\hoffsetTop,0) + (0,-\voffset) + (.7,1.5) $);	
	\draw[thick, zString] ($ (b2) + 3*(\hoffset,0) + (0,-\voffset)$) .. controls ++(90:.8cm) and ++(270:.6cm) .. ($ (b1) + 4*(\hoffsetTop,0) + (0,-\voffset) + (.7,1.5) $);
	
	\draw[more thick, densely dotted, yString] ($ (a2) + 3*(\hoffset,0) + (0,\voffset)$) .. controls ++(90:.6cm) and ++(270:.8cm) .. ($ (b2) + 2*(\hoffset,0) + (0,-.1) $) .. controls ++(90:.8cm) and ++(270:.9cm) .. ($ (b1) + 3*(\hoffsetTop,0) + (0,\voffset) + (.7,1.5) $);	
	
\end{tikzpicture}
$$
\end{rem}

\begin{proof}
We start from the left hand side, whose mate is given by
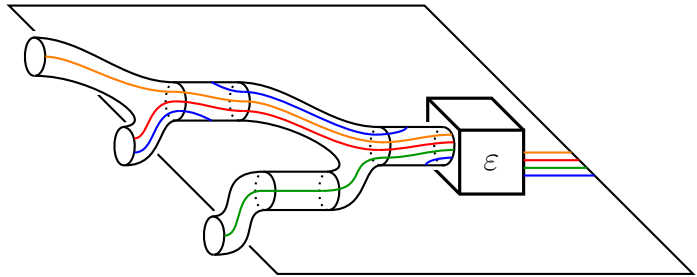
\newcommand{\halfDottedEllipseRev}[3]{
	\draw[thick] ($#1 + (0.001,.3)$) .. controls ++(75:.2cm) and ++(90:.2cm) .. ($#1 + (.6,-.3)$);
	\draw[thick, dotted] ($#1 + (0,.3)$) .. controls ++(270:.2cm) and ++(270:.2cm) .. ($#1 + (.6,-.3)$);
}

\newcommand{\halfDottedEllipseRevSolid}[3]{
	\draw[thick] ($#1 + (0.001,.3)$) .. controls ++(75:.2cm) and ++(90:.2cm) .. ($#1 + (.6,-.3)$);
	\draw[thick] ($#1 + (0,.3)$) .. controls ++(270:.2cm) and ++(270:.2cm) .. ($#1 + (.6,-.3)$);
}

\newcommand{\pairOfPantsRev}[2]{
	\draw[thick] ($ #1 + (0,.28) $) .. controls ++(90:.8cm) and ++(270:.8cm) .. ($ #1 + (.7,1.5+.28) $);
	\draw[thick] ($ #1 + (2,-.28) $) .. controls ++(90:.6cm) and ++(270:.3cm) .. ($ #1 + (2,0) + (-.7,1.5-.28) $);
	\draw[thick] ($ #1 + (.6,-.28) $).. controls ++(90:.9cm) and ++(95:.7cm) .. ($ #1 + (1.4,.28) $); 
	\halfDottedEllipseRev{($ #1 + (.7,1.5) $)}{.3}{.1}
	\halfDottedEllipseRev{#1}{.3}{.1}
	\halfDottedEllipseRev{($ #1 + (1.4,0) $)}{.3}{.1}
}
\newcommand{\pairOfPantsRevSolid}[2]{
	\draw[thick] ($ #1 + (0,.28) $) .. controls ++(90:.8cm) and ++(270:.8cm) .. ($ #1 + (.7,1.5+.28) $);
	\draw[thick] ($ #1 + (2,-.28) $) .. controls ++(90:.6cm) and ++(270:.3cm) .. ($ #1 + (2,0) + (-.7,1.5-.28) $);
	\draw[thick] ($ #1 + (.6,-.28) $).. controls ++(90:.9cm) and ++(95:.7cm) .. ($ #1 + (1.4,.28) $); 
	\halfDottedEllipseRev{($ #1 + (.7,1.5) $)}{.3}{.1}
	\halfDottedEllipseRevSolid{#1}{.3}{.1}
	\halfDottedEllipseRevSolid{($ #1 + (1.4,0) $)}{.3}{.1}
}

\newcommand{\LeftSlantCylinderSolid}[2]{
	\draw[thick] ($ #1 + (0,.28) $) .. controls ++(90:.45cm) and ++(270:.8cm) .. ($ #1 + (-.7,1.5+.28) $);
	\draw[thick] ($ #1 + (.6,0-.28) $).. controls ++(90:.8cm) and ++(270:.45cm) .. ($ #1 + (-.1,1.5-.28) $); 
	\halfDottedEllipseRev{($ #1 + (-.7,1.5) $)}{.3}{.1}
	\halfDottedEllipseRevSolid{#1}{.3}{.1}
}

\newcommand{\topCylinderRev}[2]{
	\draw[thick] ($ #1 + (0,.28) $) -- ($ #1 + (0,1+.28) $);
	\draw[thick] ($ #1 + (.6,-.28) $) -- ($ #1 + (.6,1-.28) $);
	\halfDottedEllipseRev{($ #1 + (0,1) $)}{.3}{.1}
}

\[\begin{tikzpicture}[scale=.85]

	\draw[thick] (2.1,-2.7) -- ++(-4.2,4.2) -- ++(6.5,0) -- ++(4.2,-4.2) -- cycle;
	\CMboxSize{box}{(4.95,-1.45)}{1}{1}{.5}{$\varepsilon$}{1.1}
	\fill[unshaded] (4.35,-1.075) rectangle +(.3,.75);
	\fill[unshaded] (-.1,-.65) circle (.35+.02);
	\fill[unshaded] ($(-.45,-.65)+(1.8,-1.37)$) circle (.37+.02);
	\fill[unshaded] (-1.83-.02,.12-.02) rectangle +(1.3+.02+.02,.95+.02+.02);

\draw[thick, wString] (5.95,-.8) -- +(.75,0); 
\draw[thick, xString] (5.95,-.9-.02) -- +(.87,0); 
\draw[thick, zString] (5.95,-1-.02-.02) -- +(.99,0); 
\draw[thick, yString] (5.95,-1.1-.02-.02-.02) -- +(1.11,0); 

\pgftransformxslant{-1}
\pgftransformrotate{-90}

	\coordinate (a) at (-1,-1);
	\coordinate (a1) at ($ (a) + (1.4,0)$);
	\coordinate (a2) at ($ (a1) + (1.4,0)$);
	\coordinate (b) at ($ (a) + (.7,1.5)$);
	\coordinate (b2) at ($ (b) + (1.4,0)$);
	\coordinate (c) at ($ (b) + (0,1)$);
	\coordinate (c2) at ($ (c) + (1.4,0)$);
	\coordinate (d) at ($ (c) + (.7,1.5)$);	
	
	\pairOfPantsRevSolid{(a)}{}
	\pairOfPantsRev{(c)}{}
	\topCylinderRev{($ (a) + (1.4,4) $)}{}
	\draw[thick] ($ (b) + (0,.28) $) -- ($ (b) + (0,1+.28) $);
	\draw[thick] ($ (b) + (.6,-.28) $) -- ($ (b) + (.6,1-.28) $);
	\draw[thick] ($ (b) + (1.4,.28) $)  -- ($ (b) + (1.4,1+.28) $);
	\draw[thick] ($ (b) + (2,-.28) $) -- ($ (b) + (2,1-.28) $);
	\LeftSlantCylinderSolid{($ (a) + (2.8,0) $)}{}

	\draw[thick, wString] ($ (b) + (.15,-.08)  + (0,.2)$) .. controls ++(90:.4cm) and ++(270:.4cm) .. ($ (b) + 2*(.15,0) + (0,.9)  + (0,.2)$);
	\draw[thick, xString] ($ (b) + 2*(.15,0)$) .. controls ++(90:.4cm) and ++(270:.4cm) .. ($ (b) + 3*(.15,0) + (0,.92) $);		
	\draw[thick, yString] ($ (b) + 3*(.15,0) + (0,-.08) $) .. controls ++(90:.2cm) and ++(225:.1cm) .. ($ (b) + 4*(.15,0) + (0,-.08) + (0,.35)$);
	\draw[thick, yString] ($ (b) + (.15,1) + (0,-.08)  + (0,.3)$) .. controls ++(270:.2cm) and ++(45:.1cm) .. ($ (b) + (0,1) + (0,-.08) + (0,-.45) + (0,.4)$);

	\draw[thick, wString] ($ (d) + 2*(.12,0) + (0,-.1) + (0,.2)$) .. controls ++(90:.4cm) and ++(270:.4cm) .. ($ (d) + (.12,0) + (0,-.08) + (0,1) + (0,.39)$);
	\draw[thick, xString] ($ (d) + 3*(.12,0) + (0,-.08) $) .. controls ++(90:.4cm) and ++(270:.4cm) .. ($ (d) + 2*(.12,0) + (0,.9) + (0,.32)$);		
	\draw[thick, yString] ($ (d) + (.12,0) + (0,-.08)  + (0,.2)$) .. controls ++(90:.2cm) and ++(-45:.1cm) .. ($ (d) + (0,-.08) + (0,.45) + (0,.35)$);
	\draw[thick, yString] ($ (d) + 4*(.12,0) + (0,1) + (0,-.08) + (0,.04) $) .. controls ++(270:.2cm) and ++(135:.1cm) .. ($ (d) + 5*(.12,0) + (0,-.08) + (0,-.45) + (0,.95)$);
	\draw[thick, zString] ($ (d) + 4*(.12,0) + (0,-.08) $) .. controls ++(90:.4cm) and ++(270:.4cm) .. ($ (d) + 3*(.12,0) + (0,-.08) + (0,1) + (0,.18)$);

	\draw[thick, wString] ($ (a) + 2*(.15,0) + (0,-.1) + (0,.27)$) .. controls ++(90:.5cm) and ++(270:.8cm) .. ($ (b) + (.15,0) + (0,-.08) + (0,.2) $);	
	\draw[thick, xString] ($ (a1) + 2*(.15,0) + (0,-.1) + (-.12,.375)$) .. controls ++(90:.25cm) and ++(270:.45cm) .. ($ (b) + 2*(.15,0)$);	
	\draw[thick, yString] ($ (a1) + 3*(.15,0) + (0,-.08) + (-.05,.14)$) .. controls ++(90:.35cm) and ++(270:.45cm) .. ($ (b) + 3*(.15,0) + (0,-.08) $);	
	\draw[thick, zString] ($ (a2) + 2*(.15,0) + (0,-.09) + (0,.26)$) .. controls ++(90:.3cm) and ++(270:.45cm) .. ($ (b2) + 2*(.15,0) + (0,-.05) $) -- ($ (c2) + 2*(.15,0) + (0,-.1) $);	
	
	\draw[thick, yString] ($ (c) + (.15,0) + (0,-.08) + (0,.3)$) .. controls ++(90:.6cm) and ++(270:.6cm) .. ($ (d) + (.12,0) + (0,-.08) + (0,.2) $);
	\draw[thick, wString] ($ (c) + 2*(.15,0) + (0,-.1) + (0,.2)$) .. controls ++(90:.6cm) and ++(270:.8cm) .. ($ (d) + 2*(.12,0) + (0,-.08) + (0,.2) $);	
	\draw[thick, xString] ($ (c) + 3*(.15,0) + (0,-.08)$) .. controls ++(90:.7cm) and ++(270:.75cm) .. ($ (d) + 3*(.12,0) + (0,-.08) $);	
	\draw[thick, zString] ($ (c2) + 2*(.15,0) + (0,-.1)$) .. controls ++(90:.8cm) and ++(270:.7cm) .. ($ (d) + 4*(.12,0) + (0,-.08) $);	

\end{tikzpicture}
\]
By successive application of Equations \eqref{pic: twopics - b}, \eqref{eq: pants o epsilon}, \eqref{pic: twopics - a}, and then once again \eqref{eq: pants o epsilon},
we can rewrite this as:
\begin{equation}
\begin{tikzpicture}[baseline=.9cm, scale=.8]
	\plane{(-.5,-1)}{5}{4}

	\draw[thick, wString] (-1.8,2) -- (1.5,2);
	\draw[thick, xString] (-.1,.9) -- (2.6,.9);
	\draw[thick, zString] (2.4,-.2) -- (3.7,-.2);
	\draw[thick, yString] (-.1,.7) .. controls ++(0:.4cm) and ++(180:.4cm) .. (1.6,-.8) -- (4.3,-.8);

	\CMbox{box}{(-2.6,1.6)}{.8}{.8}{.4}{$\varepsilon_w$}
	\CMbox{box}{(-1.3,.3)}{1.2}{.8}{.4}{$\varepsilon_{x\otimes y}$}
	\CMbox{box}{(1.6,-.6)}{.8}{.8}{.4}{$\varepsilon_{z}$}

	\straightTubeWithString{(-2.8,1.9)}{.1}{1.5}{wString}
	\straightTubeTwoStrings{(-1.5,.6)}{1.5}{xString}{yString}
	\straightTubeWithString{(1.4,-.3)}{.1}{3.5}{zString}

\node[scale=.87] at (8.5,1.3) {$\in \cM\big(\Phi(\Tr_\cC(w)\otimes \Tr_\cC(x\otimes y)\otimes \Tr_\cC(z)), w\otimes x\otimes z\otimes y\big).$};
\end{tikzpicture}
\label{eqn:CommonMap}
\end{equation}
Starting instead from the right hand side of \eqref{eqn:RHS}, if we take its mate as above, and apply
\eqref{eq: pants o epsilon}, then 
\eqref{pic: twopics - b}, then
\eqref{eq: pants o epsilon}, and finally \eqref{pic: twopics - a},
we obtain the same picture \eqref{eqn:CommonMap}.
\end{proof}

We now prove a compatibility between the traciator and the unit $\eta$.

\begin{lem}\label{lem:StringOverCap}
For $x\in \cM$ and $c\in \cC$, the following diagram commutes:
\[
\xymatrix{
&c \ar[dr]^{\eta} \ar[d]^{\eta} \ar[dl]_{\eta}&
\\
 \Tr_\cC(\Phi(c)) \ar[d]_{\Tr_\cC(\id\otimes \tilde\coev_x)} &  \Tr_\cC(\Phi(c)) \ar[d]|(.48){\,\Tr_\cC((e_{\Phi(c),x}\otimes \id)\circ(\id\otimes \coev_x))\phantom{\textstyle |}} & \Tr_\cC(\Phi(c))\ar[d]^{\Tr_\cC(\tilde\coev_x\otimes \id)}
\\
\Tr_\cC(\Phi(c)\otimes x^* \otimes x) \ar[r]_(.49){\tau^+} & \Tr_\cC(x \otimes \Phi(c) \otimes x^*) & \Tr_\cC(x^* \otimes x\otimes \Phi(c)) \ar[l]^(.48){\tau^-}
}
\]
In diagrams,
$$
\begin{tikzpicture}[baseline=.9cm, xscale=1.1]

	\coordinate (a1) at (0,0);
	\coordinate (b1) at (0,2);
	\coordinate (c) at (.15,.4);
	
	\draw[thick] (a1) -- (b1);
	\draw[thick] ($ (a1) + (.6,0) $) -- ($ (b1) + (.6,0) $);
	\halfDottedEllipse{(0,.15)}{.3}{.1}
	\draw[thick] ($ (b1) + (.3,0) $) ellipse (.3cm and .1cm);
	\draw[thick] (a1) arc (-180:0:.3cm);		

	\draw[thick, xString] (c)+(.15,0) arc (-180:0:.08cm)  .. controls ++(90:.2cm) and ++(225:.2cm) .. ($ (c) + (.45,.6) $);		
	\draw[thick, xString] ($ (b1) + (.15,-.08) $) .. controls ++(270:.2cm) and ++(45:.2cm) .. ($ (c) + (-.15,.8) $);
	\draw[thick, xString, dotted] ($ (c) + (-.15,.8) $) -- ($ (c) + (.45,.6) $);	
	\draw[thick, xString] (c)++(.15,0) .. controls ++(90:.8cm) and ++(270:.8cm) .. ($ (b1) + (.45,-.08) $);	
\draw[super thick, white] (.3,1.95) .. controls ++(270:.8cm) and ++(90:.8cm) .. (.15,.5) -- (.15,0) .. controls ++(270:.2cm) and ++(90:.2cm) .. ++(-.2,-.5) -- +(0,-.1);
\draw[thick, cString] (.3,1.95) .. controls ++(270:.8cm) and ++(90:.8cm) .. (.15,.5) -- (.15,0) .. controls ++(270:.2cm) and ++(90:.2cm) .. ++(-.2,-.5) -- +(0,-.1);
\end{tikzpicture}
\,\,=\,\,
\begin{tikzpicture}[baseline=.9cm, xscale=1.1]

	\coordinate (a1) at (0,0);
	\coordinate (b1) at (0,2);
	\coordinate (c) at (.15,.4);
	
	\draw[thick] (a1) -- (b1);
	\draw[thick] ($ (a1) + (.6,0) $) -- ($ (b1) + (.6,0) $);
	\halfDottedEllipse{(0,.15)}{.3}{.1}
	\draw[thick] ($ (b1) + (.3,0) $) ellipse (.3cm and .1cm);
	\draw[thick] (a1) arc (-180:0:.3cm);		

	\draw[thick, xString] ($ (b1) + (.15,-.08) $) -- ++ (0,-.3) .. controls ++(270:.8cm) and ++(90:.8cm) .. (.3,.5) -- ++(0,-.08) arc (-180:0:.08cm) -- ($ (b1) + (.45,-.08) $);	
\draw[super thick, white] (.3,1.95) -- ++ (0,-.3) .. controls ++(270:.8cm) and ++(90:.8cm) .. (.15,.5) -- (.15,0) .. controls ++(270:.2cm) and ++(90:.2cm) .. ++(-.2,-.5) -- +(0,-.1);
\draw[thick, cString] (.3,1.95) -- ++ (0,-.3) .. controls ++(270:.8cm) and ++(90:.8cm) .. (.15,.5) -- (.15,0) .. controls ++(270:.2cm) and ++(90:.2cm) .. ++(-.2,-.5) -- +(0,-.1);
\end{tikzpicture}
\,\,=\,\,
\begin{tikzpicture}[baseline=.9cm, xscale=-1.1]

	\coordinate (a1) at (0,0);
	\coordinate (b1) at (0,2);
	\coordinate (c) at (.15,.4);
	
	\draw[thick] (a1) -- (b1);
	\draw[thick] ($ (a1) + (.6,0) $) -- ($ (b1) + (.6,0) $);
	\halfDottedEllipse{(0,.15)}{.3}{.1}
	\draw[thick] ($ (b1) + (.3,0) $) ellipse (.3cm and .1cm);
	\draw[thick] (a1) arc (-180:0:.3cm);		

	\draw[thick, xString] (c)+(.15,0) arc (-180:0:.08cm)  .. controls ++(90:.2cm) and ++(225:.2cm) .. ($ (c) + (.45,.6) $);		
	\draw[thick, xString] ($ (b1) + (.15,-.08) $) .. controls ++(270:.2cm) and ++(45:.2cm) .. ($ (c) + (-.15,.8) $);
	\draw[thick, xString, dotted] ($ (c) + (-.15,.8) $) -- ($ (c) + (.45,.6) $);	
	\draw[thick, xString] (c)++(.15,0) .. controls ++(90:.8cm) and ++(270:.8cm) .. ($ (b1) + (.45,-.08) $);	
\draw[super thick, white] (.3,1.95) .. controls ++(270:.8cm) and ++(90:.8cm) .. (.15,.5) -- (.15,.1) .. controls ++(270:.3cm) and ++(90:.3cm) .. ++(.3,-.6) -- +(0,-.1);
\draw[thick, cString] (.3,1.95) .. controls ++(270:.8cm) and ++(90:.8cm) .. (.15,.5) -- (.15,.1) .. controls ++(270:.3cm) and ++(90:.3cm) .. ++(.3,-.6) -- +(0,-.1);
\end{tikzpicture}
$$
\end{lem}
\begin{proof}
We only prove that the left hand diagram commutes (the other one is similar).
The mate of the leftmost map can be simplified in the following way:
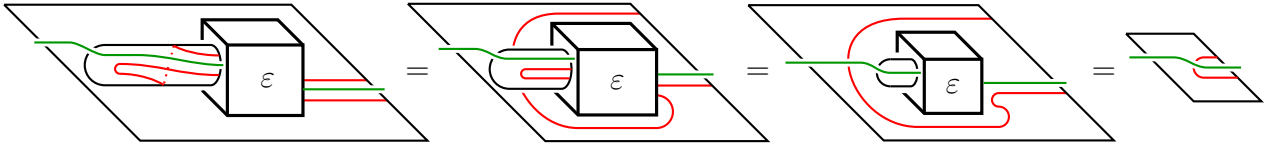
\[
\begin{tikzpicture}[baseline=.22cm, scale=.67]
	\draw[thick, xString] (1.4,-.1) -- (3.4,-.1);
	\draw[thick, xString] (1.4,.3) -- (3,.3);
	\CMbox{box}{(.23,-.4)}{1.5}{1.4}{.5}{$\,\varepsilon$}
	\straightTubeNoString{(-.1,.2)}{.2}{3}
\pgftransformyshift{17}
\pgftransformyscale{1.3}
\pgftransformxshift{-59}
\pgftransformxscale{2}
\pgftransformrotate{90}
	\draw[thick, xString] (-.15,-1.07) .. controls ++(90:.4cm) and ++(270:.4cm) .. (.1-.08,-.08);		
	\draw[thick, xString] (.15,-1.07) .. controls ++(90:.2cm) and ++(225:.1cm) .. (.3,-.62);		
	\draw[thick, xString] (-.1-.03,-.08) .. controls ++(270:.2cm) and ++(45:.1cm) .. (-.3,-.52);
	\draw[thick, xString, dotted] (.3,-.62) -- (-.3,-.52);	
\pgftransformrotate{-90}
\pgftransformxscale{1/2}
\pgftransformxshift{59}
\pgftransformyscale{1/1.3}
\pgftransformyshift{-17}
	\fill[white] (-4,.1) rectangle (-1.9,1.1);
	\draw[thick] (-1.9,1) -- ++(-.3,0) arc (90:270:.4) -- ++(.3,0);
	\plane{(-1.5,-.9)}{5.7}{2.7}
	\draw [xString,thick] (-1.9,.624) arc (90:270:.098);
\draw[super thick, white] (1.9,.12) -- (3.35,.12);
\draw[thick, cString] (1.73,.12) -- (3.35,.12);
\draw[super thick, white] (.15,.6) .. controls ++(180:.6cm) and ++(0:.6cm) .. ++(-1.7,.2) -- ++(-.65,0) .. controls ++(180:.2cm) and ++(0:.2cm) .. ++(-.7,.25) -- +(-.7,0);
\draw[thick, cString] (.15,.6) .. controls ++(180:.6cm) and ++(0:.6cm) .. ++(-1.7,.2) -- ++(-.65,0) .. controls ++(180:.2cm) and ++(0:.2cm) .. ++(-.7,.25) -- +(-.7,0);
\end{tikzpicture}
\!\!\!\!\!\!=\!\!\!\!\!\!
\begin{tikzpicture}[baseline=.133cm, scale=.87]
	\draw[thick, xString] (1.2,-.1) arc (90:-90:.25cm) -- (0,-.6) .. controls ++(180:1cm) and ++(180:1.2cm) .. (-.3,1.15) -- (.9+.05,1.15);
	\draw[thick, xString] (1.2,.2-.15) -- (1.9+.15,.2-.15);
	\CMbox{box}{(0,-.4)}{1.2}{1}{.4}{$\varepsilon$}
	\filldraw[line width=3, white] (-.25,0) arc (-90:90:.15 and .3) -- ++(-.8,0) arc (90:270:.3) -- cycle;
	\draw[thick] (-.25,0) arc (-90:90:.15 and .3) -- ++(-.8,0) arc (90:270:.3) -- cycle;
	\draw [xString,thick] (-.8,.3) arc (90:270:.07) -- ++(.67,0) (-.8,.3) -- ++(.7,0);
	\plane{(-.5,-.8)}{3.2+.2}{2.1}
\draw[super thick, white](1.5,.05+.17) -- (2.2,.05+.17);
\draw[thick, cString] (1.22,.05+.17) -- (2.2-.13,.05+.17);
\draw[super thick, white]  (-.05,.3+.16) -- ++(-1.03,0) .. controls ++(180:.2cm) and ++(0:.2cm) .. ++(-.5,.15) -- +(-.5,0);
\draw[thick, cString] (-.05,.3+.16) -- ++(-1.03,0) .. controls ++(180:.2cm) and ++(0:.2cm) .. ++(-.5,.15) -- +(-.55,0);
\end{tikzpicture}
\!\!\!\!\!\!=\!\!\!\!\!\!
\begin{tikzpicture}[baseline=.1cm, scale=.9]
	\draw[thick, xString] (2.2,-.1) -- (1.2,-.1) arc (90:270:.1cm) arc (90:-90:.15cm) -- (0,-.6) .. controls ++(180:1cm) and ++(180:1.4cm) .. (-.2,1) -- (1.1,1);
	\CMbox{box}{(.1,-.4)}{.85}{.8}{.4}{$\varepsilon$}
	\straightTubeWithCap{(-.1,0)}{.1}{.3}
	\plane{(-.5,-.8)}{3.2+.2}{2}
\draw[super thick, white](1.5,.05) -- (2.2,.05);
\draw[thick, cString] (.974,.05) -- (2.2,.05);
\draw[super thick, white]  (.05,.2) -- ++(-.45,0) .. controls ++(180:.2cm) and ++(0:.2cm) .. ++(-.45,.15) -- +(-1.1,0);
\draw[thick, cString] (.05,.2) -- ++(-.45,0) .. controls ++(180:.2cm) and ++(0:.2cm) .. ++(-.45,.15) -- +(-1.1,0);
\end{tikzpicture}
\!\!\!\!\!\!=
\begin{tikzpicture}[baseline=.3cm, scale=.9]
	\plane{(0,0)}{1}{1}
	\draw [xString,thick] (.35,.65) -- +(-.2,0) arc (90:270:.15) -- +(.5,0);
\draw[super thick, white] (.7,.5) -- ++(-.5,0) .. controls ++(180:.2cm) and ++(0:.2cm) .. ++(-.5,.15) -- +(-.65,0);
\draw[thick, cString] (.7,.5) -- ++(-.5,0) .. controls ++(180:.2cm) and ++(0:.2cm) .. ++(-.5,.15) -- +(-.65,0);
\end{tikzpicture}
\]
where we have used Equation \eqref{pic: twopics - a}, then
Lemma \ref{lem:MoveThroughNaturalMap} (naturality of $\varepsilon$), and finally Equation \eqref{eq:MAdjointRelation}.
It agrees with the mate of the middle map:
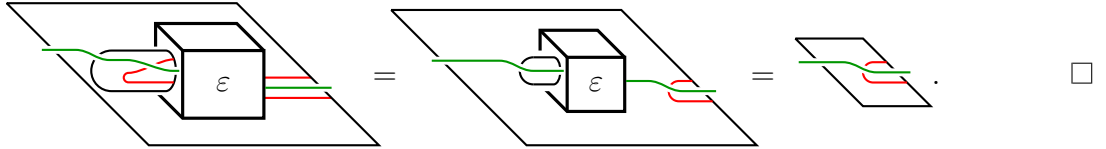
\begin{align*}
\begin{tikzpicture}[baseline=.1cm, scale=.9]
	\draw[thick, xString] (1.2,-.1) -- (2.2,-.1);
	\draw[thick, xString] (1.2,.2) -- (1.7+.2,.2);
	\CMbox{box}{(0,-.4)}{1.2}{1}{.4}{$\varepsilon$}
	\filldraw[line width=3, white] (-.25,0) arc (-90:90:.15 and .3) -- ++(-.8,0) arc (90:270:.3) -- cycle;
	\draw[thick] (-.25,0) arc (-90:90:.15 and .3) -- ++(-.8,0) arc (90:270:.3) -- cycle;
	\draw [xString,thick] (-.8,.27) arc (90:270:.07) -- ++(.67,0) (-.8,.27) .. controls ++(0:.2cm) and ++(180:.2cm) .. ++(.67,.2);
	\plane{(-.5,-.8)}{3.2+.2}{2.1}
	\draw[super thick, white](1.5,.05) -- (2.2,.05);
	\draw[thick, cString] (1.22,.05) -- (2.2,.05);
\draw[super thick, white]  (-.05,.3) -- ++(-.1,0) .. controls ++(180:.2cm) and ++(0:.2cm) .. ++(-.67,.16) -- ++(-.26,0) .. controls ++(180:.2cm) and ++(0:.2cm) .. ++(-.5,.15) -- +(-.5,0);
\draw[thick, cString] (-.05,.3) -- ++(-.1,0) .. controls ++(180:.2cm) and ++(0:.2cm) .. ++(-.67,.16) -- ++(-.26,0) .. controls ++(180:.2cm) and ++(0:.2cm) .. ++(-.5,.15) -- +(-.5,0);
\end{tikzpicture}
\!\!\!&=\!\!\!
\begin{tikzpicture}[baseline=.1cm, scale=.9]
	\CMbox{box}{(.1,-.3)}{.85}{.8}{.4}{$\varepsilon$}
	\straightTubeWithCap{(-.1,.1)}{.1}{.3}
	\plane{(-.5,-.8)}{3.2+.2}{2}
\draw [xString,thick] (1.35+.6,.15) -- +(-.2,0) arc (90:270:.15) -- +(.5,0);
\draw[super thick, white] (1.7+.6,0) -- ++(-.5,0) .. controls ++(180:.2cm) and ++(0:.2cm) .. ++(-.5,.15);
\draw[thick, cString] (1.7+.6,0) -- ++(-.5,0) .. controls ++(180:.2cm) and ++(0:.2cm) .. ++(-.5,.15) -- +(-.33,0);
\draw[super thick, white]  (.05,.3) -- ++(-.45,0) .. controls ++(180:.2cm) and ++(0:.2cm) .. ++(-.5,.15) -- +(-1,0);
\draw[thick, cString] (.05,.3) -- ++(-.45,0) .. controls ++(180:.2cm) and ++(0:.2cm) .. ++(-.5,.15) -- +(-1,0);
\end{tikzpicture}
\!\!\!=\,\,
\begin{tikzpicture}[baseline=.3cm, scale=.9]
	\plane{(0,0)}{1}{1}
	\draw [xString,thick] (.35,.65) -- +(-.2,0) arc (90:270:.15) -- +(.5,0);
\draw[super thick, white] (.7,.5) -- ++(-.5,0) .. controls ++(180:.2cm) and ++(0:.2cm) .. ++(-.5,.15) -- +(-.65,0);
\draw[thick, cString] (.7,.5) -- ++(-.5,0) .. controls ++(180:.2cm) and ++(0:.2cm) .. ++(-.5,.15) -- +(-.65,0);
\end{tikzpicture}
.
\qedhere
\end{align*}
\end{proof}

We now assume that $\cC$ and $\Phi^\ssZ$ are braided, so that $\Tr_\cC$ is the categorified trace associated to a pivotal module tensor category. In this context, we establish certain relations between $\Tr_\cC$ and the braiding and twist of $\cC$. We begin by examining the traciator and the twist.


\begin{lem}\label{lem:PullUpTheta}
Let $c\in \cC$ and $x\in \cM$ be objects. 
Then for every $f\in\cM(\Phi(c), x)$, we have
\[
f\circ\Phi(\theta_c)
=	(\id_x\otimes \tilde\ev_x)
\circ	(\id_x\otimes f\otimes \id_{x^*})
\circ	(e_{\Phi(c),x}\otimes \id_{x^*})
\circ	(\id_{\Phi(c)}\otimes \coev_x),
\]
where $\tilde \ev_x:x\otimes x^*\to 1$ is as in Definition \ref{defn:Traciator}. 
In diagrams:
$$
\begin{tikzpicture}[baseline=.1cm, scale=.8]
	\plane{(-.5,-.8)}{3}{2}
	\draw[thick, xString] (.8,0) -- (1.7,0);
	
	\CMbox{box}{(0,-.4)}{.8}{.8}{.4}{$f$}
	\invisibleTube{(-.3,.1)}
	\fill[unshaded] (-1.3,1.2) circle (.1cm);
	\draw[thick, cString] (-.2,.2) -- (-.3,.2) .. controls ++(180:.6cm) and ++(315:.6cm) .. (-1.35,.8);
	\draw[thick, cString] (-2,1.7) .. controls ++(0:1.1cm) and ++(310:.4cm) .. (-1.7,.8) .. controls ++(135:.2cm) and ++(135:.2cm) .. (-1.5,.95);
\end{tikzpicture}
=
\begin{tikzpicture}[baseline=.1cm, scale=.8]
	\plane{(-.5,-.8)}{3}{2}
	\draw[thick, xString] (.8,0) arc (90:-90:.3cm) -- (0,-.6) arc (270:90:.8cm) -- (.7,1);
	\CMbox{box}{(0,-.4)}{.8}{.8}{.4}{$f$}
	\invisibleTubeWithString{(-.3,.1)}{cString}
	\draw[thick, cString] (-1.4,1.4) .. controls ++(90:.2cm) and ++(0:.2cm) .. (-2,1.7);
\end{tikzpicture}.
$$
\end{lem}
\begin{proof}
Starting from the right hand side,
we bring $f$ under the crossing using pivotality:\vspace{-.2cm}
\[
\begin{tikzpicture}[baseline=-.1, xscale=1.1]
	\draw[thick, xString] (1.3,0) arc (90:-90:.3cm) -- (-.7,-.6) arc (270:90:.3cm) .. controls ++(0:.3cm) and ++(180:.3cm) .. (.7,.6) -- (1.6,.6);
	\fill[unshaded] (0,.3) circle (.2cm);
	\draw[thick, cString] (-1,.6) -- (-.7,.6) .. controls ++(0:.3cm) and ++(180:.3cm) .. (.7,0);
	\nbox{unshaded}{(1,0)}{.3}{0}{0}{$f$}
	\node[scale=1.1] at (-.8,.85) {$\scriptstyle \Phi(c)$};
	\node[scale=1.1] at (1.4,.85) {$\scriptstyle x$};
	\node at (1.9,.-.3) {$\scriptstyle \tilde\ev_x$};
	\node[left] at (-.9,.-.335) {$\scriptstyle \coev_x$};
\end{tikzpicture}
=
\begin{tikzpicture}[baseline=-.1, xscale=-1, xscale=1.1]
	\draw[thick, xString] (-1,.6) -- (-.7,.6) .. controls ++(0:.3cm) and ++(180:.3cm) .. (.7,0);
	\fill[unshaded] (0,.3) circle (.2cm);
	\draw[thick, cString] (1.3,0) arc (90:-90:.3cm) -- (-.7,-.6) arc (270:90:.3cm) .. controls ++(0:.3cm) and ++(180:.3cm) .. (.7,.6) -- (1.6,.6);
	\nbox{unshaded}{(1,0)}{.3}{0}{0}{$f$}
	\node at (2.15,.-.335) {$\scriptstyle \coev_{\Phi(c)}$};
	\node at (-1.45,.-.3) {$\scriptstyle \tilde\ev_{\Phi(c)}$};
\end{tikzpicture}
=
\begin{tikzpicture}[baseline=-.1, xscale=1.1]
	\draw[thick, cString] (.7,0) arc (90:-90:.3cm) -- (-.7,-.6) arc (270:90:.3cm) .. controls ++(0:.3cm) and ++(180:.3cm) .. (.7,.6);
	\fill[unshaded] (0,.3) circle (.2cm);
	\draw[thick, cString] (-1,.6) -- (-.7,.6) .. controls ++(0:.3cm) and ++(180:.3cm) .. (.7,0);
	\draw[thick, xString] (1.3,.6) -- (1.8,.6);
	\nbox{unshaded}{(1,.6)}{.3}{0}{0}{$f$}
	\node at (1.45,.-.3) {$\scriptstyle \tilde\ev_{\Phi(c)}$};
	\node at (-1.55,.-.335) {$\scriptstyle \coev_{\Phi(c)}$};
\end{tikzpicture}.\vspace{.1cm}
\]
To conclude the argument, note that $(\id_{\Phi(c)}\otimes\tilde\ev_{\Phi(c)})\circ(e_{\Phi(c),\Phi(c)}\otimes \id_{\Phi(c)^*})\circ(\id_{\Phi(c)}\otimes\coev_{\Phi(c)})$
is the image of
$(\id_c\otimes\tilde\ev_c)\circ(\beta_{c,c}\otimes \id_{c^*})\circ(\id_c\otimes\coev_c)=\theta_c$
under $\Phi=F\circ \Phi^{\scriptscriptstyle \cZ}$, because
$\Phi^{\scriptscriptstyle \cZ}:\cC\to \cZ(\cM)$ is a balanced functor (by Lemma \ref{lem: it is a balanced functor}).
\end{proof}

\begin{prop}
\label{prop:ThetaAndTraciator}
The map $\tau_{1_\cM,x}:\Tr_\cC(1_\cM \otimes x) \to \Tr_\cC(x \otimes 1_\cM )$ is equal to the twist $\theta_{\Tr_\cC(x)}$.
In diagrams:
$$
\theta_{\Tr_\cC(x)}\,\,=\,\,
\begin{tikzpicture}[baseline=-.1cm]

	\draw[thick] (-.3,-1) -- (-.3,1);
	\draw[thick] (.3,-1) -- (.3,1);
	\draw[thick] (0,1) ellipse (.3cm and .1cm);
	\halfDottedEllipse{(-.3,-1)}{.3}{.1}
	
	\draw[thick, xString] (0,-1.1) .. controls ++(90:.2cm) and ++(225:.2cm) .. (.3,-.2);		
	\draw[thick, xString] (0,.9) .. controls ++(270:.2cm) and ++(45:.2cm) .. (-.3,.2);
	\draw[thick, xString, dotted] (-.3,.2) -- (.3,-.2);	
\end{tikzpicture}\,.
$$
Similarly, we have $\theta_{\Tr_\cC(x)}^{-1}=\tau^-_{x,1_\cM}$.
\end{prop}
\begin{proof}
Taking $c=\Tr_\cC(x)$ and $f=\varepsilon_x$ in the previous lemma, and using Equation \eqref{pic: twopics - a}, we get
\(
\varepsilon_x\circ\Phi(\theta_{\Tr_\cC(x)})=\varepsilon_x\circ\Phi(\tau_{1_\cM,x})
\),
which is the mate of our equation.
\end{proof}

\begin{cor}
\label{cor:ComposeTraciatorsToGetBigTheta}
For $x,y\in \cM$, the map $\theta_{\Tr_\cC(x\otimes y)}$ is equal to $\tau_{y,x}\circ \tau_{x,y} : \Tr_\cC(x\otimes y)\to \Tr_\cC(x\otimes y)$.
\end{cor}
\begin{proof}
By Lemma \ref{lem:TraciatorComposition}, $\tau_{y,x}\circ \tau_{x,y}$ is equal to $\tau_{x \otimes y,1_\cM} : 1_\cM\otimes x\otimes y \to x\otimes y\otimes 1_\cM$ (modulo unitors and associators which we suppress).
The latter is equal to $\theta_{\Tr_\cC(x\otimes y)}$ by the previous proposition.
\end{proof}

We now establish the relationship between the traciator, the braiding and the twist.

\begin{lem}\label{lem:ThetaAndBraiding}
Given morphisms $f:\Phi(c)\to x$ and $g:\Phi(d)\to y$ in $\cM$, the following equation holds:
\[
(\id_{x\otimes y}\otimes \tilde\ev_x)\circ(\id_x\otimes g\otimes f\otimes \id_{x^*})\circ e_{\Phi(d\otimes c),x}\circ(\id_{\Phi(d\otimes c)}\otimes \coev_x)
=
(f\otimes g)\circ\Phi(\beta_{d,c}\circ (\id_d\otimes \theta_c)).
\]
In diagrams:
$$
\begin{tikzpicture}[baseline=.5cm, scale=.8]
	\plane{(-.5,-.8)}{3}{2.8}

	\draw[thick, xString] (.8,0) arc (90:-90:.3cm) -- (0,-.6) .. controls ++(180:1cm) and ++(180:1.4cm) .. (-1.2,1.8) -- (-.1,1.8);
	\draw[thick, yString] (0,1) -- (.7,1);

	\CMbox{box}{(-.8,.4)}{.8}{.8}{.4}{$g$}
	\CMbox{box}{(0,-.4)}{.8}{.8}{.4}{$f$}
	\fill[unshaded] (-1.3,.1) rectangle (-.3,.3);
	\fill[unshaded] (-2.65,2) circle (.1);
	\invisibleTubeWithString{(-1.1,.9)}{dString}
	\invisibleTubeWithString{(-1.3,.1)}{cString}

	\draw[thick, dString] (-2.2,2.2) arc (0:90:.6cm);
	\draw[thick, cString] (-2.4,1.4) arc (0:90:.9cm);
	\draw[thick, cString] (-1.2,.2) -- (-.2,.2);

\node at (-3.2,2.48) {$\scriptstyle c$};
\node at (-2.65,3.05) {$\scriptstyle d$};
\node at (.2,1.85) {$\scriptstyle x$};
\node at (1.03,1.05) {$\scriptstyle y$};
\end{tikzpicture}
=
\begin{tikzpicture}[baseline=.5cm, scale=.8]
	\plane{(-.5,-.8)}{3}{2.8}

	\draw[thick, xString] (0,1) -- (.7,1);
	\draw[thick, yString] (.4,.2) -- (1.5,.2);

	\CMbox{box}{(-.8,.4)}{.8}{.8}{.4}{$f$}
	\CMbox{box}{(0,-.4)}{.8}{.8}{.4}{$g$}
	\invisibleTube{(-1.1,.9)}
	\fill[unshaded] (-1.75,2) circle (.1cm);
	\fill[unshaded] (-2.15,2) circle (.1cm);
	\draw[thick, cString] (-1,1) -- (-1.1,1) .. controls ++(180:.6cm) and ++(315:.6cm) .. (-2.1,1.55);
	\draw[thick, cString] (-2.8,2.3) .. controls ++(0:1.1cm) and ++(310:.4cm) .. (-2.5,1.5) .. controls ++(135:.2cm) and ++(135:.2cm) .. (-2.3,1.75);
	
	\fill[unshaded] (-.5,.1) rectangle (-.3,.3);
	\invisibleTubeWithString{(-.5,.1)}{dString}
	\draw[thick, dString] (-.5,.2) -- (-.2,.2);
	\draw[thick, dString] (-1.6,1.4) arc (0:90:1.3cm);

\node at (-3.07,2.3) {$\scriptstyle c$};
\node at (-3.15,2.75) {$\scriptstyle d$};
\node at (1.02,1.05) {$\scriptstyle x$};
\node at (1.85,.22) {$\scriptstyle y$};
\end{tikzpicture}.
$$
\end{lem}
\begin{proof}
This follows from Lemma \ref{lem:PullUpTheta} applied to $f$, together with passing the morphism $f$ under the $d$ strand.
\end{proof}

\begin{lem}
\label{lem:TwistMultiplicationAndTraciators}
Given $x,y\in \cM$, we have
\[
\tau_{x,y}\circ\mu_{x,y}=\mu_{y,x}\circ \beta_{\Tr_\cC(x),\Tr_\cC(y)} \circ (\id_{\Tr_\cC(x)} \otimes\, \theta_{\Tr_\cC(y)}).
\]
In diagrams: 
$$
\begin{tikzpicture}[baseline=.4cm]

	\draw[thick] (-.3,.5) -- (-.3,1.5);
	\draw[thick] (.3,.5) -- (.3,1.5);
	\draw[thick] (0,1.5) ellipse (.3cm and .1cm);
	
	\draw[thick, xString] (-.1,.43) .. controls ++(90:.6cm) and ++(270:.6cm) .. (.1,1.42);		
	\draw[thick, yString] (.1,.43) .. controls ++(90:.2cm) and ++(225:.2cm) .. (.3,.95);		
	\draw[thick, yString] (-.1,1.42) .. controls ++(270:.2cm) and ++(45:.2cm) .. (-.3,1.05);
	\draw[thick, yString, dotted] (.3,.95) -- (-.3,1.05);	

	\draw[thick] (-1,-1) .. controls ++(90:.8cm) and ++(270:.8cm) .. (-.3,.5);
	\draw[thick] (1,-1) .. controls ++(90:.8cm) and ++(270:.8cm) .. (.3,.5);
	\draw[thick] (-.4,-1) .. controls ++(90:.8cm) and ++(90:.8cm) .. (.4,-1); 
	\halfDottedEllipse{(-.3,.5)}{.3}{.1}
	\halfDottedEllipse{(-1,-1)}{.3}{.1}
	\halfDottedEllipse{(.4,-1)}{.3}{.1}
	
	\draw[thick, xString] (-.8,-1.07) .. controls ++(90:.8cm) and ++(270:.8cm) .. (-.1,.42);		
	\draw[thick, yString] (.8,-1.07) .. controls ++(90:.8cm) and ++(270:.8cm) .. (.1,.42);		
\end{tikzpicture}
=
\begin{tikzpicture}[baseline=-1.1cm]

	\draw[thick] (-1,-2) -- (-1,-3);
	\draw[thick] (-.4,-2) -- (-.4,-3);
	\halfDottedEllipse{(-1,-3)}{.3}{.1}
	
	\draw[thick, xString] (-.7,-2.1) -- (-.7,-3.1);

	\draw[thick] (1,-2) -- (1,-3);
	\draw[thick] (.4,-2) -- (.4,-3);
	\halfDottedEllipse{(.4,-3)}{.3}{.1}

	\draw[thick, yString] (.7,-3.1) .. controls ++(90:.2cm) and ++(225:.2cm) .. (1,-2.55);		
	\draw[thick, yString] (.7,-2.1) .. controls ++(270:.2cm) and ++(45:.2cm) .. (.4,-2.45);
	\draw[thick, yString, dotted] (.4,-2.45) -- (1,-2.55);		

	\draw[thick] (1,-2) .. controls ++(90:.7cm) and ++(270:.7cm) .. (-.4,0);
	\draw[thick] (.4,-2) .. controls ++(90:.7cm) and ++(270:.7cm) .. (-1,0);
	\draw[thick, yString] (-.7,-.07) .. controls ++(270:.6cm) and ++(90:.8cm) .. (.7,-2.1);		

	\fill[unshaded] (-1.2,-2) .. controls ++(90:.8cm) and ++(270:.8cm) .. (.2,0) -- (1.2,0)  .. controls ++(270:.8cm) and ++(90:.8cm) .. (-.2,-2);
	\draw[thick] (-1,-2) .. controls ++(90:.7cm) and ++(270:.7cm) .. (.4,0);
	\draw[thick] (-.4,-2) .. controls ++(90:.7cm) and ++(270:.7cm) .. (1,0);
	\draw[thick, xString] (.7,-.07) .. controls ++(270:.6cm) and ++(90:.8cm) .. (-.7,-2.1);	

	\halfDottedEllipse{(-1,-2)}{.3}{.1}
	\halfDottedEllipse{(.4,-2)}{.3}{.1}

	\draw[thick] (-1,0) .. controls ++(90:.8cm) and ++(270:.8cm) .. (-.3,1.5);
	\draw[thick] (1,0) .. controls ++(90:.8cm) and ++(270:.8cm) .. (.3,1.5);
	\draw[thick] (-.4,0) .. controls ++(90:.8cm) and ++(90:.8cm) .. (.4,0); 
	\draw[thick] (0,1.5) ellipse (.3cm and .1cm);
	\halfDottedEllipse{(-1,0)}{.3}{.1}
	\halfDottedEllipse{(.4,0)}{.3}{.1}
	
	\draw[thick, yString] (-.7,-.07) .. controls ++(90:.8cm) and ++(270:.8cm) .. (-.1,1.42);		
	\draw[thick, xString] (.7,-.07) .. controls ++(90:.8cm) and ++(270:.8cm) .. (.1,1.42);		
\end{tikzpicture}.
$$
Similarly, we have
\(
\tau^-_{x,y}\circ\mu_{x,y}=\mu_{y,x}\circ \beta^-_{\Tr_\cC(x),\Tr_\cC(y)} \circ (\theta^{-1}_{\Tr_\cC(x)}\otimes \id_{\Tr_\cC(y)})
\).
\end{lem}
\begin{proof}
We only prove the first claim.
By taking mates, it is equivalent to the equation
\begin{equation}\label{eq: two thick tubes that braid}
\begin{tikzpicture}[baseline=.5cm, scale=.8]
	\plane{(-.5,-.8)}{3}{3}

	\draw[thick, yString] (.8,0) arc (90:-90:.3cm) -- (0,-.6) .. controls ++(180:1cm) and ++(180:1.4cm) .. (-1.2,2) -- (-.3,2);
	\draw[thick, xString] (-.2,1) -- (.7,1);

	\CMbox{box}{(-1,.6)}{.8}{.8}{.4}{$\varepsilon_x$}
	\CMbox{box}{(0,-.4)}{.8}{.8}{.4}{$\varepsilon_y$}
	\straightTubeWithString{(-1.2,.9)}{.1}{2}{xString}
	\straightTubeWithString{(-.2,-.1)}{.1}{2}{yString}

\end{tikzpicture}
\!\!\!\!=\,\,\,\,\,\,\,\,
\begin{tikzpicture}[baseline=.5cm, scale=.8]
	\plane{(-.5,-.8)}{3}{2.8}

	\draw[thick, yString] (-.2,1) -- (.7,1);
	\draw[thick, xString] (.4,0) -- (1.7,0);

	\pgfmathsetmacro{\distAtoB}{1};

	\coordinate (a) at (-1,.6);
	\coordinate (b) at ($ (a) + (\distAtoB,-\distAtoB)$);
	\CMbox{box}{(a)}{.8}{.8}{.4}{$\varepsilon_y$}
	\CMbox{box}{(b)}{.8}{.8}{.4}{$\varepsilon_x$}

	\coordinate (ZZqa) at ($ (a) + (-.2,.3)$);
	\coordinate (ZZqb) at ($ (b) + (-.2,.3)$);
	\pgfmathsetmacro{\tubeLength}{2};
	\pgfmathsetmacro{\tubeWidth}{.1};
	\pgfmathsetmacro{\buffer}{.05};	

	\fill[unshaded]  ($ (ZZqa) + 3/2*(-\tubeLength,0) + (0,-\distAtoB) + 2*(0,-\buffer) $) -- ($ (ZZqa) + (-\tubeLength,-\distAtoB) + 2*(0,-\buffer) $) .. controls ++(0:.6cm) and ++(180:.6cm) .. ($ (ZZqa) + 2*(0,-\buffer) $) arc(-90:90:{\tubeWidth+\buffer} and {2*(\tubeWidth+\buffer)}) .. controls ++(180:.6cm) and ++(0:.6cm) .. ($ (ZZqa) + (-\tubeLength,-\distAtoB) + 4*(0,\tubeWidth) + 2*(0,\buffer) $) -- ($ (ZZqa) + 3/2*(-\tubeLength,0) + (0,-\distAtoB) + 4*(0,\tubeWidth) + 2*(0,\buffer) $) ;

	\draw[unshaded, thick]  ($ (ZZqa) + 3/2*(-\tubeLength,0) + (0,-\distAtoB) $) -- ($ (ZZqa) + (-\tubeLength,-\distAtoB) $) .. controls ++(0:.6cm) and ++(180:.6cm) .. (ZZqa) arc(-90:90:{\tubeWidth} and {2*\tubeWidth}) .. controls ++(180:.6cm) and ++(0:.6cm) .. ($ (ZZqa) + (-\tubeLength,-\distAtoB) + 4*(0,\tubeWidth) $) -- ($ (ZZqa) + 3/2*(-\tubeLength,0) + (0,-\distAtoB) + 4*(0,\tubeWidth) $) ;

	\draw[thick, yString] ($ (ZZqa) + (\tubeWidth,0) + 2*(0,\tubeWidth) $) .. controls ++(180:.6cm) and ++(0:.4cm) .. ($ (ZZqa) + (-\tubeLength,-\distAtoB) + (\tubeWidth,0) + 2*(0,\tubeWidth) $) .. controls ++(180:.2cm) and ++(-45:.1cm) .. ($ (ZZqa) + 10/8*(-\tubeLength,0) + (0,-\distAtoB) + (\tubeWidth,0) + 4*(0,\tubeWidth) $); 
	\draw[thick, yString] ($ (ZZqa) + 3/2*(-\tubeLength,0) + (0,-\distAtoB) + (\tubeWidth,0) + 2*(0,\tubeWidth) $) .. controls ++(0:.2cm) and ++(135:.1cm) .. ($ (ZZqa) + 10/8*(-\tubeLength,0) + (0,-\distAtoB) + (\tubeWidth,0) $) ;
	\draw[thick, yString, dotted] ($ (ZZqa) + 10/8*(-\tubeLength,0) + (0,-\distAtoB) + (\tubeWidth,0) + 4*(0,\tubeWidth) $) -- ($ (ZZqa) + 10/8*(-\tubeLength,0) + (0,-\distAtoB) + (\tubeWidth,0) $);

	\fill[unshaded]  ($ (ZZqb) + 3/2*(-\tubeLength,0) + 2*(-\distAtoB,0) + (0,\distAtoB) + 2*(0,-\buffer) $) -- ($ (ZZqb) + (-\tubeLength,0) + 2*(-\distAtoB,0) + (0,\distAtoB) + 2*(0,-\buffer) $) .. controls ++(0:.6cm) and ++(180:.6cm) .. ($ (ZZqb) + 2*(0,-\buffer) $) arc(-90:90:{\tubeWidth + \buffer} and {2*(\tubeWidth+\buffer)}) .. controls ++(180:.6cm) and ++(0:.6cm) .. ($ (ZZqb) + (-\tubeLength,0) + 2*(-\distAtoB,0) + (0,\distAtoB) + 4*(0,\tubeWidth) + 2*(0,\buffer) $) -- ($ (ZZqb) + 3/2*(-\tubeLength,0) + 2*(-\distAtoB,0) + (0,\distAtoB) + 4*(0,\tubeWidth) + 2*(0,\buffer) $);

	\draw[unshaded, thick]  ($ (ZZqb) + 3/2*(-\tubeLength,0) + 2*(-\distAtoB,0) + (0,\distAtoB) $) -- ($ (ZZqb) + (-\tubeLength,0)+2*(-\distAtoB,0) + (0,\distAtoB) $) .. controls ++(0:.6cm) and ++(180:.6cm) .. (ZZqb) arc(-90:90:{\tubeWidth} and {2*\tubeWidth}) .. controls ++(180:.6cm) and ++(0:.6cm) .. ($ (ZZqb) + (-\tubeLength,0) + 2*(-\distAtoB,0) + (0,\distAtoB) + 4*(0,\tubeWidth) $) -- ($ (ZZqb) + 3/2*(-\tubeLength,0) + 2*(-\distAtoB,0) + (0,\distAtoB) + 4*(0,\tubeWidth) $);

	\draw[thick] ($ (ZZqa) + 3/2*(-\tubeLength,0) +(0,-\distAtoB) $) arc(-90:90:{\tubeWidth} and {2*\tubeWidth}) arc(90:270:{\tubeWidth} and {2*\tubeWidth});

	\draw[thick] ($ (ZZqb) + 3/2*(-\tubeLength,0) + 2*(-\distAtoB,0) + (0,\distAtoB) $) arc(-90:90:{\tubeWidth} and {2*\tubeWidth}) arc(90:270:{\tubeWidth} and {2*\tubeWidth});

	\draw[thick, xString] ($ (ZZqb) + (\tubeWidth,0) + 2*(0,\tubeWidth) $) .. controls ++(180:.6cm) and ++(0:.6cm) .. ($ (ZZqb) + (-\tubeLength,0) + 2*(-\distAtoB,0) + (0,\distAtoB) + (\tubeWidth,0) + 2*(0,\tubeWidth) $) -- ($ (ZZqb) + 3/2*(-\tubeLength,0) + 2*(-\distAtoB,0) + (0,\distAtoB) + (\tubeWidth,0) + 2*(0,\tubeWidth) $);
\end{tikzpicture}
\end{equation}
where we have used \eqref{pic: twopics - a} and Lemma \ref{lem:MultiplicationCompatibleWithNatural} for the left hand side, and Lemma \ref{lem:MultiplicationCompatibleWithNatural} for the right hand side.
Equation \eqref{eq: two thick tubes that braid} now follows from Proposition \ref{prop:ThetaAndTraciator} and Lemma \ref{lem:ThetaAndBraiding}.
\end{proof}



\section{Traces of algebras and semisimplicity}\label{sec:Pivotal}

We have seen in Section~\ref{sec:InternalTrace} that our categorified trace is a lax monoidal functor.
As a result, it takes algebra objects to algebra objects.
In fact, given two algebras $A$ and $B$, we will see that $\Tr_\cC(A \otimes B)$ is also naturally equipped with an algebra structure.
What is more, under mild additional assumptions, we will see that $\Tr_\cC(A \otimes B)$ is semisimple when both $A$ and $B$ are.
We begin by investigating the semisimplicity of $\Tr_\cC(A)$ under weaker assumptions than those necessary to study $\Tr_\cC(A \otimes B)$.

\subsection{The algebra \texorpdfstring{$\Tr_\cC(A)$}{trace of A}}
\label{sec:Algebras}
\definecolor{dString}{named}{blue}

If $A=(A,m_A,i_A)$ is an algebra object in $\cM$, then $\Tr_\cC(A)$ is an algebra object in $\cC$, where
the multiplication and unit maps are given by:
\[
m_{\Tr_\cC(A)}:=\Tr_\cC(m_A)\circ \mu_{A,A}=
\begin{tikzpicture}[baseline=.6cm, scale=.9]
	\topPairOfPants{(0,0)}{}
	\filldraw[AString] (1, 1.1) circle (.025cm);
	\draw[thick, AString] (.3,-.09) .. controls ++(92:.8cm) and ++(270:.4cm) .. (1,1.1);		
	\draw[thick, AString] (1.7,-.09) .. controls ++(90:.8cm) and ++(270:.4cm) .. (1,1.1);		
	\draw[thick, AString] (1,1.1) -- (1,1.41);
\end{tikzpicture}
\quad
\qquad i_{\Tr_\cC(A)}:=\Tr_\cC(i_A)\circ i=\,
\begin{tikzpicture}[baseline=.3cm, scale=.9]
	\topCylinder{(.7,0)}{.3}{1}
	\draw[thick] (.7,0) arc (-180:0:.3cm);		
	\halfDottedEllipse{(.7,0)}{.3}{.1}
	\filldraw[AString] (1, .3) circle (.025cm);
	\draw[thick, AString] (1,.91) -- (1,.3);
\end{tikzpicture}\,\,.
\]

Recall that an algebra object $A$ in a semisimple tensor category $\cM$ is called semisimple if its category of $A$-modules $\Mod_\cM(A)$ is semisimple.
The following is the main result of this subsection:

\begin{thm}
\label{thm:SemiSimple}
Let $\cC$ and $\cM$ be rigid semisimple tensor categories.
Let $\Phi:\cC\to \cM$ be a tensor functor, with right adjoint $\Tr_\cC$.
If $A$ is a semisimple algebra object in $\cM$, then $\Tr_\cC(A)$ is a semisimple algebra object in $\cC$.
\end{thm}

\begin{proof}
Let us write $B$ for $\Tr_\cC(A)$.  We prove the theorem by finding an equivalent, manifestly semisimple, description of the category $\Mod_\cC(B)$ of $B$-modules.

Let $\cN_0$ be the essential image of the functor $\cC\to \Mod_\cM(A): c\mapsto A\otimes \Phi(c)$, and let $\cN$ be the idempotent completion of $\cN_0$.
The category $\cN$ can equivalently be described as the full subcategory of $\Mod_\cM(A)$ generated by the simple objects that occur as summands of $A \otimes \Phi(c)$ for $c\in\cC$. 
\newcommand*{\longhookrightarrow}{\ensuremath{\lhook\joinrel\relbar\joinrel\relbar\joinrel\rightarrow}}
The trace functor induces a functor
\[
\Tr_\cC^{\Mod(A)}:\Mod_\cM(A) \to \Mod_\cC(B),
\] 
which in turn restricts to a functor
\[
\cN \longhookrightarrow \Mod_\cM(A) \xrightarrow{\Tr_\cC^{\Mod(A)}} \Mod_\cC(B).
\]
 We will prove that the latter is an equivalence of categories $\cN\cong \Mod_\cC(B)$.
 Since $\cN$ is semisimple, this will immediately imply the theorem.\medskip

$\bullet$ \emph{The functor $\Tr_\cC^{\Mod(A)}\!|_\cN$ is fully faithful.}
It suffices to show that the resitriction of $\Tr_\cC^{\Mod(A)}$ to $\cN_0$ is fully faithful.
By definition, any object of $\cN_0$ is of the form $A \otimes \Phi(c)$ for some $c \in \cC$.
We must prove that the map
\begin{equation*}
\Tr_\cC^{\Mod(A)} : \cN_0\big(A \otimes \Phi(c), A \otimes \Phi(d)\big) \to \cC_{B}\big(\Tr_\cC(A\otimes \Phi(c)), \Tr_\cC(A\otimes \Phi(d))\big)
\end{equation*}
given by
\begin{align}
\begin{tikzpicture}[baseline=.1cm, scale=.8]
	\plane{(-.4,-.8)}{2.8}{1.7}
	\draw[thick, AString] (.8,.2) -- (1.4,.2);
	\draw[thick, AString] (-1.4,.2) -- (-.6,.2);
	\CMbox{box}{(0,-.5)}{.8}{.8}{.8}{$f$}
	\draw[super thick, white] (-1.65,0) -- (-.2,0);
	\draw[thick, cString] (-1.8,0) -- (-.2,0);
	\draw[super thick, white] (1,0) -- (2.2,0);
	\draw[thick, dString] (.8,0) -- (2.2,0);
	\node[xshift=-5, yshift=8] at (-1.65,-.2) {\scriptsize{$c$}};
	\node at (2,.2) {\scriptsize{$d$}};
\end{tikzpicture}
&\,\longmapsto\,
\Tr_\cC
\left(
\begin{tikzpicture}[baseline=.1cm, scale=.8]
	\plane{(-.4,-.8)}{2.8}{1.7}
	\draw[thick, AString] (.8,.2) -- (1.4,.2);
	\draw[thick, AString] (-1.4,.2) -- (-.6,.2);
	\CMbox{box}{(0,-.5)}{.8}{.8}{.8}{$f$}
	\draw[super thick, white] (-1.65,0) -- (-.2,0);
	\draw[thick, cString] (-1.8,0) -- (-.2,0);
	\draw[super thick, white] (1,0) -- (2.2,0);
	\draw[thick, dString] (.8,0) -- (2.2,0);
	\node[xshift=-5, yshift=8] at (-1.65,-.2) {\scriptsize{$c$}};
	\node at (2,.2) {\scriptsize{$d$}};
\end{tikzpicture}
\,
\right)
\,=\,\,
\begin{tikzpicture}[baseline=.6cm, scale=1.1]
	\coordinate (a1) at (0,0);
	\coordinate (b1) at (0,1.4);
	\draw[thick] (a1) -- (b1);
	\draw[thick] ($ (a1) + (.6,0) $) -- ($ (b1) + (.6,0) $);
	\draw[thick] ($ (b1) + (.3,0) $) ellipse (.3 and .1);
	\halfDottedEllipse{(a1)}{.3}{.1}
	\draw[thick, AString] (.3,-.1) -- (.3,1.3);
	\CMboxSize{box}{(.4,.3)}{.4}{.44}{.2}{$f$}{.9}
	\draw[super thick, white] (.5,-.15) -- (.5,0);
	\draw[thick, cString] (.5,-.15) -- (.5,.3);
	\draw[super thick, white] (.5,.82) -- (.5,1.4);
	\draw[thick, dString] (.5,.82) -- (.5,1.4);
\end{tikzpicture}
\label{map:TrC}
\end{align}
is an isomorphism.

We claim that its inverse is provided by
\begin{equation}
\begin{split}
\begin{tikzpicture}[baseline=.4cm, scale=.8]
	\CMbox{box}{(0,0)}{1.4}{.8}{.4}{$g$}
	\upTubeWithString{(0.1,1.1)}{.8}{.5}{AString}
	\downTubeWithString{(.1,0)}{.8}{.4}{AString}
	\draw[super thick, white] (.8,1.55) -- (.8,.9);
	\draw[thick, dString] (.8,1.55) -- (.8,.9);
	\draw[super thick, white] (.8,-.8) -- (.8,-.2);
	\draw[thick, cString] (.8,-.7) -- (.8,0);
	\node at (1.05,1.4) {\scriptsize{$d$}};
	\node at (1,-.7) {\scriptsize{$c$}};  
\end{tikzpicture}
\,\,\,\,\,\longmapsto\,\,\,&
\begin{tikzpicture}[baseline=.4cm, scale=.8]
	\plane{(-2.6,-.8)}{6}{2.5}
	\CMbox{box}{(0,-.5)}{.8}{1}{.8}{$\varepsilon$}
	\coordinate (a) at (1.6,.6);
	\filldraw[AString] (a) circle (.05cm);
	\draw[thick, AString] (-4.9,1.5) -- (-.2,1.5) .. controls ++(0:.6cm) and ++(180:.4cm) .. (a);
	\draw[thick, AString] (.8,.2) -- (1,.2) .. controls ++(0:.4cm) and ++(180:.2cm) .. (a);
	\draw[thick, AString] (a) -- (2,.6);
	\node at (3,.2) {\scriptsize{$d$}};
	\straightTubeWithString{(-.4,-.1)}{.2}{1.2}{AString}
	\draw[super thick, white] (-1.6,0) -- (-.2,0);
	\draw[thick, dString] (-1.8,0) -- (-.2,0);
	\draw[super thick, white] (1,0) -- (3.4,0);
	\draw[thick, dString] (.8,0) -- (3.4,0);
	\CMbox{box}{(-2,-.5)}{.8}{1}{.8}{$g$}
	\straightTubeWithCap{(-2.4,-.1)}{.2}{.5}
	\draw[thick, AString] (-2.8,.3) -- (-2.2,.3);
	\filldraw[AString] (-2.8,.3) circle (.05cm);
%
\def\sh{.2}
\draw[super thick, white] (-4.3+.1-\sh,.3-.1+\sh) -- (-3.4-\sh,.3-.1+\sh) .. controls ++(0:.2cm) and ++(180:.2cm) .. (-2.9,.1-.1) -- (-2.37+.2,.1-.1) ;
\draw[thick, cString] (-4.3+.1-\sh,.3-.1+\sh) -- (-3.4-\sh,.3-.1+\sh) .. controls ++(0:.2cm) and ++(180:.2cm) .. (-2.9,.1-.1) -- (-2.37+.2,.1-.1) ;
	\node[xshift=-6, yshift=8] at (-3.8-.1-\sh,\sh) {\scriptsize{$c$}};
\end{tikzpicture}
\label{map:TrCInverse}
\\\\
g\mapsto (m_A\otimes \id_{\Phi(d)})\circ &(\id_A \otimes [\varepsilon_{A\otimes \Phi(d)} \circ \Phi(g \circ \Tr_\cC(i_A\otimes \id_{\Phi(c)}) \circ \eta_c)])
\end{split}
\end{equation}
We show that \eqref{map:TrC} and \eqref{map:TrCInverse} compose both ways to the identity.
One direction goes as follows:
\begin{align*}
\bigg(\eqref{map:TrCInverse}\circ \eqref{map:TrC}\bigg)(f)
&\,=\!\!\!\!
\begin{tikzpicture}[baseline=.1cm, scale=.8]
	\plane{(-2.4,-.8)}{5.8}{2.5}
	\CMbox{box}{(0,-.5)}{.8}{1}{.8}{$\varepsilon$}
	\coordinate (a) at (1.6,.6);
	\filldraw[AString] (a) circle (.05cm);
	\draw[thick, AString] (-4.7,1.5) -- (-.2,1.5) .. controls ++(0:.6cm) and ++(180:.4cm) .. (a);
	\draw[thick, AString] (.8,.2) -- (1,.2) .. controls ++(0:.4cm) and ++(180:.2cm) .. (a);
	\draw[thick, AString] (a) -- (2,.6);
	\straightTubeWithCap{(-.4,-.1)}{.2}{2.2}
	\draw[super thick, white] (-1.6,0) -- (-.15,0);
	\draw[thick, dString] (-1.8,0) -- (-.15,0);
	\draw[super thick, white] (1,0) -- (3.4,0);
	\draw[thick, dString] (.8,0) -- (3.4,0);
	\node at (3,.2) {\scriptsize{$d$}};
	\draw[thick, AString] (-2.8,.3) -- (-.2,.3);
	\filldraw[AString] (-2.8,.3) circle (.05cm);
	\CMbox{box}{(-2,-.3)}{.6}{.6}{.3}{$f$}
	\draw[super thick, white] (-3.6,0) -- (-2.1,0);
	\draw[thick, cString] (-3.8,0) -- (-2.1,0);
	\node[xshift=-8, yshift=8] at (-3.4,-.2) {\scriptsize{$c$}};
\useasboundingbox ($(current bounding box.north east)+(0,.2)$) rectangle ($(current bounding box.south west)+(0,-.6)$);
\end{tikzpicture}
=\!\!\!\!
\begin{tikzpicture}[baseline=.1cm, scale=.8]
	\plane{(-2.4,-.8)}{5.8}{2.4}
	\CMbox{box}{(.1,-.5)}{.8}{.8}{.8}{$f$}
	\coordinate (a) at (1.6,.6);
	\filldraw[AString] (a) circle (.05cm);
	\draw[thick, AString] (-4.6,1.4) -- (-.2,1.4) .. controls ++(0:.6cm) and ++(180:.4cm) .. (a);
	\draw[thick, AString] (.9,.2) -- (1,.2) .. controls ++(0:.4cm) and ++(180:.2cm) .. (a);
	\draw[thick, AString] (a) -- (2,.6);
	\draw[super thick, white] (-1.6,0) -- (-.2,0);
	\draw[thick, cString] (-1.8,0) -- (-.1,0);
	\draw[super thick, white] (1,0) -- (3.4,0);
	\draw[thick, dString] (.9,0) -- (3.4,0);
%
	\node at (3,.2) {\scriptsize{$d$}};
	\CMbox{box}{(-2,-.5)}{.8}{.8}{.4}{$\varepsilon$}
	\straightTubeWithCap{(-2.2,-.2)}{.1}{.5}
	\draw[super thick, white] (-4,.2) -- (-3.1,.2) .. controls ++(0:.2cm) and ++(180:.2cm) .. (-2.6,0) -- (-2.07,0) ;
	\draw[thick, cString] (-4,.2) -- (-3.1,.2) .. controls ++(0:.2cm) and ++(180:.2cm) .. (-2.6,0) -- (-2.07,0) ;
	\node[xshift=-8, yshift=8] at (-3.6,0) {\scriptsize{$c$}};
	\draw[thick, AString] (-1,.2) -- (-.6,.2);
	\filldraw[AString] (-1,.2) circle (.05cm);
\end{tikzpicture}
\\
&=
\begin{tikzpicture}[baseline=.1cm, scale=.8]
	\plane{(-1,-.8)}{4.4}{2.4}
	\CMbox{box}{(.1,-.5)}{.8}{.8}{.8}{$f$}
	\coordinate (a) at (1.6,.6);
	\filldraw[AString] (a) circle (.05cm);
	\draw[thick, AString] (-3.2,1.4) -- (-.2,1.4) .. controls ++(0:.6cm) and ++(180:.4cm) .. (a);
	\draw[thick, AString] (.9,.2) -- (1,.2) .. controls ++(0:.4cm) and ++(180:.2cm) .. (a);
	\draw[thick, AString] (a) -- (2,.6);
	\draw[super thick, white] (-2.4,0) -- (-.2,0);
	\draw[thick, cString] (-2.4,0) -- (-.1,0);
	\draw[super thick, white] (1,0) -- (3.4,0);
	\draw[thick, dString] (.9,0) -- (3.4,0);
	\node[xshift=-8, yshift=8] at (-2,-.2) {\scriptsize{$c$}};
	\node at (3,.2) {\scriptsize{$d$}};
	\draw[thick, AString] (-1.4,.2) -- (-.6,.2);
	\filldraw[AString] (-1.4,.2) circle (.05cm);
\end{tikzpicture}
=
\begin{tikzpicture}[baseline=.1cm, scale=.8]
	\plane{(-1,-.8)}{3.4}{2.2}
	\CMbox{box}{(.1,-.5)}{.8}{.8}{.8}{$f$}
	\coordinate (a) at (-1,.2);
	\filldraw[AString] (a) circle (.05cm);
	\draw[thick, AString] (-2.8,1) -- (-2.4,1) .. controls ++(0:.6cm) and ++(180:.4cm) .. (a);
	\draw[thick, AString] (.9,.2) -- (1.4,.2);
	\draw[super thick, white] (-2.4,0) -- (-.2,0);
	\draw[thick, cString] (-2.4,0) -- (-.1,0);
	\draw[super thick, white] (1,0) -- (2.4,0);
	\draw[thick, dString] (.9,0) -- (2.4,0);
	\node[xshift=-8, yshift=8] at (-2,-.2) {\scriptsize{$c$}};
	\node at (2,.2) {\scriptsize{$d$}};
	\draw[thick, AString] (-1.6,.2) -- (-.6,.2);
	\filldraw[AString] (-1.6,.2) circle (.05cm);
\end{tikzpicture}
\,=\,\,
f.
\end{align*}
The second equality holds by the naturality of $\varepsilon$, the third equality is Equation \eqref{eq:MAdjointRelation}, and the fourth one follows from $f$ being an $A$-module map.

For the other direction, it is convenient to precompose both sides of the equation by the isomorphism
\begin{equation}\label{eq: recall some lemma}
\begin{tikzpicture}[baseline=.5cm, scale=.9]
	\coordinate (a1) at (0,-.5);
	\coordinate (b1) at (0,.4);
	\coordinate (a2) at (1.4,0);
	\coordinate (b2) at (1.4,.4);
	\draw[thick] (a1) -- (b1);
	\draw[thick] ($ (a1) + (.6,0) $) -- ($ (b1) + (.6,0) $);
	\halfDottedEllipse{(a1)}{.3}{.1}
	\draw[thick] (a2) -- (b2);
	\draw[thick] ($ (a2) + (.6,0) $) -- ($ (b2) + (.6,0) $);
	\topPairOfPants{(b1)}{}
	\draw[thick] (a2) arc (-180:0:.3cm);
	\draw[thick, AString] ($ (a1) + (.3,-.1) $) -- ($ (b1) + (.3,-.1) $) .. controls ++(90:.8cm) and ++(270:.8cm) .. ($ (b1) + (.95,1.41) $);
	\draw[super thick, white] ($ (a2) + (0,-.6) $) --($ (a2) + (0,-.5) $) .. controls ++(90:.2cm) and ++(270:.2cm) .. ($ (a2) + (.3,0) $) -- ($ (b2) + (.3,-.05) $) .. controls ++(90:.8cm) and ++(270:.8cm) .. ($ (b1) + (1.1,1.45) $);
	\draw[thick, cString] ($ (a2) + (0,-.6) $) --($ (a2) + (0,-.5) $) .. controls ++(90:.2cm) and ++(270:.2cm) .. ($ (a2) + (.3,0) $) -- ($ (b2) + (.3,-.05) $) .. controls ++(90:.8cm) and ++(270:.8cm) .. ($ (b1) + (1.1,1.45) $);
\end{tikzpicture}
\,\,:\, B \otimes c\to \Tr_\cC(A\otimes \Phi(c))
\end{equation}
provided by Lemma \ref{lem:TrCSplitting}:
\begin{align*}
\bigg(\eqref{map:TrC}\circ \eqref{map:TrCInverse}\bigg)(g)
\circ \eqref{eq: recall some lemma}
\,\,\,&=
\begin{tikzpicture}[baseline=-.1cm]
	\coordinate (a1) at (-.7,-2.4);
	\coordinate (a2) at (.7,-1.9);
	\coordinate (b1) at (-.7,-1.5);
	\coordinate (b2) at (.7,-1.5);
	\coordinate (c1) at (0,0);
	\coordinate (d1) at ($ (c1) + (0,2.5) $);
	\coordinate (z) at ($ (c1) + (.4,.4) $);
	\coordinate (y) at ($ (d1) + (.2,-.25) $);
	\draw[thick] (c1) -- (d1);
	\draw[thick] ($ (c1) + (.6,0) $) -- ($ (d1) + (.6,0) $);
	\draw[thick] ($ (d1) + (.3,0) $) ellipse (.3 and .1);
	\halfDottedEllipse{(c1)}{.3}{.1}
	\draw[thick] (a1) -- (b1);
	\draw[thick] ($ (a1) + (.6,0) $) -- ($ (b1) + (.6,0) $);
	\halfDottedEllipse{(a1)}{.3}{.1}
	\draw[thick] (a2) -- (b2);
	\draw[thick] ($ (a2) + (.6,0) $) -- ($ (b2) + (.6,0) $);
	\draw[thick] (a2) arc (-180:0:.3cm);
	\draw[thick,unshaded] (z) -- ($ (z) + (0,-.15) $) arc (-180:0:.13cm) -- ($ (z) + (.26,0) $);
	\CMbox{box}{(z)}{.44}{.44}{.2}{$g$}
	\downTube{($ (z) + (0,1) $)}{.26}{.45}
	\CMbox{box}{($ (z) + (0,1) $)}{.44}{.44}{.2}{$\varepsilon$}
	\pairOfPants{(b1)}{.3}{.1}
	\draw[thick, AString] ($ (z) + (.09,1) $) -- ($ (z) + (.09, .5) $);
	\draw[thick, AString] ($ (z) + (.09,0) $) -- ($ (z) + (.09, -.12) $);
	\filldraw[AString] ($ (z) + (.09,-.12) $) circle (.025cm);
	\draw[super thick, white] ($ (z) + (.17, -.15) $) .. controls ++(270:.2cm) and ++(90:.2cm) .. ($ (b1) + (1.15,1.35) $);
	\draw[thick, cString] ($ (z) + (.17,0) $) -- ($ (z) + (.17, -.15) $) .. controls ++(270:.2cm) and ++(90:.2cm) .. ($ (b1) + (1.15,1.35) $);
	\draw[super thick, white] ($ (z) + (.17,.47) $) -- ($ (z) + (.17,.95) $);
	\draw[thick, dString] ($ (z) + (.17,.48) $) -- ($ (z) + (.17,1) $);
	\draw[super thick, white] ($ (d1) + (.5,0) $) -- ($ (d1) + (.5,-.55) $);
	\draw[thick, dString] ($ (d1) + (.5,0) $) -- ($ (d1) + (.5,-.55) $);
	\filldraw[AString] (y) circle (.05cm);
	\draw[thick, AString] ($ (z) + (-.05,1.62) $) .. controls ++(90:.1cm) and ++(270:.1cm) .. (y) -- ($ (d1) + (.2,-.08) $);
	\draw[thick, AString] ($ (a1) + (.3,-.1) $) -- ($ (b1) + (.3,-.1) $) .. controls ++(90:.8cm) and ++(270:.8cm) .. ($ (c1) + (.12,0) $) -- ($ (c1) + (.12,2) $) .. controls ++(90:.1cm) and ++(270:.1cm) .. (y);
	\draw[super thick, white] ($ (a2) + (0,-.6) $) --($ (a2) + (0,-.5) $) .. controls ++(90:.4cm) and ++(270:.4cm) .. ($ (a2) + (.3,.2) $) -- ($ (b2) + (.3,-.05) $) .. controls ++(90:.8cm) and ++(270:.8cm) .. ($ (b1) + (1.15,1.35) $);
	\draw[thick, cString] ($ (a2) + (0,-.6) $) --($ (a2) + (0,-.5) $) .. controls ++(90:.4cm) and ++(270:.4cm) .. ($ (a2) + (.3,.2) $) -- ($ (b2) + (.3,-.05) $) .. controls ++(90:.8cm) and ++(270:.8cm) .. ($ (b1) + (1.15,1.35) $);
\end{tikzpicture}
=\,\,
\begin{tikzpicture}[baseline=-.1cm]
	\coordinate (a1) at (-.7,-2.4);
	\coordinate (a2) at (.7,-1.9);
	\coordinate (b1) at (-.7,1);
	\coordinate (b2) at (.7,-1.5);
	\coordinate (c1) at (.7,-1.5);
	\coordinate (d1) at ($ (c1) + (0,2.5) $);
	\coordinate (z) at ($ (c1) + (.4,.4) $);
	\coordinate (y) at ($ (b1) + (.9,1) $);
	\draw[thick] (c1) -- (d1);
	\draw[thick] ($ (c1) + (.6,0) $) -- ($ (d1) + (.6,0) $);
	\halfDottedEllipse{(c1)}{.3}{.1}
	\halfDottedEllipse{(d1)}{.3}{.1}
	\draw[thick] (a1) -- (b1);
	\draw[thick] ($ (a1) + (.6,0) $) -- ($ (b1) + (.6,0) $);
	\halfDottedEllipse{(a1)}{.3}{.1}
	\draw[thick] (a2) -- (b2);
	\draw[thick] ($ (a2) + (.6,0) $) -- ($ (b2) + (.6,0) $);
	\draw[thick] (a2) arc (-180:0:.3cm);
	\draw[thick,unshaded] (z) -- ($ (z) + (0,-.15) $) arc (-180:0:.13cm) -- ($ (z) + (.26,0) $);
	\CMbox{box}{(z)}{.44}{.44}{.2}{$g$}
	\downTube{($ (z) + (0,1) $)}{.26}{.45}
	\CMbox{box}{($ (z) + (0,1) $)}{.44}{.44}{.2}{$\varepsilon$}
	\topPairOfPants{(b1)}{.3}{.1}
	\draw[thick, AString] ($ (z) + (.09,1) $) -- ($ (z) + (.09, .5) $);
	\draw[thick, AString] ($ (z) + (.09,0) $) -- ($ (z) + (.09, -.12) $);
	\filldraw[AString] ($ (z) + (.09,-.12) $) circle (.025cm);
	\draw[super thick, white] ($ (z) + (.17,.47) $) -- ($ (z) + (.17,.95) $);
	\draw[thick, dString] ($ (z) + (.17,.48) $) -- ($ (z) + (.17,1) $);
	\draw[super thick, white] ($ (d1) + (.5,-.55) $)  -- ($ (d1) + (.5,0) $) .. controls ++(90:.5cm) and ++(270:.6cm) .. ($ (b1) + (1.1,1.5) $) ; 
	\draw[thick, dString] ($ (d1) + (.5,-.55) $)  -- ($ (d1) + (.5,0) $) .. controls ++(90:.5cm) and ++(270:.6cm) .. ($ (b1) + (1.1,1.5) $) ; 
	\filldraw[AString] (y) circle (.05cm);
	\draw[thick, AString] ($ (z) + (-.05,1.62) $) -- ($ (b1) + (1.75,-.1) $) .. controls ++(90:.5cm) and ++(-45:.6cm) .. (y);
	\draw[thick, AString] ($ (a1) + (.3,-.1) $) -- ($ (b1) + (.3,-.1) $) .. controls ++(90:.6cm) and ++(225:.6cm) .. (y) -- ($ (b1) + (.9,1.4) $);
	\draw[super thick, white] ($ (z) + (.17,-.05) $) -- ($ (z) + (.17, -.15) $) .. controls ++(270:.2cm) and ++(90:.4cm) .. ($ (a2) + (.3,.1) $) .. controls ++(270:.4cm) and ++(90:.4cm) .. ($ (a2) + (0,-.5) $) --($ (a2) + (0,-.6) $);
	\draw[thick, cString] ($ (z) + (.17,0) $) -- ($ (z) + (.17, -.15) $) .. controls ++(270:.2cm) and ++(90:.4cm) .. ($ (a2) + (.3,.1) $) .. controls ++(270:.4cm) and ++(90:.4cm) .. ($ (a2) + (0,-.5) $) --($ (a2) + (0,-.6) $);
\end{tikzpicture}
\,\,=\,\,
\begin{tikzpicture}[baseline=-.1cm]
	\coordinate (a1) at (-.7,-2.4);
	\coordinate (a2) at (.7,-2.4);
	\coordinate (b2) at (.7,-.3);
	\coordinate (b1) at (-.7,1);
	\coordinate (c1) at (.7,-2.5);
	\coordinate (d1) at ($ (c1) + (0,3.5) $);
	\coordinate (z) at ($ (c1) + (.4,.8) $);
	\coordinate (y) at ($ (b1) + (.9,1) $);
	\draw[thick] (b2) -- (d1);
	\draw[thick] ($ (b2) + (.6,0) $) -- ($ (d1) + (.6,0) $);
	\halfDottedEllipse{(d1)}{.3}{.1}
	\draw[thick] (a1) -- (b1);
	\draw[thick] ($ (a1) + (.6,0) $) -- ($ (b1) + (.6,0) $);
	\halfDottedEllipse{(a1)}{.3}{.1}
	\draw[thick] (b2) arc (-180:0:.3cm);
	\draw[thick,unshaded] (z) -- ($ (z) + (0,-.15) $) arc (-180:0:.13cm) -- ($ (z) + (.26,0) $);
	\CMbox{box}{(z)}{.44}{.44}{.2}{$g$}
	\downTube{($ (z) + (0,1.6) $)}{.26}{1.05}
	\CMbox{box}{($ (z) + (0,1.6) $)}{.44}{.44}{.2}{$\varepsilon$}
	\topPairOfPants{(b1)}{.3}{.1}
	\draw[thick, dotted] ($ (b2) + (-.2,-.4) $) rectangle ($ (b2) + (.9,0) $);
	\node at ($ (b2) + (1.85,-.2) $) {\scriptsize{$\eta_{\Tr_\cC(A\otimes \Phi(d))}$}};
	\draw[thick, AString] ($ (z) + (.09,1.6) $) -- ($ (z) + (.09, .5) $);
	\draw[thick, AString] ($ (z) + (.09,0) $) -- ($ (z) + (.09, -.12) $);
	\filldraw[AString] ($ (z) + (.09,-.12) $) circle (.025cm);
	\draw[super thick, white] ($ (z) + (.17,.47) $) -- ($ (z) + (.17,.95) $);
	\draw[thick, dString] ($ (z) + (.17,.48) $) -- ($ (z) + (.17,1.6) $);
	\draw[super thick, white] ($ (d1) + (.5,-.55) $)  -- ($ (d1) + (.5,0) $) .. controls ++(90:.5cm) and ++(270:.6cm) .. ($ (b1) + (1.1,1.5) $) ; 
	\draw[thick, dString] ($ (d1) + (.5,-.55) $)  -- ($ (d1) + (.5,0) $) .. controls ++(90:.5cm) and ++(270:.6cm) .. ($ (b1) + (1.1,1.5) $) ; 
	\filldraw[AString] (y) circle (.05cm);
	\draw[thick, AString] ($ (z) + (-.05,2.22) $) -- ($ (b1) + (1.75,-.1) $) .. controls ++(90:.5cm) and ++(-45:.6cm) .. (y);
	\draw[thick, AString] ($ (a1) + (.3,-.1) $) -- ($ (b1) + (.3,-.1) $) .. controls ++(90:.6cm) and ++(225:.6cm) .. (y) -- ($ (b1) + (.9,1.4) $);
	\draw[super thick, white] ($ (z) + (.17,-.05) $) -- ($ (z) + (.17, -.15) $) .. controls ++(270:.4cm) and ++(90:.4cm) .. ($ (a2) + (.4,-.1) $);
	\draw[thick, cString] ($ (z) + (.17,0) $) -- ($ (z) + (.17, -.15) $) .. controls ++(270:.4cm) and ++(90:.4cm) .. ($ (a2) + (.4,-.1) $);
\end{tikzpicture}
\\&=\,\,\,\,
\begin{tikzpicture}[baseline=-.1cm]
	\coordinate (a1) at (-.7,-2);
	\coordinate (a2) at (.7,-2);
	\coordinate (b2) at (.7,-1);
	\coordinate (c1) at (-.7,0);
	\coordinate (c2) at (.7,0);
	\coordinate (y) at ($ (c1) + (.9,1) $);
	\draw[thick] (a1) -- (c1);
	\draw[thick] ($ (a1) + (.6,0) $) -- ($ (c1) + (.6,0) $);
	\halfDottedEllipse{(a1)}{.3}{.1}
	\draw[thick] (b2)-- ($ (b2) + (0,-.4) $) arc (-180:0:.3cm) -- ($ (b2) + (.6,0) $);
	\CMbox{box}{(b2)}{1}{.8}{.4}{$g$}
	\fill[unshaded] ($ (b2) + (-.1,1.1) $) rectangle ($ (b2) + (.65,1.3) $);
	\topPairOfPants{(c1)}{.3}{.1}
	\draw[thick, AString] ($ (b2) + (.2,0) $) -- ($ (b2) + (.2, -.3) $);
	\filldraw[AString] ($ (b2) + (.2,-.3) $) circle (.05cm);
	\draw[super thick, white] ($ (b2) + (.4,-.1) $) -- ($ (b2) + (.4, -.3) $) .. controls ++(270:.4cm) and ++(90:.4cm) .. ($ (a2) + (0,-.1) $);
	\draw[thick, cString] ($ (b2) + (.4,0) $) -- ($ (b2) + (.4, -.3) $) .. controls ++(270:.4cm) and ++(90:.4cm) .. ($ (a2) + (0,-.1) $);
	\draw[super thick, white] ($ (b2) + (.4,.85) $)  -- ($ (b2) + (.4,.9) $) .. controls ++(90:.6cm) and ++(270:.6cm) .. ($ (c1) + (1.1,1.5) $) ; 
	\draw[thick, dString] ($ (b2) + (.4,.85) $)  -- ($ (b2) + (.4,.9) $) .. controls ++(90:.6cm) and ++(270:.6cm) .. ($ (c1) + (1.1,1.5) $) ; 
	\filldraw[AString] (y) circle (.05cm);
	\draw[thick, AString] ($ (b2) + (.2,.9) $) -- ($ (b2) + (.2,1) $)  .. controls ++(90:.5cm) and ++(-45:.6cm) .. (y);
	\draw[thick, AString] ($ (a1) + (.3,-.1) $) -- ($ (c1) + (.3,-.1) $) .. controls ++(90:.6cm) and ++(225:.6cm) .. (y) -- ($ (c1) + (.9,1.4) $);
\end{tikzpicture}
\,\,=
\begin{tikzpicture}[baseline=.1cm]
	\coordinate (a1) at (-.7,-2);
	\coordinate (a2) at (.7,-2);
	\coordinate (b1) at (-.7,-1);
	\coordinate (b2) at (.7,-1);
	\coordinate (c1) at (0,.3);
	\coordinate (y) at ($ (b1) + (.9,1) $);
	\draw[thick] (a1) -- (b1);
	\draw[thick] ($ (a1) + (.6,0) $) -- ($ (b1) + (.6,0) $);
	\halfDottedEllipse{(a1)}{.3}{.1}
	\draw[thick] (b2)-- ($ (b2) + (0,-.4) $) arc (-180:0:.3cm) -- ($ (b2) + (.6,0) $);
	\topPairOfPants{(b1)}{.3}{.1}
	\CMbox{box}{($ (c1) $)}{1}{.8}{.4}{$g$}
	\upTube{($ (c1) + (0,1.05) $)}{.6}{.5}
	\draw[thick, AString] ($ (b2) + (.2,0) $) -- ($ (b2) + (.2, -.3) $);
	\filldraw[AString] ($ (b2) + (.2,-.3) $) circle (.05cm);
	\draw[super thick, white] ($ (b2) + (.4,-.05) $) -- ($ (b2) + (.4, -.3) $) .. controls ++(270:.4cm) and ++(90:.4cm) .. ($ (a2) + (0,-.1) $);
	\draw[thick, cString] ($ (c1) + (.4,0) $) .. controls ++(270:.6cm) and ++(90:.6cm) .. ($ (b2) + (.4,0) $) -- ($ (b2) + (.4, -.3) $) .. controls ++(270:.4cm) and ++(90:.4cm) .. ($ (a2) + (0,-.1) $);
	\draw[super thick, white] ($ (c1) + (.4,.85) $) -- ($ (c1) + (.4,1.5) $);
	\draw[thick, dString] ($ (c1) + (.4,.85) $) -- ($ (c1) + (.4,1.5) $);
	\draw[thick, AString] ($ (c1) + (.2,.9) $) -- ($ (c1) + (.2,1.4) $);
	\filldraw[AString] (y) circle (.05cm);
	\draw[thick, AString] ($ (a2) + (.2,.9) $) -- ($ (a2) + (.2,1) $)  .. controls ++(90:.5cm) and ++(-45:.6cm) .. (y);
	\draw[thick, AString] ($ (a1) + (.3,-.1) $) -- ($ (b1) + (.3,-.1) $) .. controls ++(90:.6cm) and ++(225:.6cm) .. (y) -- ($ (c1) + (.2,0) $);
\end{tikzpicture}
=
\begin{tikzpicture}[baseline=.1cm]
	\coordinate (a1) at (-.7,-2);
	\coordinate (a2) at (.7,-2);
	\coordinate (b1) at (-.7,-1);
	\coordinate (b2) at (.7,-1);
	\coordinate (c1) at (0,.3);
	\coordinate (y) at ($ (b1) + (.9,1) $);
	\draw[thick] (a1) -- (b1);
	\draw[thick] ($ (a1) + (.6,0) $) -- ($ (b1) + (.6,0) $);
	\halfDottedEllipse{(a1)}{.3}{.1}
	\draw[thick] (b2)-- ($ (b2) + (0,-.4) $) arc (-180:0:.3cm) -- ($ (b2) + (.6,0) $);
	\topPairOfPants{(b1)}{.3}{.1}
	\CMbox{box}{($ (c1) $)}{1}{.8}{.4}{$g$}
	\upTube{($ (c1) + (0,1.05) $)}{.6}{.5}
	\draw[super thick, white] ($ (c1) + (.4,.85) $) -- ($ (c1) + (.4,1.5) $);
	\draw[thick, dString] ($ (c1) + (.4,.85) $) -- ($ (c1) + (.4,1.5) $);
	\draw[thick, AString] ($ (c1) + (.2,.9) $) -- ($ (c1) + (.2,1.4) $);
	\draw[thick, AString] ($ (a1) + (.3,-.1) $) -- ($ (b1) + (.3,-.1) $) .. controls ++(90:.6cm) and ++(270:.6cm) .. ($ (c1) + (.2,0) $);
	\draw[super thick, white] ($ (c1) + (.4,-.05) $) .. controls ++(270:.6cm) and ++(90:.6cm) .. ($ (b2) + (.3,0) $) -- ($ (b2) + (.3, -.3) $) .. controls ++(270:.4cm) and ++(90:.4cm) .. ($ (a2) + (0,-.1) $);
	\draw[thick, cString] ($ (c1) + (.4,0) $) .. controls ++(270:.6cm) and ++(90:.6cm) .. ($ (b2) + (.3,0) $) -- ($ (b2) + (.3, -.3) $) .. controls ++(270:.4cm) and ++(90:.4cm) .. ($ (a2) + (0,-.1) $);
\end{tikzpicture}
\,=\,\,g\circ \eqref{eq: recall some lemma}.
\end{align*}
Here, the second and third equations follows by the naturality of $\mu$ and $\eta$.
The fourth equality is the content of Equation \eqref{eq:CAdjointRelation}, and the fifth one holds because $g$ is a $B$-module map.


$\bullet$ \emph{The functor $\Tr_\cC^{\Mod(A)}\!|_\cN$ is essentially surjective.}
Every $B$-module $c\in \cC$ fits in a coequalizer diagram $B\otimes B \otimes c \rightrightarrows B\otimes c \to c$:
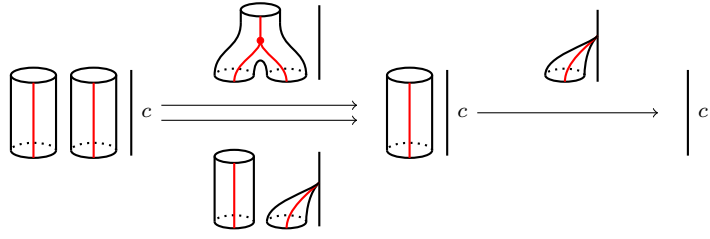
$$
\begin{tikzpicture}[baseline=.4cm]
	\coordinate (a1) at (0,0);
	\coordinate (a2) at ($ (a1) + (.8,0) $);	
	\coordinate (b1) at ($ (a2) + (.8,0) $);	
	\coordinate (a3) at ($ (a1) + (5,0) $);
	\coordinate (b2) at ($ (a3) + (.8,0) $);
	\coordinate (b3) at ($ (a3) + (4,0) $);
	\coordinate (m1) at ($ (b1) + (2.5,-.95) $);
	\coordinate (m2) at ($ (b1) + (2.5,1) $);
	\coordinate (m3) at ($ (b2) + (2,1) $);
	\topCylinder{(a1)}{1}
	\halfDottedEllipse{(a1)}{.3}{.1}
	\topCylinder{(a2)}{1}
	\halfDottedEllipse{(a2)}{.3}{.1}
	\topCylinder{(a3)}{1}
	\halfDottedEllipse{(a3)}{.3}{.1}

	\draw[thick, AString] ($ (a1) + (.3,-.1) $) -- ($ (a1) + (.3,.9) $);
	\draw[thick, AString] ($ (a2) + (.3,-.1) $) -- ($ (a2) + (.3,.9) $);
	\draw[thick] ($ (b1) + (0,-.07) $) -- ($ (b1) + (0,1.07) $);
	\draw[thick, AString] ($ (a3) + (.3,-.1) $) -- ($ (a3) + (.3,.9) $);
	\draw[thick] ($ (b2) + (0,-.07) $) -- ($ (b2) + (0,1.07) $);
	\draw[thick] ($ (b3) + (0,-.07) $) -- ($ (b3) + (0,1.07) $);
	\node at ($ (b1) + (.2,.5) $) {\scriptsize{$c$}};
	\node at ($ (b2) + (.2,.5) $) {\scriptsize{$c$}};
	\node at ($ (b3) + (.2,.5) $) {\scriptsize{$c$}};
	\draw[end>] ($ (b1) + (.4,.4) $) --  ($ (a3) + (-.4,.4) $);
	\draw[end>] ($ (b1) + (.4,.6) $) --  ($ (a3) + (-.4,.6) $);
	\draw[end>] ($ (b2) + (.4,.5) $) --  ($ (b3) + (-.4,.5) $);
\pgftransformscale{.87}
	\draw[thick] ($ (m1) + (0,-.07) $) -- ($ (m1) + (0,1.07) $);
	\halfDottedEllipse{($ (m1) + (-.8,0) $)}{.3}{.1}
	\halfDottedEllipse{($ (m1) + (-1.6,0) $)}{.3}{.1}
	\draw[thick] ($ (m1) + (-.2,0) $) .. controls ++(90:.2cm) and ++(225:.2cm) .. ($ (m1) + (0,.6) $);
	\draw[thick, AString] ($ (m1) + (-.5,-.1) $) .. controls ++(90:.3cm) and ++(225:.2cm) .. ($ (m1) + (0,.6) $);
	\draw[thick] ($ (m1) + (-.8,0) $) .. controls ++(90:.3cm) and ++(225:.2cm) .. ($ (m1) + (0,.6) $);
\draw[thick] ($ (m1) + (-1,0) $) -- ++ (0,1);
\draw[thick] ($ (m1) + (-1.6,0) $) -- ++ (0,1);
\draw[thick] ($ (m1) + (-1.3,1) $) circle (.3 and .1) ;
\draw[thick, AString] ($ (m1) + (-1.3,-.1) $) -- ++ (0,1);
%
	\draw[thick] ($ (m2) + (0,-.07) $) -- ($ (m2) + (0,1.27-.2) $);
	\halfDottedEllipse{($ (m2) + (-.8,0) $)}{.3}{.1}
	\halfDottedEllipse{($ (m2) + (-1.6,0) $)}{.3}{.1}
	\draw[thick] ($ (m2) + (-.2,0) $) .. controls ++(90:.4cm) and ++(270:.4cm) .. ($ (m2) + (-.6,.8) $);
	\draw[thick] ($ (m2) + (-1.6,0) $) .. controls ++(90:.4cm) and ++(270:.4cm) .. ($ (m2) + (-1.2,.8) $);
	\draw[thick] ($ (m2) + (-1,0) $) .. controls ++(90:.3cm) and ++(90:.3cm) .. ($ (m2) + (-.8,0) $);
\draw[thick] ($ (m2) + (-.6,.8) $) -- ++ (0,.2);
\draw[thick] ($ (m2) + (-1.2,.8) $) -- ++ (0,.2);
\draw[thick] ($ (m2) + (-.9,1) $) circle (.3 and .1) ;
\draw[thick, AString] ($ (m2) + (-.9,.53) $) -- ++ (0,.37);
	\draw[thick, AString] ($ (m2) + (-1.3,-.1) $) .. controls ++(90:.3cm) and ++(270:.2cm) .. ($ (m2) + (-.9,.53) $);
	\draw[thick, AString] ($ (m2) + (-.5,-.1) $) .. controls ++(90:.3cm) and ++(270:.2cm) .. ($ (m2) + (-.9,.53) $);
	\filldraw[AString] ($ (m2) + (-.9,.53) $) circle (.05cm);
	\draw[thick] ($ (m3) + (0,-.1) $) -- ($ (m3) + (0,1) $);
	\halfDottedEllipse{($ (m3) + (-.8,0) $)}{.3}{.1}
	\draw[thick] ($ (m3) + (-.2,0) $) .. controls ++(90:.2cm) and ++(225:.2cm) .. ($ (m3) + (0,.6) $);
	\draw[thick, AString] ($ (m3) + (-.5,-.1) $) .. controls ++(90:.3cm) and ++(225:.2cm) .. ($ (m3) + (0,.6) $);
	\draw[thick] ($ (m3) + (-.8,0) $) .. controls ++(90:.3cm) and ++(225:.2cm) .. ($ (m3) + (0,.6) $);
\end{tikzpicture}
$$
Using the isomorphism \eqref{eq: recall some lemma}, we see that
$B\otimes B\otimes c$ and $B\otimes c$ are in the essential image of $\cN$.
We are now finished by Lemma \ref{lem:EssentiallySurjective} below. 
\end{proof}

\begin{lem}
\label{lem:EssentiallySurjective}
Suppose $\cN_1,\cN_2$ are linear categories with $\cN_1$ semisimple.
Let  $T:\cN_1 \to \cN_2$ be a fully faithful linear functor.
Assume that every $b\in \cN_2$ sits in a coequalizer diagram $T(a_1)\rightrightarrows T(a_2)\to b$.
Then $T$ is essentially surjective (and thus an equivalence of categories).
\end{lem}
\begin{proof}
We may replace the coequalizer diagram by an exact sequence 
$$
T(a_1) \xrightarrow{\,f\,} T(a_2) \xrightarrow{\,\,\,} b \xrightarrow{\,\,\,} 0.
$$
Let $g:a_1\to a_2$ be the map such that $T(g) = f$.
Using semisimplicity, the map $g$ is isomorphic to one of the form
$$
\begin{pmatrix}
1 & 0
\\
0 & 0
\end{pmatrix}
:a_0\oplus a_1' \to a_0\oplus a_2'.
$$
The same holds after applying the functor $T$, since linear functors preserve direct sums.
But the cokernel of this map is clearly $T(a_2')$, which is in the image of $T$.
\end{proof}

\begin{cor}
\label{cor:FrobeniusToFrobenius}
Assume the hypotheses of Theorem \ref{thm:SemiSimple}. If $A \in \cM$ is a connected (i.e., $\mathrm{dim}\,\cC(1,A)=1$) semisimple Frobenius algebra, then $\Tr_\cC(A)$ is as well.
\end{cor}
\begin{proof}
Using that $\Phi(1_\cC)=1_\cM$, it is easy to check that $\Tr_\cC(A)$ is connected.
It is semisimple by Theorem~\ref{thm:SemiSimple}. 
Define the counit by 
\[
\epsilon_{\Tr_\cC(A)} := \epsilon \circ \Tr_\cC(\epsilon_A):\Tr_\cC(A)\to 1_\cC,
\]
where $\epsilon:\Tr_\cC(1_\cM)\to 1_\cC$ is the left inverse of $i$.
The pairing $p_{\Tr_\cC(A)} :=\epsilon_{\Tr_\cC(A)}\circ m_{\Tr_\cC(A)}$ is nondegenerate by \cite[Prop. 3.1.ii]{MR1976459}, and so $\Tr_\cC(A)$ has a natural Frobenius algebra structure by Proposition~\ref{prop   :Frobenius}.
\end{proof}

\begin{rem}
\label{rem:ThingsWeDon'tKnowHowToDo}
More generally, we would like to construct a comultiplication
$$
\Delta_{x,y} \,=
\begin{tikzpicture}[baseline=-.8cm]
\invertedPairOfPants{(0,0)}{}
\draw[thick] (.3,0) circle (.3 and .1);
\draw[thick] (1.7,0) circle (.3 and .1);
\pgftransformxscale{.9}
\pgftransformxshift{31.5}
\pgftransformyshift{-33}
\pgftransformrotate{180}
	\draw[thick, yString] (-.8,-1.07) .. controls ++(90:.6cm) and ++(270:1cm) .. (-.11,.42) node[below, xshift=.5, color=black]{$\scriptstyle y$};		
	\draw[thick, xString] (.8,-1.07) .. controls ++(90:.6cm) and ++(270:1cm) .. (.11,.42)node[below, xshift=-1, color=black]{$\scriptstyle x$};		
\end{tikzpicture}\,\,:\Tr_\cC(x\otimes y)\to\Tr_\cC(x)\otimes \Tr_\cC(y)
$$
which is compatible with the multiplication and all the other structures on $\Tr_\cC$.
The graphical calculus suggests that there should be one, and Corollary \ref{cor:FrobeniusToFrobenius} is a partial step in that direction.
The main missing ingredient is the non-degeneracy of the pairing
\[
\qquad
\begin{tikzpicture}[baseline=1.5cm]
	\coordinate (a1) at (0,0);
	\coordinate (a2) at (1.4,0);
	\coordinate (b1) at (0,1);
	\coordinate (b2) at (1.4,1);
	\coordinate (c1) at (.7,2.5);
	\draw[thick] (c1) arc (180:0:.3cm);
	\pairOfPants{(b1)}{}	
	\draw[thick, xString] ($ (b2) + (.3,-.1) $) node[below, color=black]{$\scriptstyle x$} to[in=90,out=90, looseness=2.3]($ (b1) + (.3,-.1) $) node[below, color=black]{$\scriptstyle x^*$};
\end{tikzpicture}
\,\,\,=\,\,
\epsilon
\circ
\Tr_\cC(\ev_x)
\circ
\mu_{x^*,x}
\]
(compare with Proposition \ref{prop   :Frobenius}).
\end{rem}

\subsection{The algebra \texorpdfstring{$\Tr_\cC(A \otimes B)$}{trace of A tensor B}}
\label{sec:AlgebrasAndTraciator}

\definecolor{dString}{named}{orange}

We now assume that $\cC$ and $\cM$ are pivotal tensor categories and that $\Phi^{\scriptscriptstyle \cZ}:\cC\to \cZ(\cM)$ is a tensor functor.
As always, $\Tr_\cC$ is the right adjoint of $\Phi$.
In the previous section, we saw that, under weaker hypotheses, if $A\in \cM$ is an algebra then so is $\Tr_\cC(A)$.
With our current assumptions, the same holds true for $\Tr_\cC(A \otimes B)$:


\begin{prop}
\label{prop:TwoAlgebrasPartOne}
Let $\cC$ and $\cM$ be pivotal tensor categories, let $\Phi^{\scriptscriptstyle \cZ}:\cC\to \cZ(\cM)$ be a pivotal functor, and let $\Tr_\cC$ be the right adjoint of $\Phi=F\circ\Phi^{\scriptscriptstyle \cZ}$.
If $A$ and $B$ are algebra objects in $\cM$, then $\Tr_\cC(A\otimes B)$ is an algebra object in $\cC$.
\end{prop}

\begin{proof}
The structure morphisms of $\Tr_\cC(A\otimes B)$ are given by
$$
m_{\Tr_\cC(A\otimes B)}:=\Tr_\cC(m_A\otimes m_B) \circ \tau^-_{B,A\otimes A\otimes B}\circ \mu_{B\otimes A, A\otimes B}\circ (\tau^+_{A,B}\otimes \id_{\Tr_\cC(A\otimes B)})
\,=\!
\begin{tikzpicture}[baseline=.4cm]

	\pairOfPants{(0,-1)}{}
	\bottomCylinder{(0,-2)}{.3}{1}
	\bottomCylinder{(1.4,-2)}{.3}{1}
	\emptyCylinder{(.7,.5)}{.3}{1}
	\halfDottedEllipse{(.7,1.5)}{.3}{.1}
	\topCylinder{(.7,1.5)}{.3}{1}

	\draw[thick, BString] (.2,-1.08) .. controls ++(90:.8cm) and ++(270:.8cm) .. (.8,.42);		
	\draw[thick, AString] (.4,-1.08) .. controls ++(90:.8cm) and ++(270:.8cm) .. (.9,.42);		
	\draw[thick, AString] (1.6,-1.08) .. controls ++(90:.8cm) and ++(270:.8cm) .. (1.1,.42);		
	\draw[thick, BString] (1.8,-1.08) .. controls ++(90:.8cm) and ++(270:.8cm) .. (1.2,.42);		
	
	\draw[thick, AString] (1.6,-1.08) -- (1.6,-2.08);		
	\draw[thick, BString] (1.8,-1.08) -- (1.8,-2.08);		
	\draw[thick, BString] (.4,-2.08) .. controls ++(90:.2cm) and ++(225:.2cm) .. (.6,-1.58);		
	\draw[thick, BString] (.2,-1.08) .. controls ++(270:.2cm) and ++(45:.2cm) .. (0,-1.48);
	\draw[thick, BString, dotted] (0,-1.48) -- (.6,-1.58);	
	\draw[thick, AString] (.2,-2.08) .. controls ++(90:.4cm) and ++(270:.4cm) .. (.4,-1.08);		

	\draw[thick, BString] (.8,.42) .. controls ++(90:.2cm) and ++(-45:.2cm) .. (.7,.92);		
	\draw[thick, BString] (1.2,1.42) .. controls ++(270:.2cm) and ++(135:.2cm) .. (1.3,1.02);
	\draw[thick, BString, dotted] (0.7,.92) -- (1.3,1.02);	
	\draw[thick, AString] (.9,.42) .. controls ++(90:.4cm) and ++(270:.4cm) .. (.8,1.42);		
	\draw[thick, AString] (1.1,.42) .. controls ++(90:.4cm) and ++(270:.4cm) .. (.9,1.42);		
	\draw[thick, BString] (1.2,.42) .. controls ++(90:.4cm) and ++(270:.4cm) .. (1.1,1.42);		

	\filldraw[AString] (.85, 2) circle (.025cm);
	\filldraw[BString] (1.15, 2) circle (.025cm);
	\draw[thick, AString] (.8,1.42) .. controls ++(90:.2cm) and ++(225:.1cm) .. (.85,2);		
	\draw[thick, AString] (.9,1.42) .. controls ++(90:.2cm) and ++(-45:.1cm) .. (.85,2);		
	\draw[thick, BString] (1.1,1.42) .. controls ++(90:.2cm) and ++(225:.1cm) .. (1.15,2);		
	\draw[thick, BString] (1.2,1.42) .. controls ++(90:.2cm) and ++(-45:.1cm) .. (1.15,2);		
	\draw[thick, BString] (1.15,2) -- (1.15,2.42);
	\draw[thick, AString] (.85,2) -- (.85,2.42);
\end{tikzpicture}
$$
and
$$
i_{\Tr_\cC(A\otimes B)}:= \Tr_\cC(i_A\otimes i_B)\circ i
=\,
\begin{tikzpicture}[baseline=.4cm]
	\topCylinder{(.7,0)}{.3}{1}
	\draw[thick] (.7,0) arc (-180:0:.3cm);		
	\halfDottedEllipse{(.7,0)}{.3}{.1}

	\filldraw[AString] (.9, .3) circle (.025cm);
	\filldraw[BString] (1.1, .3) circle (.025cm);
	\draw[thick, AString] (.9,.92) -- (.9,.3);
	\draw[thick, BString] (1.1,.92) -- (1.1,.3);

\end{tikzpicture}
\,.
$$
We leave it as an exercise to check that these structure maps satisfy the necessary associativity and unitality axioms (Lemma~\ref{lem:MultiplicationAssocaitive} gets used for associativity).
\end{proof}

Similarly to Remark \ref{rem:BackOfTubes}, the pictures for $m_{\Tr_\cC(A\otimes B)}$ and $i_{\Tr_\cC(A\otimes B)}$ become more intuitive if we allow to draw strands on the back of the tubes.
The structure morphisms become:
\begin{equation}
\label{pic: back of tubes}
m_{\Tr_\cC(A\otimes B)} = 
\begin{tikzpicture}[baseline=.6cm]
	\topPairOfPants{(0,0)}{}

\pgftransformxshift{1.3}
	\filldraw[AString] (.9, 1.1) circle (.025cm);
	\draw[thick, AString] (.2,-.09) .. controls ++(92:.8cm) and ++(270:.4cm) .. (.9,1.1);		
	\draw[thick, AString] (1.6,-.09) .. controls ++(90:.8cm) and ++(270:.4cm) .. (.9,1.1);		
	\draw[thick, AString] (.9,1.1) -- (.9,1.41);
\pgftransformxshift{-2.6}
	\filldraw[BString] (1.1, 1.1) circle (.025cm);
	\draw[more thick, BString, densely dotted] (.4,.08) .. controls ++(88:.6cm) and ++(270:.4cm) .. (1.1,1.1);		
	\draw[more thick, BString, densely dotted] (1.8,.08) .. controls ++(90:.6cm) and ++(270:.4cm) .. (1.1,1.1);		
	\draw[more thick, BString, densely dotted] (1.1,1.1) -- (1.1,1.58);
\end{tikzpicture}
\quad\,\,\text{ and }\quad\,\,
i_{\Tr_\cC(A\otimes B)}=
\begin{tikzpicture}[baseline=.2cm]
	\topCylinder{(.7,0)}{.3}{1}
	\draw[thick] (.7,0) arc (-180:0:.3cm);		
	\halfDottedEllipse{(.7,0)}{.3}{.1}

	\filldraw[AString] (.95, .3) circle (.025cm);
	\filldraw[BString] (1.05, .46) circle (.025cm);
	\draw[thick, AString] (.95,.91) -- (.95,.3);
	\draw[more thick, BString, densely dotted] (1.05,1.09) -- (1.05,.46);
	
\end{tikzpicture}\,\,.
\end{equation}

We now prove our main theorem, about the semisimplicity of $\Tr_\cC(A \otimes B)$. 
Recall that for an algebra object in a semisimple category, separability is a somewhat stronger condition than semisimplicity --- see Section~\ref{sec:Frobenius} for a discussion.

\begin{thm}
\label{thm:TwoAlgebrasPartTwo}
Let $\cC$ and $\cM$ be as in Proposition \ref{prop:TwoAlgebrasPartOne}, and let us assume that they are both semisimple.
Let $A,B\in \cM$ be algebra objects such that the category $\Bimod_\cM(A,B)$ of $A$-$B$-bimodules is semisimple.
Then $\Tr_\cC(A \otimes B)$ is semisimple.

In particular, if $A$ and $B$ are separable algebras (see Proposition \ref{prop: ostrik bimod semisimple}), then $\Tr_\cC(A \otimes B)$ is semisimple.
\end{thm}

\begin{proof}
The proof follows the same outline as that of Theorem \ref{thm:SemiSimple}.
Let $\cN_0$ be the essential image of the functor $\cC \to \Bimod_\cM(A,B)$ given by $c\mapsto A\otimes \Phi(c)\otimes B$, and let $\cN$ be the idempotent completion of $\cN_0$.
Since $\Bimod_\cM(A,B)$ is semisimple, $\cN$ is the full subcategory of $\Bimod_\cM(A,B)$ generated by the simple objects that occur as direct summands of $A \otimes \Phi(c) \otimes B$ for $c \in \cC$. 
The categorified trace induces a functor
$$
\Tr_\cC^{\Bimod(A,B)}: \Bimod_\cM(A,B) \to \Mod_\cC(\Tr_\cC(A \otimes B)),
$$
where the $\Tr_\cC(A\otimes B)$-module structure on $\Tr_\cC(z)$ for $z\in \Bimod_\cM(A,B)$ is given by
$$
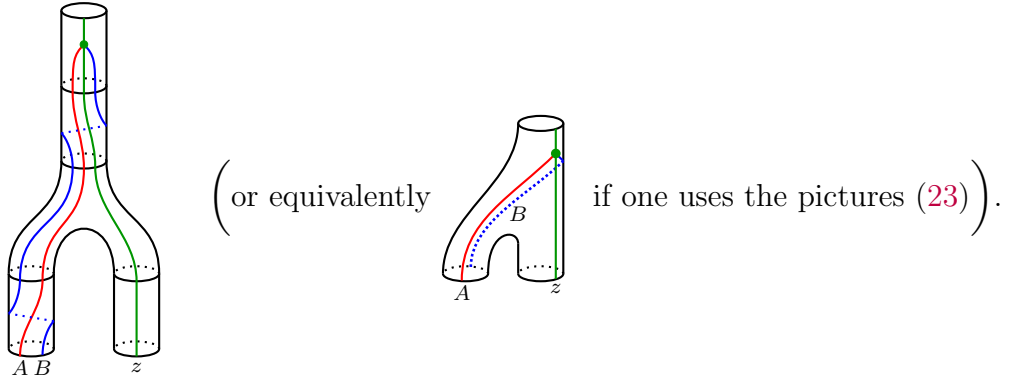
\qquad\quad\,\,
\begin{tikzpicture}[baseline = 1.9cm]
	\pgfmathsetmacro{\hoffset}{.15};
	\pgfmathsetmacro{\voffset}{.08};
	\coordinate (a1) at (0,0);
	\coordinate (a2) at (1.4,0);
	\coordinate (b1) at (0,1);
	\coordinate (b2) at (1.4,1);
	\coordinate (c1) at (.7,2.5);
	\coordinate (d1) at (.7,3.5);
	\coordinate (d2) at (1,4.05);
	\coordinate (e1) at (.7,4.5);
	\bottomCylinder{(a1)}{.3}{1}
	\bottomCylinder{(a2)}{.3}{1}
	\topCylinder{(d1)}{.3}
	\halfDottedEllipse{(d1)}{.3}{.1}
	\pairOfPants{(b1)}{}
	\draw[thick] ($ (c1) + (.6,0) $) -- ($ (d1) + (.6,0) $);
	\draw[thick] ($ (c1) $) -- ($ (d1) $);

	\draw[thick, AString] ($ (a1) + (\hoffset,0) + (0,-.1)$) .. controls ++(90:.4cm) and ++(270:.4cm) .. ($ (a1) + 3*(\hoffset,0) + (0,-\voffset) + (0,1) $);		
	\draw[thick, BString] ($ (a1) + 3*(\hoffset,0) + (0,-\voffset) $) .. controls ++(90:.2cm) and ++(225:.1cm) .. ($ (a1) + 4*(\hoffset,0) + (0,-\voffset) + (0,.45)$);
	\draw[thick, BString] ($ (a1) + (\hoffset,1) + (0,-\voffset) $) .. controls ++(270:.2cm) and ++(45:.1cm) .. ($ (a1) + (0,1) + (0,-\voffset) + (0,-.45)$);
	\draw[thick, BString, dotted] ($ (a1) + 4*(\hoffset,0) + (0,-\voffset) + (0,.45)$) -- ($ (a1) + (0,1) + (0,-\voffset) + (0,-.45)$);	

	\draw[thick, AString] ($ (c1) + 2*(\hoffset,0) + (0,-.1)$) .. controls ++(90:.4cm) and ++(270:.4cm) .. ($ (c1) + (\hoffset,0) + (0,-\voffset) + (0,1) $);
	\draw[thick, zString] ($ (c1) + 3*(\hoffset,0) + (0,-\voffset) $) .. controls ++(90:.4cm) and ++(270:.4cm) .. ($ (c1) + 2*(\hoffset,0) + (0,.9) $);		
	\draw[thick, BString] ($ (c1) + (\hoffset,0) + (0,-\voffset) $) .. controls ++(90:.2cm) and ++(-45:.1cm) .. ($ (c1) + (0,-\voffset) + (0,.45)$);
	\draw[thick, BString] ($ (c1) + 3*(\hoffset,0) + (0,1) + (0,-\voffset) $) .. controls ++(270:.2cm) and ++(135:.1cm) .. ($ (c1) + 4*(\hoffset,0) + (0,-\voffset) + (0,-.45) + (0,1)$);
	\draw[thick, BString, dotted] ($ (c1) + (0,-\voffset) + (0,.45)$) -- ($ (c1) + 4*(\hoffset,0) + (0,-\voffset) + (0,-.45) + (0,1)$);	

	\draw[thick, zString] ($ (a2) + (.3,-.1) $) -- ($ (b2) + (.3,-.1) $) .. controls ++(90:.6cm) and ++(270:.6cm) .. ($ (c1) + 3*(\hoffset,0) + (0,-\voffset) $);
	\draw[thick, AString] ($ (b1) + 3*(\hoffset,0) + (0,-\voffset) $) .. controls ++(90:.8cm) and ++(270:.8cm) .. ($ (c1) + 2*(\hoffset,0) + (0,-.1)$);
	\draw[thick, BString] ($ (b1) + (\hoffset,0) + (0,-\voffset) $) .. controls ++(90:.8cm) and ++(270:.8cm) .. ($ (c1) + (\hoffset,0) + (0,-\voffset) $);

	\draw[thick, AString] ($ (c1) + (\hoffset,0) + (0,-\voffset) + (0,1) $) .. controls ++(90:.4cm) and ++(225:.2cm) .. (d2);
	\draw[thick, BString] ($ (c1) + 3*(\hoffset,0) + (0,-\voffset) + (0,1) $) .. controls ++(90:.4cm) and ++(-45:.2cm) .. (d2);
	\draw[thick, zString] ($ (c1) + 2*(\hoffset,0) + (0,.9) $) .. controls ++(90:.4cm) and ++(270:.4cm) .. ($ (e1) + (.3,-.1) $); 
	\filldraw[zString] (d2) circle (.05cm);
	\node at ($ (a2) + (.3,-.25) $) {\scriptsize{$z$}};
	\node at ($ (a1) + (\hoffset,-.25) $) {\scriptsize{$A$}};
	\node at ($ (a1) + 3*(\hoffset,0) + (0,-.25) $) {\scriptsize{$B$}};
	
\end{tikzpicture}
\,\,\,\,\,\,\,\,\,
\bigg(\text{or equivalently}\,\,
\begin{tikzpicture}[baseline = .9cm]
	\coordinate (a1) at (0,0);
	\coordinate (a2) at (1,0);
	\coordinate (b1) at (1,.4);
	\coordinate (b2) at (1.5,1.6);
	\coordinate (c1) at (1,2);
	\halfDottedEllipse{(a1)}{.3}{.1}
	\halfDottedEllipse{(a2)}{.3}{.1}
	\draw[thick] ($ (c1) + (.3,0) $) ellipse (.3 and .1);
	\draw[thick] ($ (a2) + (.6,0) $) -- ($ (c1) + (.6,0) $);
	\draw[thick] ($ (a2) $) -- ($ (a2) + (0,.4) $);
	\draw[thick] ($ (a1) $) .. controls ++(90:.8cm) and ++(270:.8cm) .. ($ (c1) $);
	\draw[thick] ($ (a1) + (.6,0) $) .. controls ++(90:.6cm) and ++(90:.2cm) .. ($ (b1) $);
	\draw[thick, AString] ($ (a1) + (.25,-.09) $) .. controls ++(90:.8cm) and ++(225:.8cm) .. (b2);
	\draw[more thick, BString, densely dotted] ($ (a1) + (.37,.09) $) .. controls ++(90:.6cm) and ++(225:.6cm) .. ($ (b2) + (.1,-.1) $);
	\draw[thick, BString] ($ (b2) + (.1,-.1) $) -- (b2);
	\draw[thick, zString] ($ (a2) + (.5,-.07) $) -- ($ (c1) + (.5,-.07) $);
	\filldraw[thick, zString] (b2) circle (.05cm); 
	\node at ($ (a2) + (.5,-.2) $) {\scriptsize{$z$}};
	\node at ($ (a1) + (.25,-.25) $) {\scriptsize{$A$}};
	\node at ($ (a2) + (0,.8) $) {\scriptsize{$B$}};
\end{tikzpicture}\,\,\,\,
\text{if one uses the pictures \eqref{pic: back of tubes}}\bigg).
$$
We will show that the composite
\newcommand*{\longhookrightarrow}{\ensuremath{\lhook\joinrel\relbar\joinrel\relbar\joinrel\rightarrow}}
$$
\cN \longhookrightarrow \Bimod_\cM(A,B) \xrightarrow{\Tr_\cC^{\Bimod(A,B)}} \Mod_\cC(\Tr_\cC(A \otimes B))
$$
is an equivalence of categories.
As $\cN$ is semisimple, this will complete the proof.\medskip

$\bullet$ \emph{The functor $\Tr_\cC^{\Bimod(A,B)}\!|_\cN$ is fully faithful.}
It is enough to show that the restriction to $\cN_0$ is fully faithful,
so we must show that the map
$$
\cN_0\big(A \otimes \Phi(c) \otimes B, A \otimes \Phi(d) \otimes B\big) \to \cC_{\Tr_\cC(A \otimes B)}\big(\Tr_\cC(A \otimes \Phi(c) \otimes B), \Tr_\cC(A \otimes \Phi(d) \otimes B)\big)
$$
\begin{align}
\begin{tikzpicture}[baseline=.1cm, scale=.8]
	\plane{(-.4,-.8)}{2.8}{1.7}
	\draw[thick, AString] (.8,.2) -- (1.4,.2);
	\draw[thick, AString] (-1.4,.2) -- (-.6,.2);
	\draw[thick, BString] (.8,-.3) -- (1.9,-.3);
	\draw[thick, BString] (-.9,-.3) -- (-.1,-.3);
	\CMbox{box}{(0,-.5)}{.8}{.8}{.8}{$f$}
	\draw[super thick, white] (-1.6,0) -- (-.2,0);
	\draw[thick, cString] (-1.8,0) -- (-.25,0);
	\draw[super thick, white] (1,0) -- (2.2,0);
	\draw[thick, dString] (.8,0) -- (2.2,0);
	\node at (-1.6,-.2+.04) {\scriptsize{$c$}};
	\node at (2,.2) {\scriptsize{$d$}};
	\node at (-1.2,.4) {\scriptsize{$A$}};
	\node at (1.4,-.5) {\scriptsize{$B$}};
\end{tikzpicture}
&\,\longmapsto\,
\Tr_\cC
\left(
\begin{tikzpicture}[baseline=.1cm, scale=.8]
	\plane{(-.4,-.8)}{2.8}{1.7}
	\draw[thick, AString] (.8,.2) -- (1.4,.2);
	\draw[thick, AString] (-1.4,.2) -- (-.6,.2);
	\draw[thick, BString] (.8,-.3) -- (1.9,-.3);
	\draw[thick, BString] (-.9,-.3) -- (-.1,-.3);
	\CMbox{box}{(0,-.5)}{.8}{.8}{.8}{$f$}
	\draw[super thick, white] (-1.6,0) -- (-.2,0);
	\draw[thick, cString] (-1.8,0) -- (-.25,0);
	\draw[super thick, white] (1,0) -- (2.2,0);
	\draw[thick, dString] (.8,0) -- (2.2,0);
	\node at (-1.6,-.2+.04) {\scriptsize{$c$}};
	\node at (2,.2) {\scriptsize{$d$}};
	\node at (-1.2,.4) {\scriptsize{$A$}};
	\node at (1.4,-.5) {\scriptsize{$B$}};
\end{tikzpicture}
\,
\right)
\,=\,\,
\begin{tikzpicture}[baseline=.6cm, scale=1.1]
	\coordinate (a1) at (0,0);
	\coordinate (b1) at (0,1.4);
	\draw[thick] (a1) -- (b1);
	\draw[thick] ($ (a1) + (1,0) $) -- ($ (b1) + (1,0) $);
	\draw[thick] ($ (b1) + (.5,0) $) ellipse (.5 and .15);
	\halfDottedEllipse{(a1)}{.5}{.15}
	\draw[thick, BString] (.65,-.13) -- (.65,1.25);
	\draw[thick, AString] (.3,-.13) -- (.3,1.25);
	\CMboxSize{box}{(.4,.3)}{.55}{.44}{.2}{$f$}{.9}
	\draw[super thick, white] (.5,-.22) -- (.5,0);
	\draw[thick, cString] (.5,-.22) -- (.5,.3);
	\draw[super thick, white] (.5,.82) -- (.5,1.4);
	\draw[thick, dString] (.5,.82) -- (.5,1.33);
\end{tikzpicture}
\label{map:TrCBimod}
\end{align}
is an isomorphism.
The inverse of \eqref{map:TrCBimod} is given by
\begin{equation}
\begin{tikzpicture}[baseline=.4cm, scale=.8]
	\downTube{(.1,0)}{.8}{.4}
	\CMbox{box}{(0,0)}{1.4}{.8}{.4}{$g$}
	\upTube{(0.1,1.1)}{.8}{.5}
	\draw[thick, AString] (.2,.97) -- (.2,1.47);
	\draw[thick, AString] (.2,0) -- (.2,-.53);
	\draw[thick, BString] (.8,.97) -- (.8,1.47);
	\draw[thick, BString] (.8,0) -- (.8,-.53);
	\draw[super thick, white] (.5,1.55) -- (.5,.86);
	\draw[thick, dString] (.5,1.55) -- (.5,.86);
	\draw[super thick, white] (.5,-.8) -- (.5,-.2);
	\draw[thick, cString] (.5,-.7) -- (.5,0);
	\node at (.37,1.6) {\scriptsize{$d$}};
	\node at (.5,-.85) {\scriptsize{$c$}};
	\node at (.2,-.8) {\scriptsize{$A$}};
	\node at (.8,-.8) {\scriptsize{$B$}};
\end{tikzpicture}
\,\,\,\,\,\longmapsto\,\,\,
\begin{tikzpicture}[baseline=.4cm, scale=.8]
	\plane{(-2.4,-1)}{6.4}{2.7}
	\CMbox{box}{(0,-.5)}{.8}{1}{.8}{$\varepsilon$}
	\coordinate (a) at (1.6,.6);
	\filldraw[AString] (a) circle (.05cm);
	\draw[thick, AString] (-4.9,1.5) -- (-.2,1.5) .. controls ++(0:.6cm) and ++(180:.4cm) .. (a);
	\draw[thick, AString] (.8,.2) -- (1,.2) .. controls ++(0:.4cm) and ++(180:.2cm) .. (a);
	\draw[thick, AString] (a) -- (2.4,.6);
	\coordinate (b) at (2,-.6);
	\filldraw[BString] (b) circle (.05cm);
	\draw[thick, BString] (-2.6,-.8) -- (.4,-.8) .. controls ++(0:.6cm) and ++(180:.2cm) .. (b);
	\draw[thick, BString] (.8,-.2) -- (1,-.2) .. controls ++(0:.4cm) and ++(180:.2cm) .. (b);
	\draw[thick, BString] (b) -- (3.6,-.6);
	\node at (3.4,.2) {\scriptsize{$d$}};
	\node at (-4,1.3) {\scriptsize{$A$}};
	\node at (2.8,-.8) {\scriptsize{$B$}};
	\straightTubeNoString{(-.4,-.1)}{.2}{1.2}
	\draw[super thick, white] (-1.6,.3) -- (-.15,.3);
	\draw[thick, dString] (-1.8,.3) -- (-.15,.3);
	\draw[super thick, white] (1,0) -- (3.8,0);
	\draw[thick, dString] (.8,0) -- (3.8,0);
	\draw[thick, AString] (-2,.5) -- (-.23,.5);
	\draw[thick, BString] (-2,.1) -- (-.23,.1);
	\CMbox{box}{(-2,-.5)}{.8}{1}{.8}{$g$}
	\straightTubeWithCap{(-2.4,-.1)}{.2}{.5}
	\draw[thick, AString] (-2.8,.5) -- (-2.23,.5);
	\filldraw[AString] (-2.8,.5) circle (.05cm);
	\draw[thick, BString] (-2.8,.1) -- (-2.23,.1);
	\filldraw[BString] (-2.8,.1) circle (.05cm);
%
\def\sh{.2}
\draw[super thick, white] (-4.5+.1-\sh,.3+\sh+\sh) -- (-3.4-\sh,.3+\sh+\sh) .. controls ++(0:.2cm) and ++(180:.2cm) .. (-2.9,.1+\sh) -- (-2.33+.2,.1+\sh) ;
\draw[thick, cString] (-4.5+.1-\sh,.3+\sh+\sh) -- (-3.4-\sh,.3+\sh+\sh) .. controls ++(0:.2cm) and ++(180:.2cm) .. (-2.9,.1+\sh) -- (-2.33+.2,.1+\sh) ;
	\node at (-4-\sh-\sh,\sh+\sh+.1+.04) {\scriptsize{$c$}};
\end{tikzpicture}
\label{map:TrCInverse TWO}
\end{equation}
One can check the equation \eqref{map:TrCInverse TWO}$\circ$\eqref{map:TrCBimod}$=\!\id$ directly as in the proof of Theorem~\ref{thm:SemiSimple}.
On the other hand, the equation \eqref{map:TrCBimod}$\circ$\eqref{map:TrCInverse TWO}$=\!\id$ is best checked after precomposition by the isomorphism
\begin{equation}\label{eq: Trace of free bimodule}
\begin{tikzpicture}[baseline = 1.9cm]
	\pgfmathsetmacro{\hoffset}{.15};
	\pgfmathsetmacro{\voffset}{.08};
	\coordinate (a1) at (0,0);
	\coordinate (a2) at (1.4,.6);
	\coordinate (b1) at (0,1);
	\coordinate (b2) at (1.4,1);
	\coordinate (c1) at (.7,2.5);
	\coordinate (d1) at (.7,3.5);
%
	\bottomCylinder{(a1)}{.3}{1}
	\draw[thick] (a2) -- (b2);
	\draw[thick] ($ (a2) + (.6,0) $) -- ($ (b2) + (.6,0) $);
	\topCylinder{(c1)}{.3}
	\pairOfPants{(b1)}{}

	\draw[thick, AString] ($ (a1) + (\hoffset,0) + (0,-.1)$) .. controls ++(90:.4cm) and ++(270:.4cm) .. ($ (a1) + 3*(\hoffset,0) + (0,-\voffset) + (0,1) $);		
	\draw[thick, BString] ($ (a1) + 3*(\hoffset,0) + (0,-\voffset) $) .. controls ++(90:.2cm) and ++(225:.1cm) .. ($ (a1) + 4*(\hoffset,0) + (0,-\voffset) + (0,.45)$);
	\draw[thick, BString] ($ (a1) + (\hoffset,1) + (0,-\voffset) $) .. controls ++(270:.2cm) and ++(45:.1cm) .. ($ (a1) + (0,1) + (0,-\voffset) + (0,-.45)$);

	\draw[thick, AString] ($ (c1) + 2*(\hoffset,0) + (0,-.1)$) .. controls ++(90:.4cm) and ++(270:.4cm) .. ($ (c1) + (\hoffset,0) + (0,-\voffset) + (0,1) $);
	\draw[thick, BString] ($ (c1) + (\hoffset,0) + (0,-\voffset) $) .. controls ++(90:.2cm) and ++(-45:.1cm) .. ($ (c1) + (0,-\voffset) + (0,.45)$);
	\draw[thick, BString] ($ (c1) + 3*(\hoffset,0) + (0,1) + (0,-\voffset) $) .. controls ++(270:.2cm) and ++(135:.1cm) .. ($ (c1) + 4*(\hoffset,0) + (0,-\voffset) + (0,-.45) + (0,1)$);

	\draw[thick, AString] ($ (b1) + 3*(\hoffset,0) + (0,-\voffset) $) .. controls ++(90:.8cm) and ++(270:.8cm) .. ($ (c1) + 2*(\hoffset,0) + (0,-.1)$);
	\draw[thick, BString] ($ (b1) + (\hoffset,0) + (0,-\voffset) $) .. controls ++(90:.8cm) and ++(270:.8cm) .. ($ (c1) + (\hoffset,0) + (0,-\voffset) $);

%
	\node at ($ (a2) + (0,-.85) $) {\scriptsize{$c$}};
	\node at ($ (a1) + (\hoffset,-.25) $) {\scriptsize{$A$}};
	\node at ($ (a1) + 3*(\hoffset,0) + (0,-.25) $) {\scriptsize{$B$}};
%
	\draw[thick] (a2) arc (-180:0:.3cm);
	\draw[super thick, white] ($ (a2) + (0,-.7) $) --($ (a2) + (0,-.6) $) .. controls ++(90:.3cm) and ++(270:.3cm) .. ($ (a2) + (.3,0) $) -- ($ (b2) + (.3,-.05) $) .. controls ++(90:.8cm) and ++(270:.8cm) .. ($ (c1) + 3*(\hoffset,0) + (0,-\voffset) $) .. controls ++(90:.4cm) and ++(270:.4cm) .. ($ (d1) + 2*(\hoffset,0) $);
	\draw[thick, cString] ($ (a2) + (0,-.7) $) --($ (a2) + (0,-.6) $) .. controls ++(90:.3cm) and ++(270:.3cm) .. ($ (a2) + (.3,0) $) -- ($ (b2) + (.3,-.05) $) .. controls ++(90:.8cm) and ++(270:.8cm) .. ($ (c1) + 3*(\hoffset,0) + (0,-\voffset) $) .. controls ++(90:.4cm) and ++(270:.4cm) .. ($ (d1) + 2*(\hoffset,0) $);
\end{tikzpicture}
\,\,:\, \Tr_\cC(A\otimes B) \otimes c\to \Tr_\cC(A\otimes \Phi(c)\otimes B).
\end{equation}
Overall, the argument is very similar to the one used in the proof of Theorem~\ref{thm:SemiSimple}.\medskip

$\bullet$ \emph{The functor $\Tr_\cC^{\Bimod(A,B)}\!|_\cN$ is essentially surjective.}
Once again, the argument is completely parallel to the one in Theorem~\ref{thm:SemiSimple}.
Every $\Tr_\cC(A\otimes B)$-module $c$ fits in a coequalizer diagram 
$$
\Tr_\cC(A\otimes B)\otimes \Tr_\cC(A\otimes B) \otimes c \rightrightarrows \Tr_\cC(A\otimes B)\otimes c \to c.
$$
Using the isomorphism \eqref{eq: Trace of free bimodule}, we see that
$\Tr_\cC(A\otimes B)\otimes \Tr_\cC(A\otimes B)\otimes c$ and $\Tr_\cC(A\otimes B)\otimes c$ are in the essential image of $\cN$.
We are finished by Lemma \ref{lem:EssentiallySurjective}. 
%
\end{proof}

\begin{rem}
At this time, we do not know whether the algebra $\Tr(A \otimes B)$ is separable under the hypotheses of Theorem \ref{thm:TwoAlgebrasPartTwo}.
The issues are similar to the ones encountered in Remark~\ref{rem:ThingsWeDon'tKnowHowToDo}.
\end{rem}

\begin{rem}\label{rem: pivotal not needed BIS}
The proofs of \ref{prop:TwoAlgebrasPartOne} and \ref{thm:TwoAlgebrasPartTwo} go through if we only assume $\cC$ and $\cM$ to be rigid, as opposed to pivotal.
When we encounter a traciator
(as for example in the definition of $m_{\Tr_\cC(A\otimes B)}$, or in that of the $\Tr_\cC(A\otimes B)$-algebra structure on $\Tr_\cC(z)$ for $z\in \Bimod_\cM(A,B)$, or in the isomorphism \eqref{eq: Trace of free bimodule}) we need to insert a double dual at the appropriate place, as explained in Remark~\ref{rem:DoubleDual}.
It so happens that every such double duals gets undone by an inverse traciator later on.
\end{rem}

Combining Theorem \ref{thm:TwoAlgebrasPartTwo} with Remarks \ref{rem:WhenSimpleImpliesSeparable} and \ref{rem: pivotal not needed BIS}, we get the following

\begin{cor}\label{CORL our main thm}
Let $\cC$ and $\cM$ be fusion categories over a field of characteristic zero, and let $\Phi^\ssZ: \cC \to \cZ(\cM)$ be a tensor functor. 
If $A,B \in \cM$ are semisimple algebras, then so is $\Tr_\cC(A \otimes B)$.
\end{cor}

\begin{ex}\label{example: E_7}
Let $\cC:=SU(2)_{16}$, with simple objects $\bf1,\ldots, \bf17$.
As explained in \cite{MR1976459} this tensor category has three indecomposable module categories, denoted $A_{17}$, $D_{10}$, and $E_7$.
The first one, $A_{17}$, is just $\cC$ acting on itself.
$A_{17}$ and $D_{10}$ are module tensor categories, whereas $E_7$ is just a module category.
The simple objects of $A_{17}$, $D_{10}$, and $E_7$ correspond to the vertices of the Dynkin diagrams
\[
A_{17}:\,\tikz[scale=.3, baseline=-.05cm]{\draw (-6,0) -- (-5,0) -- (-4,0) -- (-3,0) -- (-2,0) -- (-1,0) -- (0,0) -- (1,0) -- (2,0) -- (3,0) -- (4,0) -- (5,0) -- (6,0) -- (7,0) -- (8,0) -- (9,0) -- (10,0);
\filldraw[fill=white] (-6,0) circle(.17);
\filldraw[fill=white] (-5,0) circle(.17);
\filldraw[fill=white] (-4,0) circle(.17);
\filldraw[fill=white] (-3,0) circle(.17);
\filldraw[fill=white] (-2,0) circle(.17);
\filldraw[fill=white] (-1,0) circle(.17);
\filldraw[fill=white] (0,0) circle(.17);
\filldraw[fill=white] (1,0) circle(.17);
\filldraw[fill=white] (2,0) circle(.17);
\filldraw[fill=white] (3,0) circle(.17);
\filldraw[fill=white] (4,0) circle(.17);
\filldraw[fill=white] (5,0) circle(.17);
\filldraw[fill=white] (6,0) circle(.17);
\filldraw[fill=white] (7,0) circle(.17);
\filldraw[fill=white] (8,0) circle(.17);
\filldraw[fill=white] (9,0) circle(.17);
\filldraw[fill=white] (10,0) circle(.17);}
\qquad\quad
D_{10}:\,\tikz[scale=.3, baseline=-.05cm]{\draw (-6,0) -- (-5,0) -- (-4,0) -- (-3,0) -- (-2,0) -- (-1,0) -- (0,0) -- (1,0) -- (2,0) (1,0) -- (1,1);
\filldraw[fill=white] (-6,0) circle(.17);
\filldraw[fill=white] (-5,0) circle(.17);
\filldraw[fill=white] (-4,0) circle(.17);
\filldraw[fill=white] (-3,0) circle(.17);
\filldraw[fill=white] (-2,0) circle(.17);
\filldraw[fill=white] (-1,0) circle(.17);
\filldraw[fill=white] (0,0) circle(.17);
\filldraw[fill=white] (1,0) circle(.17);
\filldraw[fill=white] (2,0) circle(.17);
\filldraw[fill=white] (1,1) circle(.17);}
\qquad\quad
E_7:\,\tikz[scale=.3, baseline=-.05cm]{\draw (-2,0) -- (-1,0) -- (0,0) -- (1,0) -- (2,0) -- (3,0) (1,0) -- (1,1);
\filldraw[fill=white] (-2,0) circle(.17);
\filldraw[fill=white] (-1,0) circle(.17);
\filldraw[fill=white] (0,0) circle(.17);
\filldraw[fill=white] (1,0) circle(.17);
\filldraw[fill=white] (2,0) circle(.17);
\filldraw[fill=white] (3,0) circle(.17);
\filldraw[fill=white] (1,1) circle(.17);}
\]
where the edges encode the action of $\bf 2\in \cC$.

Let $\underline{\mathrm{Hom}}_{\cC}$ denote the $\cC$-valued internal hom, explained at the beginning of the introduction.
The simple algebra objects in $\cC$ are all of the form $\underline{\mathrm{End}}_{\cC}(m):=\underline{\mathrm{Hom}}_{\cC}(m,m)$, for
some (not necessarily simple) object $m$ in one of the above module categories.

Let us write $\underline 1,\ldots,\underline 9, \underline 9'$ for the simple objects of $D_{10}$.
Triality is an action\footnote{Here, an `action' is a homomorphism to the group of isomorphism classes of tensor auto-equivalences.} of the symmetric group $S_3$ on the subcategory of $D_{10}$ spanned by $\underline 1$, $\underline 3$, $\underline 5$, $\underline 7$, $\underline 9$, $\underline 9'$ \cite[Thm.\,4.3]{MR2783128}.
The objects $\underline 1$, $\underline 5$, $\underline 7$ are fixed, whereas the objects $\underline 3$, $\underline 9$, $\underline 9'$ are permuted:
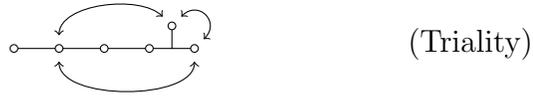
\[
\qquad\qquad\qquad\qquad\qquad
\tikz[scale=.3, baseline=-.05cm]{\draw (-6,0) -- (-5,0) -- (-4,0) -- (-3,0) -- (-2,0) -- (-1,0) -- (0,0) -- (1,0) -- (2,0) (1,0) -- (1,1);
\filldraw[fill=white] (-6,0) circle(.17);
\filldraw[fill=white] (-4,0) circle(.17);
\filldraw[fill=white] (-2,0) circle(.17);
\filldraw[fill=white] (0,0) circle(.17);
\filldraw[fill=white] (2,0) circle(.17);
\filldraw[fill=white] (1,1) circle(.17);
\draw[<->, shorten >=5, shorten <=5] (-4,0) to[out=90,in=180-45] (1,1);
\draw[<->, shorten >=5, shorten <=5] (-4,0) to[out=-90,in=-90] (2,0);
\draw[<->, shorten >=5, shorten <=5] (1,1) to[out=45,in=45, looseness=3] (2,0);
}
\qquad\qquad\qquad\text{(Triality)}
\]
The algebra $\underline{\mathrm{End}}_{D_{10}}(\underline 2)=\underline 1\oplus\underline 3$ yields, under triality, the interesting algebras
$A:=\underline 1\oplus\underline 9$ and $B:=\underline 1\oplus\underline 9'$.
Let $\Tr_\cC:D_{10}\to\cC$ be our categorified trace functor.
At the level of underlying objects, one easily computes
\[\begin{split}
\quad&\Tr_\cC(A) \cong \mathbf1\oplus\mathbf9\oplus\mathbf{17} \\
&\Tr_\cC(A\otimes A) \cong \mathbf1\oplus\mathbf1\oplus\mathbf5\oplus\mathbf9\oplus\mathbf9\oplus\mathbf9\oplus\mathbf{13}\oplus\mathbf{17}\oplus\mathbf{17} \\
&\Tr_\cC(A\otimes B) \cong \mathbf1\oplus\mathbf3\oplus\mathbf7\oplus\mathbf9\oplus\mathbf9\oplus\mathbf{11}\oplus\mathbf{15}\oplus\mathbf{17}
\end{split}\]
By compiling a list of semisimple algebra objects in $\cC$,        
one notes that the above objects have a unique such structure.
They are given by:
\[
\begin{split}
&\Tr_\cC(A)\,\cong\,\underline{\mathrm{End}}_\cC\big(\tikz[scale=.3, baseline=-.05cm]{\draw (-2,0) -- (-1,0) -- (0,0) -- (1,0) -- (2,0) -- (3,0) (1,0) -- (1,1);
\filldraw[fill=black] (-2,0) circle(.17);
\filldraw[fill=white] (-1,0) circle(.17);
\filldraw[fill=white] (0,0) circle(.17);
\filldraw[fill=white] (1,0) circle(.17);
\filldraw[fill=white] (2,0) circle(.17);
\filldraw[fill=white] (3,0) circle(.17);
\filldraw[fill=white] (1,1) circle(.17);}
\big)
\\
&\Tr_\cC(A\otimes A)\,\cong\,\underline{\mathrm{End}}_\cC\big(\underline{1}\oplus\underline{9}\big)
\\
&\Tr_\cC(A\otimes B)\,\cong\,\underline{\mathrm{End}}_\cC\big(\tikz[scale=.3, baseline=-.05cm]{\draw (-2,0) -- (-1,0) -- (0,0) -- (1,0) -- (2,0) -- (3,0) (1,0) -- (1,1);
\filldraw[fill=white] (-2,0) circle(.17);
\filldraw[fill=black] (-1,0) circle(.17);
\filldraw[fill=white] (0,0) circle(.17);
\filldraw[fill=white] (1,0) circle(.17);
\filldraw[fill=white] (2,0) circle(.17);
\filldraw[fill=white] (3,0) circle(.17);
\filldraw[fill=white] (1,1) circle(.17);}
\big).\quad
\end{split}
\]
The algebras $\Tr_\cC(A)$ and $\Tr_\cC(A\otimes B)$ lie in the Morita equivalence class $E_7$,
whereas the algebra $\Tr_\cC(A\otimes A)$ lies in the Morita equivalence class $D_{10}$.
\end{ex}

Given an algebra $A$ and an object $z$ in $\cM$, then $z\otimes A\otimes z^*$ has a canonical algebra structure given by `protecting' the multiplication and unit maps of $A$ by $z$-strands as in \cite[Proof of Prop. 5.4, p. 20]{1501.06869}:
$$
m_{z\otimes A\otimes z^*} 
= 
\begin{tikzpicture}[baseline=-.1cm]
	\filldraw[AString] (0,0) circle (.05cm);
	\draw[thick, AString] (-.4,-.4) -- (0,0) -- (.4,-.4);
	\draw[thick, AString] (0,0) -- (0,.4);
	\draw[thick, zString] (-.2,-.4)  .. controls ++(90:.15cm) and ++(90:.15cm) .. (.2,-.4); 
	\draw[thick, zString] (-.6,-.4) .. controls ++(90:.2cm) and ++(270:.4cm) .. (-.2,.4); 
	\draw[thick, zString] (.6,-.4) .. controls ++(90:.2cm) and ++(270:.4cm) .. (.2,.4); 
\end{tikzpicture}
\quad\text{ and }\quad
i_{z\otimes A\otimes z^*} 
=
\begin{tikzpicture}[baseline=-.1cm]
	\filldraw[AString] (0,0) circle (.05cm);
	\draw[thick, AString] (0,0) -- (0,.4);
	\draw[thick, zString] (-.2,.4)  -- (-.2,0) arc (-180:0:.2cm) -- (.2,.4); 
\end{tikzpicture}
$$
A straightforward calculation shows the following:

\begin{prop}
Given two algebras $A,B$ in $\cM$ and an object $z\in \cM$, the two algebras $\Tr_\cC((z\otimes A\otimes z^*)\otimes B)$ and $\Tr_\cC(A\otimes (z^*\otimes B\otimes z))$ agree up to conjugation by the traciator.
\end{prop}

\noindent
The structure morphisms can be drawn as in \eqref{pic: back of tubes}, with the $A$-strand on the front, the $B$-strand on the back, and 
the `protecting' $z$-strands on the sides:
$$
m_{\Tr_\cC(z\otimes A\otimes z^*\otimes B)} = 
\begin{tikzpicture}[baseline=.6cm]
	\topPairOfPants{(0,0)}{}

\pgftransformxshift{1.3}
	\filldraw[AString] (.9, 1.1) circle (.025cm);
	\draw[thick, AString] (.2,-.09) .. controls ++(92:.8cm) and ++(270:.4cm) .. (.9,1.1);		
	\draw[thick, AString] (1.6,-.09) .. controls ++(90:.8cm) and ++(270:.4cm) .. (.9,1.1);		
	\draw[thick, AString] (.9,1.1) -- (.9,1.41);
\pgftransformxshift{-2.6}
	\filldraw[BString] (1.1, 1.1) circle (.025cm);
	\draw[more thick, BString, densely dotted] (.4,.08) .. controls ++(88:.6cm) and ++(270:.4cm) .. (1.1,1.1);		
	\draw[more thick, BString, densely dotted] (1.8,.08) .. controls ++(90:.6cm) and ++(270:.4cm) .. (1.1,1.1);		
	\draw[more thick, BString, densely dotted] (1.1,1.1) -- (1.1,1.58);
\pgftransformxshift{1.3}
	
	\draw[thick, zString] (.05,-.05) .. controls ++(90:.85cm) and ++(270:.8cm) .. (.75, 1.45);
	\draw[thick, zString] (1.95,-.05) .. controls ++(90:.85cm) and ++(270:.8cm) .. (1.25, 1.45);
	\draw[thick, zString] (.55-.017,-.055) .. controls ++(90:.95cm) and ++(90:.95cm) .. (1.45+.017, -.055);
\end{tikzpicture}
\quad\text{ and }\quad
i_{\Tr_\cC(z\otimes A\otimes z^*\otimes B)}=
\begin{tikzpicture}[baseline=.2cm]
	\topCylinder{(.7,0)}{.3}{1}
	\draw[thick] (.7,0) arc (-180:0:.3cm);		
	\halfDottedEllipse{(.7,0)}{.3}{.1}

	\filldraw[AString] (.95, .3) circle (.025cm);
	\filldraw[BString] (1.05, .46) circle (.025cm);
	\draw[thick, AString] (.95,.91) -- (.95,.3);
	\draw[more thick, BString, densely dotted] (1.05,1.09) -- (1.05,.46);

	\draw[thick, zString] (.75,.95) -- (.75,0) arc (-180:0:.25cm) -- (1.25,.95);

\end{tikzpicture}
$$

The reader may wonder whether the two algebras $\Tr_\cC(A\otimes B)$ and $\Tr_\cC(B\otimes A)$ are related.
We show that they are, under the assumption that $\cC$ is braided pivotal and $\Phi^{\scriptscriptstyle \cZ}:\cC\to \cZ(\cM)$ is a braided pivotal functor.
Let us define the \emph{opposite algebra} of $A=(A,m,i)$ to be the algebra $A^\mathrm{op}:=(A,m \circ \beta,i)$.

\begin{prop}
\label{prop:TwoAlgebrasAreOpposite}
The traciator $\tau_{B,A}:\Tr_\cC(B\otimes A) \to \Tr_\cC(A\otimes B)$ induces an algebra isomorphism $\Tr_\cC(B\otimes A)\cong \Tr_\cC(A\otimes B)^{\mathrm{op}}$.
\end{prop}
\begin{proof}
We need to show that
\(
m_{\Tr_\cC(B\otimes A)} = \tau^- \circ m_{\Tr_\cC(A\otimes B)} \circ \beta \circ (\tau^+\otimes \tau^+).
\)
We rewrite $m_{\Tr_\cC(A\otimes B)}$ as follows
$$
m_{\Tr_\cC(A\otimes B)}
=
\begin{tikzpicture}[baseline=.4cm]

	\pairOfPants{(0,-1)}{}
	\bottomCylinder{(0,-2)}{.3}{1}
	\bottomCylinder{(1.4,-2)}{.3}{1}
	\emptyCylinder{(.7,.5)}{.3}{1}
	\halfDottedEllipse{(.7,1.5)}{.3}{.1}
	\topCylinder{(.7,1.5)}{.3}{1}

	\draw[thick, BString] (.2,-1.08) .. controls ++(90:.8cm) and ++(270:.8cm) .. (.8,.42);		
	\draw[thick, AString] (.4,-1.08) .. controls ++(90:.8cm) and ++(270:.8cm) .. (.9,.42);		
	\draw[thick, AString] (1.6,-1.08) .. controls ++(90:.8cm) and ++(270:.8cm) .. (1.1,.42);		
	\draw[thick, BString] (1.8,-1.08) .. controls ++(90:.8cm) and ++(270:.8cm) .. (1.2,.42);		
	
	\draw[thick, AString] (1.6,-1.08) -- (1.6,-2.08);		
	\draw[thick, BString] (1.8,-1.08) -- (1.8,-2.08);		
	\draw[thick, BString] (.4,-2.08) .. controls ++(90:.2cm) and ++(225:.2cm) .. (.6,-1.58);		
	\draw[thick, BString] (.2,-1.08) .. controls ++(270:.2cm) and ++(45:.2cm) .. (0,-1.48);
	\draw[thick, BString, dotted] (0,-1.48) -- (.6,-1.58);	
	\draw[thick, AString] (.2,-2.08) .. controls ++(90:.4cm) and ++(270:.4cm) .. (.4,-1.08);		

	\draw[thick, BString] (.8,.42) .. controls ++(90:.2cm) and ++(-45:.2cm) .. (.7,.92);		
	\draw[thick, BString] (1.2,1.42) .. controls ++(270:.2cm) and ++(135:.2cm) .. (1.3,1.02);
	\draw[thick, BString, dotted] (0.7,.92) -- (1.3,1.02);	
	\draw[thick, AString] (.9,.42) .. controls ++(90:.4cm) and ++(270:.4cm) .. (.8,1.42);		
	\draw[thick, AString] (1.1,.42) .. controls ++(90:.4cm) and ++(270:.4cm) .. (.9,1.42);		
	\draw[thick, BString] (1.2,.42) .. controls ++(90:.4cm) and ++(270:.4cm) .. (1.1,1.42);		

	\filldraw[AString] (.85, 2) circle (.025cm);
	\filldraw[BString] (1.15, 2) circle (.025cm);
	\draw[thick, AString] (.8,1.42) .. controls ++(90:.2cm) and ++(225:.1cm) .. (.85,2);		
	\draw[thick, AString] (.9,1.42) .. controls ++(90:.2cm) and ++(-45:.1cm) .. (.85,2);		
	\draw[thick, BString] (1.1,1.42) .. controls ++(90:.2cm) and ++(225:.1cm) .. (1.15,2);		
	\draw[thick, BString] (1.2,1.42) .. controls ++(90:.2cm) and ++(-45:.1cm) .. (1.15,2);		
	\draw[thick, BString] (1.15,2) -- (1.15,2.42);
	\draw[thick, AString] (.85,2) -- (.85,2.42);
\end{tikzpicture}
=
\begin{tikzpicture}[baseline=.7cm]

	\pairOfPants{(0,-1)}{}
	\bottomCylinder{(0,-2)}{.3}{1}
	\bottomCylinder{(1.4,-2)}{.3}{1}
	\emptyCylinder{(.7,.5)}{.3}{1}
	\halfDottedEllipse{(.7,1.5)}{.3}{.1}
	\emptyCylinder{(.7,1.5)}{.3}{1}
	\halfDottedEllipse{(.7,2.5)}{.3}{.1}
	\topCylinder{(.7,2.5)}{.3}{1}

	\draw[thick, BString] (.2,-1.08) .. controls ++(90:.8cm) and ++(270:.8cm) .. (.8,.42);		
	\draw[thick, AString] (.4,-1.08) .. controls ++(90:.8cm) and ++(270:.8cm) .. (.9,.42);		
	\draw[thick, AString] (1.6,-1.08) .. controls ++(90:.8cm) and ++(270:.8cm) .. (1.1,.42);		
	\draw[thick, BString] (1.8,-1.08) .. controls ++(90:.8cm) and ++(270:.8cm) .. (1.2,.42);		
	
	\draw[thick, AString] (1.6,-1.08) -- (1.6,-2.08);		
	\draw[thick, BString] (1.8,-1.08) -- (1.8,-2.08);		
	\draw[thick, BString] (.4,-2.08) .. controls ++(90:.2cm) and ++(225:.2cm) .. (.6,-1.58);		
	\draw[thick, BString] (.2,-1.08) .. controls ++(270:.2cm) and ++(45:.2cm) .. (0,-1.48);
	\draw[thick, AString] (.2,-2.08) .. controls ++(90:.4cm) and ++(270:.4cm) .. (.4,-1.08);		

	\draw[thick, BString] (.8,.42) .. controls ++(90:.15cm) and ++(-45:.15cm) .. (.7,.82);		
	\draw[thick, AString] (1.2,1.42) .. controls ++(270:.15cm) and ++(135:.15cm) .. (1.3,1.07);
	\draw[thick, AString] (.9,.42) .. controls ++(90:.2cm) and ++(-45:.2cm) .. (.7,.92);		
	\draw[thick, BString] (1.1,1.42) .. controls ++(270:.2cm) and ++(135:.2cm) .. (1.3,.97);
	\draw[thick, AString] (1.1,.42) .. controls ++(90:.4cm) and ++(270:.4cm) .. (.8,1.42);		
	\draw[thick, BString] (1.2,.42) .. controls ++(90:.4cm) and ++(270:.4cm) .. (.9,1.42);		

	\draw[thick, AString] (1.2,1.42) .. controls ++(90:.15cm) and ++(225:.15cm) .. (1.3,1.82);
	\draw[thick, AString] (.8,2.42) .. controls ++(270:.2cm) and ++(45:.2cm) .. (.7,1.97);		
	\draw[thick, AString] (.8,1.42) .. controls ++(90:.4cm) and ++(270:.4cm) .. (.9,2.42);		
	\draw[thick, BString] (.9,1.42) .. controls ++(90:.4cm) and ++(270:.4cm) .. (1.1,2.42);		
	\draw[thick, BString] (1.1,1.42) .. controls ++(90:.4cm) and ++(270:.4cm) .. (1.2,2.42);		

	\filldraw[AString] (.85, 3) circle (.025cm);
	\filldraw[BString] (1.15, 3) circle (.025cm);
	\draw[thick, AString] (.8,2.42) .. controls ++(90:.2cm) and ++(225:.1cm) .. (.85,3);		
	\draw[thick, AString] (.9,2.42) .. controls ++(90:.2cm) and ++(-45:.1cm) .. (.85,3);		
	\draw[thick, BString] (1.1,2.42) .. controls ++(90:.2cm) and ++(225:.1cm) .. (1.15,3);		
	\draw[thick, BString] (1.2,2.42) .. controls ++(90:.2cm) and ++(-45:.1cm) .. (1.15,3);		
	\draw[thick, BString] (1.15,3) -- (1.15,3.42);
	\draw[thick, AString] (.85,3) -- (.85,3.42);
\end{tikzpicture}
=
\begin{tikzpicture}[baseline=1cm]

	\bottomCylinder{(0,-2)}{.3}{1}
	\bottomCylinder{(1.4,-2)}{.3}{1}
	
	\draw[thick, BString] (.2,-1.08) .. controls ++(90:.8cm) and ++(270:.7cm) .. (1.6,.92);		
	\draw[thick, AString] (.4,-1.08) .. controls ++(90:.8cm) and ++(270:.7cm) .. (1.8,.92);		

	\inverseBraid{(0,-1)}{.3}{2}
	
	\pairOfPants{(0,1)}{}
	\emptyCylinder{(.7,2.5)}{.3}{1}
	\halfDottedEllipse{(0,-1)}{.3}{.1}
	\halfDottedEllipse{(1.4,-1)}{.3}{.1}
	\halfDottedEllipse{(.7,2.5)}{.3}{.1}
	\halfDottedEllipse{(.7,3.5)}{.3}{.1}
	\topCylinder{(.7,3.5)}{.3}{1}

	\draw[thick, AString] (.2,.92) .. controls ++(90:.8cm) and ++(270:.8cm) .. (.8,2.42);		
	\draw[thick, BString] (.4,.92) .. controls ++(90:.8cm) and ++(270:.8cm) .. (.9,2.42);		
	\draw[thick, BString] (1.6,.92) .. controls ++(90:.8cm) and ++(270:.8cm) .. (1.1,2.42);		
	\draw[thick, AString] (1.8,.92) .. controls ++(90:.8cm) and ++(270:.8cm) .. (1.2,2.42);		

	\draw[thick, AString] (1.6,-1.08) -- (1.6,-2.08);		
	\draw[thick, BString] (1.8,-1.08) -- (1.8,-2.08);		
	\draw[thick, AString] (.2,-2.08) .. controls ++(90:.2cm) and ++(-45:.2cm) .. (0,-1.58);		
	\draw[thick, AString] (.4,-1.08) .. controls ++(270:.2cm) and ++(135:.2cm) .. (.6,-1.48);
	\draw[thick, BString] (.4,-2.08) .. controls ++(90:.4cm) and ++(270:.4cm) .. (.2,-1.08);		

	\draw[thick, AString] (1.6,-1.08) .. controls ++(90:.8cm) and ++(270:.7cm) .. (.2,.92);		
	\draw[thick, BString] (1.8,-1.08) .. controls ++(90:.8cm) and ++(270:.7cm) .. (.4,.92);		

	\draw[thick, AString] (1.2,2.42) .. controls ++(90:.15cm) and ++(225:.15cm) .. (1.3,2.82);
	\draw[thick, AString] (.8,3.42) .. controls ++(270:.2cm) and ++(45:.2cm) .. (.7,2.97);		
	\draw[thick, AString] (.8,2.42) .. controls ++(90:.4cm) and ++(270:.4cm) .. (.9,3.42);		
	\draw[thick, BString] (.9,2.42) .. controls ++(90:.4cm) and ++(270:.4cm) .. (1.1,3.42);		
	\draw[thick, BString] (1.1,2.42) .. controls ++(90:.4cm) and ++(270:.4cm) .. (1.2,3.42);		

	\filldraw[AString] (.85, 4) circle (.025cm);
	\filldraw[BString] (1.15, 4) circle (.025cm);
	\draw[thick, AString] (.8,3.42) .. controls ++(90:.2cm) and ++(225:.1cm) .. (.85,4);		
	\draw[thick, AString] (.9,3.42) .. controls ++(90:.2cm) and ++(-45:.1cm) .. (.85,4);		
	\draw[thick, BString] (1.1,3.42) .. controls ++(90:.2cm) and ++(225:.1cm) .. (1.15,4);		
	\draw[thick, BString] (1.2,3.42) .. controls ++(90:.2cm) and ++(-45:.1cm) .. (1.15,4);		
	\draw[thick, BString] (1.15,4) -- (1.15,4.42);
	\draw[thick, AString] (.85,4) -- (.85,4.42);
	
\end{tikzpicture}
$$
and compute:
$$
\tau^{-}_{A,B} \circ m_{\Tr_\cC(A\otimes B)} \circ \beta \circ (\tau^+_{B,A}\otimes \tau^+_{B,A})
=
\begin{tikzpicture}[baseline=.2cm]

	\pairOfPants{(0,-1)}{}
	\bottomCylinder{(0,-3)}{.3}{1}
	\bottomCylinder{(1.4,-3)}{.3}{1}
	\emptyCylinder{(0,-2)}{.3}{1}
	\emptyCylinder{(1.4,-2)}{.3}{1}
	\halfDottedEllipse{(0,-2)}{.3}{.1}
	\halfDottedEllipse{(1.4,-2)}{.3}{.1}
	\emptyCylinder{(.7,.5)}{.3}{1}
	\halfDottedEllipse{(.7,1.5)}{.3}{.1}
	\emptyCylinder{(.7,1.5)}{.3}{1}
	\halfDottedEllipse{(.7,2.5)}{.3}{.1}
	\topCylinder{(.7,2.5)}{.3}{1}

	\draw[thick, AString] (.2,-1.08) .. controls ++(90:.8cm) and ++(270:.8cm) .. (.8,.42);		
	\draw[thick, BString] (.4,-1.08) .. controls ++(90:.8cm) and ++(270:.8cm) .. (.9,.42);		
	\draw[thick, BString] (1.6,-1.08) .. controls ++(90:.8cm) and ++(270:.8cm) .. (1.1,.42);		
	\draw[thick, AString] (1.8,-1.08) .. controls ++(90:.8cm) and ++(270:.8cm) .. (1.2,.42);		
	
	\draw[thick, AString] (.4,-3.08) .. controls ++(90:.2cm) and ++(225:.2cm) .. (.6,-2.58);		
	\draw[thick, AString] (.2,-2.08) .. controls ++(270:.2cm) and ++(45:.2cm) .. (0,-2.48);
	\draw[thick, BString] (.2,-3.08) .. controls ++(90:.4cm) and ++(270:.4cm) .. (.4,-2.08);		
	\draw[thick, AString] (1.8,-3.08) .. controls ++(90:.2cm) and ++(225:.2cm) .. (2,-2.58);		
	\draw[thick, AString] (1.6,-2.08) .. controls ++(270:.2cm) and ++(45:.2cm) .. (1.4,-2.48);
	\draw[thick, BString] (1.6,-3.08) .. controls ++(90:.4cm) and ++(270:.4cm) .. (1.8,-2.08);		

	\draw[thick, AString] (.2,-1.08) -- (.2,-2.08);		
	\draw[thick, BString] (.4,-1.08) -- (.4,-2.08);		
	\draw[thick, AString] (1.6,-2.08) .. controls ++(90:.2cm) and ++(-45:.2cm) .. (1.4,-1.58);		
	\draw[thick, AString] (1.8,-1.08) .. controls ++(270:.2cm) and ++(135:.2cm) .. (2,-1.48);
	\draw[thick, BString] (1.8,-2.08) .. controls ++(90:.4cm) and ++(270:.4cm) .. (1.6,-1.08);		

	\draw[thick, AString] (1.2,.42) .. controls ++(90:.15cm) and ++(225:.15cm) .. (1.3,.82);
	\draw[thick, AString] (.8,1.42) .. controls ++(270:.2cm) and ++(45:.2cm) .. (.7,.97);		
	\draw[thick, AString] (.8,.42) .. controls ++(90:.4cm) and ++(270:.4cm) .. (.9,1.42);		
	\draw[thick, BString] (.9,.42) .. controls ++(90:.4cm) and ++(270:.4cm) .. (1.1,1.42);		
	\draw[thick, BString] (1.1,.42) .. controls ++(90:.4cm) and ++(270:.4cm) .. (1.2,1.42);		

	\filldraw[AString] (.85, 2) circle (.025cm);
	\filldraw[BString] (1.15, 2) circle (.025cm);
	\draw[thick, AString] (.8,1.42) .. controls ++(90:.2cm) and ++(225:.1cm) .. (.85,2);		
	\draw[thick, AString] (.9,1.42) .. controls ++(90:.2cm) and ++(-45:.1cm) .. (.85,2);		
	\draw[thick, BString] (1.1,1.42) .. controls ++(90:.2cm) and ++(225:.1cm) .. (1.15,2);		
	\draw[thick, BString] (1.2,1.42) .. controls ++(90:.2cm) and ++(-45:.1cm) .. (1.15,2);		
	\draw[thick, BString] (1.15,2) -- (1.15,2.42);
	\draw[thick, AString] (.85,2) -- (.85,2.42);

	\draw[thick, AString] (.85,2.42) .. controls ++(90:.15cm) and ++(-45:.15cm) .. (.7,2.82);
	\draw[thick, AString] (1.15,3.42) .. controls ++(270:.2cm) and ++(135:.2cm) .. (1.3,2.97);		
	\draw[thick, BString] (1.15,2.42) .. controls ++(90:.4cm) and ++(270:.4cm) .. (.85,3.42);		
\end{tikzpicture}
=
\begin{tikzpicture}[baseline=.4cm]

	\pairOfPants{(0,-1)}{}
	\bottomCylinder{(0,-2)}{.3}{1}
	\bottomCylinder{(1.4,-2)}{.3}{1}
	\emptyCylinder{(.7,.5)}{.3}{1}
	\halfDottedEllipse{(.7,1.5)}{.3}{.1}
	\topCylinder{(.7,1.5)}{.3}{1}

	\draw[thick, AString] (.2,-1.08) .. controls ++(90:.8cm) and ++(270:.8cm) .. (.8,.42);		
	\draw[thick, BString] (.4,-1.08) .. controls ++(90:.8cm) and ++(270:.8cm) .. (.9,.42);		
	\draw[thick, BString] (1.6,-1.08) .. controls ++(90:.8cm) and ++(270:.8cm) .. (1.1,.42);		
	\draw[thick, AString] (1.8,-1.08) .. controls ++(90:.8cm) and ++(270:.8cm) .. (1.2,.42);		
	
	\draw[thick, BString] (1.6,-1.08) -- (1.6,-2.08);		
	\draw[thick, AString] (1.8,-1.08) -- (1.8,-2.08);		
	\draw[thick, AString] (.4,-2.08) .. controls ++(90:.2cm) and ++(225:.2cm) .. (.6,-1.58);		
	\draw[thick, AString] (.2,-1.08) .. controls ++(270:.2cm) and ++(45:.2cm) .. (0,-1.48);
	\draw[thick, BString] (.2,-2.08) .. controls ++(90:.4cm) and ++(270:.4cm) .. (.4,-1.08);		

	\draw[thick, AString] (.8,.42) .. controls ++(90:.2cm) and ++(-45:.2cm) .. (.7,.92);		
	\draw[thick, AString] (1.2,1.42) .. controls ++(270:.2cm) and ++(135:.2cm) .. (1.3,1.02);
	\draw[thick, BString] (.9,.42) .. controls ++(90:.4cm) and ++(270:.4cm) .. (.8,1.42);		
	\draw[thick, BString] (1.1,.42) .. controls ++(90:.4cm) and ++(270:.4cm) .. (.9,1.42);		
	\draw[thick, AString] (1.2,.42) .. controls ++(90:.4cm) and ++(270:.4cm) .. (1.1,1.42);		

	\filldraw[BString] (.85, 2) circle (.025cm);
	\filldraw[AString] (1.15, 2) circle (.025cm);
	\draw[thick, BString] (.8,1.42) .. controls ++(90:.2cm) and ++(225:.1cm) .. (.85,2);		
	\draw[thick, BString] (.9,1.42) .. controls ++(90:.2cm) and ++(-45:.1cm) .. (.85,2);		
	\draw[thick, AString] (1.1,1.42) .. controls ++(90:.2cm) and ++(225:.1cm) .. (1.15,2);		
	\draw[thick, AString] (1.2,1.42) .. controls ++(90:.2cm) and ++(-45:.1cm) .. (1.15,2);		
	\draw[thick, AString] (1.15,2) -- (1.15,2.42);
	\draw[thick, BString] (.85,2) -- (.85,2.42);
\end{tikzpicture}
=
m_{\Tr_\cC(B\otimes A)}.
$$
The remaining statement $i_{\Tr_\cC(B\otimes A)}= \tau^-\circ i_{\Tr_\cC(A\otimes B)}$ is straightforward, and left to the reader.
\end{proof}

\begin{rem}
\label{rem:TwistingAlgebraStructures}
Unless $\cC$ is symmetric, there are actually two opposite algebras: $(A, m\circ \beta^+,i)$ and $(A, m\circ \beta^-,i)$, the $(+)$-opposite and the $(-)$-opposite algebras.
However, if $\cC$ is balanced then the twist map $\theta_A:A\to A$ provides an isomorphism between $(-)$-opposite and the $(+)$-opposite algebras,
so that there is, in fact, only one opposite algebra.
\end{rem}

\appendix								


\section{Appendix}\label{App:Appendix}

\subsection{Commutative algebras in braided tensor categories}
\label{sec:Examples}

In this appendix, we expand on and fill in the details of Example \ref{ex: the short version of the thing in the appendix}.

To begin with, we introduce a convenient graphical notation for tensor products over algebra objects.
Let $a$ be an algebra object in a tensor category $\cC$, and let $x,y\in \cC$ be right and left $a$-modules.
Denoting $a$ by an orange strand and $x,y$ by green and blue strands, the tensor product over $a$ is the coequalizer $x\otimes_a y$ of the two morphisms
$$
\begin{tikzpicture}[baseline=-.1cm]
	\draw[thick, orange] (.25,-.5)  .. controls ++(90:.5cm) and ++(220:.25cm) .. (.5,.1);
	\draw[thick, DarkGreen] (0,.5) -- (0,-.5);
	\draw[thick, blue] (.5,.5) -- (.5,-.5);
\end{tikzpicture}
\,\,,\,\,\,\,
\begin{tikzpicture}[baseline=-.1cm, xscale=-1]
	\draw[thick, orange] (.25,-.5)  .. controls ++(90:.5cm) and ++(220:.25cm) .. (.5,.1);
	\draw[thick, DarkGreen] (0,.5) -- (0,-.5);
	\draw[thick, blue] (.5,.5) -- (.5,-.5);
\end{tikzpicture}\,\,:\,x\otimes a\otimes y \to x\otimes y.
$$
(we assume that $\cC$ has all the necessary colimits).
It is convenient to denote $x\otimes_a y$ by an orange ribbon bordered by green and blue edges:
$$
\begin{tikzpicture}[baseline=-.1cm]
	\filldraw[orange!25] (0,.48) rectangle (.5,-.48);
	\draw[thick, DarkGreen] (0,.5) -- (0,-.5);
	\draw[thick, blue] (.5,.5) -- (.5,-.5);
	\node[scale=.9] at (.25,-.8) {$x\otimes_a y$};
\end{tikzpicture}
$$
We draw the natural projection
$
\pi_{x,y}=\begin{tikzpicture}[baseline=-.1cm, scale=.8]
	\filldraw[orange!25] (0,-.2) -- (0,.48) -- (.5,.48) -- (.5,-.2) to[out=90+10,in=90-10, looseness=1.5] (0,-.2);
	\draw[orange] (.5,-.2) to[out=90+10,in=90-10, looseness=1.5] (0,-.2);
	\draw[thick, DarkGreen] (0,.5) -- (0,-.5);
	\draw[thick, blue] (.5,.5) -- (.5,-.5);
\end{tikzpicture}
:x\otimes y\to x\otimes_a y
$.
It makes the equation
$$
\begin{tikzpicture}[baseline=-.4cm]
	\clip (-.1,.4) rectangle (.6,-.8);
	\draw[thick, orange] (.25,-.5-.5)  .. controls ++(90:.5cm) and ++(220:.25cm) .. (.5,.1-.5);
	\filldraw[orange!25] (0,-.2) -- (0,.48) -- (.5,.48) -- (.5,-.2) to[out=90+10,in=90-10, looseness=1.5] (0,-.2);
	\draw[orange] (.5,-.2) to[out=90+10,in=90-10, looseness=1.5] (0,-.2);
	\draw[thick, DarkGreen] (0,.5) -- (0,-.5-.5);
	\draw[thick, blue] (.5,.5) -- (.5,-.5-.5);
\end{tikzpicture}
\,=\,
\begin{tikzpicture}[baseline=-.4cm, xscale=-1]
	\clip (-.1,.4) rectangle (.6,-.8);
	\draw[thick, orange] (.25,-.5-.5)  .. controls ++(90:.5cm) and ++(220:.25cm) .. (.5,.1-.5);
	\filldraw[orange!25] (0,-.2) -- (0,.48) -- (.5,.48) -- (.5,-.2) to[out=90+10,in=90-10, looseness=1.5] (0,-.2);
	\draw[orange] (.5,-.2) to[out=90+10,in=90-10, looseness=1.5] (0,-.2);
	\draw[thick, DarkGreen] (0,.5) -- (0,-.5-.5);
	\draw[thick, blue] (.5,.5) -- (.5,-.5-.5);
\end{tikzpicture}
$$
graphically motivated as
the orange strand (which we might as well have represented by an orange ribbon) is allowed to slide along the orange edge of the surface.
As in usual algebra, there are canonical isomorphisms $x\otimes_a a\cong x$ and $a\otimes_a y\cong y$.

When $\cC$ is a braided tensor category and $a\in \cC$ is a commutative algebra object
\[
\qquad\qquad\quad\begin{tikzpicture}[baseline=.3cm, thick, orange, scale=.9]
	\coordinate (a) at (-.3,.85);
	\coordinate (b) at (-.3,1.1);
	\filldraw (a) circle (.02cm);
	\draw (a) -- (b);
	\draw 	(a) +(-.3,-.5) -- +(-.3,-.33) to[in=-90-65, out =90] (a)
			(a) +(.3,-.5) -- +(.3,-.32) to[in=-90+65, out =90] (a);
	\braiding{(-.6,.18)}{.6}{.35}
\end{tikzpicture}
\,\,=\,\,
\begin{tikzpicture}[baseline=-.2cm, thick, orange, scale=.9]
	\coordinate (a) at (0,0);
	\coordinate (b) at (0,.5);
	\filldraw (a) circle (.02cm);
	\draw (a) -- (b);
	\draw (a) +(-.3,-.5) to[in=-90-45, out =90] (a)   (a) +(.3,-.5) to[in=-90+45, out =90] (a);
\end{tikzpicture}
\qquad\qquad\quad\text{(Commutativity)}
\]
then it makes sense to tensor two left $a$-modules, and the result is again an $a$-module.
Specifically, if $x$, $y$ are left $a$-modules, then $x\otimes_a y$ is the coequalizer of the two morphisms
$$
\begin{tikzpicture}[baseline=-.1cm]
	\draw[thick, orange] (.25,-.6)  .. controls ++(90:.5cm) and ++(220:.25cm) .. (.5,.1);
	\draw[thick, DarkGreen] (0,.6) -- (0,-.6);
	\draw[thick, blue] (.5,.6) -- (.5,-.6);
\end{tikzpicture}
\,\,\,\,,
\begin{tikzpicture}[baseline=-.1cm]
	\draw[thick, orange] (.25,-.6)  .. controls ++(90:.4cm) and ++(220:.5cm) .. (0,.1);
	\draw[super thick, white] (0,-.1) -- (0,-.6);
	\draw[thick, DarkGreen] (0,.6) -- (0,-.6);
	\draw[thick, blue] (.5,.6) -- (.5,-.6);
\end{tikzpicture}
\,\,:\,x\otimes a\otimes y \to x\otimes y
$$
Instead of using a ribbon as above, it is preferable in this situation to represent $x\otimes_a y$ by a thick filled orange `rope' 
with green and blue strands on its surface:
$$
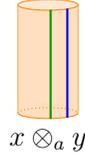
\begin{tikzpicture}[baseline=-.1cm, yscale=1.2]
	\fill [orange!25] (0,-.6) ++(-.4,0)  arc (-180:0:.4 and .06) -- ++(0,1.2) arc (0:-180:.4 and .06) -- ++(0,-1.2);
	\draw [semithick, orange!50] (0,-.6) ++(-.4,0) -- ++ (0,1.2);
	\draw [semithick, orange!50] (0,-.6) ++(.4,0) -- ++ (0,1.2);
	\fill [orange!17, draw=orange, very thin] (0,.6)  circle (.4 and .06);
	\draw [orange, very thin] (0,-.6) ++(-.4,0)  arc (-180:0:.4 and .06);
	\draw [orange, densely dotted] (0,-.6) ++(-.4,0)  arc (180:0:.4 and .06);
	\draw[thick, DarkGreen] (0,.6) ++ (-85:.4 and .06) ++(0,.005) -- ++(0,-1.21);
	\draw[thick, blue] (0,.6) ++ (-50:.4 and .06) ++(0,.005) -- ++(0,-1.21);
	\node[scale=.9] at (0,-.9) {$x\otimes_a y$};
\end{tikzpicture}
$$
Extrapolating the notation, we can also denote a single $a$-module $x$ by\,\,$
\begin{tikzpicture}[yscale=.8, xscale=.5]
	\useasboundingbox (-.55,-.3) rectangle (.55,.3);
	\fill [orange!25] (0,-.6) ++(-.4,0)  arc (-180:0:.4 and .06) -- ++(0,1.2) arc (0:-180:.4 and .06) -- ++(0,-1.2);
	\draw [semithick, orange!50] (0,-.6) ++(-.4,0) -- ++ (0,1.2);
	\draw [semithick, orange!50] (0,-.6) ++(.4,0) -- ++ (0,1.2);
	\fill [orange!17, draw=orange, very thin] (0,.6)  circle (.4 and .06);
	\draw [orange, very thin] (0,-.6) ++(-.4,0)  arc (-180:0:.4 and .06);
	\draw [orange, densely dotted] (0,-.6) ++(-.4,0)  arc (180:0:.4 and .06);
	\draw[thick, DarkGreen] (0,.6) ++ (-65:.4 and .06) ++(0,.005) -- ++(0,-1.21);
	\node[scale=1] at (.1,-.85) {$\scriptstyle x$};
\end{tikzpicture}
$,
where the presence of the orange rope is just an indication that $x$ is an $a$-module.

Let us now assume that $a$ is commutative and separable, also known as \'etale \cite[\S3]{MR3039775} (such an 
algebra is automatically Frobenius -- combine \cite[Rem.\,3.2.ii]{MR3039775}, \cite[Prop.\,3.1.ii]{MR1976459}, and Proposition \ref{prop   :Frobenius} for a proof).
The dual of an $a$-module $x$ is then naturally also an $a$-module.
The evaluation $\ev_x:x^*\otimes x\to 1$ induces a corresponding evaluation morphism $\bar\ev_x:x^*\otimes_a x\to a$ in the category of $a$-modules, and the same holds for
coevaluations \cite[\S3.3]{MR3039775}.
We write
\[
\cM := \Mod_\cC(a)
\]
for the category of $a$-modules in $\cC$.
This is a rigid tensor category under the operation $x,y\mapsto x\otimes_a y$.

Let $\Phi:\cC\to \cM$ be the free module functor, given by $\Phi(x):=a\otimes x$ on objects and $\Phi(f: x\to y) = \id_a\otimes f : a\otimes x\to a\otimes y$ on morphisms.
We then have a functor $\Phi^{\scriptscriptstyle \cZ}: \cC\to \cZ(\cM)$ given by $\Phi^{\scriptscriptstyle \cZ}(x):=(\Phi(x),e_{\Phi(x)})$, where the half-braiding $e_{\Phi(x)}$ is given by
$$
e_{\Phi(x),y}:\Phi(x)\otimes_a y\cong x\otimes y\xrightarrow{\,\,\,\beta_{x,y}\,\,\,}y\otimes x\cong y\otimes_a\Phi(x).
$$
In diagrams, this is denoted
\begin{equation}\label{eq:   e_Phi(x)(y)}
e_{\Phi(x),y}\,=\,\,
\begin{tikzpicture}[baseline=-.1cm, yscale=1.2]
	\fill [orange!25] (0,-.6) ++(-.4,0)  arc (-180:0:.4 and .06) -- ++(0,1.2) arc (0:-180:.4 and .06) -- ++(0,-1.2);
	\draw [semithick, orange!50] (0,-.6) ++(-.4,0) -- ++ (0,1.2);
	\draw [semithick, orange!50] (0,-.6) ++(.4,0) -- ++ (0,1.2);
	\fill [orange!17, draw=orange, very thin] (0,.6)  circle (.4 and .06);
	\draw [orange, very thin] (0,-.6) ++(-.4,0)  arc (-180:0:.4 and .06);
	\draw [orange, densely dotted] (0,-.6) ++(-.4,0)  arc (180:0:.4 and .06);

	\def\shiftby{(.04,-.1)}
	\path (0,.6) ++ (-100:.4 and .06) coordinate (A);
	\path (0,-.6) ++ \shiftby ++ (-100:.4 and .06) ++ (0,-.05) coordinate (B);
	\path (0,-.6) ++ (-50:.4 and .06) coordinate (C);
	\path (0,.6) ++ \shiftby ++ (-50:.4 and .06) coordinate (D);

	\draw[thick, blue] (C) .. controls ++(90:.8cm) and ++(270:.8cm) .. (A);
	\draw[line width=4, white, line cap=rect] (B) ++ (0,.05) .. controls ++(90:.8cm) and ++(270:.8cm) .. (D) -- ++ (0,.25);
	\draw[thick, DarkGreen] (B) -- ++ (0,.05) .. controls ++(90:.8cm) and ++(270:.8cm) .. (D) -- ++ (0,.28);
\end{tikzpicture}
\end{equation}
The blue strand represents $y\in \cM$, and lies on the surface of the rope.
The green strand represents $x\in \cC$, and doesn't touch the surface.
The above map is visibly invertible, and natural in $y$.
The hexagon axiom \eqref{eq: hexagon identity} for the half-braiding is the equation
\begin{equation}\label{eq: 1/2-braiding axiom check}
\begin{tikzpicture}[baseline=.5cm, yscale=.9]
	\def\posit{(.39,-.65)}
	\def\heightofcylinder{2.8}
	\def\widthof{.6}
	\def\widthofcylinder{.6 and .1}
	\fill [orange!25] \posit ++(-\widthof,0)  arc (-180:0:\widthofcylinder) -- ++(0,\heightofcylinder) arc (0:-180:\widthofcylinder) -- ++(0,-\heightofcylinder);
	\draw [semithick, orange!50] \posit ++(-\widthof,0) -- ++ (0,\heightofcylinder);
	\draw [semithick, orange!50] \posit ++(\widthof,0) -- ++ (0,\heightofcylinder);
	\fill [orange!17, draw=orange, very thin] \posit ++ (0,\heightofcylinder)  circle (.6 and .1);
	\draw [orange, very thin] \posit ++(-\widthof,0)  arc (-180:0:\widthofcylinder);
	\draw [orange, densely dotted] \posit ++(-\widthof,0)  arc (180:0:\widthofcylinder);

	\draw[thick, blue] (.5,-.74) .. controls ++(90:.8cm) and ++(270:.8cm) .. (.1,.7) -- (.1,2.065);
	\draw[thick, red] (.9,-.7) -- (.9,.7) .. controls ++(90:.8cm) and ++(270:.8cm) .. (.5,2.055);
	\draw[line width=4, white] (.1,-.87) -- (.1,-.7) .. controls ++(90:.8cm) and ++(270:.8cm) .. (.5,.7)  .. controls ++(90:.8cm) and ++(270:.8cm) .. (.9,2.1) -- (.9,2.35);
	\draw[thick, DarkGreen] (.1,-.87) -- (.1,-.7) .. controls ++(90:.8cm) and ++(270:.8cm) .. (.5,.7) .. controls ++(90:.8cm) and ++(270:.8cm) .. (.9,2.1) -- (.9,2.35);
\end{tikzpicture}
\,\,=\,\,
\begin{tikzpicture}[baseline=-.1cm]
	\def\posit{(.245,-.95)}
	\def\heightofcylinder{2}
	\def\widthof{.5}
	\def\widthofcylinder{.5 and .1}
	\fill [orange!25] \posit ++(-\widthof,0)  arc (-180:0:\widthofcylinder) -- ++(0,\heightofcylinder) arc (0:-180:\widthofcylinder) -- ++(0,-\heightofcylinder);
	\draw [semithick, orange!50] \posit ++(-\widthof,0) -- ++ (0,\heightofcylinder);
	\draw [semithick, orange!50] \posit ++(\widthof,0) -- ++ (0,\heightofcylinder);
	\fill [orange!17, draw=orange, very thin] \posit ++ (0,\heightofcylinder)  circle (.5 and .1);
	\draw [orange, very thin] \posit ++(-\widthof,0)  arc (-180:0:\widthofcylinder);
	\draw [orange, densely dotted] \posit ++(-\widthof,0)  arc (180:0:\widthofcylinder);

	\draw[thick, blue] (.5,-1.035) -- (.5,-.7) .. controls ++(90:.8cm) and ++(270:.8cm) .. (.1,.7) -- (.1,.955);
	\draw[thick, red] (.6,-1.02) -- (.6,-.7) .. controls ++(90:.8cm) and ++(270:.8cm) .. (.2,.7) -- (.2,.95);
	\draw[line width=4, white]  (.1,-1.18) -- (.1,-.7) .. controls ++(90:.8cm) and ++(270:.8cm) .. (.6,.7) -- (.6,1.25);
	\draw[thick, DarkGreen] (.1,-1.18) -- (.1,-.7) .. controls ++(90:.8cm) and ++(270:.8cm) .. (.6,.7) -- (.6,1.25);
\end{tikzpicture}
\end{equation}
It holds as both sides of \eqref{eq: 1/2-braiding axiom check} fit into the same commutative diagram
\[
\xymatrix{
y\otimes z\otimes x \ar@{->>}[rr]^(.33){\pi_{yz}\otimes 1}\ar@{<-}[d]_{(1\otimes\beta_{x,z})(\beta_{x,y}\otimes 1)\,=\,\beta_{x,y\otimes z}} && y\otimes_a z\otimes x\cong  y\otimes_a z\otimes_a \Phi(x) \ar@{<-}@<30pt>[d]^{\eqref{eq: 1/2-braiding axiom check}}
\\
x\otimes y\otimes z	\ar@{->>}[rr]^(.33){1\otimes \pi_{yz}} && x \otimes  y\otimes_a z\cong \Phi(x) \otimes_a  y\otimes_a z
}
\qquad\quad
\]
with surjective horizontal maps.
The isomorphism $(a\otimes x)\otimes_a (a\otimes y) \cong a\otimes x\otimes y$ 
endows $\Phi^{\scriptscriptstyle \cZ}$ with the structure of a tensor functor.
Indeed, the half braidings $e_{\Phi(x)\otimes_a \Phi(y)}=(e_{\Phi(x)}\otimes_a 1)\circ(1\otimes_a e_{\Phi(y)})$ and $e_{\Phi(x\otimes y)}$ are naturally isomorphic,
as required for $\Phi^{\scriptscriptstyle \cZ}$ to be a tensor functor:
$$
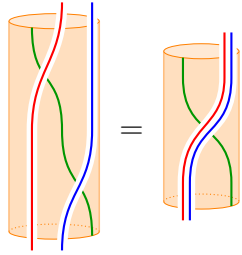
\begin{tikzpicture}[baseline=.6cm]
	\def\posit{(.395,-.65)}
	\def\heightofcylinder{2.8}
	\def\widthof{.6}
	\def\widthofcylinder{.6 and .1}
	\fill [orange!25] \posit ++(-\widthof,0)  arc (-180:0:\widthofcylinder) -- ++(0,\heightofcylinder) arc (0:-180:\widthofcylinder) -- ++(0,-\heightofcylinder);
	\draw [semithick, orange!50] \posit ++(-\widthof,0) -- ++ (0,\heightofcylinder);
	\draw [semithick, orange!50] \posit ++(\widthof,0) -- ++ (0,\heightofcylinder);
	\fill [orange!17, draw=orange, very thin] \posit ++ (0,\heightofcylinder)  circle (.6 and .1);
	\draw [orange, very thin] \posit ++(-\widthof,0)  arc (-180:0:\widthofcylinder);
	\draw [orange, densely dotted] \posit ++(-\widthof,0)  arc (180:0:\widthofcylinder);

	\draw[thick, DarkGreen] (.9,-.705) .. controls ++(90:.8cm) and ++(270:.8cm) .. (.5,.7) .. controls ++(90:.8cm) and ++(270:.8cm) .. (.1,2.065);
	\draw[line width=4, white] (.1,-.9) -- (.1,.7) .. controls ++(90:.8cm) and ++(270:.8cm) .. (.5,2.4);
	\draw[thick, red] (.1,-.9)-- (.1,.7)  .. controls ++(90:.8cm) and ++(270:.8cm) .. (.5,2.4);
	\draw[line width=4, white] (.5,-.9) .. controls ++(90:.8cm) and ++(270:.8cm) .. (.9,.7) -- (.9,2.4);
	\draw[thick, blue] (.5,-.9) .. controls ++(90:.8cm) and ++(270:.8cm) .. (.9,.7) -- (.9,2.4);
\end{tikzpicture}
\,\,=\,\,
\begin{tikzpicture}[baseline=-.1cm]
	\def\posit{(.245,-.95)}
	\def\heightofcylinder{2}
	\def\widthof{.5}
	\def\widthofcylinder{.5 and .1}
	\fill [orange!25] \posit ++(-\widthof,0)  arc (-180:0:\widthofcylinder) -- ++(0,\heightofcylinder) arc (0:-180:\widthofcylinder) -- ++(0,-\heightofcylinder);
	\draw [semithick, orange!50] \posit ++(-\widthof,0) -- ++ (0,\heightofcylinder);
	\draw [semithick, orange!50] \posit ++(\widthof,0) -- ++ (0,\heightofcylinder);
	\fill [orange!17, draw=orange, very thin] \posit ++ (0,\heightofcylinder)  circle (.5 and .1);
	\draw [orange, very thin] \posit ++(-\widthof,0)  arc (-180:0:\widthofcylinder);
	\draw [orange, densely dotted] \posit ++(-\widthof,0)  arc (180:0:\widthofcylinder);

\pgftransformxscale{.92}
	\draw[thick, DarkGreen] (.7,-1.01) -- (.7,-.7) .. controls ++(90:.8cm) and ++(270:.8cm) .. (0,.7) -- (0,.965);
	\draw[line width=4, white] (0,-1.2) -- (0,-.7) .. controls ++(90:.8cm) and ++(270:.7cm) .. (.6,.7) -- (.6,1.3);
	\draw[line width=4, white] (.1,-1.2) -- (.1,-.7) .. controls ++(90:.8cm) and ++(270:.8cm) .. (.7,.7) -- (.7,1.3);
	\draw[thick, red] (0,-1.2) -- (0,-.7) .. controls ++(90:.8cm) and ++(270:.7cm) .. (.6,.7) -- (.6,1.3);
	\draw[thick, blue] (.1,-1.2) -- (.1,-.7) .. controls ++(90:.8cm) and ++(270:.8cm) .. (.7,.7) -- (.7,1.3);
\end{tikzpicture}
$$
Finally, the functor $\Phi^{\scriptscriptstyle \cZ}$ is braided because the morphism\,\,$
\begin{tikzpicture}[baseline=-.1cm, xscale=.7, yscale=.8]
	\useasboundingbox (-.55,-.3) rectangle (.55,.3);
	\fill [orange!25] (0,-.6) ++(-.4,0)  arc (-180:0:.4 and .06) -- ++(0,1.2) arc (0:-180:.4 and .06) -- ++(0,-1.2);
	\draw [semithick, orange!50] (0,-.6) ++(-.4,0) -- ++ (0,1.2);
	\draw [semithick, orange!50] (0,-.6) ++(.4,0) -- ++ (0,1.2);
	\fill [orange!17, draw=orange, very thin] (0,.6)  circle (.4 and .06);
	\draw [orange, very thin] (0,-.6) ++(-.4,0)  arc (-180:0:.4 and .06);
	\draw [orange, densely dotted] (0,-.6) ++(-.4,0)  arc (180:0:.4 and .06);

	\def\shiftby{(0,0)}
	\path (0,.6) ++ (-100:.4 and .06) coordinate (A);
	\path (0,-.6) ++ \shiftby ++ (-100:.4 and .06) ++ (0,-.1) coordinate (B);
	\path (0,-.6) ++ (-50:.4 and .06) ++ (0,-.1) coordinate (C);
	\path (0,.6) ++ \shiftby ++ (-50:.4 and .06) coordinate (D);

	\draw[line width=3.5, white] (C) -- ++ (0,.1) .. controls ++(90:.8cm) and ++(270:.8cm) .. (A) -- ++ (0,.2);
	\draw[thick, blue] (C) -- ++ (0,.1) .. controls ++(90:.8cm) and ++(270:.8cm) .. (A) -- ++ (0,.2);
	\draw[line width=3.5, white] (B) ++ (0,.05) .. controls ++(90:.8cm) and ++(270:.8cm) .. (D) -- ++ (0,.25);
	\draw[thick, DarkGreen] (B) -- ++ (0,.1) .. controls ++(90:.8cm) and ++(270:.8cm) .. (D) -- ++ (0,.2);
\end{tikzpicture}
$
is both the image of $\beta_{x,y}$ under $\Phi^{\scriptscriptstyle \cZ}$, and a special case of \eqref{eq:   e_Phi(x)(y)}.

Let us now furthermore assume that $\cC$ is pivotal, and that $\theta_a=1$,
where the twist maps are given by \eqref{def: theta1 -- PRE}. 
Then for any $a$-module $x$, the pivotal map $\varphi_x:x\to x^{**}$ is a map of $a$-modules \cite[Thm. 1.17]{MR1936496}\footnote{To match our conventions, one should replace all over-crossings by under-crossings in \cite{MR1936496}.}.
This equips the category $\Mod_\cC(a)$ of $a$-modules with the structure of a pivotal tensor category.

We wish to check that $\Phi^{\scriptscriptstyle \cZ}:\cC\to\cZ(\cM)$ is a pivotal functor, i.e., that the equation $\delta_x^* \circ \varphi_{\Phi^{\scriptscriptstyle \cZ}(x)} =  \delta_{x^*} \circ \Phi^{\scriptscriptstyle \cZ}(\varphi_x)$ holds.
Let $\omega:a\to a^*$ be the isomorphism induced by the Frobenius pairing $\epsilon\circ\mu:a\otimes a\to 1$.
The canonical isomorphism $\delta_x:\Phi^{\scriptscriptstyle \cZ}(x^*)\to \Phi^{\scriptscriptstyle \cZ}(x)^*$ is then given by
\[
\Phi(x^*)= a\otimes x^* \xrightarrow{\,\beta^-_{a,x^*}\,} x^*\otimes a \xrightarrow{\,1\otimes \omega\,} x^*\otimes a^* \cong (a\otimes x)^*=\Phi(x)^*.
\]
We need to check that
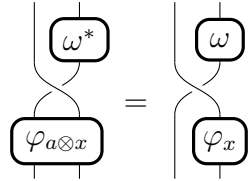
\[
\begin{tikzpicture}[baseline=-.1cm]
	\draw (-.6,-1) -- (-.6,1.35) (0,1.35) -- (0,-1);
	\braidingInverse{(-.6,.2)}{.6}{.35}
	\roundNbox{unshaded}{(-.3,-.5)}{.3}{.3}{.3}{$\varphi_{a\otimes x}$}
	\roundNbox{unshaded}{(0,.85)}{.28}{.1}{.1}{$\omega^*$}
\end{tikzpicture}
\,=\,\,
\begin{tikzpicture}[baseline=-.1cm]
	\draw (-.6,-1) -- (-.6,1.35) (0,1.35) -- (0,-1);
	\braidingInverse{(-.6,.2)}{.6}{.35}
	\roundNbox{unshaded}{(0,-.5)}{.3}{.05}{.05}{$\varphi_x$}
	\roundNbox{unshaded}{(0,.85)}{.27}{.04}{.04}{$\omega$}
\end{tikzpicture}
\]
holds.
By the monoidal property of $\varphi$, this is equivalent to the equation $\varphi_a=(\omega^*)^{-1}\circ\omega$, which is itself a consequence of $\theta_a=1$ \cite[Lem. 1.13]{MR1936496}.
This finishes the proof that $\Phi^{\scriptscriptstyle \cZ}$ is a pivotal functor.


\subsection{Braided pivotal categories}
\label{sec:Lemmas}

In this second appendix, we provide an overview of some basic properties of braided pivotal categories, since these have not received as much attention as many related notions.
Some of these results can be found in the literature \cite[\S 4.6 and 4.7]{MR2767048} \cite[Prop. 2.11]{MR1187296} (see also \cite[\S 2.2]{MR1797619} and \cite[\S 2.8.2]{MR2609644}).
We provide them here along with short proofs for the convenience of the reader.

\begin{lem}
\label{lem:RigidBalancedPivotal}
Let $\cC$ be balanced rigid category.
Then there are two ways of identifying each object with its double dual,
each one making $\cC$ into a pivotal category.
\end{lem}
\begin{proof}
For $a\in \cC$, define $\varphi_a^{\scriptscriptstyle (1)}: a\to a^{**}$ by 
\begin{equation}\label{eq: varphi_a^(1)}
\varphi_a^{\scriptscriptstyle (1)}
=
(\ev_{a}\otimes \id_{a^{**}})
\circ 
(\id_{a^*}\otimes \beta^-_{a^{**},a})
\circ
(\coev_{a^*}\otimes \id_a)
\circ 
\theta_a
=
\begin{tikzpicture}[baseline=-.1cm]
	\draw (0,-1) -- (0,1);
	\roundNbox{unshaded}{(0,-.35)}{.35}{.02}{.02}{$\theta_a$}	
	\loopIsoReverse{(0,.4)}
	\node at (-.2,-.9) {\scriptsize{$a$}};
	\node at (-.3,.9) {\scriptsize{$a^{**}$}};
\end{tikzpicture}\,.
\end{equation}
This is clearly invertible with inverse
$
\theta_a^{-1}
\circ
(\id_a \otimes \ev_{a^*})
\circ
(\beta_{a^{**},a}\otimes \id_{a^*})
\circ
(\id_{a^{**}}\otimes \coev_a)
$.
The naturality of $\varphi_a^{\scriptscriptstyle (1)}$ is left as an exercise, and it is straightforward to show
$\varphi_{a\otimes b}^{\scriptscriptstyle (1)}=\varphi_a^{\scriptscriptstyle (1)}\otimes \varphi_b^{\scriptscriptstyle (1)}$ using that $\theta_{a\otimes b} = \beta_{b,a}\circ \beta_{a,b}\circ (\theta_a\otimes \theta_b)$.

The other pivotal structure is given by
\[
\varphi_a^{\scriptscriptstyle (2)}
=
(\ev_{a}\otimes \id_{a^{**}})
\circ 
(\id_{a^*}\otimes \beta^+_{a^{**},a})
\circ
(\coev_{a^*}\otimes \id_a)
\circ 
\theta_a^{-1}
=
\begin{tikzpicture}[baseline=-.1cm]
	\draw (0,-1) -- (0,1);
	\roundNbox{unshaded}{(0,-.35)}{.35}{.05}{.05}{$\theta^{\scriptscriptstyle-1}_a$}	
	\loopIso{(0,.4)}
	\node at (-.2,-.9) {\scriptsize{$a$}};
	\node at (-.3,.9) {\scriptsize{$a^{**}$}};
\end{tikzpicture}\,.\qedhere
\]
\end{proof}

\noindent
The two pivotal structures are related by
\begin{equation}\label{two pivotal structures}
\varphi^{\scriptscriptstyle (2)}_a
\,=\,
\begin{tikzpicture}[baseline=-.1cm]
	\draw (0,-1.6) -- (0,1.6);
	\loopIso{(0,1)}
	\loopIsoReverse{(0,-1)}
	\roundNbox{unshaded}{(.1,0)}{.35}{.4}{.4}{\,$(\varphi^{\scriptscriptstyle (1)}_a)^{\scriptscriptstyle -1}$}	
	\node at (-.2,.6) {\scriptsize{$a$}};
	\node at (-.2,-1.4) {\scriptsize{$a$}};
	\node at (-.3,-.6) {\scriptsize{$a^{**}$}};
	\node at (-.3,1.4) {\scriptsize{$a^{**}$}};
\end{tikzpicture}
\,=\,
\begin{tikzpicture}[baseline=-.1cm]
	\draw (0,-1.6) -- (0,1.6);
	\loopIsoReverse{(0,1)}
	\loopIso{(0,-1)}
	\roundNbox{unshaded}{(.1,0)}{.35}{.4}{.4}{\,$(\varphi^{\scriptscriptstyle (1)}_a)^{\scriptscriptstyle -1}$}	
	\node at (-.2,.6) {\scriptsize{$a$}};
	\node at (-.2,-1.4) {\scriptsize{$a$}};
	\node at (-.3,-.6) {\scriptsize{$a^{**}$}};
	\node at (-.3,1.4) {\scriptsize{$a^{**}$}};
\end{tikzpicture}
\qquad\text{and}\qquad
\varphi^{\scriptscriptstyle (1)}_a
\,=\,
\begin{tikzpicture}[baseline=-.1cm]
	\draw (0,-1.6) -- (0,1.6);
	\loopIso{(0,1)}
	\loopIsoReverse{(0,-1)}
	\roundNbox{unshaded}{(.2,0)}{.35}{.4}{.4}{\,$(\varphi^{\scriptscriptstyle (2)}_a)^{\scriptscriptstyle -1}$}	
	\node at (-.2,.6) {\scriptsize{$a$}};
	\node at (-.2,-1.4) {\scriptsize{$a$}};
	\node at (-.3,-.6) {\scriptsize{$a^{**}$}};
	\node at (-.3,1.4) {\scriptsize{$a^{**}$}};
\end{tikzpicture}
\,=\,
\begin{tikzpicture}[baseline=-.1cm]
	\draw (0,-1.6) -- (0,1.6);
	\loopIsoReverse{(0,1)}
	\loopIso{(0,-1)}
	\roundNbox{unshaded}{(.2,0)}{.35}{.4}{.4}{\,$(\varphi^{\scriptscriptstyle (2)}_a)^{\scriptscriptstyle -1}$}	
	\node at (-.2,.6) {\scriptsize{$a$}};
	\node at (-.2,-1.4) {\scriptsize{$a$}};
	\node at (-.3,-.6) {\scriptsize{$a^{**}$}};
	\node at (-.3,1.4) {\scriptsize{$a^{**}$}};
\end{tikzpicture}\,.
\end{equation}

\begin{lem}
\label{lem:CenterPivotalBalanced}
A braided pivotal category has two sets of twists, each one making it a balanced category.
\end{lem}
\begin{proof}
For $a\in \cC$, let $\varphi_a : a \to a^{**}$ be the natural isomorphism, and define
\begin{equation}\label{def: theta1}
\theta_{a}^{\scriptscriptstyle (1)}
= 
(\id_a\otimes \ev_{a^*})
\circ
(\beta_{a^{**},a}\otimes\id_{a^*})
\circ
(\id_{a^{**}}\otimes \coev_a)
\circ
\varphi_a
=
\begin{tikzpicture}[rotate=180, baseline=.35cm]
	\draw (0,-1.6) -- (0,.8);
	\loopIso{(0,-1)}
	\roundNbox{unshaded}{(0,0)}{.35}{0}{0}{$\varphi_a$}	
	\node at (-.2+.4,.6) {\scriptsize{$a$}};
	\node at (-.2+.4,-1.4) {\scriptsize{$a$}};
	\node at (-.3+.55,-.55) {\scriptsize{$a^{**}$}};
\end{tikzpicture}\,.
\end{equation}
The other alternative is
\begin{equation}\label{def: theta2}
\theta_{a}^{\scriptscriptstyle (2)} 
= 
\varphi_a^{-1} 
\circ
(\ev_a\otimes \id_{a^{**}})
\circ
(\id_{a^*}\otimes \beta_{a^{**},a})
\circ
(\coev_{a^*}\otimes \id_a)
=
\begin{tikzpicture}[baseline=-.5cm]
	\draw (0,-1.6) -- (0,.8);
	\loopIso{(0,-1)}
	\roundNbox{unshaded}{(0,0)}{.4}{0}{0}{$\varphi_a^{\scriptscriptstyle-1}$}	
	\node at (-.2,.6) {\scriptsize{$a$}};
	\node at (-.2,-1.4) {\scriptsize{$a$}};
	\node at (-.3,-.6) {\scriptsize{$a^{**}$}};
\end{tikzpicture}\,.
\end{equation}
The naturality of $\theta^{\scriptscriptstyle (1)}$ and $\theta^{\scriptscriptstyle (2)}$ follows from that of $\varphi$.
Using the monoidal property of $\varphi$, respectively $\varphi^{-1}$, one verifies that the balance axiom holds between these twist maps and the braiding.
\end{proof}

\noindent
The two twists are related by
\begin{equation}
\theta^{\scriptscriptstyle (2)}_a
\,=\,
\begin{tikzpicture}[baseline=-.1cm]
	\draw (.6,-1) -- (.6,.4) arc (0:180:.3cm) -- (0,-.4) arc (0:-180:.3cm) -- (-.6,1);
	\roundNbox{unshaded}{(0,0)}{.4}{0}{0}{$\theta^{\scriptscriptstyle (1)}_{a^*}$}	
\end{tikzpicture}
\,=\,
\begin{tikzpicture}[baseline=-.1cm]\pgftransformxscale{-1}
	\draw (.6,-1) -- (.6,.4) arc (0:180:.3cm) -- (0,-.4) arc (0:-180:.3cm) -- (-.6,1);
	\roundNbox{unshaded}{(0,0)}{.4}{0}{0}{$\theta^{\scriptscriptstyle (1)}_{{}^*\hspace{-.2mm}a}$}	
\end{tikzpicture}
\qquad\text{and}\qquad
\theta^{\scriptscriptstyle (1)}_a
\,=\,
\begin{tikzpicture}[baseline=-.1cm]
	\draw (.6,-1) -- (.6,.4) arc (0:180:.3cm) -- (0,-.4) arc (0:-180:.3cm) -- (-.6,1);
	\roundNbox{unshaded}{(0,0)}{.4}{0}{0}{$\theta^{\scriptscriptstyle (2)}_{a^*}$}	
\end{tikzpicture}
\,=\,
\begin{tikzpicture}[baseline=-.1cm]\pgftransformxscale{-1}
	\draw (.6,-1) -- (.6,.4) arc (0:180:.3cm) -- (0,-.4) arc (0:-180:.3cm) -- (-.6,1);
	\roundNbox{unshaded}{(0,0)}{.4}{0}{0}{$\theta^{\scriptscriptstyle (2)}_{{}^*\hspace{-.2mm}a}$}	
\end{tikzpicture}\,.
\label{eq:RelateTwoTwists}
\end{equation}

Lemmas \ref{lem:RigidBalancedPivotal} and \ref{lem:CenterPivotalBalanced} now combine to give the following corollary:

\begin{cor}\label{cor: pivotal <--> balanced}
Let $\cC$ be a braided rigid category. Then there are two ways\footnote
{In \cite[Rem. 4.22]{MR2767048}, it is incorrectly stated that there are $\Z$ many ways of establishing such a correspondence.
In fact, the construction only produces $\Z/2\Z$ many distinct correspondences.}
 of establishing a one-to-one correspondence between pivotal structures on $\cC$,
and twists making $\cC$ into a balanced category (see Figure~\ref{fig:SynopticChart} in Section~\ref{sec: chart of categories}).
\end{cor}


\begin{prop}\label{prop: A braided pivotal category is spherical if and only if it is ribbon.}
Let $(\cC,\beta,\varphi^{\scriptscriptstyle (1)})$ be a braided pivotal category.
Using $\varphi^{\scriptscriptstyle (1)}$, let us define $\varphi^{\scriptscriptstyle (2)}$, $\theta^{\scriptscriptstyle (1)}$, $\theta^{\scriptscriptstyle (2)}$
by means of Equations \eqref{two pivotal structures}, \eqref{def: theta1}, \eqref{def: theta2}, respectively.
Then the following three properties are equvalent
\begin{enumerate}[$\quad$(1)]
\item $\theta^{\scriptscriptstyle (1)} = \theta^{\scriptscriptstyle (2)}$
\item $(C,\beta,\theta^{\scriptscriptstyle (i)})$ is ribbon for either $i=1,2$
\item $\varphi^{\scriptscriptstyle (1)} = \varphi^{\scriptscriptstyle (2)}$
\end{enumerate}
and imply this fourth one:
\begin{enumerate}[$\quad$(1)]
\setcounter{enumi}{3}
\item $(C,\beta,\varphi^{\scriptscriptstyle (i)})$ is spherical for either $i=1,2$.
\end{enumerate}
If moreover $\cC$ is semisimple, then the fourth property is equivalent to the first three.
\end{prop}

\begin{proof}
\mbox{}

\noindent
\underline{$(1)\Leftrightarrow (2)$:}
By definition, $(\cC,\beta,\theta^{\scriptscriptstyle(1)})$ is ribbon if and only if $(\theta_a^{\scriptscriptstyle(1)})^* = \theta_{a^*}^{\scriptscriptstyle(1)}$ holds for every $a\in\cC$.
By Equation \eqref{eq:RelateTwoTwists}, we have $(\theta_{a}^{\scriptscriptstyle(1)})^*=\theta^{\scriptscriptstyle(2)}_{a^*}$.
Therefore $(\cC,\beta,\theta^{\scriptscriptstyle(1)})$ is ribbon 
iff $\theta_{a^*}^{\scriptscriptstyle(1)}= \theta_{a^*}^{\scriptscriptstyle(2)}$ holds $\forall a$
iff
$\theta_{a}^{\scriptscriptstyle(1)}= \theta_{a}^{\scriptscriptstyle(2)}$ holds $\forall a$.
The same argument works with $(\cC,\beta,\theta^{\scriptscriptstyle(2)})$ in place of $(\cC,\beta,\theta^{\scriptscriptstyle(1)})$.

\noindent
\underline{$(2)\Leftrightarrow (3)$:}
Letting $\theta:=\theta^{\scriptscriptstyle(1)}$, then for every $a\in \cC$, we have
$$
\begin{tikzpicture}[baseline=-.1cm]
\pgftransformxscale{-1}
	\draw (-.2,-.6) -- (-.2,.3) arc (180:0:.35cm) -- (.5,-.6);
	\roundNbox{unshaded}{(.5,0)}{.32}{.05}{.05}{$\theta_a^*$}
	\node[below] at (-.2,-.6) {$\scriptstyle a$}; 
	\node[below] at (.5,-.53) {$\scriptstyle a^*$}; 
\end{tikzpicture}
\!=\,
\begin{tikzpicture}[baseline=-.1cm]
	\draw (-.2,-.6) -- (-.2,.3) arc (180:0:.35cm) -- (.5,-.6);
	\roundNbox{unshaded}{(.5,0)}{.32}{.05}{.05}{$\theta_a$}
\end{tikzpicture}
\,=\,
\begin{tikzpicture}[baseline=.4cm]
	\draw (0,-.6) -- (0,1.5) arc (180:0:.3) -- (.6,-.6);
	\roundNbox{unshaded}{(0,.5)}{.32}{.05}{.05}{$\theta_a$}
	\braidingInverse{(0,-.3)}{.6}{.4}
	\loopIsoReverse{(0,1.2)}
\end{tikzpicture}
\,=\,
\begin{tikzpicture}[baseline=.4cm]
	\draw (0,-.6) -- (0,1) arc (180:0:.3) -- (.6,-.6);
	\roundNbox{unshaded}{(0,.5)}{.32}{.1}{.1}{$\varphi^{\scriptscriptstyle(1)}_a$}
	\braidingInverse{(0,-.3)}{.6}{.4}
\end{tikzpicture}
$$
and 
$$
\,\,\,\,\begin{tikzpicture}[baseline=-.1cm]
\pgftransformxscale{-1}
	\draw (-.2,-.6) -- (-.2,.3) arc (180:0:.35cm) -- (.5,-.6);
	\roundNbox{unshaded}{(.5,0)}{.32}{.05}{.05}{$\theta_{a^*}$}
	\node[below] at (-.2,-.6) {$\scriptstyle a$}; 
	\node[below] at (.5,-.53) {$\scriptstyle a^*$}; 
\end{tikzpicture}
=
\begin{tikzpicture}[baseline=-.1cm]
	\draw (0,-1.2) -- (0,.95) arc (180:0:.35cm) -- (.7,-1.2);
	\loopIsoInverse{(0,0)}
	\loopIso{(0,-.6)}
	\roundNbox{unshaded}{(0,.63)}{.32}{.05}{.05}{$\theta_{a^*}$}
\end{tikzpicture}
\,=
\begin{tikzpicture}[baseline=-.1cm]
	\draw (0,-1.2) -- (0,.9) arc (180:0:.4cm) -- (.8,-1.2);
	\loopIso{(0,-.5)}
	\roundNbox{unshaded}{(0,.5)}{.32}{.37}{.37}{$(\varphi_{a^*}^{\scriptscriptstyle(2)})^{\scriptscriptstyle-1}$}	
\end{tikzpicture}
\,=\,
\begin{tikzpicture}[baseline=.4cm]
	\draw (0,-.6) -- (0,1) arc (180:0:.3) -- (.6,-.6);
	\roundNbox{unshaded}{(0,.5)}{.32}{.1}{.1}{$\varphi^{\scriptscriptstyle(2)}_a$}
	\braidingInverse{(0,-.3)}{.6}{.4}
\end{tikzpicture}\,.
$$
It follows that $(\theta_a^{\scriptscriptstyle(1)})^*=\theta_{a^*}^{\scriptscriptstyle(1)}$ holds if and only if $\varphi^{\scriptscriptstyle(1)}_a=\varphi^{\scriptscriptstyle(2)}_a$ holds.
Using Equations \eqref{eq:RelateTwoTwists},
the former is also easily seen to be equivalent to $(\theta_a^{\scriptscriptstyle(2)})^*=\theta_{a^*}^{\scriptscriptstyle(2)}$.

\noindent
\underline{$(3)\Rightarrow (4)$:}
Assuming $\varphi:=\varphi^{\scriptscriptstyle(1)}=\varphi^{\scriptscriptstyle(2)}$, we show that $(\cC,\beta,\varphi)$ is spherical:
$$
\tr_R(f)
= 
\begin{tikzpicture}[baseline=-.1cm]
	\draw (0,-.9) -- (0,.9) arc (180:0:.3) -- (.6,-.9) arc (0:-180:.3);
	\roundNbox{unshaded}{(0,-.5)}{.32}{0}{0}{$f$}
	\roundNbox{unshaded}{(0,.5)}{.32}{0}{0}{$\varphi$}
\end{tikzpicture}
\,=\, 
\begin{tikzpicture}[baseline=-.1cm]
	\draw (0,-.6) -- (0,.6) arc (180:0:.45) -- (.9,-.6) arc (0:-180:.45);
	\roundNbox{unshaded}{(0,0)}{.32}{0}{0}{$f$}
	\roundNbox{unshaded}{(.9,0)}{.32}{.1}{.1}{$\varphi^{\scriptscriptstyle-1}$}
\end{tikzpicture}
\,=\, 
\begin{tikzpicture}[baseline=-.1cm]
	\draw (0,-1) -- (0,1) arc (180:0:.4) -- (.8,-1) arc (0:-180:.4);
	\roundNbox{unshaded}{(0,0)}{.32}{0}{0}{$f$}
	\roundNbox{unshaded}{(.8,0)}{.32}{0}{0}{$\varphi$}
	\loopIsoInverseReverse{(.8,-.75)}
	\loopIsoInverse{(.8,.75)}
\end{tikzpicture}
=\,
\begin{tikzpicture}[baseline=-.1cm]
	\draw (0,-.6) -- (0,.6) arc (180:0:.4) -- (.8,-.6) arc (0:-180:.4);
	\roundNbox{unshaded}{(.8,0)}{.32}{0}{0}{$f$}
	\roundNbox{unshaded}{(0,0)}{.32}{0}{0}{$\varphi$}
\end{tikzpicture}
= 
\begin{tikzpicture}[baseline=-.1cm, rotate = 180]
	\draw (0,-.9) -- (0,.9) arc (180:0:.3) -- (.6,-.9) arc (0:-180:.3);
	\roundNbox{unshaded}{(0,-.5)}{.32}{0}{0}{$f$}
	\roundNbox{unshaded}{(0,.5)}{.32}{.1}{.1}{$\varphi^{\scriptscriptstyle-1}$}
\end{tikzpicture}
=
\tr_L(f).
$$
We have used Equation \eqref{two pivotal structures} for the third equality above.

Let now $\cC$ be semisimple.

\noindent
\underline{$(4)\Rightarrow (2)$:}
We again write $\varphi:=\varphi^{\scriptscriptstyle(1)}$, and assume that $(\cC,\beta,\varphi)$ is spherical.
By naturality of the twist, it is enough to verify $\theta^*_a=\theta_{a^*}$ on simple objects.
Let $\vartheta^*_a$ and $\vartheta_{a^*}$ be the scalars defined by $\theta^*_a=\vartheta^*_a\id_{a^*}$ and $\theta_{a^*}=\vartheta_{a^*}\id_{a^*}$.
By Lemma \ref{lem: semisimple + rigid ==> q-dims neq 0}, the quantum dimension $\dim(a):=\tr_L(1_a)$ of a simple object is always non-zero.
So we have $\theta^*_a=\theta_{a^*}$ if and only if $\tr_L(\theta^*_a)=\dim(a^*)\vartheta^*_a$ equals $\tr_L(\theta_{a^*})=\dim(a^*)\vartheta_{a^*}$:
\begin{align*}
\tr_L(\theta^*_a)
&=\,
\begin{tikzpicture}[baseline=-.1cm]
	\draw (0,-.9) -- (0,.85) arc (180:0:.25cm) -- (.5,.15) arc (-180:0:.25cm) -- (1,.85) arc (0:180:.8cm) -- (-.6,-.9) arc (-180:0:.3cm);
	\roundNbox{unshaded}{(.5,.5)}{.3}{0}{0}{$\theta$}
	\roundNbox{unshaded}{(0,-.6)}{.32}{.1}{.1}{$\varphi^{\scriptscriptstyle-1}$}
\end{tikzpicture}
\,=\,
\begin{tikzpicture}[baseline=-.1cm]
	\draw (0,-.9) -- (0,.8) arc (0:180:.3cm) -- (-.6,-.5) arc (0:-180:.15cm) -- (-.9,.8) arc (180:0:.75cm) -- (.6,-.9) arc (0:-180:.3);
	\roundNbox{unshaded}{(0,-.5)}{.3}{0}{0}{$\theta$}
	\roundNbox{unshaded}{(0,.5)}{.3}{0}{0}{$\varphi$}
\end{tikzpicture}
=
\begin{tikzpicture}[baseline=-.1cm]
	\draw (0,-.9) -- (0,.9) arc (180:0:.3) -- (.6,-.9) arc (0:-180:.3);
	\roundNbox{unshaded}{(0,-.5)}{.3}{0}{0}{$\theta$}
	\roundNbox{unshaded}{(0,.5)}{.3}{0}{0}{$\varphi$}
\end{tikzpicture}
=
\tr_R(\theta_{a^*})
=
\tr_L(\theta_{a^*}).
\qedhere
\end{align*}

\end{proof}

\begin{lem}\label{lem: semisimple + rigid ==> q-dims neq 0}
Let $\cC$ be a rigid semisimple tensor category (linear over some field $k$).
Then for every simple object $a$ and non-zero morphisms $c:1\to a\otimes a^*$ and $e:a\otimes a^*\to 1$, we have $e\circ c\not =0$.
\end{lem}

\begin{proof}
By the adjunction $\cC(1,a\otimes a^*)\cong \cC(a,a)\cong k$, any morphism $1\to a\otimes a^*$ is a multiple of $c$.
Decompose $a\otimes a^*$ as $u\oplus x$ with $u$ the image of $c$,
and decompose $1$ as $1'\oplus 1''$ with $1'$ the coimage of $c$.
We have $\cC(1',x)=\cC(1'',x)=\cC(1'',u)=0$, $\cC(1',u)=k$, and $c$ restricts to an isomorphism $1'\to u$.

By semisimplicity, we also have $\cC(x,1')=\cC(x,1'')=\cC(u,1'')=0$, $\cC(u,1')=k$,
and $e$ restricts to an isomorphism $u\to 1'$.
The composite $e\circ c$ is now visibly non-zero.
\end{proof}

In Lemma \ref{lem:CenterPivotalBalanced}, we have discussed the two ways of equipping a braided pivotal category with a balanced structure.
We provide the corresponding result at the level of functors:

\begin{lem}\label{lem: it is a balanced functor}
Let $\cC$, $\cD$ be braided pivotal categories, equipped with the balanced structure~\eqref{def: theta1}.
Then a braided pivotal functor $F:\cC\to \cD$ is automatically balanced.
The same result holds using the other balanced structure~\eqref{def: theta2}.
\end{lem}

\begin{proof}
Recall from Section~\ref{sec:  Tensor categories} that the isomorphism $\delta_a:F(a^*)\to F(a)^*$ is defined by requiring that $$F(\ev_a) = i \circ \ev_{F(a)}\circ(\delta_a\otimes \id_{F(a)})\circ \nu_{a^*,a}^{-1}.$$
It follows that $\ev_{F(a)}=i^{-1} \circ F(\ev_a)\circ \nu_{a^*,a}\circ(\delta_a^{-1}\otimes \id_{F(a)})$.
The map $\coev_{F(a)}$ is characterised by the zig-zag equation $(\id\otimes\ev_{F(a)})\circ(\coev_{F(a)}\otimes\id) = \id_{F(a)}$.
As the morphism $(\id_{F(a)}\otimes \delta_a) \circ\nu_{a,a^*}^{-1} \circ F(\coev_a)\circ i$ satisfies that equation, it must be equal to $\coev_{F(a)}$.
So we get that $$F(\coev_a) = \nu_{a,a^*} \circ(\id_{F(a)}\otimes\, \delta_a^{-1}) \circ\coev_{F(a)}\circ\, i^{-1}.$$
We can now prove that $F:(\cC,\theta^{\scriptscriptstyle (1)})\to (\cD, \theta^{\scriptscriptstyle (1)})$ is a balanced functor:
\[
F(\theta^{\scriptscriptstyle (1)}_a)
\,=\,
F\Big(\!
	\begin{tikzpicture}[baseline=.35cm]
	\draw (0,1.6) -- (0,-.8);
	\loopIsoInverseReverse{(0,1)}
	\roundNboxSize{unshaded}{(0,0)}{.32}{0}{0}{$\varphi_a$}{.9}	
	\node at (-.2,-.6) {\scriptsize{$a$}};
	\node at (-.2,1.4) {\scriptsize{$a$}};
	\node at (-.25,.55) {\scriptsize{$a^{**}$}};
\end{tikzpicture}
\,\,\Big)
=\,
\begin{tikzpicture}[baseline=.9cm, yscale=1.5]
\draw (0,2.9) -- (0,-1.5);
\draw (1.1,0) -- (1.1,2) (.3,1) to[in=90,out=-90] (.8,0) (.3,1) to[in=-90,out=90] (.8,2);
\draw[densely dotted] (1,0) -- +(0,-.5);
\draw[densely dotted] (1,2) -- +(0,.5);
	\roundNboxSize{unshaded}{(.2,1)}{.2}{.4}{.4}{$\scriptstyle F(\beta)$}{1.1}
	\roundNboxSize{unshaded}{(0,-1)}{.2}{.32}{.32}{$\scriptstyle F(\varphi_a)$}{1.1}
	\roundNboxSize{unshaded}{(1,0)}{.2}{.45}{.45}{$\scriptstyle F(\coev)$}{1.1}	
	\roundNboxSize{unshaded}{(1,2)}{.2}{.45}{.45}{$\scriptstyle F(\ev)$}{1.1}
\def \shi {.05}	
	\draw[very thin, fill=white, rounded corners=1.5] (-.6-\shi,-.4) rectangle (1.95-\shi,-.6);
	\draw[very thin, fill=white, rounded corners=1.5] (-.6-\shi,.6) rectangle (1.95-\shi,.4);
	\draw[very thin, fill=white, rounded corners=1.5] (-.6-\shi,1.6) rectangle (1.95-\shi,1.4);
	\draw[very thin, fill=white, rounded corners=1.5] (-.6-\shi,2.6) rectangle (1.95-\shi,2.4);
	\node[scale=.8] at (.675-\shi,-.5) {$\scriptstyle \id\otimes\, i$};
	\node[scale=.8] at (.675-\shi,.5) {$\scriptstyle (\nu\;\!\otimes\;\!\id)\circ (\id\otimes\;\! \nu^{-1})$};
	\node[scale=.8] at (.675-\shi,1.5) {$\scriptstyle (\id\otimes\;\! \nu)\circ(\nu^{-1}\otimes\;\!\id)$};
	\node[scale=.8] at (.675-\shi,2.5) {$\scriptstyle \id\otimes\, i^{-1}$};
\end{tikzpicture}
\,\,=\,
\begin{tikzpicture}[baseline=-.95cm, scale=1.05]
\draw (0,1.5) -- (0,-3.1);
\draw (.6,-.3) -- +(0,-.7) arc(-180:0:.3) -- +(0,2) arc(0:180:.3);
	\braiding{(0,0)}{.6}{.4}

	\roundNboxSize{unshaded}{(.6,.7)}{.3}{.05}{.05}{$\scriptstyle \delta_{a^*}$}{1.1}
	\roundNboxSize{unshaded}{(0,-.7)}{.3}{.05}{.05}{$\scriptstyle \delta_{a^*}^{-1}$}{1.1}	
	\roundNboxSize{unshaded}{(0,-1.6)}{.3}{.05}{.05}{$\scriptstyle \delta_a^*$}{1.1}	
	\roundNboxSize{unshaded}{(0,-2.5)}{.32}{.15}{.15}{$\scriptstyle \varphi_{F(a)}$}{1.1}
	\roundNboxSize{unshaded}{(1.2,-.7)}{.3}{.05}{.05}{$\scriptstyle \delta_a^{-1}$}{1.1}
\end{tikzpicture}
\,\,=\!\!\!
\begin{tikzpicture}[baseline=-.7cm, scale=1.05]
\draw (0,1.3) -- (0,-2.2);
\draw (.6,-.3) -- +(0,-.5) arc(-180:0:.3) -- +(0,1.6) arc(0:180:.3)  -- +(0,-.4);
	\braiding{(0,0)}{.6}{.4}

	\roundNboxSize{unshaded}{(1.2,.45)}{.3}{.05}{.05}{$\scriptstyle \delta_{a}$}{1.1}
	\roundNboxSize{unshaded}{(1.2,-.45)}{.3}{.05}{.05}{$\scriptstyle \delta_a^{-1}$}	{1.1}
	\roundNboxSize{unshaded}{(0,-1.5)}{.32}{.17}{.17}{$\scriptstyle \varphi_{F(a)}$}{1.1}
\end{tikzpicture}
\,\,=\,
\begin{tikzpicture}[baseline=.35cm, scale=1.05]
	\draw (0,1.6) -- (0,-.8);
	\loopIsoInverseReverse{(0,1)}
	\roundNbox{unshaded}{(0,0)}{.3}{.17}{.17}{$\scriptstyle \varphi_{F(a)}$}	
\end{tikzpicture}
\,=\,\,
\theta^{\scriptscriptstyle (1)}_{F(a)}
\]
The proof that $F:(\cC,\theta^{\scriptscriptstyle (2)})\to (\cD, \theta^{\scriptscriptstyle (2)})$ is a balanced functor is similar.
\end{proof}


\bibliographystyle{amsalpha}
{\footnotesize{
\bibliography{../../../bibliography}
}}
\end{document}